\documentclass[a4paper,twoside,bibtotoc,idxtotoc]{scrbook} 
\usepackage{amsmath,amssymb,amsfonts,amsthm} 
\usepackage{graphics,graphicx}                 
\usepackage{color}                    
\usepackage{hyperref,fancyhdr}                 
\usepackage[all]{xy}
\usepackage[applemac]{inputenc}
\usepackage[ngerman,english]{babel}
\usepackage{url}
\usepackage[split]{splitidx}\makeindex

\newindex[Index of Notations]{n}
\newindex[Index of Terms]{idx} 

\renewcommand{\sectionmark}[1]{\markright{A Geometric Approach to Noncommutative Principal Bundles}}
\fancypagestyle{plain}{\fancyhf{}
\lhead[\fancyplain{}{\thepage}]{}
\rhead[]{\fancyplain{}{\thepage}}
}

\newtheorem{theorem}{Theorem}[section]
\newtheorem{proposition}[theorem]{Proposition}
\newtheorem{corollary}[theorem]{Corollary}
\newtheorem{lemma}[theorem]{Lemma}

\newtheorem{theorem-example}[theorem]{Theorem $\backslash$ Example}
\renewenvironment{proof}{\begin{sloppypar}\noindent{\bf Proof}}{\hfill$\blacksquare$\end{sloppypar}}
\theoremstyle{definition}
\newtheorem{definition}[theorem]{Definition}
\newtheorem{example}[theorem]{Example}
\newtheorem{beispiel}[theorem]{Beispiel}
\newtheorem{bemerkung}[theorem]{Bemerkung}
\newtheorem{remark}[theorem]{Remark}
\newtheorem{convention}[theorem]{Convention}
\newtheorem{open problem}[theorem]{Open Problem}
\newtheorem{construction}[theorem]{Construction}
\newtheorem{notation}[theorem]{Notation}
\DeclareMathOperator{\Aut}{Aut}
\DeclareMathOperator{\Inn}{Inn}
\DeclareMathOperator{\Idem}{Idem}
\DeclareMathOperator{\im}{im}
\DeclareMathOperator{\id}{id}
\DeclareMathOperator{\ev}{ev}
\DeclareMathOperator{\GL}{GL}

\DeclareMathOperator{\M}{M}

\DeclareMathOperator{\per}{per}
\DeclareMathOperator{\Hom}{Hom}
\DeclareMathOperator{\Out}{Out}
\DeclareMathOperator{\Diff}{Diff}
\DeclareMathOperator{\Gau}{Gau}
\DeclareMathOperator{\End}{End}

\DeclareMathOperator{\SU}{SU}

\DeclareMathOperator{\spec}{spec}

\DeclareMathOperator{\Ad}{Ad}
\DeclareMathOperator{\ad}{ad}
\DeclareMathOperator{\pr}{pr}

\DeclareMathOperator{\sgn}{sgn}

\DeclareMathOperator{\Alt}{Alt}
\DeclareMathOperator{\Sym}{Sym}
\DeclareMathOperator{\Iso}{Iso}

\DeclareMathOperator{\rank}{rank}

\DeclareMathOperator{\Ext}{Ext}
\DeclareMathOperator{\Fr}{Fr}
\DeclareMathOperator{\Mult}{Mult}
\DeclareMathOperator{\der}{der}
\DeclareMathOperator{\Det}{Det}
\DeclareMathOperator{\supp}{supp}

\DeclareMathOperator{\Prim}{Prim}
\DeclareMathOperator{\Lin}{Lin}
\DeclareMathOperator{\Span}{span}
\DeclareMathOperator{\ver}{ver}
\DeclareMathOperator{\Bun}{Bun}
\DeclareMathOperator{\Vbun}{Vbun}

\date{06.05.2011}
\title{\bf{A Geometric Approach to Noncommutative Principal Bundles}}
\author{Stefan Wagner}

\begin{document}
\thispagestyle{empty}
\begin{center}
\Huge{\textbf{A Geometric Approach to Noncommutative Principal Bundles}}\\
$ $\\
$ $\\
\Huge{\textbf{Ein Geometrischer Zugang zu Nichtkommutativen Hauptfaserb\"undeln}}\\
$ $\\
$ $\\
\LARGE{Der Naturwissenschaftlichen Fakult\"at\\der Friedrich-Alexander-Universit\"at Erlangen-N\"urnberg\\zur\\Erlangung des Doktorgrades}
$ $\\
$ $\\
$ $\\
$ $\\
$ $\\
$ $\\
\LARGE{vorgelegt von\\ Dipl.-Math. Stefan Wagner}\footnote{
Department of Mathematics\\
FAU Erlangen-N\"urnberg\\
Bismarckstrasse 1 1/2, 91054 Erlangen\\ 
\url{wagners@mi.uni-erlangen.de}}

$ $\\
$ $\\
\LARGE{aus Offenbach am Main}
\end{center}
$ $\\
$ $\\

\newpage
\thispagestyle{empty}
$ $\\
$ $\\
$ $\\
$ $\\
\begin{center}
\LARGE{Als Dissertation genehmigt von der Naturwissen-\\schaftlichen Fakult\"at der Universit\"at Erlangen-N\"urnberg}
\end{center}
$ $\\
$ $\\
$ $\\
$ $\\
$ $\\
$ $\\
$ $\\
$ $\\
$ $\\
$ $\\
\begin{table}[h]
\begin{tabular}{ll}
{\LARGE Tag der m\"undlichen Pr\"ufung:} & {\LARGE 14. Juli 2011}\\
 & \\
{\LARGE Vorsitzender der } & \\ 
{\LARGE Promotionskommission:} & {\LARGE Prof. Dr. R. Fink}\\
 & \\
{\LARGE Erstberichterstatter:} & {\LARGE Prof. Dr. K.-H. Neeb}\\
 & \\
{\LARGE Zweitberichterstatter:} & {\LARGE Prof. Dr. S. Echterhoff}\\
\end{tabular}
\end{table}

\chapter*{Danksagung}

An dieser Stelle m\"ochte ich mich bei den Personen bedanken, die mich w\"ahrend meiner Zeit als Doktorand ma\ss geblich unterst\"utzt  haben.

Zuallererst m\"ochte ich meinem Doktorvater Prof. Dr. Karl-Hermann Neeb ganz herzlich f\"ur seine einzigartige und vorbildliche Betreuung sowie f\"ur seine inspirierende Art und Zusammen\-arbeit danken. Seit Beginn meines Studiums profitiere ich von seiner au\ss erordentlichen wissenschaftlichen Kompetenz und seiner praktisch st\"andigen Ansprechbarkeit. Au\ss erdem danke ich ihm daf\"ur, dass er mir beim Forschen einen gro\ss en Freiraum lie{\ss} und f\"ur die vielf\"altigen M\"oglichkeiten mich wissenschaftlich weiterzuentwickeln.


Desweiteren danke ich den Personen in meinem direkten Arbeitsumfeld, die zu einer gro\ss artigen Atmosph\"are und einem unvergleichbaren Miteinander beigetragen haben. Insbesondere danke ich Henrik Sepp\"anen f\"ur etliche Einblicke und spannende Diskussionen rund um die Mathematik und weit (teils sehr weit) dar\"uber hinaus. Ein besonders gro\ss er Dank gilt auch meinen Doktorbr\"udern Hasan G\"undo\u gan und Christoph Zellner f\"ur die gute Zusammenarbeit bei allen ($!$) auftauchenden Fragen sowie f\"ur das Korrekturlesen meiner Dissertation. Ich bewundere Hasan f\"ur seine st\"andige Hilfsbereitschaft und f\"ur seine Geduld gegen\"uber Christoph und mir. Ferner m\"ochte ich meinem Mitbewohner und Kollegen St\'ephane Merigon f\"ur die tolle gemeinsame Zeit in Erlangen danken. Erst durch ihn habe ich die Freude an einem guten Wein kennen und sch\"atzen gelernt. Gerlinde Gehring sei f\"ur ihr offenes Ohr und die Unterst\"utzung in administrativen Angelegenheiten gedankt.

Ein besonderer Dank gilt auch meinen Freunden, die mich in den vergangenen Jahren unabl\"assig unterst\"utzt haben und auch in schwierigen Zeiten immer f\"ur mich da waren. Vielen Dank Rachid und Katha, Nadiem, Phillipp (oder so), Clemens und Erika, Janni, Dimi, Bernd, Emanuel und Stephi.

Der Studienstiftung des deutschen Volkes danke ich f\"ur die Finanzierung meines Dissertationsprojekts und f\"ur viele unglaublich sch\"one Erinnerungen und Momente als Stipendiat. Die Promotionsf\"orderung durch die Studienstiftung hat mir erlaubt meine Arbeit zeitlich flexibel zu gestalten und mich voll und ganz auf meine Forschung zu konzentrieren. Gleichzeitig hat sie mir erlaubt meinen Hobbys nachzugehen und geliebten Menschen ein guter Freund zu sein, was unter anderen Umst\"anden kaum in dem selben Ma\ss e m\"oglich gewesen w\"are. Au\ss erdem m\"ochte ich meiner Vertrauensdozentin Prof. Dr. Barbara Drossel f\"ur die herzliche sowie kompetente Betreuung und so manches interessante Gespr\"ach danken.

Zu guter Letzt und von ganzem Herzen danke ich Nada f\"ur ihre endlose Geduld und Unterst\"utzung, ohne die die vergangenen Monate nicht vorstellbar gewesen w\"aren!

\newpage
\thispagestyle{empty}

\tableofcontents

\cleardoublepage


\pagenumbering{arabic}

\chapter{Introduction}

\section{Why Introducing Fibre Bundles ?}

The recognition of the domain of mathematics called fibre bundles took place in the period 1935-1940. The first general definitions were given by H. Withney. His work and that of H.\,Hopf and E. Stiefel demonstrated the importance of the subject for the applications of topology, differential geometry and theoretical physics. If $q:P\rightarrow M$ is a smooth map between manifolds, we often refer to the inverse images $q^{-1}(m)$ of points as the fibres of the map $q$. A fibre bundle is such a map $q:P\rightarrow M$ in which the fibres can all be identified with a fixed space $F$ in a systematic way. In most cases one additionally requires that the bundle looks locally like a product of an open subset $U\subseteq M$ and the fibre $F$, i.e.,
\[q^{-1}(U)\cong U\times F
\]holds locally. The theory depends on what sort of identifications we choose to allow; for example, if the fibre $F$ is a vector space, we want that the identifications are linear. Such a bundle is called a \emph{vector bundle}. A \emph{principal fibre bundle} is defined to be a fibre bundle whose fibre $F$ is a Lie group and the identifications are right equivariant. Among all fibre bundles, the principal fibre bundles and the vector bundles are of particular interest. 

One of the most popular examples for the importance of fibre bundles is general relativity, which is formulated in terms of manifolds and the curvature of vector bundles. The pioneering idea of A. Einstein was that any point and any coordinate system of the manifold should describe equal physical laws. This assumption leads to a theory which is invariant under diffeomorphisms by assumption. Thus general relativity may be viewed as a theory formulated in terms of manifolds $M$ and their tangent bundles $TM$, which has the Lie group $\Diff(M)$ as a symmetry group.

Moreover, among all field theories, \emph{classical gauge theory} is the most favoured tool used to probe elementary particles at a classical level. The theory of gauge fields and their associated fields, such as Yang--Mills--Higgs fields, was developed by physicists to explain and unify the fundamental forces of nature. The theory of connections in a principal fibre bundle was developed by mathematicians during approximately the same time period, but that they are closely related was not noticed for many years. Since then substantial progress has been made in understanding this relationship and in applying it successfully to problems in physics and mathematics. In physical applications one is usually interested in a fixed Lie group $G$ called the \emph{gauge group} (or structure group in the mathematical literature) which represents an internal or local symmetry of the field. The base manifold $M$ of the principal bundle is usually the space-time manifold or its Euclidean version, i.e., a Riemannian manifold of dimension 4. But in some physical applications, such as superspace, Kaluza--Klein and string theories, the base manifold can be an arbitrary manifold. Nevertheless, classical gauge theory is far away from describing the behaviour of elementary particles in a satisfying way. This is because a theory of elementary particles has to be a quantum theory. Nowadays the most favoured quantum theory for elementary particles is the so called quantum field theory. In view of the mathematical structure of the classical gauge theory it is quite natural to try to find a ``noncommutative fibre bundle" theory which is closely related to quantum field theory.

\section{A Short Survey of Noncommutative Geometry}

The correspondence between \emph{geometric spaces} and \emph{commutative algebras} is a familiar and basic idea of algebraic geometry. In fact, if $A$ is a complex algebra and
\[\Gamma_A:=\Hom_{\text{alg}}(A,\mathbb{C})\backslash\{0\}
\]
the corresponding set of characters of $A$, then a well-known theorem from algebraic geometry states that the functor $A\mapsto\Gamma_A$ is a (contravariant) equivalence from the category of affine complex algebras to the category of affine algebraic varieties.
The aim of \emph{Noncommutative Geometry} is to reformulate as much as possible the basic ideas of topology, measure theory and differential geometry in terms of algebraic concepts and then to generalize the corresponding classical results to the setting of noncommutative algebras. For example, if $X$ is a compact space, then it is well-known that the algebra $C(X)$ of continuous functions on $X$ carries the structure of a commutative unital $C^*$-algebra. The \emph{Noncommutative Topology}
started with the famous Gelfand--Naimark Theorems: 

\begin{theorem}\label{gelfand-naimark}{\bf(Gelfand--Naimark).}\index{Theorem!of Gelfand--Naimark}
\begin{itemize}
\item[\emph{(a)}]
Every commutative unital $C^*$-algebra is the algebra of continuous functions on a compact space $X$.
\item[\emph{(b)}]
Every $C^*$-algebra can be represented as an algebra of bounded operators on a Hilbert space.
\end{itemize}
\end{theorem}

\begin{proof}
\,\,\,For a proof of this statement we refer to [Conw90], Chapter VIII, Theorem 2.1 and Theorem 5.14.
\end{proof}

Thus, a noncommutative $C^*$-algebra can be viewed as ``the algebra of continuous functions vanishing at infinity'' on a ``quantum space''. If the space $X$  is not only compact and Hausdorff, but has some extra structure, then $X$ can sometimes be recovered from a smaller algebra. The algebra of smooth function on a manifold $M$ is a case in point. Moreover, if $G$ is a compact group, then $G$ can be recovered by the algebra $R(G,\mathbb{R})$ of real representative functions on $G$, i.e., those $\mathbb{R}$-valued functions $f$ on $G$ whose right translates 
\[f_h(g):=f(gh),\,\,\,g,h\in G,
\]generate a finite-dimensional subspace of $C(G,\mathbb{R})$. In fact, we have the following so-called Tannaka--Krein duality:

\begin{theorem}\label{tannaka-krein}{\bf(Tannaka--Krein).} 
\begin{itemize}
\item[\emph{(a)}]
If $G$ is a compact group, then the algebra $R(G,\mathbb{R})$ of real representative functions on $G$ carries the structure of a commutative Hopf algebra.
\item[\emph{(b)}]
Conversely, every real commutative Hopf algebra is the algebra of real representative functions on a compact group $G$.
\end{itemize}
\end{theorem}\index{Theorem!of Tannaka--Krein}

\begin{proof}
\,\,\,A nice reference for this duality is [GVF01], Chapter 1, Theorem 1.30 and Theorem 1.31.
\end{proof}

In particular, the noncommutative geometry of (compact) groups is encoded in the theory of Hopf algebras. The study of von Neumann algebras is sometimes referred to as \emph{Noncommutative Measure Theory}. This is due to the following theorem:

\begin{theorem}\label{von-neumann}{\bf(Commutative von Neumann Algebras).}\index{Noncommutative!Measure Theory}
\begin{itemize}
\item[\emph{(a)}]
For every $\sigma$-finite measure space $(X,\Sigma,\mu)$, the algebra $L^{\infty}(X,\Sigma,\mu)$ of essentially bounded functions on $X$ carries the structure of a commutative von Neumann algebra.
\item[\emph{(b)}]
Conversely, every commutative von Neumann algebra $\mathcal{A}$ is isomorphic to $L^{\infty}(X,\Sigma,\mu)$ for some measure space $(X,\Sigma,\mu)$.
\end{itemize}
\end{theorem}

\begin{proof}
\,\,\,A proof of this statement can be found in [Tak02], Chapter III, Theorem 1.18.
\end{proof}

A very exiting problem is to describe the noncommutative geometry of manifolds. Alain Connes, the creator of modern Noncommutative Geometry, proposed that the definition of a noncommutative manifold is a so-called spectral triple $(A,D,\mathcal{H})$, consisting of a Hilbert space $\mathcal{H}$, an algebra $A$ of operators on $\mathcal{H}$ (usually closed under taking adjoints) and a densely defined self adjoint operator $D$ satisfying $\Vert[D,a]\Vert_{\text{op}}\leq\infty$ for any $a\in A$. Starting with a closed Riemannian spin manifold $(M,g)$, the canonical associated spectral triple is given by 
\[(C^{\infty}(M),L^2(M,S(M)),D),
\]where $L^2(M,S(M))$ denotes the Hilbert space of $L^2$-sections of the spinor bundle $S(M)$ over $M$ and $D$ is the corresponding Dirac Operator (cf. [GVF01], Chapter 11). Without being too precise, we have the following theorem:

\begin{theorem}\label{connes manifold reconstruction}{\bf(Spectral Characterization of Manifolds).}\index{Spectral Characterization of Manifolds}
The algebra $A$ of a commutative unital spectral triple $(A,D,\mathcal{H})$ satisfying several additional conditions is the algebra of smooth functions on a compact manifold.
\end{theorem}

For a detailed background on spectral triples and a complete proof of this theorem we refer to the very beautiful paper [Co08] (cf. [ReVa08]). We summarize the previous discussion with the help of the following table:
$ $\\
\begin{center}
\begin{tabular}{|c|c|c|}
\hline
\begin{small}``Geometric object" \end{small} & \begin{small} ``Space of functions" \end{small} & \begin{small} Algebraic structure \end{small} \\
\hline
\begin{small} $X$ compact space\end{small} & \begin{small} $C(X)$ \end{small}& \begin{small} $C^{*}$-algebra \end{small}\\
\hline
\begin{small} $G$ compact (Lie) group\end{small} &  \begin{small} $R(G,\mathbb{R})$ \end{small} & \begin{small} Hopf algebra \end{small} \\
\hline
\begin{small} $(X,\Sigma,\mu)$ measure space \end{small} & $L^{\infty}(X,\Sigma,\mu)$ \begin{small} \end{small} & \begin{small} von Neumann algebra \end{small} \\
\hline
\begin{small} $(M,g)$ closed. Riem. spin mfd. \end{small} &  \begin{small} $\left(C^{\infty}(M),L^2(M,S(M)),D\right)$ \end{small} & \begin{small} Spectral triple \end{small} \\
\hline
\end{tabular}
\end{center}
$ $\\

Noncommutative Geometry may be seen as a main avenue from physics of the XXth century to physics of the XXIth century. The importance of this theory rests on two essential points:

i) There exist many natural objects and theories for which the classical set-theoretical tools of topology, measure theory and differential geometry lose their pertinence, but which correspond very naturally to a noncommutative algebra. For example, classical mechanics describes the motion of macroscopic objects, from billiard balls or footballs to rockets, planets, stars and galaxies. Nevertheless, this theory fails in the context of microscopic objects like atoms or photons. Indeed, classical mechanics cannot describe the line spectrum of hydrogen. At this point quantum mechanics enters the picture. The heart of quantum mechanics is its mathematical formalism, which abstractly describes physical systems with the help of vectors and operators of a Hilbert space. Thus, quantum mechanics is a noncommutative theory containing classical mechanics as an approximation for large systems. Another example is provided by the geometry of space-time. In fact, it is quite natural to use a noncommutative algebra to describe the geometry of space-time in a very economical way. For a more detailed discussion of the last statement we refer to the survey paper [Co06].

ii) In general the passage from a commutative theory to a corresponding noncommutative theory is highly non-trivial. Completely new phenomena appear in the noncommutative world. Therefore the noncommutative case leads to a better understanding of the classical phenomena and provides much more possibilities of modelling, for example to describe phenomena of quantum mechanics in a geometric setting.

\section{Noncommutative Fibre Bundles}

The question of whether there is a way to translate the geometric concept of fibre bundles to Noncommutative Geometry is quite natural and highly interesting. In the case of vector bundles the theorem of Serre and Swan gives the essential clue: 

\begin{theorem}{\bf(Serre--Swan).}\label{1}
The category of vector bundles over a compact space $X$ is equivalent to the category of finitely generated projective modules over $C(X)$.
\end{theorem}

\begin{proof}
\,\,\,A proof of this statement can be found in [GFV01], Chapter 2, Theorem 2.3.
\end{proof}

\noindent
This theorem does for vector bundles what the Gelfand-Naimark theorem does for compact spaces. Indeed, it invites us to extend the fruitful notions from the theory of vector bundles to the category of right modules over (noncommutative) algebras and vice versa (when dealing with projective modules it is convenient to consider right modules, because in this case endomorphisms act by left matrix multiplications). It is therefore reasonable to regard finitely generated projective modules over an arbitrary algebra $A$ as ``noncommutative vector bundles" with ``base space" $A$. An important remark in this context is that the theorem of Serre-Swan also holds in the smooth category, i.e., it is possible to obtain a Serre-Swan equivalence between the category of smooth vector bundles over a (not necessarily compact) manifold $M$ and finitely generated projective modules over $C^{\infty}(M)$. In fact, we give a proof of this statement in Section \ref{serre-swan smooth category}. Another hint for the quality of this definition of ``noncommutative vector bundles" is:

\begin{theorem}{\bf(Cuntz--Quillen).}\label{Cuntz-Quillen again}\index{Theorem!of Cuntz--Quillen}
A right module admits a universal connection if and only if it it is projective.
\end{theorem}

\begin{proof}
\,\,\,A sketch of this proof can be found in Chapter \ref{outlook}, Theorem \ref{Cuntz-Quillen}. 
\end{proof}

\noindent
This theorem can be viewed as a noncommutative analogue of the Narasimhan-Ramanan theorem on (commutative) universal connections on manifolds. It transfers the classical differential-geometric concept of connections on a vector bundle to ``noncommutative vector bundles" (cf. Section \ref{infinitesimal objects on noncommutative spaces}). 

The case of principal bundles is not treated in a satisfying way. From a geometrical point of view it is, so far, not really sufficiently well understood what should be a ``noncommutative principal bundle". Still, there is a well-developed purely algebraic approach using the theory of Hopf algebras. An important handicap of this approach is the ignorance of any topological and geometrical aspects. We report the basic elements of the theory of \emph{Hopf--Galois extensions} in Appendix \ref{appendix A}.

In view of the previous discussion we can extend our ``dictionary" of Noncommutative Geometry in the following way:
$ $\\
\begin{center}\label{table II}
\begin{tabular}{|c|c|c|}
\hline
\begin{small}``Geometric object" \end{small} & \begin{small} ``Space of functions" \end{small} & \begin{small} Algebraic structure \end{small} \\
\hline
\hline
\begin{small} $X$ compact space\end{small} & \begin{small} $C(X)$ \end{small}& \begin{small} $C^{*}$-algebra \end{small}\\
\hline
\begin{small} $G$ compact (Lie) group\end{small} &  \begin{small} $R(G,\mathbb{R})$ \end{small} & \begin{small} Hopf algebra \end{small} \\
\hline
\begin{small} $(X,\Sigma,\mu)$ measure space \end{small} & $L^{\infty}(X,\Sigma,\mu)$ \begin{small} \end{small} & \begin{small} von Neumann algebra \end{small} \\
\hline
\begin{small} $(M,g)$ closed. Riem. spin mfd. \end{small} &  \begin{small} $\left(C^{\infty}(M),L^2(M,SM),D\right)$ \end{small} & \begin{small} Spectral triple \end{small} \\
\hline
\begin{small} $(\mathbb{V},M,V,q)$ smooth. vect. bdl. \end{small} & $\Gamma\mathbb{V}$ \begin{small} \end{small} & \begin{small} Fin. gen. proj. module\end{small} \\
\hline
\begin{small} $(P,M,G,\sigma,q)$ prin. bdl.\end{small} & \begin{small} {\bf ?} \end{small} & \begin{small} {\bf ?} \end{small} \\
\hline
\end{tabular}
\end{center}

\section{Aim of this Thesis}

From a geometrical point of view it is, so far, not sufficiently well understood what should be a ``noncommutative principal bundle". Still, there is a well-developed abstract algebraic approach using the theory of Hopf algebras. An important handicap of this approach is the ignorance of any topological and geometrical aspects. The aim of this thesis is to develop a geometrically oriented approach to the noncommutative geometry of principal bundles based on dynamical systems and the representation theory of the corresponding transformation group. In fact, our approach starts with the following observation (cf. Proposition \ref{smoothness of the group action on the algebra of smooth functions}):

\begin{proposition}
If $\sigma:M\times G\rightarrow M$ is a smooth \emph{(}right-\emph{)} action of a Lie group $G$ on a finite-dimensional manifold $M$ \emph{(}possibly with boundary\emph{)} and $E$ is a locally convex space, then the induced \emph{(}left-\emph{)} action 
\[\alpha:G\times C^{\infty}(M,E)\rightarrow C^{\infty}(M,E),\,\,\,\alpha(g,f)(m):=(g.f)(m):=f(\sigma(m,g))
\]of $G$ on the locally convex space $C^{\infty}(M,E)$ is smooth.
\end{proposition}

The previous proposition invites us to consider (smooth) dynamical systems $(A,G,\alpha)$, i.e., triples $(A,G,\alpha)$ consisting of a (probably noncommutative) unital locally convex algebra $A$, a (Lie) topological group $G$ and a group homomorphism $\alpha:G\rightarrow\Aut(A)$, which induces a (smooth) continuous action of $G$ on $A$.

\begin{example}
Each principal bundle $(P,M,G,q,\sigma)$ (cf. Definition \ref{principal bundles I}) induces 
a smooth dynamical system $(C^{\infty}(P),G,\alpha)$, consisting of the Fr\'echet algebra of smooth functions on the total space $P$, the structure group $G$ and a group homomorphism $\alpha:G\rightarrow\Aut(C^{\infty}(P))$, induced by the smooth action $\sigma:P\times G\rightarrow P$ of $G$ on $P$. 
\end{example}

Next, given a manifold $P$ and a Lie group $G$, we note that each smooth dynamical system $(C^{\infty}(P),G,\alpha)$ induces a smooth action of $G$ on $P$ (cf. Proposition \ref{smoothness of the group action on the set of characters}):

\begin{proposition}\label{introduction II}
If $P$ is a manifold, $G$ a Lie group and $(C^{\infty}(P),G,\alpha)$ a smooth dynamical system, then the homomorphism $\alpha:G\rightarrow\Aut(C^{\infty}(P))$ induces a smooth \emph{(}right-\emph{)} action
\begin{align*}
\sigma:P\times G\rightarrow P,\,\,\,(\delta_p,g)\mapsto\delta_p\circ\alpha(g)
\end{align*}
of the Lie group $G$ on the manifold $P$. Here, we have identified $P$ with the set of characters.
\end{proposition}

\begin{remark} We recall that if the action $\sigma$ is free and proper, then we obtain a smooth principal bundle 
\[(P,P/G,G,\pr,\sigma),
\]where $\pr:P\rightarrow P/G$ denotes the canonical orbit map (cf. [To00], Kapitel VIII, Abschnitt 21).
\end{remark}

We are now going to formulate the single most important problem that this thesis deals with:
$ $\\
$ $\\
{\bf Main Problem:} \emph{Find a geometric approach to ``noncommutative principal bundles"}.
$ $\\

In order to treat this problem we will mainly be concerned with discussing the following three questions:
$ $\\
$ $\\
{\bf Question 1:} Let $P$ be a manifold and $G$ a Lie group. Do there exist natural \emph{algebraic conditions} on a smooth dynamical system $(C^{\infty}(P),G,\alpha)$ which ensure that the action $\sigma$ of Proposition \ref{introduction II} is free?

\noindent
{\bf Question 2:} If the action $\sigma$ is free and proper, what can be said about the \emph{structure} of the induced principal bundle 
\[(P,P/G,G,\pr,\sigma)?
\]

In classical differential geometry, the relation between locally and globally defined objects is important for many constructions and applications. In particular, we recall that the principal bundles in our context are automatically locally trivial. This leads to:
$ $\\
$ $\\
{\bf Question 3:} Does there exist a good \emph{localization method} in Noncommutative Geometry?
$ $\\

Our main idea for treating Question 1 is to study a (smooth) dynamical system $(A,G,\alpha)$ with the help of the representation theory of the ``structure group" $G$. In fact, given a dynamical system $(A,G,\alpha)$ we will provide conditions including representations of $G$ which ensure that the corresponding right action
\begin{align}
\sigma:\Gamma_A\times G\rightarrow\Gamma_A,\,\,\,(\chi,g)\mapsto\chi\circ\alpha(g)\label{action introduction}
\end{align}
of $G$ on the spectrum $\Gamma_A$ of $A$ is free. In particular, this ``new characterization of free group actions"  will lead in the case of a compact abelian Lie group to an answer of Question 2. To be a little bit more precise, we will see that if the isotypic components of such a dynamical system $(A,G,\alpha)$ contain invertible elements, then the action (\ref{action introduction}) becomes free. From a geometrical point of view, it will, astonishingly, turn out that this condition exactly characterizes the trivial principal $G$-bundles and therefore gives rise to a reasonable approach to so-called \emph{trivial noncommutative principal bundles} with compact abelian structure group. While in classical (commutative) differential geometry there exists up to isomorphy only one trivial principal $G$-bundle over a given manifold $M$, 
we will see that the situation completely changes in the noncommutative world. For $G=\mathbb{T}^n$, we will provide a complete classification of trivial noncommutative principal $\mathbb{T}^n$-bundles (up to completion) in terms of a suitable cohomology theory.

Inspired by localization methods from algebraic geometry, we will provide a convenient smooth localization method for non-commutative algebras (Question 3). In fact, we will show that, given a manifold $M$ and an open subset $U$ of $M$, then it is possible to reconstruct $C^{\infty}(U)$ out of data from $C^{\infty}(M)$. Moreover, we will see that this reconstruction also remains true in the context of sections of algebra bundles with possibly infinite-dimensional fibre.
Given a compact abelian Lie group $G$, the step from the trivial to the non-trivial case will then be carried out with the help of this localization method, i.e., by saying that a (smooth) dynamical system $(A,G,\alpha)$ is a \emph{non-trivial noncommutative principal $G$-bundle} if ``localization" around characters of the fixed point algebra of the action of $G$ on the center $C_A$ of $A$ are trivial noncommutative principal $G$-bundles. Finally, we will discuss various examples.

Of course, we will see that this is only a small part of the questions that arise from the goal to develop a geometrically oriented approach to the noncommutative geometry of principal bundles.

\section{Zusammenfassung}

Die nichtkommutative Geometrie von Hauptfaserb\"undeln ist bisher noch nicht hinreichend gut verstanden. Allerdings gibt es bereits einen rein algebraischen Ansatz, in dem Hopfalgebren eine wichtige Rolle spielen. Ein entscheidender Nachteil dieses Ansatzes ist jedoch die Vernachl\"assigung von topologischen und geometrischen Aspekten. Ein Hauptziel dieser Dissertation besteht nun darin, einen geometrisch fundierten Zugang zu der nichtkommutativen Geometrie von Hauptfaserb\"undeln zu finden. Unser Ansatz beginnt mit der folgenden Beobachtung (vgl. Proposition \ref{smoothness of the group action on the algebra of smooth functions}):

\begin{proposition}
Ist $\sigma:M\times G\rightarrow M$ eine glatte Rechtswirkung einer Lie--Gruppe $G$ auf einer endlichdimensionalen Mannigfaltigkeit $M$ und $E$ ein lokalkonvexer Raum, dann ist die induzierte Linkswirkung
\[\alpha:G\times C^{\infty}(M,E)\rightarrow C^{\infty}(M,E),\,\,\,\alpha(g,f)(m):=(g.f)(m):=f(\sigma(m,g))
\]von $G$ auf dem lokalkonvexen Raum $C^{\infty}(M,E)$ glatt.
\end{proposition}

Die vorangehende Proposition l\"adt uns dazu ein, (glatte) dynamische Systeme $(A,G,\alpha)$ zu betrachten, d.h., Tripel $(A,G,\alpha)$
bestehend aus einer lokalkonvexen Algebra $A$, einer (Lie--Gruppe) topologischen Gruppe $G$ und einem Gruppenhomomorphismus $\alpha:G\rightarrow\Aut(A)$, der eine (glatte) stetige Wirkung von $G$ auf $A$ definiert.

\begin{beispiel}
Jedes Hauptfaserb\"undel $(P,M,G,q,\sigma)$ (vgl. Definition \ref{principal bundles I}) induziert ein glattes dynamisches System $(C^{\infty}(P),G,\alpha)$ bestehend aus der Fr\'echet Algebra der glatten Funktionen auf $P$, der Strukturgruppe $G$ des Hauptfaserb\"undels und dem Gruppenhomomorphismus $\alpha:G\rightarrow\Aut(C^{\infty}(P))$, definiert durch die Gruppenwirkung $\sigma:P\times G\rightarrow P$ von $G$ auf $P$. 
\end{beispiel}

Desweiteren bemerken wir, dass jedes glatte dynamische System der Form $(C^{\infty}(P),G,\alpha)$ eine glatte Wirkung von $G$ auf $P$ induziert (vgl. Proposition \ref{smoothness of the group action on the set of characters}):

\begin{proposition}\label{zusammenfassung II}
Ist $P$ eine Mannigfaltigkeit, $G$ eine Lie--Gruppe und $(C^{\infty}(P),G,\alpha)$ ein glattes dynamisches System, dann induziert der Gruppenhomomorphismus $\alpha:G\rightarrow\Aut(C^{\infty}(P))$ eine glatte Rechtswirkung
\begin{align*}
\sigma:P\times G\rightarrow P,\,\,\,(\delta_p,g)\mapsto\delta_p\circ\alpha(g)
\end{align*}
der Lie--Gruppe $G$ auf der Mannigfaltigkeit $P$. Hierbei identifizieren wir $P$ mit der zugeh\"origen Menge von Charakteren von $C^{\infty}(P)$.
\end{proposition}

\begin{bemerkung}
Ist die Wirkung $\sigma$ frei und eigentlich, dann erhalten wir ein Hauptfaserb\"undel
\[(P,P/G,G,\pr,\sigma),
\]wobei $\pr:P\rightarrow P/G$ die kanonische Orbitabbildung bezeichnet (cf. [To00], Kapitel VIII, Abschnitt 21).
\end{bemerkung}

Wir formulieren jetzt das Hauptproblem, mit dem sich diese Arbeit auseinandersetzt:
$ $\\
$ $\\
{\bf Hauptproblem:} \emph{Finde einen geometrisch fundierten Zugang zu der nichtkommutativen Geometrie von Hauptfaserb\"undeln}.
$ $\\

Um dieses Problem zu l\"osen, werden wir haupts\"achlich damit befasst sein die folgenden drei Fragen zu kl\"aren:
$ $\\
$ $\\
{\bf Frage 1:} Sei $P$ eine Mannigfaltigkeit und $G$ eine Lie--Gruppe. Existieren nat\"urliche \emph{algebraische Bedingungen} an ein glattes dynamisches System $(C^{\infty}(P),G,\alpha)$, welche die Freiheit der induzierten Wirkung $\sigma$ von Proposition \ref{zusammenfassung II} garantieren?

\noindent
{\bf Frage 2:} Angenommen die Wirkung $\sigma$ ist frei und eigentlich. Was kann \"uber die Struktur des induzierten Hauptfaserb\"undels  
\[(P,P/G,G,\pr,\sigma)
\]gesagt werden?

In der klassischen Differentialgeometrie ist die Beziehung zwischen lokal und global definierten Objekten eminent wichtig f\"ur viele Konstruktionen und Anwendungen. Insbesondere sind die Hauptfaserb\"undel in unserem Kontext automatisch lokal trivial. Dies f\"uhrt zu folgender Frage:
$ $\\
$ $\\
{\bf Frage 3:} Existiert eine gute \emph{Lokalisierungsmethode} in der Nichtkommutativen Geometrie?
$ $\\

Um Frage 1 zu behandeln, werden wir (glatte) dynamische Systeme $(A,G,\alpha)$ mit Hilfe der Darstellungstheorie der Symmetriegruppe $G$ studieren. Dieser Ansatz wird nat\"urliche Bedingungen f\"ur die Freiheit der zugeh\"origen Rechtswirkung
\begin{align}
\sigma:\Gamma_A\times G\rightarrow\Gamma_A,\,\,\,(\chi,g)\mapsto\chi\circ\alpha(g)\label{action zusammenfassung}
\end{align}
von $G$ auf dem Spektrum $\Gamma_A$ von $A$ liefern. Ist $G$ kompakt und abelsch, dann werden wir sehen, dass die induzierte Wirkung (\ref{action zusammenfassung}) frei ist, falls die isotypischen Komponenten des dynamischen Systems $(A,G,\alpha)$ invertierbare Elemente enthalten. Desweiteren werden wir sehen, dass diese Bedingung gerade die trivialen $G$-Hauptfaserb\"undel charakterisiert und daher zu einer vern\"unftigen Definition von \emph{trivialen nichtkommutativen $G$-Hauptfaserb\"undeln} f\"uhrt. 
W\"ahrend es in der klassischen Differentialgeometrie bis auf Isomorphie nur genau ein triviales $G$-Hauptfaserb\"undel \"uber einer gegebenen Mannigfaltigkeit $M$ existiert, werden wir sehen, dass sich die Situation in der nichtkommutativen Welt v\"ollig anders verh\"alt. Insbesondere werden wir alle trivialen nichtkommutativen $\mathbb{T}^n$-Hauptfaserb\"undel mit Hilfe einer geeigneten Kohomologietheorie klassifizieren.

Inspiriert durch Lokalisierungsmethoden aus der Algebraischen Geometrie, werden wir eine geeignete (glatte) Lokalisierungsmethode f\"ur nichtkommutative Algebren entwickeln (Frage 3). Ist $M$ eine Mannigfaltigkeit und $U\subseteq M$ eine offene Teilmenge, dann werden wir mit Hilfe dieser Lokalisierungsmethode zeigen, dass es m\"oglich ist $C^{\infty}(U)$ aus $C^{\infty}(M)$ zu rekonstruieren. Insbesondere werden wir sehen, dass diese Rekonstruktion auch im Kontext von Schnitten von Algebrenb\"undeln gilt. Der Schritt zu (nicht-trivialen) nichtkommutativen Hauptfaserb\"undeln mit kompakt abelscher Stukturgruppe erfolgt ebenfalls mit Hilfe dieser Lokalisierungsmethode: Wir werden ein (glattes) dynamisches System $(A,G,\alpha)$ ein \emph{nicht-triviales nichtkommutatives $G$-Hauptfaserb\"undel} nennen, falls die Lokalisierung um Charakter der Fixpunktalgebra der Wirkung von $G$ auf dem Zentrum $C_A$ von $A$ zu trivialen nichtkommutativen Hauptfaserb\"undeln f\"uhrt. Schlie\ss lich werden wir einige Beispiele diskutieren.

Nat\"urlich ist dies nur ein kleiner Teil der Fragen, die auftauchen, wenn man einen geometrisch fundierten Zugang zu der nichtkommutativen Geometrie von Hauptfaserb\"undeln entwickeln m\"ochte.

\section{Preliminaries and Notation} 

In order to simplify the notation a little we write ``NCP" for ``noncommutative principal". Given a unital algebra $A$, we write $\Aut(A)$\sindex[n]{$\Aut(A)$} for the corresponding automorphisms in the category of $A$. By a (smooth) dynamical system we mean a triple $(A,G,\alpha)$, consisting of a unital locally convex algebra $A$, a (Lie) topological group $G$ and a group homomorphism $\alpha:G\rightarrow\Aut(A)$, which induces a (smooth) continuous action of $G$ on $A$. Moreover, we write $C_A$ for the center of $A$. All manifolds appearing in this thesis are assumed to be finite-dimensional, paracompact, second countable and smooth if not mentioned otherwise. If $M$ is a manifold, we usually write $C^{\infty}(M)$ for the algebra of smooth $\mathbb{C}$-valued functions on $M$.

\section{Outline}

We now give a rough outline of the results that can be found in this thesis, without going too much into detail.

\subsection*{Chapter 2. Background on Differential Geometry of Fibre Bundles}

Chapter 2 contains the basic concepts of the theory of principal fibre bundles. Indeed, in the first section we present the basic definitions of principal bundles, trivial principal bundles (which play a key role in this thesis) and vector bundles. We further recall how to associate a vector bundle to a principal bundle. This construction is very important, since the sections of an associated vector bundle have a very nice description in terms of the principal bundle. In the second part of the chapter we recall the concept of connections on principal bundles; they are used in the development of curvature and hence are also important in the Chern--Weil theory of characteristic classes of principal bundles.

\subsection*{Chapter 3. Noncommutative Vector Bundles}
 
The goal of this chapter is to prove a slightly more general version of Theorem \ref{1}
in the smooth context, i.e., to establish that the categories of finitely generated projective modules over $C^{\infty}(M)$ and the category of smooth vector bundles are equivalent for arbitrary manifolds $M$ of finite dimension: 
\vspace*{0,2cm}

\noindent 
{\bf Theorem} {\bf(Serre--Swan).}
{\it Let $M$ be a manifold. Then the functor
\[\mathbb{V}\mapsto\Gamma(\mathbb{V})
\]defines an equivalence between the category of vector bundles over $M$ and the category of finitely generated projective modules over $C^{\infty}(M)$.}

\noindent
In particular one should consider finitely generated modules of noncommutative algebras as \emph{noncommutative vector bundles}. We have not found such a theorem discussed explicitly in the literature. We further present some results concerning the theory of locally convex modules of locally convex algebras.

\subsection*{Chapter 4. Smooth Localization in Noncommutative Geometry}

As is well-known from classical differential geometry, the relation between locally and globally defined objects is important for many constructions and applications. For example, a (non-trivial) principal bundle $(P,M,G,q,\sigma)$ can be considered as a geometric object that is glued together from local pieces which are trivial, i.e., which are of the form $U\times G$ for some open subset $U$ of $M$. Unfortunately, there is no obvious and straightforward localization method in Noncommutative Geometry. 
Inspired by methods from algebraic geometry, we present in Section \ref{SLNS} an appropriate method of localizing (possibly noncommutative) algebras in a smooth way:
\vspace*{0,2cm}

\noindent 
{\bf Definition} {\bf(Smooth localization).}
For a unital locally convex algebra $A$ we write
\[A\{t\}:=C^{\infty}(\mathbb{R},A)
\]for the unital locally convex algebra of smooth $A$-valued arcs. We further define for $a\in A$ a smooth $A$-valued function on $\mathbb{R}$ by $f_a:\mathbb{R}\rightarrow A,\,\,\,t\mapsto 1_A-ta$ and write $I_a:=\overline{\langle f_a\rangle}$ for the closure of the two-sided ideal generated by this function. Finally, we write
\[A_{\{a\}}:=A\{t\}/I_{a}
\]for the corresponding locally convex quotient algebra and call it the (\emph{smooth}) \emph{localization} of $A$ with respect to $a$.

In Section \ref{Localizing C(M)} we show that if $A$ is a unital Fr\'echet algebra, then the smooth localization of $C^{\infty}(M,A)$ with respect to a smooth nonzero function $f:M\rightarrow\mathbb{R}$ is isomorphic as a unital Fr\'echet algebra to $C^{\infty}(M_f,A)$, where
\[M_f:=\{m\in M:\,f(m)\neq 0\}.
\]More, specifically, we prove the following theorem:
\vspace*{0,2cm}

{\bf Theorem} {\bf(Smooth localization of sections of trivial algebra bundles).}
{\it Let $A$ be a unital Fr\'echet algebra. If $M$ is a manifold and $f:M\rightarrow\mathbb{R}$ a smooth function which is nonzero, then the map
\[\phi_f:C^{\infty}(M,A)_{\{f\}}\rightarrow C^{\infty}(M_f,A),\,\,\,[F]\mapsto F\circ\left(\frac{1}{f}\times\id_{M_f}\right)
\]is an isomorphism of unital Fr\'{e}chet algebras.}

\noindent
In particular, we conclude that it is possible to reconstruct $C^{\infty}(U)$ out of data from $C^{\infty}(M)$, since for each open subset $U$ of $M$ there exists a smooth nonzero function $f:M\rightarrow\mathbb{R}$ satisfying $U=f^{-1}(\mathbb{R}^{\times})$ (cf. Theorem \ref{whitneys theorem}).

Finally, in Section \ref{ACNCPT^nB} we present a method of localizing dynamical systems $(A,G,\alpha)$ with respect to elements of the commutative fixed point algebra of the induced action of $G$ on the center $C_A$ of $A$. This construction will be crucial for our approach to NCP bundles.

\subsection*{Chapter 5. Some Comments on Sections of Algebra Bundles}

In this short chapter we are dealing with algebra bundles $q:\mathbb{A}\rightarrow M$ with a possibly infinite-dimensional fibre $A$ over a finite-dimensional manifold $M$ and show how to endow the corresponding space $\Gamma\mathbb{A}$ of sections with a topology that turns it into a (unital) locally convex algebra. Moreover, we show that if $A$ is a Fr\'echet algebra, then the smooth localization of $\Gamma\mathbb{A}$ with respect to an element $f\in C^{\infty}(M,\mathbb{R})$ (cf. Definition \ref{sm. loc.}) is isomorphic (as a unital Fr\'echet algebra) to $\Gamma\mathbb{A}_{M_f}$, where 
\[M_f:=\{m\in M:\,f(m)\neq 0\}.
\]More, specifically, we prove the following theorem: 
\vspace*{0,2cm}

\noindent 
{\bf Theorem} {\bf(Smooth localization of sections of algebra bundles).}
{\it Let $A$ be a unital Fr\'echet algebra. If $(\mathbb{A},M,A,q)$ is an algebra bundle and $f:M\rightarrow\mathbb{R}$ a smooth function which is nonzero, then the map
\[\phi_f:\Gamma\mathbb{A}_{\{f\}}\rightarrow\Gamma\mathbb{A}_{M_f},\,\,\,[F]\mapsto F\circ\left(\frac{1}{f}\times\id_{M_f}\right)
\]is an isomorphism of unital Fr\'echet algebras.}

\subsection*{Chapter 6. Free Group Actions from the Viewpoint of Dynamical Systems}

Chapter 6 is dedicated to the foundations of our geometric approach to noncommutative principal bundles. Our main result is a new characterization of free group actions, involving dynamical systems and representations of the corresponding transformation group. In Section \ref{section:free dynamical systems} we introduce the concept of a \emph{free dynamical system}. Loosely speaking, we call a dynamical system $(A,G,\alpha)$ free, if the locally convex algebra $A$ is \emph{commutative} and the topological group $G$ admits ``sufficiently many" representations such that a certain family of maps defined on the corresponding modules associated to $A$ are surjective.
\vspace*{0,2cm}

\noindent 
{\bf Theorem} {\bf(Freeness of the induced action).}
{\it If $(A,G,\alpha)$ is a free dynamical system, then the induced action
\[\sigma:\Gamma_A\times G\rightarrow\Gamma_A,\,\,\,(\chi,g)\mapsto\chi\circ\alpha(g)
\]of $G$ on the spectrum $\Gamma_A$ of $A$ is free.}

In Section \ref{ANCFGACG} we apply the constructions of Section \ref{section:free dynamical systems} to dynamical systems arising from group actions in classical geometry. The astonishing result is a new characterization of free group actions:
\vspace*{0,2cm}

\noindent 
{\bf Theorem} {\bf(Characterization of free group actions).}
{\it Let $P$ be a manifold, $G$ a compact Lie group and $(C^{\infty}(P),G,\alpha)$ a smooth dynamical system. Then the following statements are equivalent:
\begin{itemize}
\item[\emph{(a)}]
The smooth dynamical system $(C^{\infty}(P),G,\alpha)$ is free.
\item[\emph{(b)}]
The induced smooth group action $\sigma:P\times G\rightarrow P$ is free.
\end{itemize}
In particular, in this situation the two concepts of freeness coincide.}
\vspace*{0,2cm}

Section \ref{FACAG} is devoted to discussing dynamical systems $(A,G,\alpha)$ with compact structure group $G$ with the help of the ``fine structure" of $A$ with respect to $G$, i.e., with the Fourier decomposition of the system $(A,G,\alpha)$. Our main result is the extendability of characters on fixed point algebras of CIAs. Indeed, we show that if $A$ is a complete commutative CIA, $G$ a compact group and $(A,G,\alpha)$ a dynamical system, then each character of $B:=A^G$ can be extended to a character of $A$. In particular, the natural map $\Gamma_A\rightarrow\Gamma_B$, $\chi\mapsto\chi_{\mid B}$ is surjective:
\vspace*{0,2cm}

\noindent 
{\bf Theorem} {\bf(Extendability of characters).}
{\it Let $A$ be a complete commutative CIA, $G$ a compact group and $(A,G,\alpha)$ a dynamical system. If $B$ is the corresponding fixed point algebra, then each character $\chi:B\rightarrow\mathbb{C}$ extends to a character $\widetilde{\chi}:A\rightarrow\mathbb{C}$.}

In the remaining part of this chapter we are concerned with rewriting and discussing the freeness condition for a dynamical system $(A,G,\alpha)$ with compact abelian structure group $G$. Here, the character group $\widehat{G}:=\Hom_{\text{grp}}(G,\mathbb{T})$ of $G$ plays an important role. In particular, we present natural conditions which ensure the freeness of such a dynamical system. These conditions may even be formulated if the algebra $A$ is noncommutative.
\vspace*{0,2cm}

\noindent
{\bf Theorem} {\bf(The freeness condition for compact abelian groups).}
{\it Let $A$ be a commutative unital locally convex algebra, $G$ a compact abelian group and $(A,G,\alpha)$ a dynamical system. Further, let
\[A_{\varphi}:=\{a\in A:\,(\forall g\in G)\,\,\alpha(g).a=\varphi(g)\cdot a\}
\]be the isotypic component corresponding to the character $\varphi\in\widehat{G}$. The dynamical system $(A,G,\alpha)$ is free if the map
\[\ev^{\varphi}_{\chi}: A_{\varphi}\rightarrow \mathbb{C},\,\,\,a\mapsto\chi(a)
\]is surjective for all $\varphi\in\widehat{G}$ and all $\chi\in\Gamma_A$.}
\vspace*{0,2cm}

\noindent
{\bf Theorem} {\bf(Invertible elements in isotypic components).}
{\it Let $A$ be a commutative unital locally convex algebra, $G$ a compact abelian group and $(A,G,\alpha)$ a dynamical system. If each isotypic component $A_{\varphi}$ contains an invertible element, then the dynamical system $(A,G,\alpha)$ is free.}

\subsection*{Chapter 7. Trivial NCP Torus Bundles}

This chapter is concerned with a new, geometrically oriented approach to the noncommutative geometry of trivial principal $\mathbb{T}^n$-bundles based on dynamical systems of the form $(A,\mathbb{T}^n,\alpha)$. In the first section we introduce the concept of trivial noncommutative principal $\mathbb{T}^n$-bundles.:
\vspace*{0,2cm}

\noindent 
{\bf Definition} {\bf(Trivial NCP $\mathbb{T}^n$-bundles).}
Let $A$ be a unital locally convex algebra. A (smooth) dynamical system $(A,\mathbb{T}^n,\alpha)$ is called a (\emph{smooth}) \emph{trivial NCP $\mathbb{T}^n$-bundle}, if each isotypic component contains an invertible element. 

\noindent
This definition is inspired by the following observation: A principal bundle $(P,M,\mathbb{T}^n,q,\sigma)$ is trivial if and only if it admits a trivialization map. Such a trivialization map consists basically of $n$ smooth functions $f_i:P\rightarrow\mathbb{T}$ satisfying $f_i(\sigma(p,z))=f_i(p)\cdot z_i$ for all $p\in P$ and $z\in\mathbb{T}^n$. From an algebraical  point of view this condition means that each isotypic component of the (naturally) induced dynamical system $(C^{\infty}(P),\mathbb{T}^n,\alpha)$ contains an invertible element. Conversely, we show that each trivial noncommutative principal $\mathbb{T}^n$-bundle of the form $(C^{\infty}(P),\mathbb{T}^n,\alpha)$ induces a trivial principal $\mathbb{T}^n$-bundle of the form $(P,P/\mathbb{T}^n,\mathbb{T}^n,\pr,\sigma)$. The crucial point here is to verify the freeness of the induced action of $\mathbb{T}^n$ on $P$.
\vspace*{0,2cm}

\noindent 
{\bf Theorem} {\bf(Trivial principal $\mathbb{T}^n$-bundles).}
{\it If $P$ is a manifold, then the following assertions hold:
\begin{itemize}
\item[\emph{(a)}]
If $(C^{\infty}(P),\mathbb{T}^n,\alpha)$ is a smooth trivial NCP $\mathbb{T}^n$-bundle, then the corresponding principal bundle $(P,P/\mathbb{T}^n,\mathbb{T}^n,\pr,\sigma)$ is trivial.
\item[\emph{(b)}]
Conversely, if $(P,M,\mathbb{T}^n,q,\sigma)$ is a trivial principal $\mathbb{T}^n$-bundle, then the corresponding smooth dynamical system $(C^{\infty}(P),\mathbb{T}^n,\alpha)$ is a smooth trivial NCP $\mathbb{T}^n$-bundle.
\end{itemize}}

In the second section we present various examples including noncommutative tori, topological dynamical systems and certain crossed product constructions. In particular, we show that the group $C^{*}$-algebra of the discrete Heisenberg group is a trivial NCP $\mathbb{T}^2$-bundle; the group $C^{*}$-algebra of the discrete Heisenberg group plays an important role in the nice paper [ENOO09].

The remainder of Chapter 7 is devoted to a complete classification of trivial noncommutative principal $\mathbb{T}^n$-bundles up to completion. It turns out that each trivial noncommutative principal $\mathbb{T}^n$-bundle possesses an underlying algebraic structure of a $\mathbb{Z}^n$-graded unital associative algebra. This structure may be considered as an algebraic counterpart of a trivial noncommutative principal $\mathbb{T}^n$-bundle and can be classified with methods from the extension theory of groups and a suitable cohomology theory. We further present some nice examples of these algebraically trivial principal $\mathbb{T}^n$-bundles. 
\vspace*{0,2cm}

\noindent 
{\bf Theorem} {\bf(Classification of trivial NCP $\mathbb{T}^n$-bundles).}
{\it  Let $n\in\mathbb{N}$ and $B$ be a unital algebra. Then the map
\[\chi:\Ext(\mathbb{Z}^n,B)\rightarrow H^2(\mathbb{Z}^n,B),\,\,\,[A]\mapsto\chi(A)
\]is a well-defined bijection.}

\subsection*{Chapter 8. Trivial NCP Cyclic Bundles}

For each $n\in\mathbb{N}$ we write 
\[C_n:=\{z\in\mathbb{C}^{\times}:\,z^n=1\}=\{\zeta^k:\,\zeta:=\exp(\frac{2\pi i}{n}),\,k=0,1,\ldots,n-1\}
\]for the cyclic subgroup of $\mathbb{T}$ of $n$-th roots of unity. The goal of this chapter is to present a new, geometrically oriented approach to the noncommutative geometry of trivial principal $C_n$-bundles based on dynamical systems of the form $(A,C_n,\alpha)$. Our main idea is to characterize the functions appearing in a trivialization of a trivial principal $C_n$-bundle $(P,M,C_n,q,\sigma)$ as elements in $C^{\infty}(P)$. This approach leads to the following concept: Given a unital locally convex algebra $A$, a dynamical system $(A,C_n,\alpha)$ is called a trivial NCP $C_n$-bundle if the isotypic component corresponding to $\zeta$ contains an invertible element $a$ with the additional property that there exists an element $b$ in the corresponding fixed point algebra with $b^n=a^n$. Moreover, we also treat the case of products of finite cyclic groups (i.e. finite abelian groups) and present some examples.
\vspace*{0,2cm}

In the last section of Chapter 8 we present some general facts on principal bundles with finite structure group and deduce some algebraic properties of their corresponding smooth function algebras. This discussion in particular leads to a possible approach to ``noncommutative finite coverings''. One notable aspect is given by the following observation: If $q:P\rightarrow M$ is a finite covering, i.e., a principal bundle $(P,M,G,q,\sigma)$ with finite structure group $G$, then $C^{\infty}(P)$ is a finitely generated projective $C^{\infty}(M)$-module. We have not found such a theorem discussed in the literature and prove it with methods coming from representation theory. 

\subsection*{Chapter 9. NCP Bundles with Compact Abelian Structure Group}

Let $G$ be a compact abelian Lie group. The main goal of Chapter 9 is to present a reasonable approach to NCP $G$-bundles. Since, for example, the isotypic components of a dynamical system $(A,\mathbb{T}^n,\alpha)$ do in general not contain invertible elements, we have to come up with algebraic conditions on a dynamical system $(A,G,\alpha)$ that still ensures the freeness of the induced (right-) action
\[\sigma:\Gamma_A\times G\rightarrow\Gamma_A,\,\,\,(\chi,g)\mapsto \chi\circ\alpha(g).
\]of $G$ on the spectrum $\Gamma_A$ of $A$. Because the freeness of a group action is a local condition (cf. Remark \ref{free=locally free}), our main idea is inspired by the classical setting: Loosely speaking, a dynamical system $(A,G,\alpha)$ is called a NCP $G$-bundle, if it is ``locally" a trivial NCP $G$-bundle, i.e., once the concept of a trivial NCP $G$-bundle is known for a certain (compact) group $G$, Section 4 enters the picture. In view of Chapter \ref{trivial ncp torus bundles} and Chapter \ref{trivial ncp cyclic bundles}, we restrict our attention to compact abelian Lie groups $G$. In fact, [HoMo06], Proposition 2.42 implies that each compact abelian Lie group $G$ is isomorphic to $\mathbb{T}^n\times \Lambda$ for some natural number $n$ and finite abelian group $\Lambda$. 
\vspace*{0,2cm}

\noindent
{\bf Definition} {\bf(NCP $G$-bundles).}
Let $G$ be a compact abelian Lie group. We call a (smooth) dynamical system $(A,G,\alpha)$ a \emph{\emph{(}smooth\emph{)} NCP $G$-bundle} if for each character $\chi$ of the fixed point algebra $Z=C_A^G$ ($C_A$ denotes the center of $A$) there exists an element $z\in Z$ with $\chi(z)\neq 0$ such that the corresponding (smooth) localized dynamical system is a (smooth) trivial NCP $G$-bundle. 

\noindent
We prove that this approach extends the classical theory of principal bundles:
\vspace*{0,2cm}

\noindent
{\bf Theorem} {\bf(Reconstruction Theorem).}
{\it For a manifold $P$, the following assertions hold:
\begin{itemize}
\item[\emph{(a)}]
If $P$ is compact and $(C^{\infty}(P),G,\alpha)$ is a smooth NCP $G$-bundle, then we obtain a principal $G$-bundle $(P,P/G,G,\pr,\sigma)$.
\item[\emph{(b)}]
Conversely, if $(P,M,G,q,\sigma)$ is a principal $G$-bundle, then the corresponding smooth dynamical system $(C^{\infty}(P),G,\alpha)$ is a smooth NCP $G$-bundle.
\end{itemize}}

In the remainder we present various examples like sections of algebra bundles with trivial NCP $G$-bundle as fibre, sections of algebra bundles which are pull-backs of 
principal $G$-bundles and sections of trivial equivariant algebra bundles. Moreover, we show that each trivial NCP $G$-bundle carries the structure of a NCP $G$-bundle in its own right.
\subsection*{Chapter 10. Characteristic Classes of Lie Algebra Extensions}

The Chern--Weil homomorphism of a principal bundle $q:P\rightarrow M$ is an algebra homomorphism from the algebra of polynomials invariant under the adjoint action of a Lie group $G$ on the corresponding Lie algebra $\mathfrak{g}$ into the even de Rham cohomology $H_{\text{dR}}^{2\bullet}(M,\mathbb{K})$ of the base space $M$ of a principal bundle $P$. This map is achieved by evaluating an invariant polynomial $f$ of degree $k$ on the curvature $\Omega$ of a connection $\omega$ on $P$ and thus obtaining a closed form on the base.
Around 1970, another set of characteristic classes called the \emph{secondary characteristic classes} have been discovered. The secondary characteristic classes are also topological invariants of principal bundles which are derived from the curvature of adequate connections. They appear for example in the Lagrangian formulation of modern quantum field theories. The best known of these classes are the Chern--Simons classes. In 1985, P. Lecomte described a general cohomological construction which generalizes the classical Chern--Weil homomorphism. This construction associates characteristic classes to every extension of Lie algebras. The classical construction of Chern and Weil arises in this context from an extension of Lie algebras commonly known as the \emph{Atiyah sequence} (cf. Example \ref{Attia sequence}). The aim of this chapter is to define \emph{secondary characteristic classes in the setting of Lie algebra extensions}. This will also provide a new proof of Lecomte's construction.

\subsection*{Chapter 11. Outlook: More Aspects, Ideas and Problems on NCP Bundles}

This final chapter is devoted to discussing more aspects of our approach to NCP bundles. In particular, it contains additional ideas and open problems which might serve for further studies on this topic. In the first section we present the problem of embedding our approach to NCP bundles with compact abelian structure group in a theory of NCP bundles with compact structure group. In the second section we suggest a possible approach to a classification theory for (non-trivial) NCP $\mathbb{T}^n$-bundles. The third section is dedicated to a theory of ``infinitesimal" objects on noncommutative spaces, i.e., we discuss possible ``geometric" invariants for our approach to NCP  bundles. This discussion particularly makes use of Lecomte's Chern--Weil homomorphism of Chapter \ref{chapter characteristic classes of lie algebra extensions}. Finally, we present some ideas of associating characteristic classes to finitely generated projective modules of CIA's.

\subsection*{Appendix A. Hopf--Galois Extensions and Differential Structures}

The noncommutative geometry of principal bundles is, so far, not sufficiently well understood, but there is a well-developed purely algebraic approach using the theory of bialgebras and Hopf algebras commonly known as Hopf--Galois extensions. In this appendix we report on the basics of Hopf--Galois extensions.

\subsection*{Appendix B. Some Background on Lie Theory}

Here, we provide the basic definitions and concepts of Lie theory which appear throughout this thesis. In particular, we provide the basic definitions related to extensions of Lie groups. Since every discrete group can be viewed as a Lie group, our discussion includes in particular the algebraic context of group extensions. This discussion is important for the classification of our trivial NCP $\mathbb{T}^n$-bundles. Moreover, we provide the basic tools of Lie algebra cohomology which are necessary in order to understand Lecomte's generalization of the Chern--Weil map and define secondary characteristic classes of Lie algebra extensions.

\subsection*{Appendix C. Continuous Inverse Algebras}

In this appendix we introduce an important class of algebras whose groups of units are Lie groups. They may be seen as the infinite-dimensional generalization of matrix algebras and are encountered in K-theory and noncommutative geometry, usually as dense unital subalgebras of $C^*$-algebras. In fact, these algebras play a central role in this thesis, since, for a compact manifold $M$, the space $C^{\infty}(M)$ is the prototype of such a continuous inverse algebra. 

\subsection*{Appendix D. Topology and Smooth Vector-Valued Function Spaces}

This appendix is devoted to topological aspects and results which are needed throughout this thesis. We start with some results on the extendability of densely defined maps. Then we discuss the projective tensor product of locally convex spaces and further the projective tensor product of locally convex modules. Furthermore, we present some useful results on smooth vector-valued function spaces. In particular, the smooth exponential law will be used several times in this thesis. We finally discuss a special class of locally convex spaces in which the uniform boundedness principle (``Theorem of Banach--Steinhaus") still remains true.

\subsection*{Appendix E. An Important Source of NC Spaces: Noncommutative Tori}

A prominent class of noncommutative spaces are the so called ``noncommutative tori" or ``quantum tori".
These spaces will also provide an important class of examples throughout this thesis. For this reason we provide an overview on the structure of NC tori.

\subsection*{Structure of this Thesis}
The thesis is organized as follows. At the beginning of each chapter and section, we give a rough outline of our aims. As a rule of thumb,  we provide motivating comments in each section which should illustrate the flow of ideas. Terminology and notation can mainly be found in remarks and definitions, as long as they are important for the sequel.

\chapter{Background on Differential Geometry of Fibre Bundles}

In the following we give an overview of the differential topology of fibre bundles. All manifolds are assumed to be finite-dimensional, paracompact, second countable and smooth.

\section{Fibre Bundles}

For a geometric analysis of smooth maps the following concept turns out to be successful: A smooth map $q: B\rightarrow M$ between manifolds $B$ and $M$ can be considered as a smooth family of its inverse images $q^{-1}(m)$, which are parametrized by the set $M$. Of particular interest is the case where all fibres are isomorphic to a certain space $F$. In this section we give a short overview of the theory of fibre bundles. We will in particular be concerned with the case where the fibres are Lie groups and vector spaces. 

\begin{definition}\label{fibre bundles}{\bf(Fibre bundles).}\index{Bundles!Fibre}
A \emph{smooth fibre bundle} is a quadruple $(B,M,F,q)$\sindex[n]{$(B,M,F,q)$}, consisting of manifolds $B$, $M$ and $F$ and a smooth map $q:B\rightarrow M$ with the following property of local triviality: Each point $m\in M$ has an open neighbourhood $U$ for which there exists a diffeomorphism
\[\varphi_U:U\times F\rightarrow q^{-1}(U),
\]satisfying
\[q\circ\varphi_U=\pr_U:U\times F\rightarrow U,\,\,\,(u,f)\mapsto u.
\]\sindex[n]{$\varphi_U$}
\end{definition}

We use the following terminology:

\begin{itemize}
\item[$\bullet$]
$B$ is called the \emph{total space}.
\item[$\bullet$]
$M$ is called the \emph{base space}.
\item[$\bullet$]
$F$ is called the \emph{fibre type}.
\item[$\bullet$]
$q$ is called the \emph{bundle projection}.
\item[$\bullet$]
The sets $B_q:=q^{-1}(b)$ are called the \emph{fibers of} $q$.
\item[$\bullet$]
$\varphi_U$ is called a \emph{bundle chart}.
\item[$\bullet$]
$B_U:=q^{-1}(U)$ is called the \emph{restriction of} $E$ \emph{to} $U$.
\item[$\bullet$]
$(B,M,F,q)$ is called an $F$-\emph{bundle} over $M$.
\end{itemize}

\begin{definition}\label{principal bundles I}{\bf(Principal bundles).}\index{Bundles!Principal}
Let $G$ be a Lie group. A \emph{principal bundle} is a quintuple $(P,M,G,q,\sigma)$\sindex[n]{$(P,M,G,q,\sigma)$}, where $\sigma:P\times G\rightarrow P$ is a smooth action, with the property of local triviality: Each point $m\in M$ has an open neighbourhood $U$ for which there exists a diffeomorphism
\[\varphi_U:U\times G\rightarrow q^{-1}(U),
\]satisfying $q\circ\varphi_U=\pr_U$ and the equivariance property
\[\varphi_U(u,gh)=\varphi_U(u,g).h\,\,\,\text{for}\,\,\,u\in U, g,h\in G.
\]
\end{definition}

\begin{remark}\label{principal bundles II}
(a) For each principal bundle $(P,M,G,q,\sigma)$, the quadruple $(P,M,G,q)$ is a smooth fibre bundle.

(b) The right action of $G$ on $P$ is free and proper and the natural map 
\[\Phi:P/G\mapsto M,\,\,\,p.G\mapsto q(p),
\]is a diffeomorphism.

(c) Conversely, in view of the Quotient Theorem\index{Quotient Theorem} (cf. [To00], Kapitel VIII, Abschnitt 21), each free and proper right smooth action $\sigma:P\times G\rightarrow P$ defines the principal bundle $(P,P/G,G,\pr,\sigma)$.
\end{remark}

\begin{example}\label{hopf fibration}{\bf(The Hopf fibration).}\index{Hopf Fibration}
The Hopf fibration
\[q:\mathbb{S}^3\rightarrow\mathbb{S}^3/\mathbb{T}\cong\mathbb{S}^2\cong\mathbb{P}_1(\mathbb{C}),\,\,\,(z_1,z_2)\mapsto\mathbb{T}.(z_1,z_2),
\]is a $\mathbb{T}$-principal bundle over $\mathbb{S}^2$.
\end{example}

\begin{definition}\label{trivial principal bundle I}{\bf(Trivial principal bundles).}\index{Bundles!Trivial Principal}
A principal bundle $(P,M,G,q,\sigma)$ is called \emph{trivial} if it is equivalent to the bundle $(M\times G,M,G,q_M,\sigma_G)$, where $q_M(m,g)=m$ and $\sigma_G((m,g),h)=(m,gh)$.
\end{definition}

The next lemma will be crucial for our approach to trivial noncommutative principal bundles:

\begin{lemma}\label{trivial principal bundle II}\index{Sections!of Principal Bundles}
A principal bundle $(P,M,G,q,\sigma)$ is trivial if and only if it has a smooth section, i.e., if there exists a smooth map $s:M\rightarrow P$ such that $q\circ s=\id_M$.
\end{lemma}

\begin{proof}
\,\,\,($``\Rightarrow"$) By Definition \ref{trivial principal bundle I}, the principal bundle $(P,M,G,q,\sigma)$ is trivial if it is equivalent to the bundle $(M\times G,M,G,q_M,\sigma_G)$. If $\phi:P\rightarrow M\times G$ denotes the corresponding equivalence and $s_{1_G}$ the section of $M\times G$ given by 
\[s_{1_G}:M\rightarrow M\times G,\,\,\,m\mapsto (m,1_G),
\]then a short calculation shows that $s:=\phi^{-1}\circ s_{1_G}$ is a smooth section of $(P,M,G,q,\sigma)$.

($``\Leftarrow"$) On the other hand, if $s:M\rightarrow P$ is a smooth section of $(P,M,G,q,\sigma)$, then 
\[\phi_s:M\times G\rightarrow P ,\,\,\,(m,g)\mapsto s(m)g
\]defines an equivalence of principal bundles.
\end{proof}

\begin{definition}\label{Aut(P)}{\bf(The group of bundle automorphisms).}\index{Group of Bundle Automorphisms}\sindex[n]{$\Aut(P)$}\sindex[n]{$\Diff(P)$}\sindex[n]{$\mathfrak{aut}(P)$}
Let $(P,M,G,q,\sigma)$ be a principal bundle. For fixed $g\in G$ we write $\sigma_g(p):=p.g$ for the corresponding orbit-map. The group of \emph{bundle automorphisms} of $P$ is given by
\[\Aut(P):=\{\phi\in\Diff(P):\,(\forall g\in G)\,\phi\circ\sigma_g=\sigma_g\circ\phi\}.
\]Motivated by this definition we consider the Lie algebra 
\[\mathfrak{aut}(P):=\{X\in\mathcal{V}(P):\,(\forall g\in G)\,(\sigma_g)_*X=X\}
\]of $G$-invariant vector fields on $P$.
\end{definition}

\begin{lemma}\label{aut(P)}
The following assertions hold for $\mathfrak{aut}(P)$:
\begin{itemize}
\item[\emph{(a)}]
$\mathfrak{aut}(P)$ is a Lie subalgebra of $\mathcal{V}(P)$.
\item[\emph{(b)}]
The map
\[q_*:\mathfrak{aut}(P)\rightarrow\mathcal{V}(M),\,\,\,(q_*X)(q(p)):=T_p(q)X(p)
\]is well-defined and a homomorphism of Lie algebras.
\item[\emph{(c)}]
$\mathfrak{aut}(P)$ is a $C^{\infty}(M)$-module with respect to
\[(fX)(p):=f(q(p))X(p)
\]and $q_*$ is $C^{\infty}(M)$-linear.
\end{itemize}
\end{lemma}

\begin{proof}\,\,\,(a) That $X\in\mathfrak{aut}(P)$ is equivalent to $X$ being $\sigma_g$-related to itself for each $g\in G$. Therefore, the Related Vector Field Lemma implies that $\mathfrak{aut}(P)$ is a Lie subalgebra.

(b) The relation $(\sigma_g)_*X=X$ means that for each $p\in P$ we have
\[T(\sigma_g)X(p)=X(p.g).
\]Hence,
\[T(q)X(p.g)=T(q)T(\sigma_g)X(p)=T(q\circ\sigma_g)X(p)=T(q)X(p)
\]implies that $q_*X$ is well-defined. Since $(q_*X)\circ q=T(q)\circ X$ is smooth and $q:P\rightarrow M$ is a submersion, the vector field $q_*X$ is smooth. For $X,Y\in\mathfrak{aut}(P)$ we use that the vector fields $X$ and $q_*X$ are $q$-related to see that $[X,Y]$ and $[q_*X,q_*Y]$ are $q$-related. This implies that 
\[q_*[X,Y]=[q_*X,q_*Y].
\]

(c) For $g\in G$ we have $(\sigma_g)_*(fX)=f((\sigma_g)_*X)=f X$.
\end{proof}

\begin{remark}\label{gau(P)}\sindex[n]{$\mathfrak{gau}(P)$}
The ideal
\[\mathfrak{gau}(P):=\ker(q_*)
\]describes those vector fields which are tangent to the $G$-orbits in $P$.
\end{remark}

\begin{definition}\label{vector bundle}{\bf(Vector bundles).}\index{Bundles!Vector}
Let $V$ be a $\mathbb{K}$-vector space ($\mathbb{K}=\mathbb{R},\mathbb{C}$). A \emph{vector bundle} is a smooth fibre bundle $(\mathbb{V},M,V,q)$ for which all fibres $\mathbb{V}_m$, $m\in M$, carry vector space structures, and each point $m\in M$ has an open neighbourhood $U$ for which there exists a diffeomorphism
\[\varphi_U:U\times V\rightarrow q^{-1}(U)=\mathbb{V}_U,
\]satisfying $q\circ\varphi_U=p_U$ and all maps
\[\varphi_{U,x}:V\rightarrow\mathbb{V}_x,\,\,\,v\mapsto\varphi_U(x,v)
\]are linear isomorphisms.
\end{definition}

\begin{example}\label{tangent bundle}{\bf(The tangent bundle of a manifold $M$).}\index{Bundles!Tangent}
For each manifold $M$ the tangent bundle $(TM,M,\mathbb{R}^n,q)$ is a vector bundle with $q(X)=m$ for $X\in T_mM$.
\end{example}

\begin{definition}\label{sections}{\bf(The space of sections of a vector bundle).}\index{Sections! of Vector Bundles}
If $q:\mathbb{V}\rightarrow M$ is a smooth vector bundle, then the corresponding space
\[\Gamma\mathbb{V}:=\Gamma(M,\mathbb{V}):=\{s\in C^{\infty}(M,\mathbb{V}):q\circ s=\id_M\}
\]of \emph{smooth sections} carries a vector space structure defined by
\[(s_1+s_2)(m):=s_1(m)+s_2(m)\,\,\,\text{and}\,\,\,(\lambda s)(m):=\lambda s(m).
\]Moreover, for each $f\in C^{\infty}(M)$ and $s\in\Gamma\mathbb{V}$ the product
\[(fs)(m):=f(m)s(m)
\]is a smooth section of $\mathbb{V}$. We thus obtain on $\Gamma\mathbb{V}$ the structure of a $C^{\infty}(M)$-module, i.e., the following relations hold for $f,f_1,f_2\in C^{\infty}(M)$ and $s,s_1,s_2\in\Gamma\mathbb{V}$:
\[(f_1+f_2)s=f_1s+f_2s,\,\,\,(f_1f_2)s=f_1(f_2s),\,\,\,f(s_1+s_2)=fs_1+fs_2.
\]
\end{definition}

We will give a more precise description of the structure of $\Gamma\mathbb{V}$ in Chapter \ref{NCVB}.

\begin{construction}\label{associated vector bundle}{\bf(Associated vector bundles).}\index{Bundles!Associated Vector}
We now consider again a principal bundle $(P,M,G,q,\sigma)$ and, in addition, a vector space $V$. If $(\pi,V)$ is a smooth representation of $G$, 
then $(p,v).g:=(p.g,g^{-1}.v)$ defines a free and proper action of $G$ on $P\times V$. In particular, we obtain a bundle 
\[\mathbb{V}:=P\times_{\pi}V:=P\times_GV:=(P\times V)/G
\]over $M$ with bundle projection $q_{\mathbb{V}}:\mathbb{V}\rightarrow M$ given by $q_{\mathbb{V}}([p,v]):=q(p)$. Next, we define $\mathbb{K}\times \mathbb{V}\rightarrow \mathbb{V}$ by
\[(\lambda,[p,v])\mapsto\lambda\cdot [p,v]:=[p,\lambda v]
\]and an addition of two elements $x=[p,v],y=[p,w]\in \mathbb{V}$ by
\[x+y:=[p,v+w].
\]As $G$ acts linearly on $V$, these operations are well-defined. This results in the following structure: Each fibre of $\mathbb{V}=P\times_{\pi}V$ is a vector space over $\mathbb{K}$, and for every $m\in M$ there exists a neighbourhood $U$ of $m$ and a diffeomorphism
\[\Phi_U:U\times V\rightarrow q^{-1}(U)
\]which is linear in each fibre. Hence, $(\mathbb{V},M,q_{\mathbb{V}},V)$ is a vector bundle. It is called \emph{vector bundle with fibre $V$ associated to the principal bundle $(P,M,G,q,\sigma)$}.
\end{construction}


\begin{proposition}\label{vector bundles are associated to principal bundles}
Let $V$ be a vector space and $(\pi,V)$ the identical representation of $\GL(V)$ on $V$. Then the following assertions hold:
\begin{itemize}
\item[\emph{(a)}]
If $(P,M,\GL(V),q,\sigma)$ is a $\GL(V)$-principal bundle, then $\mathbb{V}:=P\times_{\pi}V$ is a vector bundle over $M$ with fibre $V$.
\item[\emph{(b)}]
If, conversely, $(\mathbb{V},M,V,q)$ is a vector bundle over $M$, then its frame bundle
\[\Fr(\mathbb{V}):=\bigcup_{m\in M}\Iso(V,\mathbb{V}_m)
\]carries the structure of a $\GL(V)$-principal bundle with respect to the action 
\[\sigma:\Fr(\mathbb{V})\times \GL(V)\rightarrow\Fr(\mathbb{V}),\,\,\,\sigma(\varphi,g):=\varphi\circ g.
\]Moreover, the evaluation map
\[\Fr(\mathbb{V})\times V\rightarrow\mathbb{V},\,\,\,(\varphi,v)\mapsto\varphi(v)
\]induces a bundle equivalence
\[\Fr(\mathbb{V})\times_{\pi} V\rightarrow\mathbb{V},\,\,\,[(\varphi,v)]\mapsto\varphi(v).
\]
\end{itemize}
\end{proposition}

\begin{proof}\,\,\,(a) The first part of the proposition follows from Construction \ref{associated vector bundle}.

(b) For the second assertion we first note that the bundle projection $q_{\Fr(\mathbb{V})}$ is given by taking the fibre $\Iso(V,\mathbb{V}_m)$ to $m$. Further, bundle charts of the vector bundle $(\mathbb{V},M,V,q)$ induce bundle charts of the frame bundle
$q_{\Fr(\mathbb{V})}:\Fr(\mathbb{V})\rightarrow M$. In this way we get a Hausdorff topology on $\Fr(\mathbb{V})$. In particular, the induced bundle charts define on $\Fr(\mathbb{V})$ the structure of a smooth manifold. A local consideration now shows that the action $\sigma$ is smooth and we conclude that the the fibration $q_{\Fr(\mathbb{V})}:\Fr(\mathbb{V})\rightarrow M$ is a $\GL(V)$-principal bundle over $M$. The proof of the last statement of part (b) is left as an easy exercise to the reader.
\end{proof}

\begin{remark}\label{remark on vector bundles are associated to principal bundles}
The preceding theorem easily implies that for two vector bundles $\mathbb{V}_1$ and $\mathbb{V}_2$ with fibre $V$ over $M$ we have
\[\mathbb{V}_1\sim\mathbb{V}_1\Longleftrightarrow\Fr(\mathbb{V}_1)\sim\Fr(\mathbb{V}_2).
\]This leads to a bijection of the set $\Bun(M,\GL(V))$ of equivalence classes of $\GL(V)$-bundles over $M$ and the set $\Vbun(M,V)$ of equivalence classes of $V$-vector bundles over $M$.
\end{remark}

The following proposition gives a nice characterization of the space of sections of an associated vector bundle as the space of $G$-equivariant functions from $P$ to $V$:

\begin{proposition}\label{sections of an associated vector bundle}\index{Sections!of Associated Vector Bundles}\sindex[n]{$C^{\infty}(P,V)^G$}
Let $(P,M,G,q,\sigma)$ be a principal bundle and $(\pi,V)$ be a smooth representation of $G$ defining the associated bundle $\mathbb{V}:=P\times_{\pi} V$ over $M$. If we write
\[C^{\infty}(P,V)^G:=\{f:P\rightarrow V:\,(\forall g\in G)\,f(p.g)=\pi(g^{-1}).f(p)\}
\]for the space of equivariant smooth functions, then the map
\[\Psi_{\pi}:C^{\infty}(P,V)^G\rightarrow\Gamma\mathbb{V},\,\,\,\text{defined by}\,\,\,\Psi(f)(q(p)):=[p,f(p)],
\]is an isomorphism of $C^{\infty}(M)$-modules.
\end{proposition}

\begin{proof}
\,\,\,We leave the proof of this proposition as an easy exercise to the reader.
\end{proof}

\section{Connections on Principal Bundles}\label{conpb}

Connections on principal bundles are the fundamental tool in the theory of smooth fibre bundles. They are used in the development of curvature and therefore also important in the Chern--Weil theory of characteristic classes of principal bundles. We will see in Section 
\ref{infinitesimal objects on noncommutative spaces} how to make the concept of connections applicable to our approach to Noncommutative Geometry. In the following let $(P,M,G,q,\sigma)$ be a principal bundle and $\mathfrak{g}$ the Lie algebra of $G$. We recall that the derived action
\[\dot{\sigma}:\mathfrak{g}\rightarrow\mathcal{V}(P),\,\,\,\dot{\sigma}(x)(p)=\left.\frac{d}{dt}\right|_{t=0}p.\exp(tx)
\]is an homomorphism of Lie algebras, and that for each $p\in P$ the map
\[\dot{\sigma}_p:\mathfrak{g}\rightarrow\ker(T_p(q))\subseteq T_p(P),\,\,\,x\mapsto\dot{\sigma}(x)(p)=T_{1_G}(\sigma^p)(x)
\]is a linear isomorphism. 

\begin{definition}\label{vertical subspace}{\bf(Vertical subspace).}\index{Vertical Subspace}\sindex[n]{$V_p(P)$}
For each $p\in P$ we define the \emph{vertical subspace} of $T_p(P)$ by
\[V_p(P):=\dot{\sigma}_p(\mathfrak{g})=\ker(T_p(q)).
\]Its elements are called \emph{vertical} tangent vectors.
\end{definition}

Connections on principal bundles are most conveniently defined in terms of Lie algebra-valued 1-forms:

\begin{definition}\label{connection 1-forms}{\bf(Connection 1-forms).}\index{Connection 1-Forms}
A 1-form $\theta\in\Omega^1(P,\mathfrak{g})$ is called a \emph{connection 1-form} if
\begin{itemize}
\item[(C1)]
$\theta(\dot{\sigma}(x))=x$ for each $x\in\mathfrak{g}$.
\item[(C2)]
$\sigma^*_g\theta=\Ad(g)^{-1}\circ\theta$ for each $g\in G$, i.e., $\theta\in\Omega^1(P,\mathfrak{g})^G$.
\end{itemize}
We write $\mathcal{C}(P)$\sindex[n]{$\mathcal{C}(P)$} for the space of connection 1-forms. This is obviously an affine subspace of $\Omega^1(P,\mathfrak{g})$ whose translation vector space is $\Omega^1(M,\Ad(P))$, where $\Ad(P)=P\times_{\Ad}\mathfrak{g}$\sindex[n]{$\Ad(P)$} is the vector bundle associated to $P$ by the adjoint representation $(\Ad,\mathfrak{g})$ of $G$.
\end{definition}

\begin{definition}\label{horizontal subspace}{\bf(Horizontal subspace).}\index{Horizontal!Subspace}\sindex[n]{$H_p(P)$}
If $\theta\in\mathcal{C}(P)$ is a connection 1-form, then
\[H_p(P):=\ker\theta_p
\]is a vector space complement to $V_p(P)$, i.e.,
\begin{align}
T_p(P)=V_p(P)\oplus H_p(P).\label{geom distribution I}
\end{align}The space $H_p(P)$ is called the \emph{horizontal subspace} (with respect to $\theta$) and vectors in $H_p(P)$ are called \emph{horizontal}. We further have 
\begin{align}
T(\sigma_g)H_p(P)=H_{p.g}(P),\,\,\,T(\sigma_g)V_p(P)=V_{p.g}(P)\,\,\,\text{for}\,\,\,g\in G,\label{geom distribution II}
\end{align}
i.e., the decomposition into vertical and horizontal subspace is $G$-invariant.
\end{definition}

\begin{remark}\label{geometric distribution}{\bf(Geometric distribution).}\index{Geometric Distribution}
Note that each connection 1-form $\theta\in\mathcal{C}(P)$ can be reconstructed from the family of horizontal subspaces. In fact, for each $p\in P$ we have
\[T_p(P)=\dot{\sigma}_p(\mathfrak{g})\oplus H_p(P).
\]Thus, if $v=\dot{\sigma}_p(x)+w$ for some $x\in\mathfrak{g}$ and $w\in H_p(P)$, then $\theta_p(v)=x$. Therefore, the subspace $H_p(P)$ completely determines the linear map $\theta_p:T_p(P)\rightarrow\mathfrak{g}$. Conversely, each connection 1-form $\theta$ can be described in terms of a \emph{geometric distribution}, i.e., a family of subspaces satisfying (\ref{geom distribution I}) and (\ref{geom distribution II}).
\end{remark}

\begin{remark}\label{c 1-forms on trivial pb}
If $P=M\times G$ is the trivial principal bundle with $\sigma_g(x,h)=(x,hg)$, then $\Omega^1(M,\mathfrak{g})$ parametrizes the space of connection 1-forms. If $p_G:M\times G\rightarrow G$ is the $G$-projection, $p_M:M\times G\rightarrow M$ the $M$-projection and $\kappa_G\in\Omega^1(G,\mathfrak{g})$ denotes the Maurer-Cartan form\index{Maurer-Cartan form} defined by
\[\kappa_G(T_{\mathbf{1}}(\lambda_g)x)=x\,\,\,\text{for}\,\,\,g\in G, x\in\mathfrak{g},
\]then it is possible to show that each connection 1-form $\theta\in\mathcal{C}(M\times G)$ is of the form
\[\theta=p^*_G\kappa_G+(\Ad\circ p_G)^{-1}.p^*_MA,\,\,\,A\in\Omega^1(M,\mathfrak{g}).
\]
\end{remark}

\begin{proposition}\label{existence of c 1-forms}
If $M$ is paracompact, then each principal bundle $(P,M,G,q,\sigma)$ possesses a connection 1-form.
\end{proposition}

\begin{proof}
\,\,\,For the proof we refer to [Ne08b], Proposition 5.4.6.
\end{proof}

\begin{construction}\label{horizontal lifts}{\bf (Horizontal lifts).}\index{Horizontal!Lifts} Let $\theta\in\mathcal{C}(P)$ be a connection 1-form. For each $X\in\mathcal{V}(M)$, we define a vector field $\widetilde{X}\in\mathcal{V}(P)$ by
\[\widetilde{X}(p):=(T_p(q)_{\mid H_p(P)})^{-1}X(q(p)).
\]The vector field $\widetilde{X}$ clearly satisfies
\[q_*\widetilde{X}:=T_p(q)(\widetilde{X}(q(p))=X\,\,\,\text{and}\,\,\,\theta(\widetilde{X})=0.
\]That $\widetilde{X}$ is smooth can be verified locally, so that we may assume that $P=M\times G$ is trivial. By Remark \ref{c 1-forms on trivial pb}, $\theta$ is of the form
$\theta=p^*_G\kappa_G+(\Ad\circ p_G)^{-1}.p^*_MA$ for some $A\in\Omega^1(M,\mathfrak{g})$. Now $\theta_{(m,g)}(v,w)=0$ is equivalent to
\[g^{-1}.w+\Ad(g)^{-1}A_m(v)=0,
\]which is equivalent to $w=-A_m(v).g$. Hence, the horizontal lift $\widetilde{X}$ of $X\in\mathcal{V}(M)$ is given by
\[\widetilde{X}(m,g)=(-A_m(X(m)).g, X(m)),
\]which is smooth.
\end{construction}

\begin{lemma}\label{surj of q_*}
For any $\theta\in C(P)$ the horizontal lift defines a $C^{\infty}(M)$-linear map
\[\tau_{\theta}:\mathcal{V}(M)\rightarrow\mathfrak{aut}(P),\,\,\,X\mapsto\widetilde{X}\,\,\,\text{with}\,\,\,q_*\circ\tau_{\theta}=\id_{\mathcal{V}(M)}.
\]In particular, $q_*:\mathfrak{aut}(P)\rightarrow\mathcal{V}(M)$ is surjective.
\end{lemma}

\begin{proof}
\,\,\,We first show that $\widetilde{X}\in\mathfrak{aut}(P)$:
\begin{align}
\widetilde{X}(p.g)&=(T_{p.g}(q)_{\mid H_{p.g}(P)})^{-1}X(q(p))=T(\sigma_g)(T_p(q)_{\mid H_p(P)})^{-1}X(q(p))\notag\\
&=T(\sigma_g)\widetilde{X}(p).\notag
\end{align}
The claim now follows from $\widetilde{(f.X)}(p)=f(q(p))\widetilde{X}(p)=(f.\widetilde{X})(p)$.
\end{proof}

We also show the converse:

\begin{proposition}\label{C(M)-linear map induces connections form}
For each $C^{\infty}(M)$-linear section $\tau$ of $q_*$, there exists a unique $\theta\in C(P)$ with $\tau=\tau_{\theta}$.
\end{proposition}

\begin{proof}
\,\,\,(i) Let $\tau:\mathcal{V}(M)\rightarrow\mathfrak{aut}(P)$ be a $C^{\infty}(M)$-linear cross section of $q_*$. Since $\tau$ is $C^{\infty}(M)$-linear, it defines over each open subset $U\subseteq M$ a $C^{\infty}(U)$-linear cross section of $q_*$. If $P_U$ is trivial, then we thus obtain a $C^{\infty}(U)$-linear map
\[\tau_U:\mathcal{V}(U)\rightarrow\mathfrak{aut}(P_U)\cong C^{\infty}(U,\mathfrak{g})\rtimes\mathcal{V}(U).
\]
This map is of the form
\[\tau_U(X)=(-A(X),X)
\]with a $C^{\infty}(U)$-linear map $A:\mathcal{V}(U)\rightarrow C^{\infty}(U,\mathfrak{g})$. This in particularly means that $A\in\Omega^1(U,\mathfrak{g})$ is a $\mathfrak{g}$-valued 1-form. In terms of vector fields on $U\times G$, we then have
\[\tau_U(X)(m,g)=(-A(X)(m).g,X(m))
\](cf. Construction \ref{horizontal lifts}). We conclude that $\tau_U$ coincides with the horizontal lift with respect to the connection 1-form
\[\theta_U=p^*_G\kappa_G+(\Ad\circ p_G)^{-1}.p^*_UA.
\](ii) Now, $\theta_U\in C(P_U)$ is a connection 1-form for which $\tau_U$ is the corresponding horizontal lift. From the uniqueness of $\theta_U$ we derive that for each open subset $V\subseteq U$ we have $\theta_V={\theta_U}_{\mid P_U}$. Thus, the collection of the $\theta_U$'s defines an element $\theta\in C(P)$ with $\tau_{\theta}=\tau$.
\end{proof}

\begin{corollary}\label{1:1 correspondence between section and connections}
If $M$ is paracompact, then the homomorphism $q_*:\mathfrak{aut}(P)\rightarrow\mathcal{V}(M)$ is surjective and we obtain a short exact sequence of Lie algebras
\[0\longrightarrow\mathfrak{gau}(P){\longrightarrow}\mathfrak{aut}(P)\stackrel{q_*}{\longrightarrow}\mathcal{V}(M)\longrightarrow 0.
\]Moreover, there is a 1:1 correspondence between connection 1-forms and $C^{\infty}(M)$-linear cross sections of $q_*$.
\end{corollary}

\begin{proof}
\,\,\,The assertions are a direct consequence of Lemma \ref{surj of q_*} and Proposition \ref{C(M)-linear map induces connections form}.
\end{proof}

\chapter{Noncommutative Vector Bundles}\label{NCVB}

The goal of this chapter is to prove a slightly more general version of Theorem \ref{1}
in the smooth context, i.e., to establish that the categories of finitely generated projective modules over $C^{\infty}(M)$ and the category of smooth vector bundles are equivalent for arbitrary manifolds $M$ of finite dimension. We further provide some insight in the theory of locally convex modules of locally convex algebras. All modules are assumed to be right modules.

\section{Projective Modules over Unital Algebras}\label{Projective Modules over Unital Algebras}

In this section we start with the notions relating to projective modules over unital algebras.

\begin{definition}\label{free module}{\bf(Free modules).}\index{Modules!Free}
Let $A$ be a unital algebra. An $A$-module $E$ is called \emph{free} if $E$ is isomorphic to a module direct sum of copies of the $A$-module $A$, i.e., if there exists an index set $I$ such that
\[E\cong\bigoplus_{i\in I} A_i\,\,\,\text{for}\,\,\,A_i=A.
\]
\end{definition}

\begin{definition}\label{Appendix B1}{\bf(Projective modules).}\index{Modules!Projective}
Let $A$ be a unital algebra. An $A$-module $E$ is called \emph{projective} if there exists an $A$-module $E'$ such that the ($A$-module) direct sum $E\oplus E'$ is free.
\end{definition}


\begin{example}\label{Appendix B2}
Every free $A$-module $E$ is projective. Indeed, this is a direct consequence of Definition \ref{Appendix B1}.
\end{example}

\begin{proposition}\label{Appendix B3}
Let $I$ denote an arbitrary index set and let $E=\bigoplus_{i\in I}E_i$ be the module direct sum of $A$-modules $E_i$. Then $E$ is projective if and only if all $E_i$ are projective.
\end{proposition}

\begin{proof}\,\,\,($``\Rightarrow"$) Let $i_0\in I$ and suppose that $E$ is projective. Then there exists an $A$-module $E'$ such that $E\oplus E'$ is free and the projectivity of $E_{i_0}$ immediately follows from the decomposition
\[E_{i_0}\oplus\left(\bigoplus_{i\neq i_0}E_i\oplus E'\right)=E\oplus E'.
\]

($``\Leftarrow"$) On the other hand, suppose that all $E_i$ are projective. Then for each $i\in I$ there exists an $A$-module $E'_i$ such that $E_i\oplus E'_i$ is free. If we now put $E':=\bigoplus_{i\in I}E'_i$, the claim easily follows from the fact that $E\oplus E'$ is free.



\end{proof}


\begin{proposition}\label{Appendix B4}
Let $E$ be an $A$-module. Then the following statements are equivalent:
\begin{itemize}
\item[\emph{(a)}]
$E$ is projective.
\item[\emph{(b)}]
There exists a free $A$-module $H$ and an element $p\in\End_A(H)$ with $p^2=p$ such that $E\cong p\cdot H$.
\end{itemize}
\end{proposition}
\begin{proof}
\,\,\,
(a) $\Rightarrow$ (b):
We choose $H:=E\oplus E'$ and $p$ to be the projection from $H$ to $E$ with kernel $E'$.

(b) $\Rightarrow$ (a):
Let $H$ be a free $A$-module and $p$ be an idempotent $A$-endomorphism of $H$ such that $E\cong p\cdot H$. If we define $E':=(\id_H-p)\cdot H$, then every element $h\in H$ can be uniquely written as 
\[h=(\id_H-p)\cdot h)+p\cdot h
\]and the projectivity of $E$ follows from $E\oplus E'\cong p\cdot H\oplus E'=H$.





\end{proof}
%

\begin{remark}\label{projective lifting property}{\bf(Projective lifting property).}\index{Projective!Lifting Property}
An $A$-module $E$ is said to have the \emph{projective lifting property} if for every surjective morphism $f:F\rightarrow G$ of $A$-modules and every $A$-linear map $g:E\rightarrow G$, there exists an $A$-linear map $h:E\rightarrow F$ such that $f\circ h=g$. This property is equivalent to $E$ being projective:

($``\Rightarrow"$) Indeed, let $(e_i)_{i\in I}\subseteq E$ be a set of generators of $E$. We choose a free $A$-module $H$ with basis $(h_i)_{i\in I}$. Next, we define a surjective morphism $f:H\rightarrow E$ of $A$-modules by $f(h_i)=e_i\,\,\,\text{for all}\,\,\,i\in I$. We then get the following short exact sequence
\[0\longrightarrow \ker f{\longrightarrow} H\stackrel{}{\longrightarrow}E\longrightarrow0
\]of $A$-modules. If $E$ satisfies the projective lifting property, then there exists an $A$-linear map $h:E\rightarrow  H$ such that $f\circ h=\id_{E}$ (take $g=\id_{E}$), i.e. $H\cong E\oplus\underbrace{\ker f}_{:=E'}$.

($``\Leftarrow"$) For the other direction we first observe that free modules satisfy the projective lifting property: In this case $h$ is determined by lifting the image of a basis. To see that all projective modules satisfy the lifting property, we extend $g$ to a map from a free module $E\oplus E'$ to $G$ and lift that.
\end{remark}

The following corollary will be crucial in the next chapter: 

\begin{corollary}\label{f.g.pr.mod}{\bf(Finitely generated projective modules).}\index{Modules!Finitely Generated Projective}
Let $A$ be a unital algebra and $E$ be an $A$-module. If $n\in\mathbb{N}$ and $\Idem_n(A)$\sindex[n]{$\Idem_n(A)$} denotes the set of idempotent $A$-linear maps of $A^n$, then the following statements are equivalent:
\begin{itemize}
\item[\emph{(a)}]
$E$ is finitely generated and projective.
\item[\emph{(b)}]
$E$ is isomorphic to a direct summand of $A^n$.
\item[\emph{(c)}]
There exists an element $p\in\Idem_n(A)$ such that $E\cong p\cdot A^n$.
\end{itemize}
\end{corollary}

\begin{proof}
\,\,\,(a) $\Rightarrow$ (b): Suppose $E$ is a finitely generated and projective $A$-module. Then there exists a finite set of generators ${e_1,\ldots,e_n}$ of $E$. Now, we just copy the first part of the equivalence of Remark \ref{projective lifting property} to see that $E$ is isomorphic to a direct summand of $A^n$. 

(b) $\Rightarrow$ (a): If $E$ is isomorphic to a direct summand of $A^n$, i.e., if there exists an $A$-module $E'$ with $E\oplus E'\cong A^n$, then $E$ is obviously projective by Definition \ref{Appendix B1}. To see that $E$ is also finitely generated, we choose a basis $(a_1,\ldots,a_n)$ of $A^n$ and let $p$ be the projection onto $E$ with kernel $E'$. Then $(p\cdot a_1,\ldots,p\cdot a_n)$ is a finite set of generators of $E$.

The equivalence of (b) and (c) follows exactly as in Proposition \ref{Appendix B4}.
\end{proof}

\section{The Theorem of Serre and Swan in the Smooth Category}\label{serre-swan smooth category}

Since vector subbundles will be crucial in the following, we start with recalling the definition of vector subbundles and provide some suitable conditions for checking that certain subspaces of a vector bundle are subbundles. We then present quite interesting results concerning vector bundles, which will be necessary for proving the \emph{Theorem of Serre and Swan} in the smooth setting. By a morphism of vector bundles (over a given manifold $M$), we shall always mean a vector bundle homomorphism which covers $\id_M$.

\begin{definition}{\bf(Vector subbundles).}\label{vector subbundle}\index{Bundles!Vector Sub-}
Given a vector bundle $(\mathbb{V},M,V,q)$, a vector subbundle of $\mathbb{V}$ is a submanifold $\mathbb{U}$ of $\mathbb{V}$ such that for each $m\in M$ the fibre $\mathbb{U}_m:=\mathbb{U}\cap\mathbb{V}_m$ is a linear subspace of $\mathbb{V}_m$ and $q_{|\mathbb{U}}:\mathbb{U}\rightarrow M$ is a vector bundle in its own right.
\end{definition}
%
%

\begin{remark}
We recall that if $N$ is a submanifold of a manifold $M$, then the tangent bundle $TN$ is a subbundle of the restricted tangent bundle \[TM_N:=q^{-1}(N)
\](cf. Example \ref{tangent bundle}).
\end{remark}

\begin{theorem}\label{im ker}{\bf(Vector subbundle criterion).}\index{Vector Subbundle Criterion}
Let $\mathbb{V}$ and $\mathbb{W}$ be vector bundles over a manifold $M$. Further, let $\phi:\mathbb{V}\rightarrow\mathbb{W}$ be a morphism of vector bundles such that the map $M\rightarrow \mathbb{N}_0$, $m\mapsto\rank\phi_m$ is locally constant. Then the following assertions hold:
\begin{itemize}
\item[\emph{(a)}]
$\ker\phi$ given by $(\ker\phi)_m=\ker\phi_m$ is a vector subbundle of $\mathbb{V}$.
\item[\emph{(b)}]
$\im\phi$ given by $(\im\phi)_m=\im\phi_m$ is a vector subbundle of $\mathbb{W}$.
\end{itemize}
\end{theorem}

\begin{proof}
\,\,\,A very nice proof can be found in [Hu08], Chapter 2, Proposition 3.5.
\end{proof}

\begin{corollary}\label{injective}
Let $\phi:\mathbb{V}\rightarrow\mathbb{W}$ be a morphism of vector bundles that is injective or, equivalently, $\phi_m$ is injective for all $m\in M$. Then $\im\phi$ is a vector subbundle of $\mathbb{W}$.
\end{corollary}

\begin{proof}
\,\,\,This follows from Theorem \ref{im ker} since $\phi$ has constant rank.
\end{proof}

\begin{corollary}\label{surjective}
Let $\phi:\mathbb{V}\rightarrow\mathbb{W}$ be a morphism of vector bundles that is surjective or, equivalently, $\phi_m$ is surjective for all $m\in M$. Then $\ker\phi$ is a vector subbundle of $\mathbb{V}$.
\end{corollary}

\begin{proof}
\,\,\,Again, this follows from Theorem \ref{im ker} since $\phi$ has constant rank.
\end{proof}

\begin{lemma}\label{orth. complement}
Let $\mathbb{V}$ be a vector bundle over a manifold $M$ and let $\mathbb{U}$ be a vector subbundle of $\mathbb{V}$. Further, let $g$ be a bundle metric on $\mathbb{V}$. Then the orthogonal complement $\mathbb{U}^{\perp}$ \emph{(}with respect to $g$\emph{)}, given by $(\mathbb{U}^{\perp})_m:=\mathbb{U}_m^{\perp}$, is also a vector subbundle of $\mathbb{V}$.
\end{lemma}

\begin{proof}
\,\,\,We define a surjective morphism $\phi:\mathbb{V}\rightarrow \mathbb{U}$ of vector bundles by projecting an element $e\in\mathbb{V}_m$ onto its image in $\mathbb{U}_m$ and use Corollary \ref{surjective} to see that the orthogonal complement $\mathbb{U}^{\perp}$ is a vector subbundle of $\mathbb{V}$.

\end{proof}

\begin{proposition}\label{s.e.s of VB splits}
Every short exact sequence 
\[0\longrightarrow \mathbb{V}'\stackrel{\psi}{\longrightarrow}\mathbb{V}\stackrel{\phi}{\longrightarrow}\mathbb{V}''\longrightarrow0
\]of vector bundles over a manifold $M$ splits.
\end{proposition}

\begin{proof}
\,\,\,We first choose a bundle metric $g$ on $\mathbb{V}$ and note that, by Corollary \ref{injective}, $\psi(\mathbb{V}')$ is a vector subbundle of $\mathbb{V}$. Then Lemma \ref{orth. complement} implies that $\mathbb{V}''':=\psi(\mathbb{V}')^{\perp}$ is also a vector subbundle of $\mathbb{V}$. Therefore, the map
\[\phi_{|\mathbb{V}'''}:\mathbb{V}'''\rightarrow \mathbb{V}''
\]defines an isomorphism of bundles. Finally, a right inverse of $\phi$ is given by
\[\chi:=(\phi_{|\mathbb{V}'''})^{-1}:\mathbb{V}''\rightarrow \mathbb{V}'''\subseteq \mathbb{V}.
\]
\end{proof}

\begin{remark}{\bf(Finite open cover property).}\label{finite open cover}\index{Finite Open Cover Property}
If $(\mathbb{V},M,V,q)$ is a vector bundle over a compact manifold $M$, then it is obviously true that there exists a \emph{finite} open cover $\{U_i\}_{1\leq i\leq m}$ of $M$ such that each $\mathbb{V}_{U_i}:=q^{-1}(U_i)$ is trivial. Note that this property is also true for \emph{arbitrary} manifolds: 

Indeed [Con93], Theorem 7.5.16 states that if $(\mathbb{V},M,V,q)$ is a vector bundle over a manifold $M$ of dimension $n$, there is a cover $\{U_i\}_{1\leq i\leq n+1}$ of $M$, such that $\mathbb{V}_{U_i}:=q^{-1}(U_i)$ is trivial. 
\end{remark}

\begin{proposition}\label{existence of trivialising complements}{\bf(Existence of trivializing complements).}\index{Trivializing Complements}
If $\mathbb{V}$ is a vector bundle over a manifold $M$, then we can find another vector bundle $\mathbb{V}'$ over $M$ such that the Whitney sum $\mathbb{V}\oplus \mathbb{V}'$ is a trivial vector bundle over $M$.
\end{proposition}

\begin{proof}
\,\,\,(i) By Remark \ref{finite open cover}, there is a finite open covering $\{U_i\}_{1\leq i\leq k}$ of $M$ such that there are $n=\dim V$ linearly independent local sections $s_{i1},\ldots,s_{in}$ in each $\Gamma(U_i,\mathbb{V})$. 

(ii) Let $\{t_i\}_{1\leq i\leq k}$ be a smooth partition of unity subordinate to $\{U_i\}_{1\leq i\leq k}$. For $1\leq i\leq k$ and $1\leq j\leq n$ we define
\[\sigma_{ij}:M\rightarrow \mathbb{V}\,\,\,\text{as}\,\,\,
\begin{cases}
t_is_{ij} &\text{on}\ U_i\\
0 &\text{otherwise}
\end{cases}
\]
and note that the vectors $\sigma_{ij}(m)$ span each fibre $\mathbb{V}_m$. We further put $N:=kn$ and define a map $\phi:M\times\mathbb{K}^N\rightarrow \mathbb{V}$ by
\[\phi(m,z):=\sum_{i,j}z_{ij}\sigma_{ij}(m).
\]A short observation shows that $\phi$ is a surjective bundle map. Thus, Corollary \ref{surjective} implies that $\mathbb{V}':=\ker\phi$ is a vector subbundle of $\mathbb{V}$ giving the exact sequence
\[0\longrightarrow \mathbb{V}'\longrightarrow M\times\mathbb{K}^N\stackrel{\phi}{\longrightarrow}\mathbb{V}\longrightarrow 0
\]of vector bundles. By Proposition \ref{s.e.s of VB splits} this sequence splits, yielding $\mathbb{V}\oplus \mathbb{V}'\cong M\times\mathbb{K}^N$.
\end{proof}

For the next proposition we recall that the space $\Gamma\mathbb{V}$ of smooth sections of a vector bundle $(\mathbb{V},M,V,q)$ carries the structure of a $C^{\infty}(M)$-module (cf. Definition \ref{sections}):

\begin{proposition}\label{canonical isomorphisms of A-modules}
If $\mathbb{V}$ and $\mathbb{V}'$ are vector bundles over a manifold $M$, then there are canonical isomorphisms of $C^{\infty}(M)$-modules:
\[\Gamma\mathbb{V}\oplus\Gamma\mathbb{V}'\cong\Gamma(\mathbb{V}\oplus \mathbb{V}')
\]and
\[\Gamma\mathbb{V}\otimes_{C^{\infty}(M)}\Gamma\mathbb{V}'\cong\Gamma(\mathbb{V}\otimes \mathbb{V}').
\]
\end{proposition}

\begin{proof}
\,\,\, We give a sketch for the second isomorphism, since the first isomorphism is quite obvious:
For $s\in\Gamma\mathbb{V},s'\in\Gamma\mathbb{V}'$ we provisionally write $s\odot s'$ for the section $m\mapsto s(m)\otimes s'(m)$ of the tensor product bundle $\mathbb{V}\otimes \mathbb{V}'$ over $M$. We further write $s\otimes s'$ for the element in $\Gamma\mathbb{V}\otimes_{C^{\infty}(M)}\Gamma\mathbb{V}'$ given by the tensor product of $C^{\infty}(M)$-modules. The desired map is then determined by mapping $s\otimes s'$ to $s\odot s'$. In fact, a complete proof of this proposition can be found in [GVF01], Chapter 2.2, Proposition 2.6.
\end{proof}

\begin{proposition}\label{V is a f.g.p. C-infty (M)-module}
Let $M$ be a manifold and $\mathbb{V}$ be a vector bundle over $M$. Then $\Gamma\mathbb{V}$ is a finitely generated projective $C^{\infty}(M)$-module.
\end{proposition}

\begin{proof}
\,\,\,We first choose a trivializing complement $\mathbb{V}'$ of $\mathbb{V}$ and let $N$ be the rank of $\mathbb{V}\oplus \mathbb{V}'$. By Proposition \ref{canonical isomorphisms of A-modules}, we get
\[\Gamma\mathbb{V}\oplus\Gamma\mathbb{V'}\cong\Gamma(\mathbb{V}\oplus \mathbb{V}')\cong\Gamma(M\times\mathbb{K}^N)\cong C^{\infty}(M)^N.
\]Now, Corollary \ref{f.g.pr.mod} implies that $\Gamma\mathbb{V}$ is a finitely generated and projective $C^{\infty}(M)$-module.
\end{proof}

\begin{lemma}\label{smoothness of induced map}
Let $M$ be a manifold and $V$ a finite-dimensional vector space. Further, let $f_1,\ldots,f_n\in C^{\infty}(M,V)$ be a finite set of smooth $V$-valued functions on $M$. Then the map
\[F:M\times\mathbb{K}^n\rightarrow V,\,\,\,(m,\lambda_1,\ldots,\lambda_n)\mapsto\sum_{i=1}^n\lambda_if_i(m)
\]is smooth.
\end{lemma}

\begin{proof}
\,\,\,Obviously the map 
\[b:V^n\times\mathbb{K}^n\rightarrow V,\,\,\,(v_1,\ldots,v_n,\lambda_1,\ldots,\lambda_n)\mapsto\sum_{i=1}^n\lambda_iv_i
\]is smooth. Moreover, [NeWa07], Proposition I.2 implies the smoothness of the evaluation map
\[\ev_M:C^{\infty}(M,V)\times M\rightarrow V,\,\,\,(f,m)\mapsto f(m).
\]Hence, the map
\[F=b\circ\left((\ev_M(f_1),\ldots,\ev_M(f_n))\times\id_{\mathbb{K}^n}\right)
\]is smooth as a composition of smooth maps.
\end{proof}

\begin{proposition}\label{morphism of VB-morphism of f.g.p. modules}
Let $M$ be a manifold and $\mathbb{V},\mathbb{W}$ be two vector bundles over $M$. Then the following two assertions hold:
\begin{itemize}
\item[\emph{(a)}]
Every morphism $\phi:\mathbb{V}\rightarrow\mathbb{W}$ of vector bundles over $M$ induces a $C^{\infty}(M)$-linear map $\Gamma\phi$ between the corresponding spaces of sections $\Gamma\mathbb{V}$ and $\Gamma\mathbb{W}$.
\item[\emph{(b)}]
 Conversely, every $C^{\infty}(M)$-linear map $\Phi:\Gamma\mathbb{V}\rightarrow\Gamma\mathbb{W}$ induces a morphism of vector bundles $\phi:\mathbb{V}\rightarrow \mathbb{W}$ which covers $\id_M$, such that $\Gamma\phi=\Phi$.
\end{itemize}
\end{proposition}

\begin{proof}
\,\,\,(a) If $\phi:\mathbb{V}\rightarrow\mathbb{W}$ is a morphism of vector bundles over $M$ which covers $\id_M$, then we obtain a $C^{\infty}(M)$-linear map $\Gamma\phi:\Gamma\mathbb{V}\rightarrow\Gamma\mathbb{W}$ by defining
\[(\Gamma\phi)(s)(m):=\phi(s(m)).
\] 

(b) Conversely, let $\Phi:\Gamma\mathbb{V}\rightarrow\Gamma\mathbb{W}$ be a $C^{\infty}(M)$-linear map. We first choose a trivializing complement $\mathbb{V}'$ of $\mathbb{V}$ (cf. Proposition \ref{existence of trivialising complements}) and define a $C^{\infty}(M)$-linear map by
\[\widehat{\Phi}:=\Phi\oplus 0:\Gamma\mathbb{V}\oplus\Gamma\mathbb{V}'\rightarrow\Gamma\mathbb{W}.
\]Since $\mathbb{V}\oplus \mathbb{V}'$ is trivial (let's say of rank $N$), there exist smooth linear independent sections $s_1,\ldots,s_N\in\Gamma(\mathbb{V}\oplus \mathbb{V}')$. In particular, we can write each $e\in \mathbb{V}_m$ uniquely as $e=\sum^N_{i=1}\lambda_is_i(m)$. From this we conclude that the map 
\[\phi:\mathbb{V}\rightarrow \mathbb{W},\,\,\,\phi(e):=\sum^N_{i=1}\lambda_i\widehat{\Phi}(s_i)(m)
\]defines a morphism of vector bundles: Indeed, $\phi$ is fibrewise linear by definition. Its smoothness follows from Lemma \ref{smoothness of induced map}, because each 
$\widehat{\Phi}(s_i)$ is locally (let's say in an open neighbourhood $U$ of a point in $M$) given by a smooth function in $C^{\infty}(U,W)$. A short calculation now shows that 
\[(\Gamma\phi)(s)(m)=\widehat{\Phi}(s)(m)=\Phi(s)(m)
\]holds for all $s\in\Gamma\mathbb{V}$ and $m\in M$. 
Hence, $\Gamma\phi=\Phi$ as desired.
\end{proof}


\begin{theorem}{\bf(Serre--Swan).}\label{Serre-Swan}\index{Theorem!of Serre--Swan}
Let $M$ be a manifold. Then the functor
\[\mathbb{V}\mapsto\Gamma(\mathbb{V})\,\,\,\text{and}\,\,\,\phi\mapsto\Gamma\phi
\]defines an equivalence between the category of vector bundles over $M$ and the category of finitely generated projective modules over $C^{\infty}(M)$.
\end{theorem}

\begin{proof}
\,\,\,In view of Proposition \ref{V is a f.g.p. C-infty (M)-module} and Remark \ref{morphism of VB-morphism of f.g.p. modules}, it only remains to prove that every finitely generated projective $C^{\infty}(M)$-module $E$ is of the form $\Gamma\mathbb{V}$ for some vector bundle $\mathbb{V}$ over $M$: 

(i) As $E$ is finitely generated and projective, we find some $N\in\mathbb{N}$ and $p\in\Idem_N(C^{\infty}(M))$ such that $E\cong p\cdot C^{\infty}(M)^N$ (cf. Proposition \ref{f.g.pr.mod}). By Proposition \ref{morphism of VB-morphism of f.g.p. modules} (b), the endomorphism $p$ induces a bundle map 
\[\phi:M\times\mathbb{K}^N\rightarrow M\times\mathbb{K}^N
\]with $\Gamma\phi=p$. 

(ii) Next, we claim that the map $M\rightarrow \mathbb{N}_0$, $m\mapsto\rank\phi_m$ is locally constant: Since $\phi=\phi^2$, the map $\id-\phi$ also is an idempotent bundle map. Therefore, both maps $M\rightarrow \mathbb{N}_0$, $m\mapsto\rank\phi_m$ and $M\rightarrow \mathbb{N}_0$, $m\mapsto\rank(1-\phi_m)$ are lower-semicontinuous. Furthermore, 
\[\rank(\phi_m)+\rank(1-\phi_m)=N
\]shows that $M\rightarrow \mathbb{N}_0$, $m\mapsto\rank\phi_m$ is also upper-semicontinuous, hence continuous and thus locally constant. Now, Theorem \ref{im ker} implies that $\mathbb{V}:=\phi(M\times\mathbb{K}^N)$ is a subbundle of $M\times\mathbb{K}^N$. In particular, this is the desired bundle, since
\[\Gamma\mathbb{V}=\{\phi\circ s:\,s\in\Gamma(M\times\mathbb{K}^N)\}=\im p\cong E.
\]
\end{proof}

\begin{remark}{\bf(Noncommutative vector bundles).}
The previous theorem justifies to consider finitely generated modules over unital algebras as \emph{noncommutative vector bundles}.
\end{remark}

\section{Topological Aspects of Finitely Generated Projective Modules}

In this short section we show how to topologize a finitely generated projective module of a unital locally convex algebra $A$. In particular, we will see that each finitely generated projective module of a unital Fr\'{e}chet algebra $A$ admits a unique topology of a Fr\'{e}chet $A$-module. We further present some results on the automatic continuity of algebraic morphism between modules of Fr\'{e}chet algebras.

\begin{definition}{\bf(Locally convex modules).}\index{Modules!Locally Convex}
Let $A$ be a unital locally convex algebra and $E$ be a locally convex space, which is at the same time an algebraic $A$-module. Then $E$ is called a \emph{locally convex $A$-module} if the bilinear map
\[\rho:E\times A\rightarrow E,\,\,\,(s,a)\mapsto \rho(s,a):=s.a
\]defining the $A$-module structure on $E$ is continuous.
\end{definition}

\begin{construction}{\bf(Topologizing modules).}\label{topologizing modules}\index{Modules!Topologizing}
Let $A$ be a unital locally convex algebra and $E$ be a finitely generated projective $A$-module. The following procedure turns $E$ into a locally convex $A$-module:

By Corollary \ref{f.g.pr.mod}, there exists an element $p\in\Idem_n(A)$ such that $E\cong p\cdot A^n$. Since $p\in\Idem_n(A)$, the $A$-linear map $\id-p:A^n\rightarrow A^n$ is continuous with kernel $p\cdot A^n$. Hence, $p\cdot A^n$ carries the structure of a closed locally convex $A$-submodule of $A^n$. Therefore, we can endow $E$ with the locally convex topology for which $E\cong p\cdot A^n$ becomes an isomorphism of locally convex $A$-modules. Additionally we note that if $A$ is complete, then the same holds for $E$ as $p\cdot A^n$ is a closed subspace of $A^n$.
\end{construction}

\begin{proposition}{\bf(Open mapping theorem.)}\label{open mapping theorem}\index{Open Mapping Theorem}
Every surjective continuous linear map $f:E\rightarrow F$ between Fr\'{e}chet spaces is open, i.e., it maps open sets into open sets.
\end{proposition}

\begin{proof}
\,\,\,A reference for this statement is, for example, [Sch99], Chapter III, Section 2.2.
\end{proof}

\begin{proposition}\label{modules of frechet algebras}
Let $A$ be a unital Fr\'{e}chet algebra. Then the following assertions hold:
\begin{itemize}
\item[\emph{(a)}]
If $E$ is a finitely generated Fr\'{e}chet $A$-module and $F$ is a locally convex $A$-module, then each \emph{(}algebraic\emph{)} morphism $f:E\rightarrow F$ of $A$-modules is continuous.
\item[\emph{(b)}]
Each finitely generated projective $A$-module $E$ admits a unique topology of Fr\'{e}chet $A$-module. 
\end{itemize}
\end{proposition}

\begin{proof}
\,\,\,(a) We first note that there exist $n\in\mathbb{N}$ and a surjective morphism $p:A^n\rightarrow E$ of Fr\'{e}chet $A$-modules defined by some finite generating set of $E$. By Proposition \ref{open mapping theorem}, the map $p$ is open. Now, the claim follows from the fact that any morphism $A^n\rightarrow F$ of locally convex $A$-modules is continuous.

(b) In view of Construction \ref{topologizing modules}, $E$ carries the structure of a Fr\'{e}chet $A$-module. That this Fr\'{e}chet topology is unique is a consequence of part (a).
\end{proof}

\begin{remark}{\bf(Idempotent maps).}
(a) Each idempotent map $p:E\rightarrow E$ on a locally convex space $E$ is open onto its image. Indeed, if $U$ is a 0-neighbourhood of $E$, then we easily conclude that $U\cap p(E)\subseteq p(U)$.

(b) Suppose that $A$ is a unital locally convex algebra. In view of part (a) of this remark, Proposition \ref{modules of frechet algebras} remains true for locally convex finitely generated projective $A$-modules, which are topologically isomorphic to a direct summand of a free $A$-module.
\end{remark}

\begin{corollary}\label{space of section is frechet-module}
Let $\mathbb{V}$ be a vector bundle over a manifold $M$. Then the corresponding space of sections $\Gamma\mathbb{V}$ carries a unique topology of Fr\'{e}chet $C^{\infty}(M)$-module.
\end{corollary}

\begin{proof}
\,\,\,The claim follows from Proposition \ref{V is a f.g.p. C-infty (M)-module} and Proposition \ref{modules of frechet algebras} (b).
\end{proof}

\begin{corollary}\label{sections of an associated vector bundle top}
Let $(P,M,G,q,\sigma)$ be a principal bundle. Further, let $(\pi,V)$ be a finite-dimensional representation of $G$ defining the associated bundle $\mathbb{V}:=P\times_{\pi} V$ over $M$ \emph{(}cf. Construction \ref{associated vector bundle} \emph{)}. If we write
\[C^{\infty}(P,V)^G:=\{f:P\rightarrow V:\,(\forall g\in G)\,f(p.g)=\pi(g^{-1}).f(p)\}
\]for the space of equivariant smooth functions, then the map
\[\Psi_{\pi}:C^{\infty}(P,V)^G\rightarrow\Gamma\mathbb{V},\,\,\,\text{defined by}\,\,\,\Psi(f)(q(p)):=[p,f(p)],
\]is a topological isomorphism of $C^{\infty}(M)$-modules.
\end{corollary}

\begin{proof}
\,\,\,The assertion follows from Proposition \ref{sections of an associated vector bundle}, Proposition \ref{V is a f.g.p. C-infty (M)-module} and Proposition \ref{modules of frechet algebras} (a).
\end{proof}

\begin{lemma}\label{top change of base}
Let $A$ be a unital Fr\'{e}chet algebra and $E$ be a finitely generated projective $A$-module. If $\varphi:A\rightarrow B$ is a morphism of unital Fr\'{e}chet algebras, then $E\otimes_A B$ is a finitely generated projective Fr\'{e}chet $B$-module. In particular, we conclude that
\[E\otimes_A B=E\widehat{\otimes}_A B,
\]where $\widehat{\otimes}_A$ denotes the completion of the projective tensor product topology of $A$-modules \emph{(}cf. Definition \ref{tensor product of $A$ modules}\emph{)}.
\end{lemma}

\begin{proof}
\,\,\,The claim is a consequence of Proposition \ref{modules of frechet algebras} (b) and Remark \ref{mod and ringhom I}.
\end{proof}

\begin{example}
Let $(\mathbb{V},M,V,q)$ be a vector bundle and $U$ an open subset of $M$.
We recall from Corollary \ref{space of section is frechet-module} that $\Gamma\mathbb{V}$ is a finitely generated projective Fr\'{e}chet $C^{\infty}(M)$-module.  Therefore, Lemma \ref{top change of base} implies that $\Gamma\mathbb{V}\otimes_{C^{\infty}(M)}C^{\infty}(U)$ is a finitely generated projective Fr\'{e}chet $C^{\infty}(U)$-module with respect to the restriction map $r:C^{\infty}(M)\rightarrow C^{\infty}(U)$.
\end{example}
 
In fact, the following proposition tells us that $\Gamma\mathbb{V}\otimes_{C^{\infty}(M)}C^{\infty}(U)$ is isomorphic to the space $\Gamma\mathbb{V}_U$ of sections of the ``restricted" vector bundle $\mathbb{V}_U:=q^{-1}(U)$:

\begin{proposition}\label{gamma V_U=gamma VC(U)}
Let $(\mathbb{V},M,V,q)$ be a vector bundle and $U$ an open subset of $M$. Then $\Gamma\mathbb{V}\otimes_{C^{\infty}(M)}C^{\infty}(U)$ and $\Gamma\mathbb{V}_U$ are isomorphic as Fr\'{e}chet $C^{\infty}(U)$-modules.
\end{proposition}

\begin{proof}
\,\,\,(i) In view of Proposition \ref{V is a f.g.p. C-infty (M)-module} and Corollary \ref{f.g.pr.mod}, we find some $n\in\mathbb{N}$ and a matrix $p\in\Idem_n(C^{\infty}(M))$ such that $\Gamma\mathbb{V}\cong p\cdot C^{\infty}(M)^n$. This implies that $\Gamma\mathbb{V}_U\cong p_U\cdot C^{\infty}(U)^n$, where $p_U$ denotes the matrix in $\Idem_n(C^{\infty}(U))$ obtained from $p$ by restricting each of its entries to $U$. 

(ii) On the other hand, Remark \ref{mod and ringhom I} applied to the restriction map $r:C^{\infty}(M)\rightarrow C^{\infty}(U)$, implies that
\[\Gamma\mathbb{V}\otimes_{C^{\infty}(M)}C^{\infty}(U)\cong p_U\cdot C^{\infty}(U)^n.
\]Hence, we conclude that
\[\Gamma\mathbb{V}\otimes_{C^{\infty}(M)}C^{\infty}(U)\cong\Gamma\mathbb{V}_U\label{iso of C-infty (M)-module}
\]holds as $C^{\infty}(U)$-modules. That this last isomorphism is actually an isomorphism of Fr\'{e}chet $C^{\infty}(U)$-modules is a consequence of Corollary \ref{space of section is frechet-module}, Lemma \ref{top change of base} and Proposition \ref{modules of frechet algebras}\,(a).
\end{proof}

\chapter{Smooth Localization in Noncommutative Geometry}\label{SLNG}

The idea of localization comes from algebraic geometry: Given a point $x$ in some affine variety $X$, one likes to investigate the nature of $X$ in an arbitrarily small neighbourhood of $x$ in the Zariski topology. Now, small neighbourhoods of $x$ in $X$ correspond to large algebraic subsets $Y$. For example, let $Y$ be the zero set of some algebraic function $f$ on $X$, which does not vanish at $x$. Then the affine ring $\mathbb{K}[(X\backslash Y)]$ is obtained from $\mathbb{K}[X]$ by adjoining a multiplicative inverse for $f$, i.e., taking the coproduct of $\mathbb{K}[X]$ with the free ring in one generator and dividing out the ideal $I_f:=\langle 1-ft\rangle$; this is called \emph{inverting} $f$. Nevertheless, this construction is not valid for the algebra of smooth functions on some manifold $M$. In this chapter we present an appropriate method of localizing (possibly noncommutative) algebras in a smooth way. We start with a little reminder on the classical setting.\sindex[n]{$\mathbb{K}[X]$}

\section{Some Facts on ``Classical" Localization}

\begin{construction}\label{S/R}{\bf(Localization of a ring).}\index{Localization!of a Ring}
Let $R$ be a commutative unital ring. A subset $S$ is called \emph{multiplicatively closed}\index{Multiplicatively Closed}, if $ab$ is in $S$ for all $a,b\in S$. For a multiplicatively closed subset $S$ we define an equivalence relation on $R\times S$ by
\[(r,s)\sim(r',s')\Longleftrightarrow u(rs'-r's)=0\,\,\,\text{for some} \,\,\,u\in S.
\]The set of equivalence classes is denoted by $S^{-1}R$\sindex[n]{$S^{-1}R$} and the class of $(r,s)\in R\times S$ is denoted by $\frac{r}{s}$. Addition and multiplication is defined in the obvious way:
\[\frac{r}{s}+\frac{r'}{s'}:=\frac{rs'+r's}{ss'},\,\,\,\frac{r}{s}\frac{r'}{s'}=\frac{rr'}{ss'}.
\]It is easily checked that these operations are well defined and that $S^{-1}R$ becomes a ring with $0$ and $1$ given by $\frac{0}{a}$, respectively $\frac{a}{a}$ for any $a\in S$. Moreover, we have a natural ring homomorphism
\[j:R\rightarrow S^{-1}R,\,\,\,r \mapsto\frac{rs}{s}.
\]The ring $S^{-1}R$ is called the \emph{localization} of $R$ with respect to the multiplicatively closed subset $S$.
\end{construction}

\begin{proposition}\label{univ. prop of loc.}
The pair $(S^{-1}R,j)$ has the following universal property: every element $s\in S$ maps to a unit in $S^{-1}R$, and any other homomorphism $\varphi:R\rightarrow R'$ with $\varphi(S)\subseteq(R')^{\times}$ factors uniquely through $j$, i.e., there exists a unique homomorphism $\psi:S^{-1}R\rightarrow R'$ with $\psi\circ j=\varphi$.
\end{proposition}

\begin{proof}
\,\,\,We first show the uniqueness: If $\psi$ exists, then the assumption and $\psi(s)\cdot\psi(\frac{r}{s})=\psi(r)$ implies that $\psi(\frac{r}{s})=\varphi(r)\cdot\varphi(s)^{-1}$. Thus $\psi$ is unique.

To show the existence we define $\psi(\frac{r}{s}):=\varphi(r)\cdot\varphi(s)^{-1}$. A short calculation shows that this map is well-defined and a homomorphism of rings.
\end{proof}

\begin{example}\label{1,a,a^2}
Let $R$ be a commutative unital ring and $a\in R$. Then $S_a:=\{1,a,a^2,\ldots\}$ is a multiplicatively closed subset of $R$ and we define $R_a:=S^{-1}_aR$\sindex[n]{$R_a$}. Thus, every element of $R_a$ can be written in the form $\frac{r}{a^n}$, $r\in R$ and $n\in\mathbb{N}$. Moreover,
\[\frac{r}{a^n}=\frac{r'}{a^m}\Longleftrightarrow a^N(ra^m-r'a^n)=0\,\,\, \text{for some}\,\,\,N\in\mathbb{N}.
\]If $a$ is nilpotent, then $R_a=0$, and if $R$ is an integral domian with field of fractions $F$ and $a\neq 0$, then $R_a$ is the subring of $F$ of elements of the form $\frac{r}{a^n}$, $r\in R$, $n\in\mathbb{N}$.
\end{example}

\begin{example}\label{prime ideal}
Let $R$ be a commutative unital ring and $\mathfrak{p}$ be a prime ideal of $R$. Then $S_{\mathfrak{p}}:=R\backslash\mathfrak{p}$ is a multiplicatively closed subset of $R$, and we let $R_{\mathfrak{p}}:=S^{-1}_{\mathfrak{p}}R$. Thus each element of $R_{\mathfrak{p}}$ can be written in the form $\frac{r}{s}$, $s\notin\mathfrak{p}$, and 
\[\frac{r}{s}=\frac{r'}{s'}\Longleftrightarrow u(rs'-r's)=0,\,\,\,\text{for some} \,\,\,u\notin \mathfrak{p}.
\]The subset $\mathfrak{m}:=\{\frac{r}{s}:\,r\in\mathfrak{p}, s\notin\mathfrak{p}\}$ is a maximal ideal in $R_{\mathfrak{p}}$, and it is the only maximal ideal, i.e., $R_{\mathfrak{p}}$ is a local ring. When $R$ is an integral domian with field of fractions $F$, $R_{\mathfrak{p}}$ is the subring of $F$ consisting of elements expressible in the form $\frac{r}{s}, r\in R, s\notin\mathfrak{p}$.
\end{example}

\begin{proposition}\label{isom of localization}
For any ring $R$ and $a\in R$, the map
\[R_{\{a\}}:=R[t]/\langle 1-at\rangle\rightarrow R_a,\,\,\,\left[\sum r_kt^k\right]\mapsto\sum\frac{r_k}{a^k}
\]defines an isomorphism of rings.
\end{proposition}

\begin{proof}
\,\,\,If $a=0$, both rings are zero, and so we may assume $a\neq 0$. In the ring $R_{\{a\}}$, $1=at$, and so $a$ is a unit. Let $\alpha:R\rightarrow R'$ be a morphism of rings such that $\alpha(a)$ is a unit in $R'$. The morphism 
\[R[t]\rightarrow R',\,\,\,\sum r_kt^k\mapsto\sum\alpha(r_k)\alpha(a)^{-k}
\]factors through $R_{\{a\}}$ because $1-at\mapsto 1-\alpha(a)\alpha(a)^{-1}=0$, and because $\alpha(a)$ is a unit in $R'$, this is the unique extension of $\alpha$ to $R_{\{a\}}$. Therefore, $R_{\{a\}}$ has the same universal property as $R_a$, and so the map 
\[R_{\{a\}}:=R[t]/\langle 1-at\rangle\rightarrow R_a,\,\,\,\left[\sum r_kt^k\right]\mapsto\sum\frac{r_k}{a^k}
\]defines an isomorphism of rings.
\end{proof}

Let $X$ be an affine variety. For every nonzero element $f$ in the coordinate ring $\mathbb{K}[X]$ we define
\[X_f:=\{x\in X:\,f(x)\neq0\}.
\]\sindex[n]{$X_f$}Then we have a natural map $\mathbb{K}[X]_f\rightarrow\mathcal{O}_X(X_f)$, which maps a quotient $\frac{g}{f^n}$ in $\mathbb{K}[X]_f$ to the function $X_f\rightarrow\mathbb{K}$ mapping the point $x$ to $\frac{g(x)}{f^n(x)}$. It turns out that this map is an isomorphism:

\begin{proposition}\label{K(X)_f=O(X_f)}
Let $X$ be an affine variety. Then the map 
\[\mathbb{K}[X]_f\rightarrow \mathcal{O}_X(X_f),\,\,\,\frac{g}{f^n}\mapsto\left(x\mapsto\frac{g(x)}{f^n(x)}\right)
\]is an isomorphism for every element $f\in\mathbb{K}[X]$.
\end{proposition}

\begin{proof}\,\,\,We just show the injectivitiy; for a complete proof we refer, for example, to [Har87], Chapter II, Proposition 2.2 (b): We assume that a quotient $\frac{g}{f^n}$ is mapped to zero in $\mathcal{O}_X(X_f)$. Then $g(x)=0$ for $x\in X_f$. However, then $(fg)(x)=0$ for all $x\in X$. This exactly means $fg=0$ in $\mathbb{K}[X]$. Hence, $\frac{g}{f^n}=0$ in $\mathbb{K}[X]_f$.
\end{proof}

\section{Smooth Localization of Noncommutative Spaces}\label{SLNS}

Now, from the viewpoint of noncommutative differential geometry, one might ask whether Proposition \ref{K(X)_f=O(X_f)} is still true for the algebra of smooth function on a manifold $M$. Unfortunately, if $f$ is a nonzero element in $C^{\infty}(M)$ and 
\[M_f:=\{m\in M:\,f(m)\neq 0\},
\]\sindex[n]{$M_f$}then not every smooth function $g:M_f\rightarrow\mathbb{K}$ is of the form $\frac{h}{f^n}$ for $h\in C^{\infty}(M)$ and some $n\in\mathbb{N}$.
In this section we provide a construction which will fix this problem. For the following propositions we recall the smooth compact open topology for smooth vector-valued function spaces of Definition \ref{smooth compact open topology} and the projective tensor product topology of Definition \ref{proj. tensor product III}.

\begin{proposition}\label{C(M,A) loc. con. algebra}
If $M$ is a manifold and $A$ a locally convex algebra, then $C^{\infty}(M,A)$ is a locally convex algebra.
\end{proposition}

\begin{proof}
\,\,\,To prove the claim we just have to verify that the multiplication in $C^{\infty}(M,A)$ is continuous:

(i) For this we first note that the tangent space $TA$ of $A$ carries a natural locally convex algebra structure given by the tangent functor $T$, i.e., defined by 
\[(a,v)(a',v'):=(aa',av'+va').
\]If $m_A:A\times A\rightarrow A$ is the multiplication of $A$, then $T(m_A):TA\times TA\cong T(A\times A)\rightarrow TA$ is the multiplication of $TA$. Iterating this process, we obtain a locally convex algebra structure on $T^nA$ for each $n\in\mathbb{N}$.

(ii) For elements $f,g\in C^{\infty}(M,A)$ the functoriality of $T$ implies that
\[T^n(fg)=T^n(m_A\circ(f,g))=T^n(m_A)\circ T^n(f,g)=T^n(m_A)\circ(T^n f,T^n g)=T^n f\cdot T^n g.
\]In particular, we conclude that the embedding
\[C^{\infty}(M,A)\hookrightarrow\prod_{n\in\mathbb{N}_0}C(T^nM,T^nA),\,\,\,f\mapsto(T^nf)_{n\in\mathbb{N}_0},
\]is a morphism of algebras. By Proposition \ref{C(X,A) loc. con. algebra}, the multiplication on the product algebra on the right is continuous, which implies the continuity of the multiplication on $C^{\infty}(M,A)$.
\end{proof}

\begin{proposition}\label{proj. tensor product for algebras}
If $M$ is a manifold and $A$ is a complete locally convex algebra, then there is a unique locally convex algebra structure on the locally convex space $C^{\infty}(M)\widehat{\otimes}A$ such that
\[C^{\infty}(M)\widehat{\otimes}A\cong C^{\infty}(M,A)
\]as \emph{(}complete\emph{)} locally convex algebras.
\end{proposition}

\begin{proof}
\,\,\,According to Proposition \ref{AoB as l.c. algebra}, $C^{\infty}(M)\otimes A$ is a locally convex algebra. Thus, Corollary \ref{completion of top rings II} applied to the multiplication map of $C^{\infty}(M)\otimes A$ implies that $C^{\infty}(M)\widehat{\otimes}A$ is also a locally convex algebra. If $\widehat{m}$ denotes the multiplication map of $C^{\infty}(M)\widehat{\otimes}A$ and $m$ the multiplication map of the locally convex algebra $C^{\infty}(M,A)$ (cf. Proposition \ref{C(M,A) loc. con. algebra}), then it remains to verify the identity
\begin{align}
\Phi\circ\widehat{m}=m\circ(\Phi\times\Phi),\label{verify identity}
\end{align}
where 
\[\Phi:C^{\infty}(M)\widehat{\otimes}A\rightarrow C^{\infty}(M,A)
\]denotes the isomorphism of locally convex spaces stated in Theorem \ref{proj. tensor product IV} (b): In fact, an easy observation shows that (\ref{verify identity}) holds on $C^{\infty}(M)\otimes A$ and thus the claim follows from the principle of extension of identities.
\end{proof}

The next definition is crucial for the aim of the thesis since it is the beginning of a smooth localization method:

\begin{definition}\label{sm. loc.}{\bf(Smooth localization).}\index{Localization!Smooth}
For $n\in\mathbb{N}$ and a unital locally convex algebra $A$ we write
\[A\{t_1,\ldots,t_n\}:=C^{\infty}(\mathbb{R}^n,A)
\]\sindex[n]{$A\{t_1,\ldots,t_n\}$}\sindex[n]{$C^{\infty}(\mathbb{R}^n,A)$}for the unital locally convex algebra of smooth $A$-valued functions on $\mathbb{R}^n$ (cf. Proposition \ref{C(M,A) loc. con. algebra}). We further define for each $1\leq i\leq n$ and $a\in A$ a smooth $A$-valued function on $\mathbb{R}^n$ by 
\[f^i_a:\mathbb{R}^n\rightarrow A,\,\,\,(t_1,\ldots,t_n)\mapsto 1_A-t_ia.
\]If $a_1,\ldots,a_n\in A$, then we write $I_{a_1,\ldots,a_n}:=\overline{\langle f^1_{a_1},\ldots,f^n_{a_n}\rangle}$\sindex[n]{$I_{a_1,\ldots,a_n}$} for the closure of the two-sided ideal generated by these functions. Finally, we write
\[A_{\{a_1,\ldots,a_n\}}:=A\{t_1,\ldots,t_n\}/I_{a_1,\ldots,a_n}
\]for the corresponding locally convex quotient algebra and
\[\pi_{\{a_1,\ldots,a_n\}}:A\{t_1,\ldots,t_n\}\rightarrow A_{\{a_1,\ldots,a_n\}}
\]for the corresponding continuous quotient homomorphism. The algebra $A_{\{a_1,\ldots,a_n\}}$\sindex[n]{$A_{\{a_1,\ldots,a_n\}}$} is called the (\emph{smooth}) \emph{localization} of $A$ with respect to $a_1,\ldots,a_n$.
\end{definition}

\begin{remark}
For $n\in\mathbb{N}$ and a complete locally convex algebra $A$, Proposition \ref{proj. tensor product for algebras}, applied to $M=\mathbb{R}^n$, implies that
\[C^{\infty}(\mathbb{R}^n)\widehat{\otimes}A\cong C^{\infty}(\mathbb{R}^n,A)
\]as locally convex algebras. This result corresponds to the classical picture in commutative algebra of adjoining $n$ indeterminate elements to a ring $R$ by taking the coproduct of $R$ and the free $\mathbb{K}$-algebra in $n$ generators: 
\[\mathbb{K}[t_1,\ldots,t_n]\otimes R\cong R[t_1,\ldots,t_n].
\]In particular, the algebra $C^{\infty}(\mathbb{R}^n,A)$ may be thought of the outcome of adjoining $n$ indeterminates to the algebra $A$ in a smooth way.
\end{remark}

\begin{remark}\label{rem to sm. loc.}
Unlike in ordinary commutative algebra, we do, not even in the case $n=1$, have an explicit description of the elements of $A_{\{a_1,\ldots,a_n\}}$. In particular, the construction in Definition \ref{sm. loc.} is, in general, quite different from the one in Example \ref{1,a,a^2}; for example, if one localizes an algebra $A$ with respect to a non-zero quasi-nilpotent element $a$, i.e., an element with $\spec(a)=\{0\}$ which is not nilpotent, then $A_{\{a\}}=0$ since the function $f_a$ is invertible in $A\{t\}$. On the other hand we have $0\neq\frac{a}{1}\in A_a$. We will point out another difference in Corollary \ref{C(M)_f=C(M_f)}. Nevertheless, by Proposition \ref{univ. prop of loc.}, we have a canonical homomorphism from $A_a$ to $A_{\{a\}}$.
\end{remark}

\begin{remark}\label{inverting 1 and 0} 
Let $A$ be an arbitrary (possibly noncommutative) unital algebra. Then localizing $A$ with respect to 0 leads to the zero-algebra, i.e.,
\[A_0={\bf0}\,\,\,\text{and}\,\,\,A_{\{0\}}={\bf0}.
\]On the other hand, if $a$ in $A$ is an invertible element, then localization of $A$ with respect to $a$ changes nothing, i.e., there is a canonical isomorphism between $A$ and $A_a$. This is not clear at all for the smooth localization $A_{\{a\}}$.
\end{remark}


\begin{proposition}\label{inverting 1}
Let $A$ be a complete unital locally convex algebra. Then the map
\[\ev_{1}:A\{t\}\rightarrow A,\,\,\,f\mapsto f(1)
\]\sindex[n]{$\ev_{1}$} is a surjective morphism of \emph{(}complete\emph{)} unital locally convex algebras with kernel 
\[I_{1_A}:=(t-1)\cdot A\{t\}.
\]
\end{proposition}

\begin{proof}
\,\,\,(i) According to [NeWa07], Proposition I.2, the evaluation map
\[\ev_{\mathbb{R}}:C^{\infty}(\mathbb{R},A)\times\mathbb{R}\rightarrow A,\,\,\,(f,r)\mapsto f(r)
\]is smooth. In particular, the map $\ev_{1}$ is continuous. Further, a short observation shows that $\ev_{1}$ is surjective and a homomorphism of algebras. Therefore, $\ev_{1}$ is a surjective morphism of (complete) unital locally convex algebras.

(ii) It remains to determine the kernel $I:=\ker\ev_{1}$: Clearly, $I_{1_A}\subseteq I$. Therefore, let $f\in A\{t\}$ with $f(1)=0$. We define a smooth $A$-valued function on $\mathbb{R}$ by
\[g:\mathbb{R}\rightarrow A,\,\,\,g(t):=\int^1_0f'(s(t-1)+1)\,ds.
\]Now, an easy calculation leads to
\[f(t)=(t-1)\cdot g(t),
\]i.e., $f\in I_{1_A}$ and thus $I_{1_A}=I$.
\end{proof}

\begin{corollary}\label{inverting 1 iso}
In the situation of Proposition \ref{inverting 1}, the map
\[\varphi:A_{\{1_A\}}\rightarrow A,\,\,\,f+I_{1_A}\mapsto f(1)
\]is an isomorphism of \emph{(}complete\emph{)} locally convex algebras.
\end{corollary}

\begin{proof}
\,\,\,In view of Proposition \ref{inverting 1} and the definition of the quotient topology, the map $\varphi$ is a bijective morphism of locally convex algebras. Further, a short observation shows that the map 
\[\psi:=\pi_{\{1_A\}}\circ i:A\rightarrow A_{\{1_A\}},
\]where $i:A\rightarrow A\{t\}$ denotes the canonical inclusion, is a continuous inverse of the map $\varphi$.
\end{proof}

\begin{remark}
(a) We will mainly be interested in the case $n=1$.

(b) If $A=C^{\infty}(M)$ is the algebra of smooth functions on some manifold $M$, then Lemma \ref{smooth exp law} implies that 
\[A\{t_1,\ldots,t_n\}\cong C^{\infty}(\mathbb{R}^n\times M).
\]In particular, $A\{t\}=C^{\infty}(\mathbb{R}\times M)$.
\end{remark}

In the forthcoming chapter we will need the following property of $A_{\{a\}}$:

\begin{lemma}\label{A_{a...} is a (left) A-modul}
For all $a_1,\ldots,a_n\in A$ there is a continuous $A$-bimodule structure on $A_{\{a_1,\ldots,a_n\}}$, given for all $a,a'\in A$ and $f\in A\{t_1,\ldots,t_n\}$ by 
\[a.[f].a':=[afa']=afa'+I_{a_1,\ldots,a_n}.
\]
\end{lemma}

\begin{proof}
\,\,\,The claim follows from the continuity of the $A$-bimodule structure on $A\{t_1,\ldots,t_n\}$, given for all $a,a'\in A$ and $f\in A\{t_1,\ldots,t_n\}$ by 
\[(a.f.a')(t_1,\ldots,t_n):=af(t_1,\ldots,t_n)a',
\]and the definition of the quotient topology.
\end{proof}

\section{The Spectrum of $A_{\{a_1,\ldots,a_n\}}$}\label{the spectrum of A_a section}

Before coming up with concrete examples of smoothly localized algebras, we will first be concerned with the problem of calculating the spectrum of $A_{\{a_1,\ldots,a_n\}}$ for some unital locally convex algebra $A$ and elements $a_1,\ldots,a_n\in A$. We recall that the \emph{spectrum} of an algebra $A$ is defined as 
\[\Gamma_A:=\Hom_{\text{alg}}(A,\mathbb{K})\backslash\{0\},\,\,\,\mathbb{K}=\mathbb{R},\mathbb{C},
\]\sindex[n]{$\Gamma_A$}(with the topology of pointwise convergence on $A$). The elements of $\Gamma_A$ are called characters. Moreover, if $A$ is a topological algebra, then $\Gamma^{\text{cont}}_A$\sindex[n]{$\Gamma^{\text{cont}}_A$} denotes the set of continuous characters of $A$. We start with determining the characters of the algebra of smooth functions on a manifold. The proof of the following proposition originates from a unpublished paper of H. Grundling and K.-H. Neeb:

\begin{theorem}\label{spec of C(M,R) set}
Let $M$ be a manifold and let $C^{\infty}(M,\mathbb{R})$ be the unital Fr\'echet algebra of smooth functions on $M$. Then the following assertions hold:
\begin{itemize}
\item[\emph{(a)}]
Each closed maximal ideal of $C^{\infty}(M,\mathbb{R})$ is the kernel of an evaluation homomorphism \[\delta_m:C^{\infty}(M,\mathbb{R})\rightarrow\mathbb{R},\,\,\,f\mapsto f(m)\,\,\,\text{for some}\,\,\,m\in M.
\]
\item[\emph{(b)}]
Each character $\chi:C^{\infty}(M,\mathbb{R})\rightarrow\mathbb{R}$ is an evaluation in some point $m\in M$.
\end{itemize}
\end{theorem}
\begin{proof}
\,\,\,(a) Let $I\subseteq C^{\infty}(M,\mathbb{R})$ be a closed maximal ideal. If all functions vanish in the point $m$ in $M$, then the maximality of $I$ implies that $I=\ker\delta_m$. So we have to show that such a point exists. Let us assume that this is not the case. From that we shall derive the contradiction $I=C^{\infty}(M,\mathbb{R})$:

(i) Let $K\subseteq M$ be a compact set. Then for each $m\in K$ there exists a function $f_m\in I$ with $f_m(m)\neq 0$. The family $(f_m^{-1}(\mathbb{R}^{\times}))_{m\in K}$ is an open cover of $K$, so that there exist $m_1,\ldots,m_n\in K$ and a smooth function $f_K:=\sum_{i=1}^n f_{m_i}^2>0$ on $K$.

(ii) If $M$ is compact, then we thus obtain a function $f_M\in I$ which is nowhere zero. This leads to the contradiction $f_M\in C^{\infty}(M,\mathbb{R})^{\times}\cap I$. Next, suppose that $M$ is non-compact. Then there exists a sequence $(M_n)_{n\in N}$ of compact subsets with $M=\bigcup_n M_n$ and $M_n\subseteq M_{n+1}^0$. Let $f_n\in I$ be a non-negative function supported by $M_{n+1}\backslash M_{n-1}^0$ with $f_n>0$ on the compact set $M_n\backslash M_{n-1}^0$. Here the requirement on the support can be achieved by multiplying with a smooth function supported by $M_{n+1}\backslash M_{n-1}^0$ which equals 1 on $M_n\backslash M_{n-1}^0$. Then $f:=\sum_n f_n$ is a smooth function in $\overline{I}=I$ with $f>0$. Hence, $f$ is invertible, which is a contradiction.

(b) The proof of this assertion is divided into four parts:

(i) Let $\chi:C^{\infty}(M,\mathbb{R})\rightarrow\mathbb{R}$ be a character. If $f\in C^{\infty}(M,\mathbb{R})$ is non-negative, then for each $c>0$ we have $f+c=h^2$ for some smooth function $h\in C^{\infty}(M,\mathbb{R})^{\times}$, and this implies that $\chi(f)+c=\chi(f+c)\geq 0$, which leads to $\chi(f)\geq -c$, and consequently to $\chi(f)\geq 0$.

(ii) Now, let $F:M\rightarrow\mathbb{R}$ be a smooth function for which the sets $F^{-1}(]-\infty,c])$, $c\in\mathbb{R}$, are compact. Such a function can easily be constructed from a sequence $(M_n)_{n\in\mathbb{N}}$ as above.

(iii) We consider the ideal $I=\ker\chi$. If $I$ has a zero, then $I=\ker\delta_m$ for some $m$ in $M$ and this implies $\chi=\delta_m$. Hence, we may assume that $I$ has no zeros. Then the argument of (a) provides for each compact subset $K\subseteq M$ a compactly supported function $f_K\in I$ with $f_K>0$ on $K$. If $h\in C^{\infty}(M,\mathbb{R})$ is supported by $K$, we therefore find a $\lambda>0$ with $\lambda f_K-h\geq 0$, which leads to
\[0\leq\chi(\lambda f_K-h)=\chi(-h),
\]and hence to $\chi(h)\leq0$. Replacing $h$ by $-h$, we also get $\chi(h)\geq 0$ and hence $\chi(h)=0$. Therefore, $\chi$ vanishes on all compactly supported functions.

(iv) For $c>0$ we now pick a non-negative function $f_c\in I$ with $f_c>0$ on the compact subset $F^{-1}(]-\infty,c])$. Then there exists a $\mu>0$ with $\mu f_c+F\geq c$ on $F^{-1}(]-\infty,c])$. Now $\mu f_c+F\geq c$ holds on all of  $M$, and therefore $\chi(F)=\chi(\mu f_c+F)\geq c$.
Since $c>0$ was arbitrary, we arrive at a contradiction.
\end{proof}

\begin{corollary}\label{spec of C(M,K) set}
Let $M$ be a manifold and let $C^{\infty}(M,\mathbb{C})$ be the unital Fr\'echet algebra of smooth complex-valued functions on $M$. Then each character $\chi:C^{\infty}(M,\mathbb{C})\rightarrow\mathbb{C}$ is an evaluation in some point $m\in M$.
\end{corollary}

\begin{proof}
\,\,\,This assertion easily follows from Theorem \ref{spec of C(M,R) set} (b): Indeed, we just have to note that each element $f\in C^{\infty}(M,\mathbb{C})$ can be written as a sum of smooth real-valued functions $f_1$ and $f_2$, i.e., 
\[f=f_1+if_2\,\,\,\text{for}\,\,\,f_1,f_2\in C^{\infty}(M,\mathbb{R}),
\]and that the restriction of the character $\chi$ to $C^{\infty}(M,\mathbb{R})$ is real-valued, i.e., defines a character of $C^{\infty}(M,\mathbb{R})$.
\end{proof}

The preceding proposition shows that the correspondence between $M$ and $\Gamma_{C^{\infty}(M)}$ is actually a topological isomorphism:

\begin{proposition}\label{spec of C(M) top}
Let $M$ be a manifold. Then the map
\begin{align}
\Phi:M\rightarrow \Gamma_{C^{\infty}(M)},\,\,\,m\mapsto \delta_m\notag.
\end{align}
is a homeomorphism.
\end{proposition}

\begin{proof}
\,\,\,(i) The surjectivity of $\Phi$ follows from Corollary \ref{spec of C(M,K) set}.
To show that $\Phi$ is injective, choose elements $m\neq m'$ of $M$. Since $M$ is manifold, there exists a function $f$ in $C^{\infty}(M)$ with $f(m)\neq f(m')$. Then 
\[\delta_m(f)=f(m)\neq f(m')=\delta_{m'}(f)
\]implies that $\delta_m\neq\delta_{m'}$, i.e., $\Phi$ is injective. 

(ii) Next, we show that $\Phi$ is continuous: Let $m_n\rightarrow m$ be a convergent sequence in $M$. Then we have 
\[\delta_{m_n}(f)=f(m_n)\rightarrow f(m)=\delta_m(f)\,\,\,\text{for all}\,\,\,f\,\,\,\text{in}\,\,\,C^{\infty}(M),
\]i.e., $\delta_{m_n}\rightarrow\delta_m$ in the topology of pointwise convergence. Hence, $\Phi$ is continuous. 

(iii) We complete the proof by showing that $\Phi$ is an open map: For this let $U$ be an open subset of $M$, $m_0$ in $U$ and $h$ a smooth real-valued function with $h(m_0)\neq 0$ and $\supp(h)\subset U$. Since the map
\[\delta_h:\Gamma_{C^{\infty}(M)}\rightarrow\mathbb{K},\,\,\,\delta_m\mapsto h(m)
\]is continuous, a short calculations shows that $\Phi(U)$ is a neighbourhood of $m_0$ containing the open subset $\delta_h^{-1}(\mathbb{K}^{\times})$. Hence, $\Phi$ is open.
\end{proof}
%

We now describe the spectrum of the tensor product of two unital algebras. 

\begin{proposition}\label{spectrum of tensor products}
Let $A$ and $B$ be two unital locally convex algebras. Then the map
\[\Phi:\Gamma^{\emph{\text{cont}}}_A\times\Gamma^{\emph{\text{cont}}}_B\rightarrow\Gamma^{\emph{\text{cont}}}_{A\otimes B},\,\,\,(\chi_A,\chi_B)\rightarrow \chi_A\otimes\chi_B
\]is a homeomorphism.
\end{proposition}
\begin{proof}
\,\,\,(i) We first note that Lemma \ref{continuity of maps EoF to E'oF'} implies that the map $\Phi$ is well-defined. To show that the map $\Phi$ is bijective, we recall that each simple tensor $a\otimes b$ can be written as $(a\otimes 1_B)\cdot(1_A\otimes b)$. Therefore, every continuous character 
\[\chi:A\otimes B\rightarrow\mathbb{K}
\]is uniquely determined by the elements of the form $(a\otimes 1_B)$ and $(1_A\otimes b)$ for $a\in A$ and $b\in B$. In particular, the restriction of $\chi$ to the subalgebra $A\otimes 1_B$, resp., $1_A\otimes B$ corresponds to a continuous character of $A$, resp., $B$, i.e.,
\[\chi=\chi_A\otimes\chi_B\,\,\,\text{for}\,\,\,\chi_A\in\Gamma_A\,\,\,\text{and}\,\,\,\chi_B\in\Gamma_B.
\]Hence, $\Phi$ is surjective. A similar argument shows that the map $\Phi$ is injective.

(ii) Finally, we leave it as an easy exercise to the reader to verify the continuity of $\Phi$ and its inverse $\Phi^{-1}$.
\end{proof}

We are now ready to prove the following theorem on continuous characters of smooth vector-valued function spaces:

\begin{theorem}\label{spectrum of  C(M,A)}
Let $M$ be a manifold and $A$ be a unital locally convex algebra. Further let $B:=C^{\infty}(M,A)$. Then the following assertions hold:
\begin{itemize}
\item[\emph{(a)}]
If $M$ is compact and $A$ is a CIA, then each maximal ideal of $B$ is closed.
\item[\emph{(b)}]
If $A$ is complete, then each continuous character $\varphi:B\rightarrow\mathbb{K}$ is an evaluation homomorphism 
\[\chi\circ\delta_m:B\rightarrow\mathbb{K},\,\,\,f\mapsto \chi(f(m))
\]for some $m\in M$ and $\chi\in\Gamma^{\emph{\text{cont}}}_A$.
\end{itemize}
\end{theorem}

\begin{proof}
\,\,\,(a) We first note that Proposition \ref{C infty (M,A) is CIA} implies that $B$ is a CIA. Hence, the unit group $B^{\times}=C^{\infty}(M,A^{\times})$ is an open subset of $B$. If $I\subseteq A$ is a maximal ideal, then $I$ intersects $B^{\times}$ trivially, and since $B^{\times}$ is open, the same holds for the closure $\bar{I}$. Hence, $\bar{I}$ also is a proper ideal, so that the maximality of $I$ implies that $I$ is closed.

(b) According to Proposition \ref{proj. tensor product for algebras}, we have 
\[C^{\infty}(M)\widehat{\otimes}A\cong C^{\infty}(M,A).
\]In particular, $C^{\infty}(M)\otimes A$ is a dense unital subalgebra of $C^{\infty}(M,A)$, and each (continuous) character $\varphi:B\rightarrow\mathbb{K}$ restricts to a character on $C^{\infty}(M)\otimes A$, which is, by Proposition \ref{spectrum of tensor products}, an evaluation in some point $(m,\chi)\in M\times\Gamma^{\text{cont}}_A$. This character is continuous with respect to the projective tensor product topology and can therefore be uniquely extended to a continuous character on $B$.
\end{proof}

\begin{corollary}\label{spectrum of C(M,A) A CIA} 
Let $M$ be a manifold and $A$ be a complete CIA. If $B:=C^{\infty}(M,A)$, then each continuous character $\varphi:B\rightarrow\mathbb{K}$ is an evaluation homomorphism 
\[\chi\circ\delta_m:B\rightarrow\mathbb{K},\,\,\,f\mapsto \chi(f(m))
\]for some $m\in M$ and $\chi\in\Gamma_A$. In particular, if $M$ is compact, then each character of $B$ is an evaluation homomorphism.
\end{corollary}

\begin{proof}
\,\,\,The first part of the Corollary immediately follows from Theorem \ref{spectrum of  C(M,A)} (b) and Lemma \ref{cont of char of CIA}, which states that every character of $A$ is continuous. 

If $M$ is compact, then Proposition \ref{C infty (M,A) is CIA} implies that $B$ is a (complete) CIA. Therefore, the second assertion follows from Lemma \ref{cont of char of CIA}.
\end{proof}

\begin{remark}
It would be nice to find a purely algebraic proof of Corollary \ref{spectrum of C(M,A) A CIA}, which shows that, even in the non-compact case, every (!) character of $B$ is an evaluation homomorphism. By Corollary \ref{spec of C(M,K) set} this is, for example, true for $A\in\{\mathbb{R},\mathbb{C}\}$ and should also work for function algebras on compact spaces, i.e., for commutative C*-algebras. For further investigations on this question, we refer to the paper [Ou07].
\end{remark}

\begin{corollary}
If $M$ is compact and $A$ a complete commutative CIA, then each maximal proper ideal of $B:=C^{\infty}(M,A)$ is the kernel of an evaluation homomorphism
\[\chi\circ\delta_m:B\rightarrow\mathbb{K},\,\,\,f\mapsto \chi(f(m))
\]for some $m\in M$ and $\chi\in\Gamma_A$.
\end{corollary}

\begin{proof}
\,\,\,According to Proposition \ref{C infty (M,A) is CIA}, $B$ is a commutative CIA. Hence, Lemma \ref{max ideals} implies that each maximal proper ideal of $B$ is the kernel of a character of $B$. Therefore, the claim follows from Corollary \ref{spectrum of C(M,A) A CIA}.
\end{proof}

The following theorem shows that, under a certain condition on the continuous spectrum of $A$, the bijection of Theorem \ref{spectrum of C(M,A)} (b) becomes an isomorphism of topological spaces:

\begin{theorem}\label{spec of C(M,A) top}
Let $M$ be a manifold and $A$ be a complete unital locally convex algebra. Further let $B:=C^{\infty}(M,A)$. Further assume that $\Gamma^{\emph{\text{cont}}}_A$ is locally equicontinuous. Then the map
\[\Phi:M\times\Gamma^{\emph{\text{cont}}}_A\rightarrow\Gamma^{\emph{\text{cont}}}_B,\,\,\,(m,\chi)\mapsto\chi\circ\delta_m
\]is a homeomorphism.
\end{theorem}

\begin{proof}
\,\,\,
(i) Since $A$ is unital, $C^{\infty}(M)$ is a central subalgebra of $B$. Moreover, we note that $A$ is embedded in $B$ as the constant-valued functions. Hence, $\Phi(m,\chi)=\Phi(m',\chi')$ implies that $m=m'$ and $\chi=\chi'$, i.e., $\Phi$ is injective. Further, Theorem \ref{spectrum of C(M,A)} (b) implies that the map $\Phi$ is surjective. 

(ii) Next, we prove the continuity of $\Phi^{-1}$: Choose a convergent net $\varphi_i=\chi_i\circ\delta_{m_i}\rightarrow\chi\circ\delta_m=\varphi$ in $\Gamma^{\text{cont}}_B$. Because $C^{\infty}(M)$ is a central subalgebra of $B$, we conclude that $\delta_{m_i}\rightarrow \delta_m$ in $\Gamma_{C^{\infty}(M)}$. By Proposition \ref{spec of C(M) top}, the last condition is equivalent to $m_i\rightarrow m$ in $M$. Since $A$ is embedded in $B$ as the constant-valued functions, we get $\chi_i\rightarrow\chi$ in $A$.

(iii) To prove continuity of $\Phi$, let $(m_0,\chi_0)\in M\times\Gamma^{\text{cont}}_A$, $\epsilon>0$ and $f\in B$. We first choose an equicontinuous neighbourhood $V$ of $\chi_0$ in $\Gamma^{\text{cont}}_A$ such that
\[V\subseteq\left\{\chi\in\Gamma^{\text{cont}}_A:\,\vert(\chi-\chi_0)(f(m_0))\vert<\frac{\epsilon}{2}\right\}.
\]Next, we choose a neighbourhood $W$ of $f(m_0)$ in $A$ such that 
\[\vert\chi(a-f(m_0))\vert<\frac{\epsilon}{2}
\]for all $a\in W$ and $\chi\in V$. Finally, we choose a neighbourhood $U$ of $m_0$ in $M$ such that $f(m)\in W$ for all $m\in U$. Then $(m,\chi)\in U\times V$ implies $f(m)\in W$ and $\chi\in V$ and therefore
\[\vert \chi(f(m)-f(m_0))\vert<\frac{\epsilon}{2}\,\,\,\text{and}\,\,\,\vert(\chi-\chi_0)(f(m_0))\vert<\frac{\epsilon}{2}.
\]It follows that
\begin{align}
&\vert\Phi(m,\chi)(f)-\Phi(m_0,\chi_0)(f)\vert=\vert\chi(f(m))-\chi_0(f(m_0))\vert\notag\\
&\leq\vert(\chi-\chi_0)(f(m_0))\vert+\vert \chi(f(m)-f(m_0))\vert<\epsilon\notag
\end{align}
for all $(m,\chi)\in U\times V$.
\end{proof}

\begin{remark}{\bf(Sources of algebras with equicontinuous spectrum).}\label{equicont}
(a) The spectrum $\Gamma_A$ of each CIA $A$ is equicontinuous. In fact, let $U$ be a balanced $0$-neighbourhood such that $U\subseteq 1_A-A^{\times}$. Then $\vert\Gamma_A(U)\vert<1$ (cf. Lemma \ref{cont of char of CIA}).


(b) Moreover, if $A$ is a \emph{$\rho$-seminormed} algebra, then [Ba00], Corollary 7.3.9 implies that $\Gamma^{\text{cont}}_A$ is equicontinuous.\index{Algebra! Seminormed}
\end{remark}

\begin{definition}\label{D(a)}
For a unital algebra $A$ and $a\in A$ we write $D(a):=\{\chi\in\Gamma_A:\,\chi(a)\neq 0\}$\sindex[n]{$D(a)$} for the set of characters which do not vanish on $a$.
Moreover, if $a_1,\ldots,a_n\in A$, then we define 
\[D(a_1,\ldots,a_n):=\bigcap^n_{i=1}D(a_i).
\]\sindex[n]{$D(a_1,\ldots,a_n)$}
\end{definition}

\begin{lemma}\label{D(a) is open}
Each set of the form $D(a_1,\ldots,a_n)$ is an open subset of $\Gamma_A$.
\end{lemma}

\begin{proof}
\,\,\,By the definition of $D(a_1,\ldots,a_n)$, it suffices to show that each set of the form $D(a)$ is open. Therefore, let $a\in A$ be arbitrary and note that the function $\ev_a:\Gamma_A\rightarrow\mathbb{K}$, $\chi\mapsto\chi(a)$ is continuous. The claim now follows from $D(a)=(\ev_a)^{-1}(\mathbb{K}^{\times})$.
\end{proof}

Next we want to describe the spectrum of $A_{\{a\}}$ for a complete locally convex algebra $A$ and an element $a\in A$. We first need the following proposition:

\begin{proposition}\label{spectrum of quotient algebras}
Let $A$ be a unital locally convex algebra and let $I$ be a closed two-sided ideal of $A$. Further, let $\pi:A\rightarrow A/I$ denote the quotient homomorphism. Then the map
\[\Phi:\left\{\chi\in\Gamma^{\emph{cont}}_A:\,\chi(I)=\{0\}\right\}\rightarrow \Gamma^{\emph{cont}}_{A/I},\,\,\,\chi\mapsto \left(a+I\mapsto \chi(a)\right)
\]is a homeomorphism with inverse given by
\[\Psi:\Gamma^{\emph{cont}}_{A/I}\rightarrow\left\{\chi\in\Gamma^{\emph{cont}}_A:\,\chi(I)=\{0\}\right\},\,\,\,\chi\mapsto \chi\circ\pi.
\]
\end{proposition}

\begin{proof}
\,\,\,Since the quotient homomorphism $\pi:A\rightarrow A/I$ becomes continuous with respect to the quotient topology, the map $\Phi$ is well-defined. Now, a short calculation shows that both maps $\Phi$ and $\Psi$ are continuous and inverse to each other. 
\end{proof}

\begin{theorem}\label{spec A_a}
Let $A$ be a complete unital locally convex algebra such that $\Gamma^{\emph{\text{cont}}}_A$ is locally equicontinuous. If $a_1,\ldots,a_n\in A$, then the map
\[\Phi_{\{a_1,\ldots,a_n\}}:D(a_1,\ldots,a_n)\rightarrow\Gamma^{\emph{cont}}_{A_{\{a_1,\ldots,a_n\}}},\,\,\,\Phi(\chi)([f]):=\chi\left(f\left(\frac{1}{\chi(a_1)},\ldots,\frac{1}{\chi(a_n)}\right)\right)
\]is a homeomorphism.
\end{theorem}

\begin{proof}
\,\,\,
(i) A short observation shows that the map
\[\phi_{\{a_1,\ldots,a_n\}}:D(a_1,\ldots,a_n)\rightarrow\mathbb{R}^n\times\Gamma^{\text{cont}}_A,\,\,\,\chi\mapsto\left(\frac{1}{\chi(a_1)},\ldots,\frac{1}{\chi(a_n)},\chi\right)
\]is clearly a homeomorphism onto its image. 

(ii) To proceed we show that the set $X:=\im(\phi_{\{a_1,\ldots,a_n\}})$ is homeomorphic to $\Gamma^{\text{cont}}_{A_{\{a_1,\ldots,a_n\}}}$: For this we first put $B:=A\{t_1,\ldots,t_n\}=C^{\infty}(\mathbb{R}^n,A)$. Then Proposition \ref{spectrum of quotient algebras} implies that the space $\Gamma^{\text{cont}}_{A_{\{a_1,\ldots,a_n\}}}$ is homeomorphic to $\{\chi\in\Gamma^{\text{cont}}_B:\,\chi(I_{a_1,\ldots,a_n})=\{0\}\}$. Furthermore, Theorem \ref{spec of C(M,A) top} implies that this last set is homeomorphic to
\[\left\{(r_1,\ldots,r_n,\chi)\in\mathbb{R}^n\times\Gamma^{\text{cont}}_A:\chi\circ\delta_{(r_1,\ldots,r_n)}(I_{a_1,\ldots,a_n})=\{0\}\right\}.
\]Next, we observe that the condition $\chi\circ\delta_{(r_1,\ldots,r_n)}(I_{a_1,\ldots,a_n})=\{0\}$ is equivalent to $r_i\chi(a_i)=1$ for all $1\leq i\leq n$, i.e., we have
\[\left\{(r_1,\ldots,r_n,\chi)\in\mathbb{R}^n\times\Gamma^{\text{cont}}_A:\delta_{(r_1,\ldots,r_n,\chi)}(I_{a_1,\ldots,a_n})=\{0\}\right\}=X.
\]The corresponding homeomorphism is given by
\[\varphi_{\{a_1,\ldots,a_n\}}:X\rightarrow\Gamma^{\text{cont}}_{A_{\{a_1,\ldots,a_n\}}},\,\,\,\varphi\left(\left(\frac{1}{\chi(a_1)},\ldots,\frac{1}{\chi(a_n)},\chi\right)\right)([f]):=\chi\left(f\left(\frac{1}{\chi(a_1)},\ldots,\frac{1}{\chi(a_n)}\right)\right)
\]

(iii) The claim now follows from $\Phi_{\{a_1,\ldots,a_n\}}=\varphi_{\{a_1,\ldots,a_n\}}\circ\phi_{\{a_1,\ldots,a_n\}}$.
\end{proof}

\section{Localizing $C^{\infty}(M,A)$}\label{Localizing C(M)}
In this section we finally want to prove a smooth analogue of Proposition \ref{K(X)_f=O(X_f)} for the algebra $C^{\infty}(M)$ of smooth functions on some manifold $M$. In fact, if $A$ is a unital Fr\'echet algebra, we show that the smooth localization of $C^{\infty}(M,A)$ with respect to a smooth nonzero function $f:M\rightarrow\mathbb{R}$ is isomorphic as a unital Fr\'echet algebra to $C^{\infty}(M_f,A)$, where
\[M_f:=\{m\in M:\,f(m)\neq 0\}.
\]We start with a very useful lemma of Hadamard:




\begin{lemma}{\bf(Hadamard's lemma).}\label{Hadamards lemma}\index{Theorem!of Hadamard}
Let $E$ be a complete locally convex space. Further, let $U$ be an open convex subset of $\mathbb{R}^s\times\mathbb{R}^{n-s}$ containing $0$ and $f:U\rightarrow E$ be a smooth function that vanishes on $U\cap\mathbb{R}^{n-s}$. If $\pr_j:\mathbb{R}^n\rightarrow\mathbb{R}$ denotes the projection to the j-th factor, then
\[f=\sum_{j=1}^{s}g_j\cdot\pr_j
\]for suitable smooth functions $g_j:U\rightarrow E$.
\end{lemma}
\begin{proof}
\,\,\,For $x=(x_1,\ldots,x_n)$ in $U$ we define $y:=(x_1,\ldots,x_s,0)$ and $z:=(0,x_{s+1},\ldots,x_n)$. We further define a smooth $E$-valued curve $h$ on $[0,1]$ by
\[h:[0,1]\rightarrow E,\,\,\, h(t):=f(z+ty).
\]Then $h'(t)=\sum_{j=1}^s\frac{\partial f}{\partial x_j}(z+ty)\cdot x_j$, and for $g_j(x):=\int^1_0\frac{\partial f}{\partial x_j}(z+ty)dt$ we obtain
\[f(x)=f(x)-f(z)=h(1)-h(0)=\int^1_0h'(t)dt=\sum^s_{j=1}g_j(x)\cdot x_j,
\]i.e., $f=\sum_{j=1}^{s}g_j\cdot\pr_j$ as desired.
\end{proof}


\begin{theorem}\label{manifold mod sub=sub}
Let $E$ be a complete locally convex space. Further, let $M$ be a manifold and $h:M\rightarrow\mathbb{R}$ a smooth function with $0\in\mathbb{R}$ as a regular value. If $H:=h^{-1}(0)$ is the corresponding closed submanifold of $M$, then the restriction map
\[R_H:C^{\infty}(M,E)\rightarrow C^{\infty}(H,E),\,\,\,f\mapsto f_{\mid H}
\]is a surjective morphism of locally convex spaces. Its kernel $\ker(R_H)$ is equal to $C^{\infty}(M,E)\cdot h$. In particular, $C^{\infty}(M,E)\cdot h$ is a closed subspace of $C^{\infty}(M,E)$.
\end{theorem}

\begin{proof}
\,\,\,The proof of this theorem is devided into three parts:
 
(i) To show that the map $R_H$ is a surjective morphism of locally convex spaces, we first note that $R_H$ is clearly linear and continuous as a restriction map. Thus, it remains to prove its surjectivity: For this we have to show that any smooth function $f:H\rightarrow E$ can be extended to a smooth function on $M$. Locally this can be done, since $H$ looks like $\mathbb{R}^{n-1}$ in $\mathbb{R}^n$. To get a global extension we have to use a partition of unity. 
 
(ii) Next, we show the equality of the kernel of $R_H$ and the subspace $h\cdot C^{\infty}(M,E)$. Since $R_H$ is continuous, this will in particular imply that $C^{\infty}(M,E)\cdot h$ is a closed subspace of $C^{\infty}(M,E)$: We immediately verify that $C^{\infty}(M,E)\cdot h\subseteq\ker(R_H)$. For the other inclusion, let $f\in C^{\infty}(M,E)$ with $R_H(f)=0$, i.e., $f\equiv 0$ on $H$. On the dense open subset $M\backslash H$ of $M$ we define a smooth $E$-valued map $g$ by
\[g:M\backslash H\rightarrow E,\,\,\,g(m):=\frac{f(m)}{h(m)}.
\]We claim that $g$ can be extended (uniquely) to a smooth $E$-valued map $\tilde{g}$ on $M$. If this is the case, then $f=\tilde{g}\cdot h$ holds on $M$ and we conclude that $f\in C^{\infty}(M,E)\cdot h$. In particular, we obtain the desired equality
\[\ker(R_H)=C^{\infty}(M,E)\cdot h.
\]

(iii) We now prove the claim of part (ii). For this, let $m\in H$. By the Implicit Function Theorem we may choose a chart $(\varphi,U)$ around $m$ such that $h\circ\varphi^{-1}:\mathbb{R}^n\rightarrow\mathbb{R}$ looks like the projection onto the first coordinate. If we write $\mathbb{R}^n=\mathbb{R}\oplus\mathbb{R}^{n-1}$ and use Lemma \ref{Hadamards lemma}, then we conclude that
\[f\circ\varphi^{-1}=(g_U\circ\varphi^{-1})\cdot(h\circ\varphi^{-1})
\]for a suitable smooth function $g_U:U\rightarrow E$. In particular, this implies that $f=g_U\cdot h$ on $U$. The same construction applied to another such chart $(\psi,V)$ around $m$ leads to a smooth function $g_V:V\rightarrow E$ and a factorization $f=g_V\cdot h$ on $V$. A short observation shows that $g_U=g_V$ on $U\cap V$. Therefore, the function $\tilde{g}:M\rightarrow E$ defined by
\[\tilde{g}(m):=
\begin{cases}
g(m) &\text{for}\,\,\,m\in M\backslash H\\
g_U(m) &\text{for}\,\,\,m\in H.
\end{cases}
\]is well-defined, smooth and an extension of $g$ to $M$. This proves the claim.
\end{proof}

Although $C^{\infty}(M)_{\{f\}}$ is not the ring of fractions with some power of $f$ as denominator, Theorem \ref{manifold mod sub=sub} yields the following desired result:

\begin{corollary}\label{C(M)_f=C(M_f)}
Let $A$ be a unital Fr\'echet algebra. If $M$ is a manifold, $f:M\rightarrow\mathbb{R}$ a smooth function which is nonzero and $M_f:=\{m\in M:\,f(m)\neq 0\}$, then the map
\[\phi_f:C^{\infty}(M,A)_{\{f\}}\rightarrow C^{\infty}(M_f,A),\,\,\,[F]\mapsto F\circ\left(\frac{1}{f}\times\id_{M_f}\right)
\]is an isomorphism of unital Fr\'{e}chet algebras.
\end{corollary}

\begin{proof}
\,\,\,We first note that $0$ is a regular value for the function 
\[h_f:\mathbb{R}\times M\rightarrow\mathbb{R},\,\,\,(t,m)\mapsto 1-tf(m)
\]and that $H:=h_f^{-1}(0)$ is diffeomorphic to $M_f$. Since the quotient of a Fr\'echet space by a closed subspace is again Fr\'echet (cf. [Bou55], \S 3.5), Theorem \ref{manifold mod sub=sub} and the definition of the quotient topology imply that the map $\phi_f$ is a bijective morphism of unital Fr\'{e}chet algebras. Thus, the assertion follows from Proposition \ref{open mapping theorem}.
\end{proof}

We now present a very useful and well-known theorem of analysis:

\begin{theorem}{\bf(Whitneys Theorem).}\label{whitneys theorem}\index{Theorem!of Whitney}
Any open subset $U$ of a manifold $M$ is the complement of the zeroset of a smooth function $f:M\rightarrow\mathbb{R}$, i.e.,
\[U=M_f:=\{m\in M:\,f(m)\neq 0\}\,\,\,\text{for some}\,\,\,f\in C^{\infty}(M,\mathbb{R}).
\]Such a function $f$ is called a $U$-defining function.\index{$U$-Defining!Function}
\end{theorem}

\begin{proof}
\,\,\,A proof for the case $M=\mathbb{R}^n$ can be found in [GuPo74], Exercise 1.5.11. A combination with a partition-of-unity argument proves the general case.
\end{proof}

The following corollary will be very important for the rest of this thesis:

\begin{corollary}\label{C(U)=C(M)_f_U}
If $A$ is a unital Fr\'echet algebra and $U$ an open subset of a manifold $M$, then the algebra $C^{\infty}(U,A)$ is isomorphic as a unital Fr\'echet algebra to the localization of $C^{\infty}(M,A)$ with respect to a $U$-defining function $f_U$. 
\end{corollary}

\begin{proof}
\,\,\,For the proof we just have to combine Corollary \ref{C(M)_f=C(M_f)} and Theorem \ref{whitneys theorem}.
\end{proof}

\section{Smooth Localization of Noncommutative Vector Bundles}\label{LNVB}

In this part of the thesis we want to apply the constructions from the previous section to suggest a possible notion of ``locally free" modules. 
In fact, Proposition \ref{gamma V_U=gamma VC(U)} invites us to give the following definition:

\begin{definition}\label{loc of NCVB}{\bf(Locally free modules).}\index{Modules!Locally Free}
Let $A$ be a unital locally convex algebra and $E$ a locally convex $A$-module. Further, let $C_A$ denote the center of $A$. We call $E$ \emph{locally free of finite rank} $n\in\mathbb{N}$ if for all $\chi\in\Gamma_{C_A}$ there exists an element $z\in C_A$ with $\chi(z)\neq 0$ such that
\begin{align}
E_{\{z\}}:=E\otimes_{C_A}(C_A)_{\{z\}}\cong (C_A)_{\{z\}}^n \label{iso of mod}
\end{align}
as locally convex $(C_A)_{\{z\}}$-modules.\sindex[n]{$E_{\{z\}}$}
\end{definition}

\begin{remark}\label{remark on loc of NCVB}
(a) From a classical point of view the condition 
in Definition \ref{loc of NCVB} means that there exists an open covering of $\Gamma_{C_A}$ of the form $(D(z_i))_{i\in I}$ such that $E$ is free of rank $n$ over each of these open subsets.

(b) When dealing with Fr\'{e}chet algebras, it is enough to require (\ref{iso of mod}) to be an algebraic isomorphism of $(C_A)_{\{z\}}$-modules:

(i) Indeed, if $A$ is a Fr\'{e}chet algebra, then the same holds for the center $C_A$, and furthermore for $(C_A)_{\{z\}}$. Here, the last assertion follows from the fact that the quotient of a Fr\'{e}chet space by a closed subspace is again Fr\'{e}chet (cf. [Bou55], \S 3.5). 

(ii) If $E_{\{z\}}\cong (C_A)_{\{z\}}^n$ as (algebraic) $(C_A)_{\{z\}}$-modules, then $E_{\{z\}}$ is a finitely generated projective $(C_A)_{\{z\}}$-module, and Proposition \ref{modules of frechet algebras} (b) implies that $E_{\{z\}}$ carries a unique topology of Fr\'{e}chet $(C_A)_{\{z\}}$-module. 

(iii) Now, Proposition \ref{modules of frechet algebras} (a) implies that each algebraic isomorphism between finitely generated projective Fr\'{e}chet modules is automatically an isomorphism of Fr\'{e}chet modules. In particular, $$E_{\{z\}}\cong (C_A)_{\{z\}}^n$$ as Fr\'{e}chet $(C_A)_{\{z\}}$-modules. 
\end{remark}

For the following examples it is important to keep Remark \ref{remark on loc of NCVB} (b)  in mind:

\begin{example}
If $(\mathbb{V},M,V,q)$ is a vector bundle over $M$ with $\dim V=n$, then $\Gamma\mathbb{V}$ is a finitely generated projective $C^{\infty}(M)$-module. Let $(U_i)_{i\in I}$ be an open cover of $M$ such that each $\mathbb{V}_{U_i}$ is a trivial vector bundle over $U_i$, i.e., such that $\mathbb{V}_{U_i}\cong U_i\times V$. Further, let $f_i$ be a $U_i$-defining function. For $m\in M$ we choose $i_m\in I$ with $m\in U_{i_m}$. Then $f_{i_m}$ is an element in $C^{\infty}(M)$ with $f_{i_m}(m)\neq 0$ and Corollary \ref{C(U)=C(M)_f_U} implies
\[\Gamma\mathbb{V}_{\{f_{i_m}\}}=\Gamma\mathbb{V}\otimes_{C^{\infty}(M)}C^{\infty}(M)_{\{f_{i_m}\}}\cong\Gamma\mathbb{V}\otimes_{C^{\infty}(M)}C^{\infty}(U_{i_m})\cong\Gamma\mathbb{V}_{U_{i_m}}\cong C^{\infty}(U_{i_m})^n.
\]
\end{example}

\begin{example}
For $\theta=\frac{n}{m}$, $n,m$ relatively prime, let $A:=\mathbb{T}^2_{\theta}\cong\Gamma(\mathbb{A})$ be the smooth 2-dimensional quantum torus from Proposition \ref{rational quantumtori}. In the following we consider the free module $A^k$ for some $k\in\mathbb{N}$: From Corollary \ref{center of rational quantumtori} we know that $C_A\cong C^{\infty}(\mathbb{T}^2)$. Now, let $(U_i)_{i\in I}$ be an open cover of $\mathbb{T}^2$ such that each $\mathbb{A}_{U_i}$ is a trivial algebra bundle over $U_i$, i.e., such that $\mathbb{A}_{U_i}\cong U_i\times\M_m(\mathbb{C}) $. Further, let $z_i$ be an element in $C_A$ corresponding to a $U_i$-defining function $f_i$. For $(u,v)\in \mathbb{T}^2$ we choose $i_0\in I$ with $(u,v)\in U_{i_0}$. Then $z_{i_0}$ is an element in $C_A$ with $\delta_{(u,v)}(z_{i_0})\neq 0$ and Corollary \ref{C(U)=C(M)_f_U} implies
\begin{align}
A^k_{\{z_{i_0}\}}&=A^k\otimes_{C_A}(C_A)_{\{z_{i_0}\}}\cong (\Gamma\mathbb{A})^k\otimes_{C^{\infty}(\mathbb{T}^2)}C^{\infty}(M)_{\{f_{i_0}\}}\cong(\Gamma\mathbb{A})^k\otimes_{C^{\infty}(\mathbb{T}^2)}C^{\infty}(U_{i_0})\notag\\
&\cong(\Gamma\mathbb{A}_{U_{i_0}})^k\cong C^{\infty}(U_{i_0},\M_m(\mathbb{C}))^k\cong C^{\infty}(U_{i_0})^{m^2k}\cong(C_A)_{\{z_{i_0}\}}^{m^2k}.\notag
\end{align}   
\end{example}

\begin{lemma}\label{complemented submodules}
Let $A$ be an algebra and $E$ a finitely generated projective $A$-module. Then each complemented submodule $F$ of $E$ is also a finitely generated projective $A$-module.
\end{lemma}

\begin{proof}
\,\,\,Let $F$ be a complemented submodule of $E$. By Corollary \ref{f.g.pr.mod}, we have to show that $F$ is a direct summand of a free $A$-module of finite rank: 

For this we choose a complement $F'$ of $F$ in $E$ and note that, since $E$ is a finitely generated projective $A$-module, there exist $n\in\mathbb{N}$ and $p\in\Idem_n(A)$ such that $E\cong p\cdot A^n$. Now, we define 
\[F'':=F'\oplus(1_n-p)\cdot A^n
\]and note that this leads to
\[F\oplus F''=F\oplus F'\oplus(1_n-p)\cdot A^n\cong E\oplus(1_n-p)\cdot A^n\cong p\cdot A^n\oplus(1_n-p)\cdot A^n\cong A^n.
\]
\end{proof}

\begin{proposition}\label{projective stuff}
Let $A$ be an algebra and $B$ a subalgebra. Further, let $A$ be finitely generated and projective as a $B$-module. Then $q\cdot A^m$ is \emph{(}also\emph{)} finitely generated and projective as a $B$-module for each $m\in\mathbb{N}$ and $q\in\Idem_m(A)$.
\end{proposition}

\begin{proof}
\,\,\,Since $A$ is a finitely generated projective $B$-module, there exist $n\in\mathbb{N}$ and $p\in\Idem_n(B)$ such that $A\cong p\cdot B^n$ as $B$-modules. For $m\in\mathbb{N}$ we define a matrix $p_m\in\M_{nm}(B)$ by
$$
p_m:=\begin{pmatrix}
 p      &  &   \\
         & \ddots &  \\
         &  &   p 
\end{pmatrix}.
$$
Then $p_m\in\Idem_{nm}(B)$ and one easily gets $A^m\cong p_m\cdot B^{nm}$ as $B$-modules. Hence, Corollary \ref{f.g.pr.mod} implies that $A^m$ is a finitely generated projective $B$-module. Alternatively one can use Proposition \ref{Appendix B3} to see that $A^m$ is a finitely generated projective $B$-module. 

Let $q\in\Idem_m(A)$. We recall that $q\cdot A^m$ is a complemented submodule of the finitely generated projective $B$-module $A^m$. Indeed, a complement is given by $(1_m-q)\cdot A^n$. Therefore, the claim follows from Lemma \ref{complemented submodules}.
\end{proof}

\begin{lemma}\label{direct sums}
Let $I$ be a finite index set and $(A_i)_{i\in I}$ a family of unital algebras. Further, let $A:=\bigoplus_{i\in I}A_i$ and $\pi_i:A\rightarrow A_i$ be the projection onto the $i$-th component. Then the following assertions hold:
\begin{itemize}
\item[\emph{(a)}]
For each $n\in\mathbb{N}$ the map
\[\Phi:\M_n(A)\rightarrow\bigoplus_{i\in I}\M_n(A_i),\,\,\,p\mapsto(p_{A_i}),
\]where $p_{A_i}$ denotes the matrix in $\M_n(A_i)$ obtained from $p$ by projecting its entries to $A_i$, is an isomorphism of algebras.
\item[\emph{(b)}]
If $n\in\mathbb{N}$ and $p\in\Idem_n(A)$, then the map
\[\Psi:p\cdot A^n\rightarrow\bigoplus_{i\in I}p_{A_i}\cdot A_i^n,\,\,\,p\cdot(a_1,\ldots,a_n)\mapsto\bigoplus_{i\in I}p_{A_i}\cdot(\pi_i(a_1),\ldots,\pi_i(a_n))
\]is an isomorphism of finitely generated projective $A$-modules. Moreover, each component is a finitely generated projective $A_i$-module.
\end{itemize}
\end{lemma}

\begin{proof}
\,\,\,The proof just consists of some simple calculation.
\end{proof}

\begin{theorem}\label{class of locally free modules}
Let $A=\Gamma(\mathbb{A})$ be the algebra of sections of a finite-dimensional algebra bundle $\mathbb{A}$ over a manifold $M$. Then the following assertions hold:
\begin{itemize}
\item[\emph{(a)}]
The algebra $A$ is a finitely generated projective $C_A$-module.
\item[\emph{(b)}]
If $C_A\cong C^{\infty}(M)^n$ for some $n\in\mathbb{N}$, then every finitely generated projective $A$-module is locally free.
\end{itemize}
\end{theorem}

\begin{proof}
\,\,\,(a) The claim follows from Remark \ref{mod and ringhom I} applied to the inclusion and the fact that the map
\[A\otimes_{C^{\infty}(M)}C_A\rightarrow A,\,\,\,a\otimes z\mapsto az
\]is an isomorphism of $C_A$-modules. 

(b) We divide the proof of this theorem into three parts: 

(i) In view of part (a) and Proposition \ref{projective stuff}, every finitely generated projective $A$-module $E$ is also a finitely generated projective  $C_A$-module, i.e., there exist $k\in\mathbb{N}$ and $p\in\Idem_k(C_A)$ such that $E\cong p\cdot C_A^k$ as $C_A$-modules. 

(ii) In particular, if $C_A\cong C^{\infty}(M)^n$ for some $n\in\mathbb{N}$, then every finitely generated projective $C_A$-module can be written as a $n$-fold direct sum of finitely generated projective $C^{\infty}(M)$-modules as in Lemma \ref{direct sums} (b), i.e.,
\[E\cong p\cdot C_A^k\cong\bigoplus_{i=1}^n p_i\cdot C^{\infty}(M)^k.
\]By Theorem \ref{Serre-Swan}, each summand $p_i\cdot C^{\infty}(M)^k$ is isomorphic to the sections of a vector bundle $\mathbb{V}_i$ of dimension $n_i$ over $M$. Hence,
\[E\cong\bigoplus_{i=1}^n\Gamma\mathbb{V}_i.
\]

(iii) Now, let $(U_j)_{j\in J}$ be an open cover of $M$ such that each $(\mathbb{V}_i)_{U_j}$ is a trivial vector bundle over $U_j$, i.e., such that $(\mathbb{V}_i)_{U_j}\cong U_j\times V_i$. Further, let $f_j$ be a $U_j$-defining function. To show that $E$ is locally free, let $(\delta_{m_1},\ldots,\delta_{m_n})$ be a character of $\Gamma_{C_A}$ corresponding to $(m_1,\ldots,m_n)\in M^n$. For each $i\in\{1,\ldots,n\}$ we choose $j_i\in J$ with $m_i\in U_{j_i}$. Then 
\[F_{j_i}:=(0,\ldots,\underbrace{f_{j_i}}_{\text{i-th component}},\ldots,0)
\]is an element in $C_A$ with $(\delta_{m_1},\ldots,\delta_{m_n})(F_{j_i})\neq 0$ and Corollary \ref{C(U)=C(M)_f_U} implies
\begin{align}
&E_{\{F_{j_i}\}}\cong E\otimes_{C_A}(C_A)_{\{F_{j_i}\}}\cong\left(\bigoplus_{i=1}^n\Gamma\mathbb{V}_i\right)\otimes_{C_A}\left({\bf0}\oplus\ldots\oplus\underbrace{C^{\infty}(U_{j_i})}_{\text{i-th component}}\oplus\ldots\oplus{\bf0}\right)\notag\\
&\cong{\bf0}\oplus\ldots\oplus\underbrace{\Gamma(\mathbb{V}_i)_{U_{j_i}}}_{\text{i-th component}}\oplus\ldots\oplus{\bf0}\cong{\bf0}\oplus\ldots\oplus\underbrace{C^{\infty}(U_{j_i})^{n_i}}_{\text{i-th component}}\oplus\ldots\oplus{\bf0}\cong(C_A)^{n_i}_{\{F_{j_i}\}}.
\end{align}

\end{proof}

\begin{corollary}
Let $A=\mathbb{T}^2_{\theta}\cong\Gamma(\mathbb{A})$ denote the smooth 2-dimensional quantum torus for $\theta=\frac{n}{m}$. Then every finitely generated projective $A$-module is locally free.
\end{corollary}

\begin{proof}
\,\,\,By Corollary \ref{center of rational quantumtori}, we have $C_A\cong C^{\infty}(\mathbb{T}^2)$. Hence, the claim follows from Theorem \ref{class of locally free modules} (b).
\end{proof}

\section{Smooth Localization of Dynamical Systems}\label{ACNCPT^nB}

A dynamical system is a triple $(A,G,\alpha)$, consisting of a unital locally convex algebra $A$, a topological group $G$ and a group homomorphism $\alpha:G\rightarrow\Aut(A)$, which induces a continuous action of $G$ on $A$. In this section we present a method of localizing dynamical systems $(A,G,\alpha)$ with respect to elements of the commutative fixed point algebra of the induced action of $G$ on $C_A$. 

\begin{notation} 
Let $(A,G,\alpha)$ be a dynamical system and $C_A$ the center of $A$. If $c\in C_A, g\in G$ and $a\in A$, then the equation $\alpha(g,c)a=a\alpha(g,c)$ shows that the map $\alpha$ restricts to a continuous action of $G$ on $C_A$ by algebra automorphisms. In the following we write $Z$ for the fixed point algebra of the induced action of $G$ on $C_A$, i.e., $Z=C_A^G$. We further write $A_{\{z\}}$\sindex[n]{$A_{\{z\}}$} for the smooth localization of $A$ with respect to $z\in Z$ and 
\[\pi_{\{z\}}:C^{\infty}(\mathbb{R},A)\rightarrow A_{\{z\}}
\]for the corresponding continuous quotient homomorphism (cf. Definition \ref{sm. loc.}). 
Moreover, we put
\[D_A(z):=\{\chi\in\Gamma_A:\,\chi(z)\neq 0\}\,\,\,\text{and}\,\,\,D(z):=D_Z(z):=\{\chi\in\Gamma_Z:\,\chi(z)\neq 0\}.
\]\sindex[n]{$D_A(z)$}
\end{notation}

\begin{lemma}\label{A_z is locally convex space again}
With the notations from above, the following assertion holds:
The map 
\[\rho_{\{z\}}:A_{\{z\}}\times Z\rightarrow A_{\{z\}},\,\,\,([f],z)\mapsto [z\cdot f]
\]defines on $A_{\{z\}}$ the structure of a right locally convex $Z$-module.
\end{lemma}

\begin{proof}
\,\,\,Since $Z$ is commutative, the assertion directly follows from Lemma \ref{A_{a...} is a (left) A-modul}.
\end{proof}

\begin{proposition}\label{A_z is algebra}
The locally convex space $A_{\{z\}}$ becomes a locally convex algebra, when equipped with the multiplication
\[m_{\{z\}}:A_{\{z\}}\times A_{\{z\}}\rightarrow A_{\{z\}},\,\,\,m([f],[f']):=[f\cdot f'].
\]
\end{proposition}

\begin{proof}
\,\,\,This assertion is an easy consequence of the construction of $A_{\{z\}}$. Indeed, if $m$ denotes the (continuous) multiplication map of the algebra $C^{\infty}(\mathbb{R},A)$, then the continuity of $m_{\{z\}}$ follows from $m_{\{z\}}\circ(\pi_{\{z\}}\times\pi_{\{z\}})=\pi_{\{z\}}\circ m$ and the fact that the map $\pi_{\{z\}}\times\pi_{\{z\}}$ is continuous, surjective and open.
\end{proof}



\begin{proposition}\label{T^n action on spec again}
If $(A,G,\alpha)$ is a dynamical system and $z\in Z$, then the map
\[\alpha_{\{z\}}:G\times A_{\{z\}}\rightarrow A_{\{z\}},\,\,\,(g,[f]\underline{})\mapsto[\alpha(g)\circ f]
\]defines a continuous action of $G$ on $A_{\{z\}}$ by algebra automorphisms. In particular, the triple $(A_{\{z\}},G,\alpha_{\{z\}})$ is a dynamical system.
\end{proposition}

\begin{proof}
\,\,\,If $g\in G$, 
\[\alpha(g)(1_A-tz)=1_A-tz
\]implies that the map $\alpha_{\{z\}}$ is well-defined. Further, a simple calculation shows that the map $\alpha_{\{z\}}$ defines an action of $G$ on $A_{\{z\}}$ by algebra automorphisms. That this action is continuous is a consequence of the continuity of the map 
\[\beta:G\times C^{\infty}(\mathbb{R},A)\rightarrow C^{\infty}(\mathbb{R},A),\,\,\,(g,f)\mapsto\alpha(g)\circ f
\](cf. Lemma \ref{dynamical system and trivial algebra bundles}), $\alpha_{\{z\}}\circ(\id_G\times\pi_{\{z\}})=\pi_{\{z\}}\circ\beta$ and the fact that $\id_G\times\pi_{\{z\}}$ is continuous, surjective and open.
\end{proof}

\begin{definition}\label{smooth localized dynamical system}\emph{\bf(Smooth localization of dynamical systems).}\index{Smooth Localization of!Dynamical Systems}
Let $(A,G,\alpha)$ be a dynamical system and $z\in Z$. We call the dynamical system
\[(A_{\{z\}},G,\alpha_{\{z\}})
\]of Proposition \ref{T^n action on spec again} the (\emph{smooth}) \emph{localization of $(A,G,\alpha)$ associated to} $z\in Z$.
\end{definition}

\begin{remark}\label{A_z frechet again}{\bf(Fr\'echet algebras).}
If $A$ is a Fr\'{e}chet algebra, then the same holds for $A_{\{z\}}$. Indeed, this assertion follows from the fact that the quotient of a Fr\'{e}chet space by a closed subspace is again Fr\'{e}chet (cf. [Bou55], \S 3.5). 
\end{remark}

In the remaining part of this section we will be concerned with the spectrum of $A_{z}$:

\begin{lemma}\label{T^n action on spec II again}
If $(A,G,\alpha)$ is a dynamical system and $z\in Z$, then the map 
\[\sigma_{\{z\}}:\Gamma^{\emph{cont}}_{A_{\{z\}}}\times G\rightarrow\Gamma^{\emph{cont}}_{A_{\{z\}}},\,\,\,(\chi,g)\mapsto \chi\circ\alpha_{\{z\}}(g)
\]defines an action of $G$ on $\Gamma^{\emph{cont}}_{A_{\{z\}}}$.
\end{lemma}

\begin{proof}
\,\,\,This claim follows from a simple calculation and is therefore left to the reader.
\end{proof}

\begin{remark}\label{spec A_z}{\bf($\Gamma^{\text{cont}}_{A_{\{z\}}}$).}
Let $A$ be a complete unital locally convex algebra such that $\Gamma^{\text{cont}}_A$ is locally equicontinuous. Then Theorem \ref{spec A_a} implies that the map
\[\Phi_{\{z\}}:D_A(z)\rightarrow\Gamma^{\text{cont}}_{A_{\{z\}}},\,\,\,\Phi(\chi)([f]):=\chi\left(f\left(\frac{1}{\chi(z)}\right)\right)
\]is a homeomorphism.
\end{remark}

\section{Another Localization Method}

Another interesting localization method can be found in [NaSa03], Chapter 3. Indeed, given a commutative $\mathbb{R}$-algebra $A$ and an open subset $U$ of the \emph{real spectrum} 
\[\Gamma_{A,\text{r}}:=\Hom_{\text{alg}}(A,\mathbb{R})\backslash\{0\}
\]of $A$, the authors of [NaSA03] define $A_U$ to be the ring of fractions of $A$ with respect to the multiplicative subset of all elements in $A$ without zeroes in $U$, i.e., elements in $A_U$ are (equivalence classes of) fractions $\frac{a}{s}$, where $a,s\in A$ and $s(\chi):=\chi(s)\neq 0$ for all $\chi\in U$ (cf. Construction \ref{S/R}). In particular, they show that if $A=C^{\infty}(\mathbb{R}^n,\mathbb{R})$, then $A_U=C^{\infty}(U,\mathbb{R})$ for each open subset $U$ of $\mathbb{R}^n$ (cf. [NaSA03], Chapter 3, Example 3.3). Of course, an important handicap of this approach is the commutativity of the algebra $A$.


\chapter{Some Comments on Sections of Algebra Bundles}

In this short chapter we are dealing with algebra bundles $q:\mathbb{A}\rightarrow M$ with a possibly infinite-dimensional fibre $A$ over a finite-dimensional manifold $M$ and show how to endow the corresponding space $\Gamma\mathbb{A}$ of sections with a topology that turns it into a (unital) locally convex algebra. Moreover, we show that if $A$ is a Fr\'echet algebra, then the smooth localization of $\Gamma\mathbb{A}$ with respect to an element $f\in C^{\infty}(M,\mathbb{R})$ (cf. Definition \ref{sm. loc.}) is isomorphic (as a unital Fr\'echet algebra) to $\Gamma\mathbb{A}_{M_f}$, where 
\[M_f:=\{m\in M:\,f(m)\neq 0\}.
\]


\section{The Space of Sections as a Locally Convex Algebra}

\begin{definition}\label{algebra bundle}{\bf(Algebra bundles).}\index{Bundles!Algebra}
Let $A$ be a unital locally convex algebra. An \emph{algebra bundle} (with fibre $A$) is a quadruple $(\mathbb{A},M,A,q)$, consisting of an infinite-dimensional manifold $\mathbb{A}$, a manifold $M$ and a smooth map $q:\mathbb{A}\rightarrow M$, with the following property: All fibres $\mathbb{A}_m$, $m\in M$, carry unital algebra structures, and each point $m\in M$ has an open neighbourhood $U$ for which there exists a diffeomorphism
\[\varphi_U:U\times A\rightarrow q^{-1}(U)=\mathbb{A}_U,
\]satisfying $q\circ\varphi_U=p_U$ and all maps
\[\varphi_{U,x}:A\rightarrow\mathbb{A}_x,\,\,\,a\mapsto\varphi_U(x,a)
\]are algebra isomorphisms.
\end{definition}

\begin{definition}{\bf(Sections of algebra bundles).}\label{sections of algebra bundles}\index{Sections!of Algebra Bundles}
If $(\mathbb{A},M,A,q)$ is an algebra bundle, then the corresponding space
\[\Gamma\mathbb{A}:=\{s\in C^{\infty}(M,\mathbb{A}):q\circ s=\id_M\}
\]of \emph{smooth sections} carries the structure of a unital algebra. Indeed, the vector space structure on $\Gamma\mathbb{A}$ is defined as in Definition \ref{sections}. Furthermore, the product  
\[(s_1\cdot s_2)(m):=s_1(m)\cdot s_2(m)
\]of two sections $s_1,s_2\in\Gamma\mathbb{A}$ again defines a section of $\mathbb{A}$, and the unit of $\Gamma\mathbb{A}$ is given by the section $s_{\bf 1}(x):=1_{\mathbb{A}_x}$.
\end{definition}

Our next goal is to endow $\Gamma\mathbb{A}$ with a topology: 

\begin{construction}{\bf(A topology on $\Gamma\mathbb{A}$).}\label{top on space of sections}\index{$\Gamma\mathbb{A}$!Topology on}
Let $(\mathbb{A},M,A,q)$ be an algebra bundle and $(\varphi_i,U_i)_{i\in I}$ a bundle atlas for $(\mathbb{A},M,A,q)$. We endow $\Gamma\mathbb{A}$ with the initial topology $\mathcal{O}_I$ generated by the maps
\[\Phi_i:\Gamma\mathbb{A}\rightarrow C^{\infty}(U_i,A),\,\,\,s\mapsto s_i:=\pr_A\circ\varphi_i^{-1}\circ s_{\mid U_i}
\]for $i\in I$. Here, the right-hand side carries the smooth compact open topology from Definition \ref{smooth compact open topology}. This topology turns $\Gamma\mathbb{A}$ into a locally convex algebra. We will see soon that the topology on $\Gamma\mathbb{A}$ does not depend on the particular choice of the bundle atlas $(\varphi_i,U_i)_{i\in I}$. 
\end{construction}

\begin{lemma}\label{locally defines sections}
Let $(\mathbb{A},M,A,q)$ be an algebra bundle and $(\varphi_i,U_i)_{i\in I}$ a bundle atlas for $(\mathbb{A},M,A,q)$. Moreover, let $\varphi_{ji}:=\varphi_j^{-1}\circ\varphi_i$ \emph{(}on $(U_i\cap U_j)\times A$\emph{)} and $s_i:=\Phi_i(s)$ for $i,j\in I$ and $s\in\Gamma\mathbb{A}$ \emph{(}cf. Construction \ref{top on space of sections}\emph{)}. Then the following assertions hold:
\begin{itemize}
\item[\emph{(a)}]
If $s\in\Gamma\mathbb{A}$ and $\widehat{s_i}(x):=(x,s_i(x))$ for $i\in I$ and $x\in U_i$, then $\widehat{s_j}=\varphi_{ji}\circ\widehat{s_i}$ for each $i,j\in I$.
\item[\emph{(b)}]
Conversely, if $(s_i)_{i\in I}\in\prod_{i\in I}C^{\infty}(U_i,A)$ satisfies $\widehat{s_j}=\varphi_{ji}\circ\widehat{s_i}$ for each $i,j\in I$, then the map 
\[s:M\rightarrow\mathbb{A},\,\,\,s(x):=(\varphi_i\circ \widehat{s_i})(x),\,\,\,x\in U_i,
\]defines a section of the bundle $(\mathbb{A},M,A,q)$ with $\Phi_i(s)=s_i$ for each $i\in I$.
\end{itemize}
\end{lemma}

\begin{proof}
\,\,\,(a) This is just a simple calculation involving the bundle charts $(\varphi_i,U_i)_{i\in I}$.

(b) Let $x\in U_i\cap U_j$. Then
\[(\varphi_j\circ \widehat{s_j})(x)=(\varphi_j\circ \varphi_{ji}\circ\widehat{s_i})(x)=(\varphi_i\circ \widehat{s_i})(x)
\]shows that the map $s$ is well-defined. It obviously defines a section of the bundle $(\mathbb{A},M,A,q)$ with $\Phi_i(s)=s_i$ for each $i\in I$.
\end{proof}

\begin{proposition}\label{imPhi}
Suppose we are in the situation of Lemma \ref{locally defines sections}. Further, consider the algebra map
\[\Phi_I:(\Gamma\mathbb{A},\mathcal{O}_I)\rightarrow\prod_{i\in I}C^{\infty}(U_i,A),\,\,\,\Phi_I(s):=(\Phi_i(s))_{i\in I}=(s_i)_{i\in I}.
\]Then the following assertions hold:
\begin{itemize}
\item[\emph{(a)}]
We have
\[\im(\Phi_I)=\{(s_i)_{i\in I}:\,(\forall i,j\in I)\,\widehat{s_j}=\varphi_{ji}\circ\widehat{s_i}\}.
\]
\item[\emph{(b)}]
The algebra map $\Phi_I$ is a topological embedding with closed image.
\end{itemize}
\end{proposition}

\begin{proof}
\,\,\,(a) The first assertion is a direct consequence of Lemma \ref{locally defines sections}.

(b) Clearly, the map $\Phi_I$ is injective. Moreover, the topology $O_I$ just defined in Construction \ref{top on space of sections} turns it into a topological embedding. To see that $\im(\Phi_I)$ is closed, let $(s_{\alpha})_{\alpha\in\Lambda}$ be a net in $\Gamma\mathbb{A}$ such that $\Phi_I(s_{\alpha})$ converges to some $(s_i)_{i\in I}\in\prod_{i\in I}C^{\infty}(U_i,A)$. Then $\lim_{\alpha}s_{\alpha,i}(x)=s_i(x)$ for each $i\in I$ and $x\in U_i$. From this we conclude that
\[\widehat{s_j}(x)=\lim_{\alpha}\widehat{s_{\alpha,j}}(x)=\lim_{\alpha}(\varphi_{ji}\circ\widehat{s_{\alpha,i}})(x)=(\varphi_{ji}\circ\widehat{s_i})(x)
\]for all $i,j\in I$ and $x\in U_i\cap U_j$ and thus that $(s_i)_{i\in I}\in\im\Phi_I$ by part (a).
\end{proof}

\begin{theorem}\label{topology unique}
Let $(\mathbb{A},M,A,q)$ be an algebra bundle and $(\varphi_i,U_i)_{i\in I}$ a bundle atlas for $(\mathbb{A},M,A,q)$. Further, let $(\psi_j,V_j)_{j\in J}$ be another bundle atlas for $(\mathbb{A},M,A,q)$. Then the identity map
\[\id:(\Gamma\mathbb{A},\mathcal{O}_I)\rightarrow(\Gamma\mathbb{A},\mathcal{O}_J)
\]is an isomorphism of unital locally convex algebras. In particular, the topology on $\Gamma\mathbb{A}$ does not depend on the particular choice of the bundle atlas. 
\end{theorem}

\begin{proof}
\,\,\,The universal property of the initial topology $\mathcal{O}_J$ implies that the identity map $\id$ is continuous if and only if the maps
\[\Phi_j:(\Gamma\mathbb{A},\mathcal{O}_I)\rightarrow C^{\infty}(V_j,A),\,\,\,s\mapsto s_j
\]are continuous for each $j\in J$. Therefore, we fix $j\in J$ and note that the continuity of $\Phi_j$ follows from Proposition \ref{imPhi} (b) and the continuity of the map
\[\prod_{i\in I}C^{\infty}(U_i,A)\rightarrow\prod_{i\in I} C^{\infty}(V_j\cap U_i,A)\,\,\,(s_i)_{i\in I}\mapsto (g_{ji}\circ(\id_{V_j\cap U_i}\times(s_i)_{\mid V_j}))_{i\in I},
\]where $g_{ji}:(V_j\cap U_i)\times A\rightarrow A$ denotes the smooth map defined by the transition function $\psi_j^{-1}\circ\varphi_i$. A similar argument shows that the ``inverse" map $\id^{-1}$ is continuous. Thus, $\id$ is an isomorphism of unital locally convex algebras.
\end{proof}

\begin{corollary}\label{A frechet GammaA frechet}
Let $A$ be a unital Fr\'echet algebra and $(\mathbb{A},M,A,q)$ an algebra bundle. Then $\Gamma\mathbb{A}$ carries a unique structure of a unital Fr\'echet algebra, when endowed with the topology of Construction \ref{top on space of sections}.
\end{corollary}

\begin{proof}
\,\,\,If $(\varphi_i,U_i)_{i\in I}$ is a countable bundle atlas for $(\mathbb{A},M,A,q)$, then Proposition \ref{imPhi} (b) implies that $\Gamma\mathbb{A}$ carries the structure of a unital Fr\'echet algebra, since the right-hand side is a unital Fr\'echet algebra and the image of $\Phi_I$ is closed. That this structure is unique is now a consequence of Theorem \ref{topology unique}.
\end{proof}

The following remark can be used to verify that certain maps to $\Gamma\mathbb{A}$ are smooth. We recall that the space $\prod_{i\in I}C^{\infty}(U_i,A)$ carries the structure of an infinite-dimensional manifold:

\begin{remark}\label{smooth structure on sections}{\bf(Verifying smoothness on $\Gamma\mathbb{A}$).}\index{$\Gamma\mathbb{A}$!Verifying Smoothness on}
If $(\mathbb{A},M,A,q)$ is an algebra bundle and $(\varphi_i,U_i)_{i\in I}$ a bundle atlas for $(\mathbb{A},M,A,q)$, then the map
\[\Phi_I:\Gamma\mathbb{A}\rightarrow\prod_{i\in I}C^{\infty}(U_i,A),\,\,\,\Phi_I(s):=(\Phi_i(s))_{i\in I}=(s_i)_{i\in I}
\]is a smooth embedding. Indeed, Proposition \ref{imPhi} (b) implies that the map $\Phi_I$ is continuous and linear. Hence, $\Phi_I$ is smooth. Since $\Phi_I$ has closed image, we conclude from [Gl05], Lemma 2.2.7, that a map $f:N\rightarrow\Gamma\mathbb{A}$ from a manifold $N$ to $\Gamma\mathbb{A}$ is smooth if and only if the composition $\Phi_I\circ f$ is smooth.
\end{remark}

\section{Smooth Localization of Sections of Algebra Bundles}

In this section we show that if $A$ is a unital Fr\'echet algebra, $(\mathbb{A},M,A,q)$ an algebra bundle and $f\in C^{\infty}(M,\mathbb{R})$, then the smooth localization of $\Gamma\mathbb{A}$ with respect to $f$ (cf. Definition \ref{sm. loc.}) is isomorphic (as a unital Fr\'echet algebra) to $\Gamma\mathbb{A}_{M_f}$, where 
\[M_f:=\{m\in M:\,f(m)\neq 0\}.
\]To be more precise, we show that the map
\[\phi_f:\Gamma\mathbb{A}_{\{f\}}\rightarrow\Gamma\mathbb{A}_{M_f},\,\,\,[F]\mapsto F\circ\left(\frac{1}{f}\times\id_{M_f}\right)
\]is an isomorphism of unital Fr\'echet algebras. We start with the following lemma:\sindex[n]{$\Gamma\mathbb{A}_{\{f\}}$}

\begin{lemma}\label{algebra section 0}
If $A$ is a unital Fr\'echet algebra, $(\mathbb{A},M,A,q)$ an algebra bundle and $N$ a manifold, then the following assertions hold:
\begin{itemize}
\item[\emph{(a)}]
If $\pr_M:N\times M\rightarrow M$ denotes the canonical projection onto $M$ and $\pr_M^{*}(\mathbb{A})$ the corresponding pull-back bundle, then the map
\[S: C^{\infty}(N,\Gamma\mathbb{A})\rightarrow\Gamma\pr_M^{*}(\mathbb{A}),\,\,\,S(F)(n,m):=(n,m,F(n)(m))
\]is an isomorphism of unital Fr\'echet algebras.
\item[\emph{(b)}]
If $(\varphi_i,U_i)_{i\in I}$ is a bundle atlas for $(\mathbb{A},M,A,q)$, then the algebra map
\[\widehat{\Phi}_I:C^{\infty}(N,\Gamma\mathbb{A})\rightarrow\prod_{i\in I}C^{\infty}(N\times U_i,A),\,\,\,F\mapsto(\Phi_i\circ F)_{i\in I}
\]is a topological embedding with closed image given by
\[\im(\widehat{\Phi}_I)=\{(F_i)_{i\in I}:\,(\forall i,j\in I, n\in N)\,\widehat{F_j(n)}=\varphi_{ji}\circ\widehat{F_i(n)}\}
\]
\emph{(}cf. Lemma \ref{locally defines sections} and Lemma \ref{smooth exp law} \emph{)}.
\end{itemize}
\end{lemma}

\begin{proof}
\,\,\,(a) We first note that both spaces, $C^{\infty}(N,\Gamma\mathbb{A})$ and $\Gamma\pr_M^{*}(\mathbb{A})$, carry the structure of a unital Fr\'echet algebra (cf. Corollary \ref{A frechet GammaA frechet}) and that the map $S$ is a morphism of unital algebras. To see that $S$ is bijective it is enough to note that each section $s$ of the pull-back bundle $\pr_M^{*}(\mathbb{A})$ has the form
\[s:N\times M\rightarrow\pr_M^{*}(\mathbb{A}),\,\,\,s(n,m)=(n,m,F_s(n,m))
\]for some smooth function $F_s:N\times M\rightarrow\mathbb{A}$ satisfying $q\circ F_s(n,\cdot)=\id_M$ for each $n\in N$, i.e., for some $F_s\in C^{\infty}(N,\Gamma\mathbb{A})$. In view of the Open Mapping Theorem (cf. Proposition \ref{open mapping theorem}), it therefore remains to show that the map $S$ is continuous: For this let $(\varphi_i,U_i)_{i\in I}$ be a bundle atlas for $(\mathbb{A},M,A,q)$. Then the continuity of $S$ follows from the definition of the topology on $\Gamma\pr_M^{*}(\mathbb{A})$ (which is induced from the bundle atlas $(\varphi_i,U_i)_{i\in I}$ and the continuity of the maps
\[C^{\infty}(N,\Gamma\mathbb{A})\rightarrow C^{\infty}(N\times U_i,A),\,\,\,F\mapsto \Phi_i\circ F
\]for $i\in I$.

(b) The second assertion immediately follows from part (a) and Proposition \ref{imPhi} applied to the pull-back bundle $\pr_M^{*}(\mathbb{A})$.
\end{proof}

\begin{proposition}\label{algebra section 1}
Let $A$ be a unital Fr\'echet algebra. If $(\mathbb{A},M,A,q)$ is an algebra bundle, $f:M\rightarrow\mathbb{R}$ a smooth function which is nonzero and $M_f:=\{m\in M:\,f(m)\neq 0\}$, then the map
\[\Phi_f:C^{\infty}(\mathbb{R},\Gamma\mathbb{A})\rightarrow\Gamma\mathbb{A}_{M_f},\,\,\,F\mapsto F\circ\left(\frac{1}{f}\times\id_{M_f}\right)
\]is a surjective morphism of unital Fr\'echet algebras. In particular, the map $\Phi_f$ is open.
\end{proposition}

\begin{proof}
\,\,\, The proof of this proposition is divided into three parts:

(i) We first note that $0$ is a regular value for the function
\[h:\mathbb{R}\times M\rightarrow\mathbb{R},\,\,\,(t,m)\mapsto 1-tf(m).
\]In particular, 
\[H:=h^{-1}(0)=\left\{\left(\frac{1}{f(x)},x\right):\,x\in M_f\right\}
\]is a closed submanifold of $\mathbb{R}\times M$. Further, we note that $H$ is diffeomorphic to $M_f$.

(ii) Next, Lemma \ref{algebra section 0} (a) applied to the algebra bundle $(\mathbb{A},M,A,q)$ and $N=\mathbb{R}$ implies that the map
\[S:C^{\infty}(\mathbb{R},\Gamma\mathbb{A})\rightarrow\Gamma\pr_M^{*}(\mathbb{A}),\,\,\,S(F)(r,m):=(r,m,F(r)(m))
\]is an isomorphism of unital Fr\'{e}chet algebras. A similar argument shows that the space $\Gamma\left(\pr_M^{*}(\mathbb{A})\right)_H$ of sections of the pull-back bundle $\pr_M^{*}(\mathbb{A})$ restricted to the submanifold $H$ is isomorphic (as a unital Fr\'echet algebra) to $\Gamma\mathbb{A}_{M_f}$. 

(iii) It therefore remains to verify that the map
\[\Theta':\Gamma\pr_M^{*}(\mathbb{A})\rightarrow\Gamma\left(\pr_M^{*}(\mathbb{A})\right)_H,\,\,\,s\mapsto s_{\mid H}
\]is surjective and continuous: In fact, the map $\Theta'$ is obviously continuous as a restriction map. For its surjectivity we have to show that any section $s:H\rightarrow\pr_M^{*}(\mathbb{A})_H$ can be extended to a global section of $\pr_M^{*}(\mathbb{A})$. Locally this can be done, since $H$ looks like $\mathbb{R}^n$ in $\mathbb{R}^{n+1}$ and $s$ like an $A$-valued function on $\mathbb{R}^n$ (for $n=\dim M$). To get a global extension, we have to use a partition of unity.
\end{proof}

\begin{remark}\label{Gamma mathbb{A}_f frechet}{\bf($\Gamma\mathbb{A}_{\{f\}}$ as Fr\'echet algebra).}
If $A$ is a Fr\'{e}chet algebra, then so is $\Gamma\mathbb{A}$ (cf. Corollary \ref{A frechet GammaA frechet}), and thus the same is true for $\Gamma\mathbb{A}_{\{f\}}$. Here, the last statement follows from the fact that the quotient of a Fr\'{e}chet space by a closed subspace is again Fr\'{e}chet (cf. [Bou55], \S 3.5). 
\end{remark}

\begin{theorem}\label{algebra section 2}{\bf(Smooth localization of sections of algebra bundles).}\index{Smooth Localization of!Algebra Bundles}
Suppose we are in the situation of Proposition \ref{algebra section 1}. The kernel $\ker(\Phi_f)$ is equal to $(1-tf)\cdot C^{\infty}(\mathbb{R},\Gamma\mathbb{A})$. In particular, $(1-tf)\cdot C^{\infty}(\mathbb{R},\Gamma\mathbb{A})$ is a closed ideal of $C^{\infty}(\mathbb{R},\Gamma\mathbb{A})$ and the map
\[\phi_f:\Gamma\mathbb{A}_{\{f\}}\rightarrow\Gamma\mathbb{A}_{M_f},\,\,\,[F]\mapsto F\circ\left(\frac{1}{f}\times\id_{M_f}\right)
\]is an isomorphism of unital Fr\'echet algebras.
\end{theorem}

\begin{proof}
\,\,\,To simplify the notation we define $K_f:=(1-tf)\cdot C^{\infty}(\mathbb{R},\Gamma\mathbb{A})$. The hard part of the proof is to show that $\ker(\Phi_f)=K_f$. Indeed, if we assume for the moment that this equality holds, then the continuity of the map $\Phi_f$ implies that $K_f$ is a closed subspace of $C^{\infty}(\mathbb{R},\Gamma\mathbb{A})$. Moreover, we conclude from the definition of $\Gamma\mathbb{A}_{\{f\}}$, Remark \ref{Gamma mathbb{A}_f frechet} and the Open Mapping Theorem (cf. Proposition \ref{open mapping theorem}) that the map $\phi_f$ is an isomorphism of unital Fr\'echet algebras. Thus, it remains to prove the assumption. In the following let $(\varphi_i,U_i)_{i\in I}$ be a bundle atlas for $(\mathbb{A},M,A,q)$. We procced as follows:

(i) We first use Proposition \ref{imPhi} to see that the algebra map
\[\Phi_{M_f}:\Gamma\mathbb{A}_{M_f}\rightarrow\prod_{i\in I}C^{\infty}(U_i\cap M_f,A),\,\,\,\Phi(s):=(\Phi_i(s))_{i\in I}=(s_i)_{i\in I}.
\]is a topological embedding with closed image given by
\[\im(\Phi_{M_f})=\{(s_i)_{i\in I}:\,(\forall i,j\in I)\,\widehat{s}_j=\varphi_{ji}\circ\widehat{s}_i\}.
\]

(ii) Next, Lemma \ref{algebra section 0} (b) applied to the algebra bundle $(\mathbb{A},M,A,q)$ and $N=\mathbb{R}$ implies that the algebra map
\[\widehat{\Phi}_I:C^{\infty}(\mathbb{R},\Gamma\mathbb{A})\rightarrow\prod_{i\in I}C^{\infty}(\mathbb{R}\times U_i,A),\,\,\,F\mapsto(\Phi_i\circ F)_{i\in I}
\]is a topological embedding with closed image given by
\[B:=\im(\widehat{\Phi}_I)=\{(F_i)_{i\in I}:\,(\forall i,j\in I, r\in\mathbb{R})\,\widehat{F_j(r)}=\varphi_{ji}\circ\widehat{F_i(r)}\}.
\] 

(iii) If $f_i:=f_{\mid U_i}$ and $K_{f_i}$ is the (two-sided) ideal in $C^{\infty}(\mathbb{R}\times U_i,A)$ generated by the element $(1-tf_i)\cdot 1_A$, then Theorem \ref{manifold mod sub=sub} implies that $K_{f_i}$ is closed and that the map
\[R:\prod_{i\in I} C^{\infty}(\mathbb{R}\times U_i,A)\rightarrow\prod_{i\in I}C^{\infty}(U_i\cap M_f,A),\,\,\,(F_i)_{i\in I}\mapsto \left(F_i\circ\left(\frac{1}{f_i}\times\id_{M_f\cap U_i}\right)\right)_{i\in I}
\]is a surjective morphism of unital Fr\'echet algebras with kernel $K:=\prod_{i\in I} K_{f_i}$.

(iv) A short calculation shows that $R\circ\widehat{\Phi}_I=\Phi_{M_f}\circ\Phi_f$. As a consequence we obtain
\[\ker(R_{\mid B})=K\cap B=\widehat{\Phi}_I(\ker(\Phi_f)).
\]In fact, the inclusion ``$\supseteq$" is obivious and for the inclusion ``$\subseteq$" we have to use the injectivity of the map $\Phi_{M_f}$.

(v) In this part of the proof we show that $K\cap B=\widehat{\Phi}_I(K_f)$. Of course, we easily get $\widehat{\Phi}_I(K_f)\subseteq K\cap B$. To verify the other inclusion, let $h=(h_i)_{i\in I}\in K\cap B$. Then $h_i=(1-tf_i)\cdot g_i$ for some $g_i\in C^{\infty}(\mathbb{R}\times U_i,A)$ and there exists $H\in C^{\infty}(\mathbb{R},\Gamma\mathbb{A})$ such that $\widehat{\Phi}_I(H)=h$. It therefore remains to prove that $H\in K_f$. For this we note that the compatibility property for $(h_i)_{i\in I}$ implies the compatibility property for $(g_i)_{i\in I}$, i.e., that
\[\widehat{g_j(r)}=\varphi_{ji}\circ\widehat{g_i(r)}
\]holds for all $i,j\in I$ and $r\in\mathbb{R}$. In particular, there exists an element $G\in C^{\infty}(\mathbb{R},\Gamma\mathbb{A})$ satisfying $\widehat{\Phi}_I(G)=(g_i)_{i\in I}$. As a consequence, $\widehat{\Phi}_I((1-tf)\cdot G)=h$ and since $\widehat{\Phi}_I$ is injective, we conclude that $H=(1-tf)\cdot G\in K_f$ as desired. 

(vi) Finally, part (iv) and (v) lead to $\widehat{\Phi}_I(\ker(\Phi_f))=\widehat{\Phi}_I(K_f)$. Hence, the injectivity of the map $\widehat{\Phi}_I$ implies $\ker(\Phi_f)=K_f$. This proves the assumption and thus the theorem.
\end{proof}



\begin{corollary}\label{algebra section 3}
If $A$ is a unital Fr\'echet algebra, $(\mathbb{A},M,A,q)$ an algebra bundle and $f_U$ a $U$-defining function corresponding to an open subset $U$ of $M$, then the map
\[\phi_U:\Gamma\mathbb{A}_{\{f_U\}}\rightarrow\Gamma\mathbb{A}_U,\,\,\,[F]\mapsto F\circ\left(\frac{1}{f_U}\times\id_U\right)
\]is an isomorphism of unital Fr\'echet algebras.
\end{corollary}

\begin{proof}
\,\,\,This assertion directly follows from Theorem \ref{algebra section 2}.
\end{proof}

\chapter{Free Group Actions from the Viewpoint of Dynamical Systems}\label{NCPB}



In this chapter we lay the foundations for our geometric approach to noncommutative principal bundles. Our main objects of interest are dynamical systems. Again, the expression dynamical system stands for a triple $(A,G,\alpha)$, consisting of a unital locally convex algebra $A$, a topological group $G$ and a group homomorphism $\alpha:G\rightarrow\Aut(A)$, which induces a continuous action of $G$ on $A$. Our main goal is to present a new characterization of free group actions, involving dynamical systems and representations of the corresponding transformation group. Indeed, given a dynamical system $(A,G,\alpha)$, we provide conditions including the existence of ``sufficiently many" representations of $G$ which ensure that the corresponding action
\[\sigma:\Gamma_A\times G\rightarrow\Gamma_A,\,\,\,(\chi,g)\mapsto\chi\circ\alpha(g)
\]of $G$ on the spectrum $\Gamma_A$ of $A$ is free. These observations may be viewed as a first step towards a geometric noncommutative principal bundle theory and will lead to a reasonable definition of noncommutative principal $\mathbb{T}^n$-bundles.

\section{Smooth Dynamical Systems}\label{NGENCG}

Since the \emph{Erlanger Programm} of Felix Klein, the defining concept in the study of a geometry has been its symmetry group. In classical differential geometry the symmetries of a manifold are measured by Lie groups, i.e., one studies smooth group actions of a Lie group $G$ acting by diffeomorphisms on a manifold $M$. From this point of view, it seems to be more natural to study the geometry of a ``noncommutative space" by smooth actions of Lie groups instead of considering coactions of Hopf algebras. Moreover, it might be interesting to construct \emph{infinitesimal objects} and \emph{characteristic classes} for our noncommutative algebras. Since for general associative algebras it is well known that the infinitesimal version of automorphisms are derivations, the derivation-based differential calculus described by Dubois--Violette in [DV88] fits into our approach: If $G$ is a Lie group acting smoothly by automorphism on a (noncommutative) algebra $A$, then the corresponding Lie algebra $\mathfrak{g}$ acts continuously by continuous derivations on $A$. The following proposition may be viewed as the origin of our approach to NCP bundles:

\begin{proposition}\label{smoothness of the group action on the algebra of smooth functions}
If $\sigma:M\times G\rightarrow M$ is a smooth \emph{(}right-\emph{)} action of a Lie group $G$ on a finite-dimensional manifold $M$ \emph{(}possibly with boundary\emph{)} and $E$ is a locally convex space, then the induced \emph{(}left-\emph{)} action 
\[\alpha:G\times C^{\infty}(M,E)\rightarrow C^{\infty}(M,E),\,\,\,\alpha(g,f)(m):=(g.f)(m):=f(\sigma(m,g))
\]of $G$ on the locally convex space $C^{\infty}(M,E)$ is smooth.
\end{proposition}

\begin{proof}
\,\,\,We first recall from [NeWa07], Proposition I.2 that the evaluation map
\[\ev_M:C^{\infty}(M,E)\times M\rightarrow E,\,\,\,(f,m)\mapsto f(m)
\]is smooth. Now, Lemma \ref{smooth exp law} implies that the action map $\alpha$ is smooth if and only if the map
\[\alpha^{\wedge}:C^{\infty}(M,E)\times M\times G\rightarrow E,\,\,\,(f,m,g)\mapsto f(\sigma(m,g))
\]is smooth. Since $$\alpha^{\wedge}=\ev_M\circ(\id_{C^{\infty}(M,E)}\times\sigma),$$ we conclude that $\alpha^{\wedge}$ is smooth as a composition of smooth maps. 
\end{proof}

From the perspective of noncommutative differential geometry, the previous proposition invites us to consider triples $(A,G,\alpha)$ consisting of a (probably noncommutative) unital locally convex algebra $A$, a Lie group $G$ and a group homomorphism $\alpha:G\rightarrow\Aut(A)$, which induces a smooth action of $G$ on $A$. 
In fact, it immediately leads us to the following definition:

\begin{definition}\label{triple}{\bf(Smooth dynamical systems).}\index{Dynamical Systems!Smooth}
We call a triple $(A,G,\alpha)$, consisting of a unital locally convex algebra $A$, a Lie group $G$ and a group homomorphism $\alpha:G\rightarrow\Aut(A)$, which induces a smooth action of $G$ on $A$, a \emph{smooth dynamical system}.
\end{definition}

\begin{example}\label{induced transformation triples}{\bf(Classical group actions).}\sindex[n]{$(C^{\infty}(P),G,\alpha)$}
(a) As the previous discussion shows, a classical example of such a smooth dynamical system is induced by a smooth action $\sigma:M\times G\rightarrow M$ of a Lie group $G$ on a manifold $M$. 


(b) In particular, each principal bundle $(P,M,G,q,\sigma)$ induces 
a smooth dynamical system $(C^{\infty}(P),G,\alpha)$, consisting of the Fr\'echet algebra of smooth functions on the total space $P$, the structure group $G$ and a group homomorphism $\alpha:G\rightarrow\Aut(C^{\infty}(P))$, induced by the smooth action $\sigma:P\times G\rightarrow P$ of $G$ on $P$. 
\end{example}

The following proposition characterizes the fixed point algebra of a smooth dynamical system, which is induced from a principal bundle, as the algebra of smooth functions on the corresponding base space:

\begin{proposition}\label{fixed point algebra of principal bundles}
Let $(P,M,G,q,\sigma)$ be a principal bundle and let $(C^{\infty}(P),G,\alpha)$ be the induced smooth dynamical system. 
Then the map
\[\Psi:C^{\infty}(P)^G\rightarrow C^{\infty}(M)\,\,\,\text{defined by}\,\,\,\Psi(f)(q(p)):=f(p)
\]is an isomorphism of Fr\'{e}chet algebras.
\end{proposition}

\begin{proof}
\,\,\,First we observe that the map $\Psi$ is well-defined and a homomorphism of algebras. Further, the universal property of submersions implies that $\Psi(f)$ defines a smooth function on $M$.  

Next, if $\Psi(f)=0$, then the $G$-invariance of $f$ implies that $f=0$. Hence, $\Psi$ is injective. To see that $\Psi$ is surjective, we choose $h\in C^{\infty}(M)$ and put $f:=h\circ q$. Then $f\in C^{\infty}(P)^G$ and $\Psi(f)=h$. The claim now follows the continuity of $\Psi$ and $\Psi^{-1}=q^*$. 
\end{proof}

In the following we will show that, if $P$ is a manifold, then each smooth dynamical system $(C^{\infty}(P),G,\alpha)$ induces a smooth action of the Lie group $G$ on $P$. As a first step we endow $\Gamma_{C^{\infty}(P)}$ with the structure of a smooth manifold:

\begin{lemma}\label{spec as manifold}
If $P$ is a manifold, then there is a unique smooth structure on $\Gamma_{C^{\infty}(P)}$ for which the map
\[\Phi:P\rightarrow \Gamma_{C^{\infty}(P)},\,\,\,p\mapsto\delta_p
\]becomes a diffeomorphism.
\end{lemma}

\begin{proof}
\,\,\,Proposition \ref{spec of C(M) top} implies that the map $\Phi$ is a homeomorphism. Therefore, $\Phi$ induces a unique smooth structure on $\Gamma_{C^{\infty}(P)}$ such that $\Phi$ becomes a diffeomorphism. 
\end{proof}

\begin{lemma}\label{characterization of smooth maps}
A continuous map $f:M\rightarrow N$ between manifolds $M$ and $N$ is smooth if and only if the composition $g\circ f:M\rightarrow\mathbb{R}$ is smooth for each $g\in C^{\infty}(N,\mathbb{R})$.
\end{lemma}

\begin{proof}
\,\,\,The `` if"-direction is clear. The proof of the other direction is divided into three parts:

(i)  We first note that the map $f$ is smooth if and only if for each $m\in M$ there is an open $m$-neighbourhood $U$ such that $f_{\mid U}:U\rightarrow M$ is smooth. Therefore, let $m\in M$, $n:=f(m)$ and $(\psi,V)$ be a chart around $n$. We now choose an open $n$-neighbourhood $W$ such that $\overline{W}\subseteq V$ and a smooth function $h:N\rightarrow\mathbb{R}$ satisfying $h_{\mid \overline{W}}=1$ and $\supp(h)\subseteq V$. We further choose an open $m$-neighbourhood $U$ such that $f(U)\subseteq W$ (here, we use the continuity of the map $f$). Since the inclusion map $i:W\rightarrow N$ is smooth, it remains to prove that $f_{\mid U}:U\rightarrow W$ is smooth.

(ii) A short observation shows that the map $f_{\mid U}:U\rightarrow W$ is smooth if and only if the map $\psi\circ f_{\mid U}:U\rightarrow\mathbb{R}^n$ is smooth. If $\psi=(\psi_1,\ldots,\psi_n)$, then the last function is smooth if and only if each of its coordinate functions $\psi_i\circ f_{\mid U}:U\rightarrow\mathbb{R}$ is smooth. 

(iii) For fixed $i\in\{1,\ldots,n\}$ we now show that the coordinate function $\psi_i\circ f_{\mid U}:U\rightarrow\mathbb{R}$ is smooth. For this we note that $h_i:=h\cdot\psi_i$ defines a smooth $\mathbb{R}$-valued function on $N$ satisfying ${h_i}_{\mid W}=\psi_i$. Hence, the assumption implies that the map $h_i\circ f:M\rightarrow\mathbb{R}$ is smooth. Since the restriction of a smooth map to an open subsets is smooth again, we conclude from $f(U)\subseteq W$ that 
\[(h_i\circ f)_{\mid U}=\psi_i\circ f_{\mid U}
\]is smooth as desired. This proves the lemma.
\end{proof}


\begin{proposition}\label{smoothness of the group action on the set of characters}
If $P$ is a manifold, $G$ a Lie group and $(C^{\infty}(P),G,\alpha)$ a smooth dynamical system, then the homomorphism $\alpha:G\rightarrow\Aut(C^{\infty}(P))$ induces a smooth \emph{(}right-\emph{)} action
\begin{align*}
\sigma:P\times G\rightarrow P,\,\,\,(\delta_p,g)\mapsto\delta_p\circ\alpha(g)
\end{align*}
of the Lie group $G$ on the manifold $P$. Here, we have identified $P$ with the set of characters via the map $\Phi$ from Lemma \ref{spec as manifold}.
\end{proposition}

\begin{proof}
\,\,\,The proof of this proposition is divided into two parts:

(i) We again use [NeWa07], Proposition I.2, which states that the evaluation map
\[\ev_P:C^{\infty}(P)\times P\rightarrow \mathbb{K},\,\,\,(f,p)\mapsto f(p)
\]is smooth. From this we conclude that the map $\sigma$ is continuous (cf. Proposition \ref{cont. action II}). 

(ii) In view of part (i), we may use Lemma \ref{characterization of smooth maps} to verify the smoothness of $\sigma$. Indeed, the map $\sigma$ is smooth if and only if the map
\[\sigma_f:P\times G\rightarrow\mathbb{R},\,\,\,(\delta_p,g)\mapsto\sigma(\delta_p,g)(f)=(\alpha(g,f))(p)
\]is smooth for each $f\in C^{\infty}(P,\mathbb{R})$. Therefore, we fix $f\in C^{\infty}(P,\mathbb{R})$ and note that  \[\sigma_f=\ev_P\circ(\id_P\times\alpha_f),
\]where $\alpha_f:G\rightarrow C^{\infty}(P)\,\,\,g\mapsto\alpha(g,f)$ denotes the smooth orbit map of $f$. Hence, the map $\sigma_f$ is smooth as a composition of smooth maps. Since $f$ was arbitrary, the map $\sigma$ is smooth.
\end{proof}

\begin{remark}\label{inverse constructions}{\bf(Inverse constructions).}
Note that the constructions of Proposition \ref{smoothness of the group action on the algebra of smooth functions} and Proposition \ref{smoothness of the group action on the set of characters} are inverse to each other.
\end{remark}

\begin{remark}\label{remark of free action in geometry}
Since we are interested in principal bundles, it is reasonable to ask if there exist natural algebraic conditions on a smooth dynamical system $(C^{\infty}(P),G,\alpha)$ which ensure the freeness of the induced action $\sigma$ of $G$ on $P$ of Proposition \ref{smoothness of the group action on the set of characters}. In fact, if this is the case and if the action is additionally proper, then we obtain a principal bundle $(P,P/G,G,\pr,\sigma)$, where $\pr:P\rightarrow P/G,\,\,\,p\mapsto p.G$ denotes the corresponding orbit map. We will treat this question in the next section.
\end{remark}

\section{Free Dynamical Systems}\label{section:free dynamical systems}

In this section we introduce the concept of a \emph{free dynamical system}. Loosely speaking, we call a dynamical system $(A,G,\alpha)$ free, if the locally convex algebra $A$ is commutative and the topological group $G$ admits ``sufficiently many" representations such that a certain family of maps defined on the corresponding modules associated to $A$ are surjective. We will in particular see how this condition implies that the induced action 
\[\sigma:\Gamma_A\times G\rightarrow\Gamma_A,\,\,\,(\chi,g)\mapsto\chi\circ\alpha(g)
\]of $G$ on the spectrum $\Gamma_A$ of $A$ is free. We start with some basics from the representation theory of (topological) groups, which will later be important for deducing the freeness property:

\begin{definition}\label{sep. the. points of G}{\bf(Separating representations).}\index{Representations!Separating}
Let $G$ be topological group. We say that a family $(\pi_j,V_j)_{j\in J}$ of (continuous) representations of $G$ \emph{separates the points of} $G$ if for each $g\in G$ with $g\neq 1_G$, there is a $j\in J$ such that $\pi_j(g)\neq\id_{V_j}$.
\end{definition}

\begin{lemma}\label{princ. bdl. cond.}
Let $G$ be a topological group and suppose that $(\pi_j,V_i)_{j\in J}$ is a family of point separating representations of $G$. If $g\in G$ is such that $\pi_j(g)=\id_{V_j}$ for all $j\in J$, then $g=1_G$.
\end{lemma}

\begin{proof}
\,\,\,The claim immediately follows from Definition \ref{sep. the. points of G}.
\end{proof}

\begin{remark}\label{faithful representations}{\bf(Faithful representations).}\index{Representations!Faithful}
We recall that each faithful representation $(\pi,V)$ of a topological group $G$ separates the points of $G$.
\end{remark}



An important class of groups that admit a family of separating representations is given by the locally compact groups:

\begin{theorem}\label{gelfand-raikov}{\bf(Gelfand--Raikov).}\index{Theorem!of Gelfand--Raikov}
Each locally compact group $G$ admits a family of continuous unitary irreducible representations that separates the points of $G$.
\end{theorem}

\begin{proof}
\,\,\,A proof of this statement can be found in the very nice paper [Yo49].
\end{proof}

\begin{definition}\label{sections again}{\bf(``Associated space of sections'').}\sindex[n]{$\Gamma_A V$}
Let $A$ be a unital locally convex algebra and $G$ a topological group. If $(A,G,\alpha)$ is a dynamical system and $(\pi,V)$ a (continuous) representation of $G$, then there is a natural (continuous) action of $G$ on the tensor product $A\otimes V$ defined on simple tensors by $g.(a\otimes v):=(\alpha(g).a)\otimes(\pi(g).v)$. We write
\begin{align*}
\Gamma_A V:=(A\otimes V)^G=\big\{s\in A\otimes V:(\forall g\in G)\,(\alpha(g)\otimes\id_V)(s)=(\id_A\otimes\pi(g^{-1}))(s)\big\}
\end{align*}
for the set of fixed elements under this action.
\end{definition}

\begin{lemma}
Let $(A,G,\alpha)$ be as in Definition \ref{sections again}. 
If $A^G$ is the corresponding fixed point algebra and $(\pi,V)$ a continuous representation of $G$, then the map
\[\rho:\Gamma_A V\times A^G\rightarrow\Gamma_A V,\,\,\,(a\otimes v,b)\mapsto ab\otimes v
\]defines on $\Gamma_A V$ the structure of a locally convex $A^G$-module.
\end{lemma}

\begin{proof}
\,\,\,According to Proposition \ref{EoF as B-module}, $A\otimes V$ carries the structure of a locally convex $B$-module. Since $B=A^G$, the same holds for the restriction to the (closed) subspace $\Gamma_A V$.
\end{proof}

\begin{example}\label{example of section}{\bf(The classical case).}
If $(P,M,G,q,\sigma)$ is a principal bundle, $(C^{\infty}(P),G,\alpha)$ the corresponding smooth dynamical system from Example \ref{induced transformation triples} (b) and $(\pi,V)$ a finite-di\-men\-sio\-nal  representation of $G$, then an easy observation shows that
\[C^{\infty}(P)\otimes V\cong C^{\infty}(P,V)
\](as Fr\'echet spaces) and further that
\[\Gamma_{C^{\infty}(P)}V\cong(C^{\infty}(P)\otimes V)^G\cong C^{\infty}(P,V)^G.
\]Now, Corollary \ref{sections of an associated vector bundle top} implies that $\Gamma_{C^{\infty}(P)}V\cong\Gamma\mathbb{V}$ is a (topological) isomorphism of $C^{\infty}(M)$-modules.
\end{example}

\begin{remark} 
In view of Example \ref{example of section}, the $A^G$-module $\Gamma_A V$ generalizes the \emph{space of sections} associated to the dynamical system $(A,G,\alpha)$ and the representation $(\pi,V)$ of $G$.
\end{remark}

We now come to the central definition of this section. Note that $A$ is assumed to be a commutative algebra, since our considerations depend on the existence of enough characters:

\begin{definition}\label{free dynamical systems}{\bf (Free dynamical systems).}\index{Dynamical Systems!Free}
Let $A$ be a commutative unital locally convex algebra and $G$ a topological group. A dynamical system $(A,G,\alpha)$ is called \emph{free} if there exists a family $(\pi_j,V_j)_{j\in J}$ of point separating representations of $G$ such that the map
\[\ev^j_{\chi}:=\ev^{V_j}_{\chi}:\Gamma_A V_j\rightarrow V_j,\,\,\,a\otimes v\mapsto\chi(a)\cdot v
\]is surjective for all $j\in J$ and all $\chi\in\Gamma_A$. 
\end{definition}

\begin{theorem}\label{freeness of induced action}{\bf(Freeness of the induced action).}\index{Freeness!of the Induced Action}
If $(A,G,\alpha)$ is a free dynamical system, then the induced action
\[\sigma:\Gamma_A\times G\rightarrow\Gamma_A,\,\,\,(\chi,g)\mapsto\chi\circ\alpha(g)
\]of $G$ on the spectrum $\Gamma_A$ of $A$ is free. 
\end{theorem}

\begin{proof}
\,\,\,We divide the proof of this theorem into four parts:

(i) In order to verify the freeness of the map $\sigma$, we have to show that the stabilizer of each element of $\Gamma_A$ is trivial: Consequently, we fix $\chi_0\in\Gamma_A$ and let $g_0\in G$ with $\chi_0\circ\alpha(g_0)=\chi_0$.

(ii) Since $(A,G,\alpha)$ is assumed to be a free dynamical system, there exists a family $(\pi_j,V_j)_{j\in J}$ of point separating representations of $G$ for which the map
\[\ev^j_{\chi}:\Gamma_A V_j\rightarrow V_j,\,\,\,a\otimes v\mapsto\chi(a)\cdot v
\]is surjective for all $j\in J$ and all $\chi\in\Gamma_A$. In particular, we can choose $j\in J$, $v\in V_j$ and $s\in\Gamma_A V_j$ with $\ev^j_{\chi_0}(s)=v$. We recall that the element $s$ satisfies the equation
\begin{align}
(\alpha(g_0)\otimes\id_{V_j})(s)=(\id_A\otimes\pi_j(g_0^{-1}))(s).\label{freeness equation}
\end{align}

(iii) Applying $\chi_0\otimes\id_{V_j}$ to the left of equation (\ref{freeness equation}) leads to
\[((\chi_0\circ\alpha(g_0))\otimes\id_{V_j})(s)=(\chi_0\otimes\pi_j(g_0^{-1}))(s).
\]Thus, we conclude from $\chi_0\circ\alpha(g_0)=\chi_0$ that
\[(\chi_0\otimes\id_{V_j})(s)=(\chi_0\otimes\pi_j(g_0^{-1}))(s)=\pi_j(g_0^{-1})((\chi_0\otimes\id_{V_j})(s)).
\]

(iv) We finally note that $s\in\Gamma_A V_j$ implies that
\[(\chi_0\otimes\id_{V_j})(s)=\ev^j_{\chi_0}(s)=v.
\]In view of part (iii) this shows that $v=\pi_j(g_0)(v)$. As $j\in J$ and $v\in V_j$ were arbitrary, we conclude that $\pi_j(g_0)=\id_{V_j}$ for all $j\in J$ and therefore that $g_0=1_G$ (cf. Lemma \ref{princ. bdl. cond.}). This completes the proof.
\end{proof}

\section{A New Characterization of Free Group Actions in Classical Geometry}\label{ANCFGACG}

In this part of the chapter we apply the results of the previous section to dynamical systems arising from group actions in classical geometry. In particular, we will see how this leads to a new characterization of free group actions. For this purpose we have to restrict our attention to Lie groups that admit a family of finite-dimensional continuous point separating representations.

\begin{definition}\label{linearizer}{\bf(The linearizer).}\index{Linearizer}
For a Lie group $G$ we define its \emph{linearizer} as the subgroup $\Lin(G)$\sindex[n]{$\Lin(G)$} which is the intersection of the kernels of all finite-dimensional continuous representations of $G$.
\end{definition}

\begin{lemma}\label{linerarizer/point sep. rep.}
Let $G$ be a Lie group. Then the following statements are equivalent:
\begin{itemize}
\item[\emph{(a)}]
The finite-dimensional continuous representations separate the points of $G$.
\item[\emph{(b)}]
$\Lin(G)=\{1_G\}$.
\end{itemize}
\end{lemma}

\begin{proof}
\,\,\,We just have to note that $\Lin(G)$ is non-trivial if and only if there exists an element $g\in G$ which lies in the kernels of all finite-dimensional continuous representation of $G$.
\end{proof}

The following theorem shows that the property $\Lin(G)=\{1_G\}$ characterizes the groups which admit a faithful finite-dimensional linear representation:

\begin{theorem}\label{linear lie groups}{\bf(Existence of faithful finite-dimensional representations).}\index{Representations!Faithful Finite-Dimensional}
For a connected Lie group $G$ the following statements are equivalent:
\begin{itemize}
\item[\emph{(a)}]
There exists a faithful finite-dimensional continuous representation of $G$.
\item[\emph{(b)}]
There exists a faithful finite-dimensional continuous representation of $G$ with closed image.
\item[\emph{(c)}]
$\Lin(G)=\{1_G\}$.
\end{itemize}
\end{theorem}

\begin{proof}
\,\,\,(a) $\Leftrightarrow$ (c): The proof of this equivalence is part of [HiNe10], Theorem 15.2.7.

(a) $\Leftrightarrow$ (b): This nontrivial statement is carried out as a bunch of exercises at the end of [HiNe10], Section 15.2.
\end{proof}

\begin{remark}
In view of Lemma \ref{linerarizer/point sep. rep.} and Theorem \ref{linear lie groups}, it is exactly the \emph{linear Lie groups} that admit a family of finite-dimensional continuous point separating representations.
\end{remark}

\begin{theorem}\label{free action theorem}
Let $P$ be a manifold and $G$ be a Lie group. Then the following assertions hold:
\begin{itemize}
\item[\emph{(a)}]
If the smooth dynamical system $(C^{\infty}(P),G,\alpha)$ is free and, in addition the induced action $\sigma$ is proper, then we obtain a principal bundle $(P,P/G,G,\pr,\sigma)$ \emph{(}cf. Remark \ref{remark of free action in geometry}\emph{)}.
\item[\emph{(b)}]
Conversely, if $G$ is a linear Lie group and $(P,M,G,q,\sigma)$ a principal bundle, then the corresponding smooth dynamical system $(C^{\infty}(P),G,\alpha)$ is free.
\end{itemize}
\end{theorem}

\begin{proof}
\,\,\,(a) We first recall that the induced action $\sigma:P\times G\rightarrow P$ is smooth by Proposition \ref{smoothness of the group action on the set of characters}. Furthermore, Theorem \ref{freeness of induced action} implies that the map $\sigma$ is free. Since $\sigma$ is additionally assumed to be proper, the claim now follows from the Quotient Theorem (cf. Remark \ref{principal bundles II}).

(b) 
Since $G$ is a linear Lie group, there exists a faithful finite-dimensional continuous representation $(\pi,V)$ of $G$. In particular, this representation separates the points of $G$ (cf. Remark \ref{faithful representations}). In order to prove the freeness of the smooth dynamical system $(C^{\infty}(P),G,\alpha)$, it would therefore be enough to show that the map
\begin{align}
\ev_{p}:C^{\infty}(P,V)^G\rightarrow V,\,\,\,f \mapsto f(p)\label{check condition}
\end{align}
is surjective for all $p\in P$. We proceed as follows: 

(i) We first observe that the surjectivity of the maps (\ref{check condition}) is a local condition. Further, we (again) recall that according to Corollary \ref{sections of an associated vector bundle top} the map
\[\Psi:C^{\infty}(P,V)^G\rightarrow\Gamma\mathbb{V},\,\,\,\Psi(f)(q(p)):=[p,f(p)],
\]where $\mathbb{V}$ denotes the vector bundle over $M$ associated to $(P,M,G,q,\sigma)$ via the representation $(\pi,V)$ of $G$, is a (topological) isomorphism of $C^{\infty}(M)$-modules.

(ii) Now, we choose $p\in P$, $v\in V$ and construct a smooth section $s\in\Gamma\mathbb{V}$ with $s(q(p))=[p,v]$. Indeed, such a section can always be constructed locally and then extended to the whole of $M$ by multiplying with a smooth bump function. The construction of $s$ implies that the function $f_s:=\Psi^{-1}(s)\in C^{\infty}(P,V)^G$ satisfies $f_s(p)=v$. As $p\in P$, $v\in V$ were arbitrary, this completes the proof.
\end{proof}

\begin{remark}
Note that Theorem \ref{free action theorem} (a) means that it is possible to test the freeness of a (smooth) group action $\sigma:P\times G\rightarrow P$ in terms of surjective maps defined on spaces of sections of associated (singular) vector bundles.
\end{remark}

\begin{remark}{\bf(Why linear Lie groups?).}
(a) The crucial idea of the proof of Theorem \ref{free action theorem} (b) is to use the identification of the space $C^{\infty}(P,V)^G$ with the space of sections $\Gamma\mathbb{V}$ of the vector bundle $\mathbb{V}$ over $M$ associated to $(P,M,G,q,\sigma)$ via the representation $(\pi,V)$ of $G$. Note that this identification only holds for smooth representations $(\pi,V)$ of $G$. Since finite-dimensional continuous representations of Lie groups are automatically smooth and there is not many literature about smooth point separating representations of Lie groups, we restrict our attention to Lie groups that admit a family of finite-dimensional continuous point separating representations. 

(b) If $G$ is an arbitrary Lie group, then an interesting observation is that the natural $G$-action on $C^{\infty}(G,\mathbb{R})$ separates the points of $G$.
\end{remark}

The following corollary gives a one-to-one correspondence between free dynamical systems and free group action in the case where the structure group $G$ is a compact Lie group:

\begin{corollary}\label{characterization of free group actions}{\bf(Characterization of free group actions).}\index{Characterization of Free Group Actions}
Let $P$ be a manifold, $G$ a compact Lie group and $(C^{\infty}(P),G,\alpha)$ a smooth dynamical system. Then the following statements are equivalent:
\begin{itemize}
\item[\emph{(a)}]
The smooth dynamical system $(C^{\infty}(P),G,\alpha)$ is free.
\item[\emph{(b)}]
The induced smooth group action $\sigma:P\times G\rightarrow P$ is free.
\end{itemize}
In particular, in this situation the two concepts of freeness coincide.
\end{corollary}

\begin{proof}
\,\,\,We first note that each compact Lie group $G$ admits a faithful finite-dimensional continuous representation, i.e., each compact Lie group $G$ is linear. Indeed, a proof of this statement can be found in [HiNe10], Theorem 11.3.9. Moreover, since $G$ is compact, the properness of the action $\sigma_P$ is automatic. Hence, the equivalence follows from Theorem \ref{free action theorem}. The last statement is now a consequence of Remark \ref{inverse constructions}.
\end{proof}

\section{Dynamical Systems with Compact Structure Group}\label{FACAG}



In this section we study dynamical systems $(A,G,\alpha)$ with compact structure group $G$ with the help of the `fine structure" of $A$ with respect to $G$, i.e., with the Fourier decomposition of the system $(A,G,\alpha)$. We show that if $A$ is a complete commutative CIA, $G$ a compact group and $(A,G,\alpha)$ a dynamical system, then each character of $B:=A^G$ can be extended to a character of $A$. In particular, the natural map $\Gamma_A\rightarrow\Gamma_B$, $\chi\mapsto\chi_{\mid B}$ is surjective.

\subsection{A Structure Theorem for Modules of Compact Groups}

The main goal of this subsection is to obtain a kind of Fourier decomposition for a dynamical systems $(A,G,\alpha)$ with compact structure group. To do this in a great generality, we have to introduce the concept of $G$-completeness. This basically means, that we restrict our attention to locally convex spaces permitting integration. Since the representation theory for compact abelian groups is well-known in the case where the ground field is $\mathbb{K}=\mathbb{C}$, let us first make the following convention:

\begin{convention}
In the remaining part of this chapter all vector spaces are assumed to be defined over the ground field $\mathbb{K}=\mathbb{C}$ if not explicitly mentioned otherwise.
\end{convention}

\begin{definition}{\bf($G$-completeness).}\index{Completeness!$G$-}
Let $G$ be a compact group and $V$ be a locally convex space. Then $V$ is called $G$-\emph{complete}, if there is a continuous linear map
\[I:C(G,V)\rightarrow V
\]such that
\[\omega(I(f))=\int_G\omega(f(g))\,dg\,\,\,\text{for all}\,\,\,\omega\in V',
\]where $dg$ denotes the normalized Haar measure on $G$.
\end{definition}

\begin{remark}\label{sources of G-complete spaces 1}
By [Homo06], Proposition 3.30, every complete locally convex space is $G$-complete for any compact group $G$.
\end{remark}

\begin{remark}\label{sources of G-complete spaces 2}{\bf(Source of examples).}
In the following we will be concerned with $G$-complete continuous representations $(\pi,V)$ of $G$. By the preceding remark, the family of all dynamical systems $(A,G,\alpha)$ with a complete unital locally convex algebra $A$ and a compact group $G$ is a source of $G$-complete locally convex $G$-modules. 
\end{remark}

We continue with a structure theorem on the general representation theory of compact groups. If $G$ is a compact group, we write $\widehat{G}$\sindex[n]{$\widehat{G}$} for the set of all equivalence classes of finite-dimensional irreducible representations of $G$:

\begin{theorem}{\bf (The structure theorem of $G$-modules).}\label{str thm of G-mod}\index{Structure Theorem of $G$-Modules}
Let $G$ be a compact group and $(\pi,V)$ be a $G$-complete continuous representation of $G$. Further, for each $\varepsilon\in\widehat{G}$ let $\chi_{\varepsilon}$ denote the character of $\varepsilon$ and $d_{\varepsilon}:=\chi_{\varepsilon}(1_G)$ the degree of $\varepsilon$ \emph{(}i.e., the dimension of $\varepsilon$\emph{)}. If
\[P_{\varepsilon}(v):=d_{\varepsilon}\cdot\int_G\overline{\chi_{\varepsilon}}(g)\cdot(\pi(g).v)\,dg,
\]where $v\in V$ and $dg$ denotes the normalized Haar measure on $G$, then the following assertions hold:\sindex[n]{$P_{\varepsilon}$}

\begin{itemize}
\item[\emph{(a)}]
For each $\varepsilon\in\widehat{G}$ the map $P_{\varepsilon}:V\rightarrow V$ is a continuous $G$-equivariant projection onto the $G$-invariant subspace $V_{\varepsilon}:= P_{\varepsilon}(V)$. In particular, $V_{\varepsilon}$ is algebraically and topologically a direct summand of $V$.

\item[\emph{(b)}]
The module direct sum $V_{\emph{\text{fin}}}:=\bigoplus_{\varepsilon\in\widehat{G}}V_{\varepsilon}$ is a dense subspace of $V$.
\end{itemize}
\end{theorem}

\begin{proof}
\,\,\,The statements follow from [HoMo06], Theorem 3.51 and Theorem 4.22.
\end{proof}

\begin{remark}{\bf (Harish--Chandra).}\label{Harish-Chandra}\index{Theorem!of Harish--Chandra}
If $G$ is a compact Lie group and $v$ a smooth vector for a complete locally convex continuous representations $(\pi,V)$ of $G$, then [War72], Theorem 4.4.2.1 implies the absolute convergence of the ``Fourier series"
\[\sum_{\varepsilon\in\widehat{G}}P_{\varepsilon}(v).
\]
\end{remark}

\begin{remark}{\bf($G$-finite elements).}\label{almost invariant elements}\index{$G$-Finite Elements}
Let $G$ be a topological group and $(\pi,V)$ be a locally convex continuous representation of $G$. An element $v\in V$ is called \emph{$G$-finite} if $\Span(Gv)$, the linear span of its orbit $Gv$, is finite dimensional. It actually turns out that in the situation of Theorem \ref{str thm of G-mod}, $V_{\text{fin}}$ coincides with the set of $G$-finite vectors in $V$ (cf. [HoMo06], Definition 3.1 and Theorem 4.22).
\end{remark}

From the preceding remark, we obtain the following proposition will be crucial for the following subsection:

\begin{proposition}\label{E_fin subalgebra}{\bf($A_{\text{fin}}$ as subalgebra).}
If $A$ is a complete unital locally convex algebra, $G$ a compact group and $(A,G,\alpha)$ a dynamical system, then $A_{\emph{\text{fin}}}$ is a \emph{(}dense\emph{)} subalgebra of $A$.
\end{proposition}

\begin{proof}
\,\,\,In view of Remark \ref{sources of G-complete spaces 1} and Theorem \ref{str thm of G-mod}, it remains to show that $A_{\text{fin}}$ is a subalgebra of $A$: As a matter of fact, this is a consequence of Remark \ref{almost invariant elements} and 
\[\Span(Gaa')\subseteq\Span((Ga)(Ga'))\subseteq\Span\left((\Span(Ga))(\Span(Ga'))\right)
\]for $a,a'\in A_{\text{fin}}$.
\end{proof}

\subsection{Extending Characters on Fixed Point Algebras}\label{extending characters on fixed point algebras}

In this subsection we present some results on the extendability of characters on fixed point algebras of CIAs. The discussion is inspired by [NeSe09], Lemma 3.1. 

\begin{lemma}\label{extension of char on fixed point algebras I} 
Let $A$ be a complete unital locally convex algebra, $G$ a compact group and $(A,G,\alpha)$ a dynamical system. Further, let $A^G$ be the corresponding fixed point algebra. Then the following assertions hold:
\begin{itemize}
\item[\emph{(a)}]
If $I$ is a proper left ideal in $A^G$, then 
\[A_{\emph{\text{fin}}}\cdot I=\bigoplus_{\varepsilon\in\widehat{G}}A_{\varepsilon}\cdot I
\]defines a proper left ideal in $A_{\emph{\text{fin}}}$ which contains $I$.
\item[\emph{(b)}]
If $I$ is a proper closed left ideal in $A^G$ and $J$ is the closure of $A_{\text{fin}}\cdot I$ in $A_{\text{fin}}$, i.e.,
\[J:=\overline{A_{\emph{\text{fin}}}\cdot I}^{A_{\emph{\text{fin}}}},
\]then
$J$ is a proper closed left ideal in $A_{\emph{\text{fin}}}$ which contains $I$.
\end{itemize}
\end{lemma}

\begin{proof}
\,\,\,
(a) We first observe that Lemma \ref{E_fin subalgebra} and a short calculation show that $A_{\text{fin}}\cdot I$ defines a left ideal of $A_{\text{fin}}$. Clearly, $A_{\text{fin}}$ contains $I$. To see that $A_{\text{fin}}\cdot I$ is proper, we assume the contrary, i.e., that $1_A\in A_{\text{fin}}\cdot I=\bigoplus_{\varepsilon\in\widehat{G}}A_{\varepsilon}\cdot I$. Then $1_A\in A^G$ implies that $1_A\in B\cdot I=I$, which contradicts the fact that $I$ is a proper ideal of $B$. Thus, $A_{\text{fin}}\cdot I$ is a proper ideal in $A_{\text{fin}}$ which contains $I$.

(b) Part (a) of the lemma and the definition of $J$ imply that $J$ is a closed left ideal in $A_{\text{fin}}$ which contains $I$. To see that $J$ is proper, we again assume the contrary, i.e., that $1_A\in J$. Then there exists a net $(a_{\alpha})_{\alpha\in\Gamma}$ in $A_{\text{fin}}\cdot I$ such that $\lim_{\alpha}a_{\alpha}=1_A$. Therefore, the continuity of the projection $P_0:A\rightarrow A$ onto $A^G$ (cf. Theorem \ref{str thm of G-mod} (a) applied to the equivalence class of the trivial representation) leads to
\[1_A=P_0(1_A)=P_0(\lim_{\alpha}a_{\alpha})=\lim_{\alpha}P_0(a_{\alpha}).
\]Since $I$ is closed in $A^G$ and $P_0(a_{\alpha})\in A^G\cdot I=I$ for all $\alpha\in\Gamma$, we conclude that $1_A\in I$. This contradicts the fact that $I$ is a proper ideal of $A^G$. Thus, $J$ is a proper closed left ideal in $A_{\text{fin}}$ which contains $I$.
\end{proof}

\begin{lemma}\label{extension of char on fixed point algebras II}
Let $A$ be a complete unital locally convex algebra and $A'$ be a dense unital subalgebra of $A$. If $I$ is a proper closed left ideal in $A'$, then $\overline{I}$ is a proper closed left ideal in $\overline{A'}=A$.
\end{lemma}

\begin{proof}
\,\,\,A short calculation shows that $\overline{I}$ is a closed left ideal in $\overline{A'}=A$. Next, we note that $I=\overline{I}\cap A'$. Indeed, the inclusion $``\subseteq "$ is obvious and for the other inclusion we use the fact that $I$ is closed in $A'$. Thus, if $\overline{I}$ is not proper, i.e., $\overline{I}=A$, then $I=A'$ which contradicts the fact that $I$ is a proper ideal of $A'$. Hence, $\overline{I}$ is a proper closed left ideal in $A$.
\end{proof}

\begin{corollary}\label{extension of char on fixed point algebras III}{\bf(Extending ideals).}\index{Extending!Ideals}
Let $A$ be a complete unital locally convex algebra, $G$ a compact group and $(A,G,\alpha)$ a dynamical system. Then each proper closed left ideal in $A^G$ is contained in a proper closed left ideal in $A$.
\end{corollary}

\begin{proof}
\,\,\,If $I$ is a proper closed left ideal in $A^G$, then Lemma \ref{extension of char on fixed point algebras I} (b) implies that $I$ is contained in a proper closed left ideal in $A_{\text{fin}}$. Since $A_{\text{fin}}$ is a dense subalgebra of $A$ (cf. Proposition \ref{E_fin subalgebra}), the claim is a consequence of Lemma \ref{extension of char on fixed point algebras II}.
\end{proof}




\begin{proposition}\label{extension of char on fixed point algebras IV}{\bf(Extending characters).}\index{Extending!Characters}
Let $A$ be a complete commutative CIA, $G$ a compact group and $(A,G,\alpha)$ a dynamical system. 
Then each character $\chi:A^G\rightarrow\mathbb{C}$ extends to a character $\widetilde{\chi}:A\rightarrow\mathbb{C}$.
\end{proposition}

\begin{proof}
\,\,\,The kernel $I:=\ker\chi$ is a proper closed ideal in the subalgebra $A^G$, so that Corollary \ref{extension of char on fixed point algebras III} implies that $I$ is contained in a proper closed ideal in $A$. In particular, it is contained in a proper maximal ideal $J$ of $A$. According to Lemma \ref{max ideals}, $J$ is the kernel of some character $\widetilde{\chi}:A\rightarrow\mathbb{C}$, i.e., $J:=\ker\widetilde{\chi}$. Since $I$ is a maximal ideal in the unital algebra $B$ and 
\[I=I\cap A^G\subseteq J\cap A^G\subseteq A^G,
\]we conclude that $I=J\cap A^G$. Thus, $A^G=I\oplus\mathbb{C}=(J\cap A^G)\oplus\mathbb{C}$ proves that $\widetilde{\chi}$ extends $\chi$.
\end{proof}

\begin{corollary}\label{extension of char on fixed point algebras V}
If $P$ is a compact manifold, $G$ a compact group and $(C^{\infty}(P),G,\alpha)$ a dynamical system, then each character $\chi:C^{\infty}(P)^G\rightarrow\mathbb{C}$ extends to a character $\widetilde{\chi}:C^{\infty}(P)\rightarrow\mathbb{C}$.
\end{corollary}

\begin{proof}
\,\,\,This follows from Proposition \ref{extension of char on fixed point algebras IV}, since $C^{\infty}(P)$ is a complete commutative CIA.
\end{proof}

\section{Free Dynamical Systems with Compact Abelian Structure Group}

In this section we first show that the characters of a compact abelian group form a family of point separating representations. We then show how Theorem \ref{str thm of G-mod} simplifies in the compact abelian case and use this observation to rewrite the freeness condition for a dynamical system $(A,G,\alpha)$ with compact abelian structure group $G$. In particular, we present natural conditions which ensure the freeness of such a dynamical system. These conditions may even be formulated if the algebra $A$ is noncommutative.

\subsection{The Character Group of a Compact Abelian Group}

Our first goal is to present a family of point separating representations in the case where $G$ is compact and abelian. We therefore need the following lemma:

\begin{lemma}\label{G abelian}
If $(\pi,V)$ is a finite-dimensional irreducible representation of a topological group $G$, then the following assertions hold:
\begin{itemize}
\item[\emph{(a)}]
The space$\End_G(V)$ is a division ring over $\mathbb{K}$. If $\mathbb{K}=\mathbb{C}$, then $\End_G(V)=\mathbb{C}\cdot\id_V$.
\item[\emph{(b)}]
If $G$ is abelian, then $\dim_{\mathbb{C}}V=1$, and there is a group homomorphism $\varphi:G\rightarrow\mathbb{C}^{\times}$ such that $\pi(g).v=\varphi(g)\cdot v$ for all $v\in V$.
\end{itemize}
\end{lemma}

\begin{proof}
\,\,\,(a) Let $\rho:V\rightarrow V$ be $G$-equivariant. Then $\ker\rho$ and $\im\rho$ are $G$-invariant subspaces. If $\rho\neq 0$, then $\ker\rho=\{0\}$ and $\im\rho=V$ follow, and so $\rho$ is bijective, that is, has an inverse. Thus $\End_G(V)$ is a division ring over $\mathbb{K}$.

If $\mathbb{K}=\mathbb{C}$ and $0\neq\rho\in\End_G(V)$, then $\varphi$ has a nonzero eigenvalue $\lambda$. Then $\rho-\lambda\cdot \id_V$ is an element in $\End_G(V)$ with nonzero kernel, hence must be zero by the preceding. Thus $\rho=\lambda\cdot\id_V$.

(b) Now, we assume that $G$ is abelian and that $\mathbb{K}=\mathbb{C}$. The commutativity of $G$ implies \[\pi(G)\subseteq\End_G(V)=\mathbb{C}\cdot\id_V.
\]
Hence, for each $g\in G$ there is a $\varphi(g)\in\mathbb{C}$ such that $\pi(g)=\varphi(g)\cdot\id_V$. A short and simple calculation immediately shows that $\varphi$ is a group homomorphism from $G$ to $\mathbb{C}^{\times}$. Moreover, every vector subspace of $V$ is $G$-invariant. Hence, the irreducibility of $V$ implies $\dim_{\mathbb{C}}V=1$.
\end{proof}

\begin{definition}\label{character group}{\bf(The character group).}\index{Character Group}
For a compact abelian group $G$, a morphism of compact groups $\varphi:G\rightarrow\mathbb{T}$ is called a \emph{character} of $G$. The set $\widehat{G}:=\Hom(G,\mathbb{T})$ of all characters is an abelian group under pointwise multiplication, and called the \emph{character group} of $G$.
\end{definition}

\begin{proposition}\label{char sep the points}{\bf(The characters separate the points).}
The characters of a compact abelian group separate the points.
\end{proposition}

\begin{proof}
\,\,\,This is a consequence of Theorem \ref{gelfand-raikov} and Lemma \ref{G abelian} (b).
\end{proof}

The following remark clarifies the relationship between the set of all equivalence classes of finite-dimensional simple modules of a compact abelian group $G$ and the character group of $G$:

\begin{remark}\label{rem G hut}
We have previously used the notation $\widehat{G}$ for the character group of a compact abelian group and once more for the set of all equivalence classes of finite-dimensional irreducible representations of an arbitrary compact group. Nevertheless, if $G$ is compact and abelian, then these sets coincide: The element $\varphi$ of $\widehat{G}$ may be considered as morphism from $G$ to $\mathbb{C}^{\times}$ and therefore induces an irreducible representation $\pi_{\varphi}:G\rightarrow\GL_1(\mathbb{C})$ given by $\pi_{\varphi}(g).v=\varphi(g)\cdot v$. The function $\varphi\mapsto [(\pi_{\varphi},\mathbb{C})]$ from the character group of $G$ to the set of equivalence classes of finite-dimensional irreducible representations of $G$ is well-defined. Conversely, every finite-dimensional irreducible representations $(\pi,V)$ over $\mathbb{C}$ is one-dimensional by Lemma \ref{G abelian} (b) and defines a unique character $\varphi:G\rightarrow\mathbb{C}^{\times}$ such that the action is given by $\pi(g).x=\varphi(g)\cdot x$. A short computation shows that isomorphic representations give the same character and the function $[(\pi,V)]\mapsto\varphi$ is well-defined and inverts the function introduced before. Thus, we have a natural bijection between the character group of $G$ and the set of all equivalence classes of finite-dimensional irreducible representations.
\end{remark}

In view of the previous remark, we can now state and prove the desired structure theorem for modules of compact abelian groups:

\begin{theorem}{\bf(The big Peter and Weyl theorem for compact abelian groups).}\label{big peter}\index{Theorem! of Peter and Weyl}
Let $G$ be a compact abelian group and $(\pi,V)$ be a $G$-complete continuous representation of $G$. Then with the same notations as in Theorem \ref{str thm of G-mod}, the following assertions hold:
\begin{itemize}
\item[\emph{(a)}]
$V_{\emph{\text{fin}}}=\bigoplus_{\varphi\in\widehat{G}}V_{\varphi}$ and $V_{\emph{\text{fin}}}$ is a dense subspace of $V$.
\item[\emph{(b)}]
$V_{\varphi}=\{v\in V:\,(\forall g\in G)\,\,\pi(g).v=\varphi(g)\cdot v\}$ for all $\varphi\in\widehat{G}$.
\end{itemize}
\end{theorem}

\begin{proof}
\,\,\,(a) The first assertion follows from Theorem \ref{str thm of G-mod} (b) and Remark \ref{rem G hut}.

(b) For the second assertion we recall that each element in $V_{\varphi}$ has the form 
\[P_{\varphi}(v)=\int_G\overline{\varphi}(g)\cdot (\pi(g).v)\,dg
\]for some $v\in V$. Thus, if $x\in V_{\varphi}$ and $g\in G$, then the left-invariance of the Haar measure implies that
\begin{align}
\pi(g).x&=\pi(g).\int_G\overline{\varphi}(h)\cdot(\pi(h).v)\,dh=\int_G\overline{\varphi}(h)\cdot(\pi(gh).v)\,dh \notag\\
&=\varphi(g)\cdot\int_G\overline{\varphi}(h)\cdot(\pi(h).v)\,dh=\varphi(g)\cdot x,\notag
\end{align}
i.e., that $x\in\{v\in V:\,(\forall g\in G)\,\,\pi(g).v=\varphi(g)\cdot v\}$. To verify the other inclusion we choose $x\in\{v\in V:\,(\forall g\in G)\,\,\pi(g).v=\varphi(g)\cdot v\}$. Then an easy calculation shows that $P_{\varphi}(x)=x$. Since the map $P_{\varphi}$ is a projection onto $V_{\varphi}$, we conclude $x\in V_{\varphi}$.
\end{proof}


\begin{corollary}\label{cor of big peter}
With the same assumptions as in Theorem \ref{big peter}, the following assertions hold:
\begin{itemize}
\item[\emph{(a)}]
If $\bf{1}$ denotes the unit element of $\widehat{G}$, then $V^G=V_{\bf{1}}$.
\item[\emph{(b)}]
If $V$ is an algebra, then $V_{\varphi}$ is a $V^G$-module for all $\varphi\in\widehat{G}$. Moreover, $V_{\emph{\text{fin}}}$ is a \emph{(}dense\emph{)} subalgebra of $V$.
\end{itemize}
\end{corollary}

\begin{proof}
\,\,\,(a) In view of the definition of $V^G$, we have
\[V^G=\{v\in V:\,(\forall g\in G)\,\,\pi(g).v=v\}=\{v\in V:\,(\forall g\in G)\,\,\pi(g).v={\bf 1}(g)\cdot v\}=V_{\bf{1}}.
\]

(b) Let $v\in V^G$ and $v_{\varphi}\in V_{\varphi}$. Then for each $g\in G$ we conclude
\[\pi(g).(v_{\varphi}v)=(\pi(g).v_{\varphi})(\pi(g).v)=\varphi(g)\cdot(v_{\varphi}v),
\]i.e., $v_{\varphi}v\in V_{\varphi}$. Hence, $V_{\varphi}$ is a $V^G$-module. Moreover, for $\varphi,\varphi'$ in $\widehat{G}$ we choose $v_{\varphi}\in V_{\varphi}$ and $v_{\varphi'}\in V_{\varphi'}$. Then for each $g\in G$ we conclude
\[\pi(g).(v_{\varphi}v_{\varphi'})=(\pi(g).v_{\varphi})(\pi(g).v_{\varphi'})=(\varphi(g)\varphi'(g))\cdot v_{\varphi}v_{\varphi'}=(\varphi\cdot\varphi')(g)\cdot(v_{\varphi}v_{\varphi'}),
\]i.e., $v_{\varphi}v_{\varphi'}\in V_{\varphi\cdot\varphi'}$. 
\end{proof}

\subsection{The Freeness Condition for Compact Abelian Groups}

In the remaining part of this section we are concerned with rewriting the freeness condition for a dynamical system $(A,G,\alpha)$ with compact abelian structure group $G$. Moreover, we present natural conditions which ensure the freeness of such a dynamical system. 

\begin{lemma}\label{A_phi=gammaC}
Let $A$ be a unital locally convex algebra, $G$ a compact abelian group and $(A,G,\alpha)$ a dynamical system. Further let $\varphi:G\rightarrow\mathbb{C^{\times}}$ be a character and $\pi_{\varphi}:G\rightarrow\GL_1(\mathbb{C})$ the corresponding representation. Then the map 
\[\Gamma_A\mathbb{C}\rightarrow A_{\varphi^{-1}},\,\,\,a\otimes 1\mapsto a
\]is an isomorphism of locally convex spaces.
\end{lemma}

\begin{proof}
\,\,\,For the proof we just note that $a\otimes 1\in\Gamma_A\mathbb{C}$ implies $\alpha(g)(a)\otimes 1=\varphi^{-1}(g)\cdot a\otimes 1$.
\end{proof}

\begin{corollary}\label{C_fin(P)}
Let $G$ be a compact abelian Lie group and $(P,M,G,q,\sigma)$ be a principal bundle. Then, with the notations of Proposition \ref{sections of an associated vector bundle top}, the following map is an isomorphism of locally convex $G$-modules:
\[C^{\infty}(P)_{\emph{\text{fin}}}\rightarrow\bigoplus_{\varphi\in\widehat{G}}\Gamma(P\times_{\varphi}\mathbb{C}),\,\,\,f_{\varphi}\mapsto \Psi_{\varphi^{-1}}(f_{\varphi}).
\]
\end{corollary}

\begin{proof}
\,\,\,Since $C^{\infty}(P)$ is a Fr\'echet algebra, we can apply Theorem \ref{big peter} to the (smooth) dynamical system ($C^{\infty}(P),G,\alpha)$ (cf. Remark \ref{sources of G-complete spaces 2}). Hence, the assertion follows from Example \ref{example of section} and Lemma \ref{A_phi=gammaC}.
\end{proof}

\begin{proposition}\label{freeness for compact abelian groups I}{\bf(The freeness condition for compact abelian groups).}\index{Freeness!Condition for Compact Abelian Groups}
Let $A$ be a commutative unital locally convex algebra and $G$ a compact abelian group. A dynamical system $(A,G,\alpha)$ is free in the sense of Definition \ref{free dynamical systems} if the map
\[\ev^{\varphi}_{\chi}: A_{\varphi}\rightarrow \mathbb{C},\,\,\,a\mapsto\chi(a)
\]is surjective for all $\varphi\in\widehat{G}$ and all $\chi\in\Gamma_A$.
\end{proposition}

\begin{proof}
\,\,\,The claim is a consequence of Proposition \ref{char sep the points} and Lemma \ref{A_phi=gammaC}. 
\end{proof}

\begin{remark}
We recall from [HoMo06], Proposition 2.42 that each compact abelian Lie group $G$ is isomorphic to $\mathbb{T}^n\times \Lambda$ for some natural number $n$ and a finite abelian group $\Lambda$. In particular, the character group $\widehat{G}$ of a compact abelian Lie group is finitely generated.
\end{remark}

\begin{corollary}\label{freeness for compact abelian groups I,5}
Let $A$ be a commutative unital locally convex algebra, $G$ a compact abelian Lie group and $(A,G,\alpha)$ a dynamical system. Further, let $(\varphi_i)_{i\in I}$ be a finite set of generators of $\widehat{G}$. Then the following two conditions are equivalent:
\begin{itemize}
\item[\emph{(a)}]
The map
\[\ev^{\varphi}_{\chi}: A_{\varphi}\rightarrow \mathbb{C},\,\,\,a\mapsto\chi(a)
\]is surjective for all $\varphi\in\widehat{G}$ and all $\chi\in\Gamma_A$.
\item[\emph{(b)}]
The map
\[\ev^{\varphi_i}_{\chi}: A_{\varphi_i}\rightarrow \mathbb{C},\,\,\,a\mapsto\chi(a)
\]is surjective for all $i\in I$ and all $\chi\in\Gamma_A$.
\end{itemize}
In particular, if one of the statements holds, then the dynamical system $(A,G,\alpha)$ is free.
\end{corollary}

\begin{proof}
\,\,\,(a) $\Rightarrow$ (b): This direction is trivial.

(b) $\Rightarrow$ (a): For the second direction we fix $\chi\in\Gamma_A$. Further, we choose for each $i\in I$ an element $a_{\varphi_i}\in A_{\varphi_i}$ with $\chi(a_{\varphi_i})\neq 0$. Now, if $\varphi\in\widehat{G}$, then there exist $k\in\mathbb{N}$ and integers $n_1,\ldots,n_k\in\mathbb{Z}$ such that 
\[\varphi=\varphi_{i_1}^{n_1}\cdots\varphi_{i_k}^{n_k}\,\,\,\text{for some}\,\,i_1,\ldots,i_k\in I.
\]Hence, the element $a_{\varphi}:=a_{\varphi_{i_1}}^{n_1}\cdots a_{\varphi_{i_k}}^{n_k}\in A_{\varphi}$ satisfies $\chi(a_{\varphi})\neq 0$.

The last assertion is a direct consequence of Proposition \ref{freeness for compact abelian groups I}.
\end{proof}

\begin{proposition}\label{freeness for compact abelian groups II}{\bf(Invertible elements in isotypic components).}\index{Invertible Elements in Isotypic Components}
Let $A$ be a commutative unital locally convex algebra, $G$ a compact abelian group and $(A,G,\alpha)$ a dynamical system. If each isotypic component $A_{\varphi}$ contains an invertible element, then the dynamical system $(A,G,\alpha)$ is free.
\end{proposition}

\begin{proof}
\,\,\,The assertion easily follows from Proposition \ref{freeness for compact abelian groups I}. Indeed, if $a_{\varphi}\in A_{\varphi}$ is invertible, then $\chi(a)\neq 0$ for all $\chi\in\Gamma_A$.
\end{proof}

\begin{remark}
Note that 
it is possible to ask for invertible elements in the isotypic components even if the algebra $A$ is noncommutative. We will use this fact in the following chapter.
\end{remark}

%


\begin{proposition}\label{set of gen}
Let $A$ be a unital locally convex algebra, $G$ a compact abelian group and $(A,G,\alpha)$ a dynamical system. Further, let $(\varphi_i)_{i\in I}$ be a finite set of generators of $\widehat{G}$. Then the following two statements are equivalent:
\begin{itemize}
\item[\emph{(a)}]
$A_{\varphi}$ contains invertible elements for all $\varphi\in\widehat{G}$.
\item[\emph{(b)}]
$A_{\varphi_i}$ contains invertible elements for all $i\in I$.
\end{itemize}
In particular, if $A$ is commutative and one of the statements holds, then the dynamical system $(A,G,\alpha)$ is free.
\end{proposition}

\begin{proof}
\,\,\,(a) $\Rightarrow$ (b): This direction is trivial.

(b) $\Rightarrow$ (a): For each $i\in I$ we choose an invertible element $a_{\varphi_i}\in A_{\varphi_i}$. Next, if $\varphi\in\widehat{G}$, then there exist $k\in\mathbb{N}$ and integers $n_1,\ldots,n_k\in\mathbb{Z}$ such that 
\[\varphi=\varphi_{i_1}^{n_1}\cdots\varphi_{i_k}^{n_k}\,\,\,\text{for some}\,\,i_1,\ldots,i_k\in I.
\]Hence, $a_{\varphi}:=a_{\varphi_{i_1}}^{n_1}\cdots a_{\varphi_{i_k}}^{n_k}$ is an invertible element in $A_{\varphi}$.

The last assertion is a direct consequence of Proposition \ref{freeness for compact abelian groups II}.
\end{proof}



The following proposition shows that if all isotypic components of a dynamical system contain invertible elements, then they are ``mutually" isomorphic to each other as modules of the fixed point algebra:

\begin{proposition}\label{iso of A_1-modules}
Let $A$ be a commutative unital locally convex algebra and $G$ a compact abelian group. Further, let $(A,G,\alpha)$ be a dynamical system and $A^G$ the corresponding fixed point algebra. If each isotypic component $A_{\varphi}$ contains an invertible element, then the map
\[\Psi_{\varphi}:A^G\rightarrow A_{\varphi} ,\,\,\,a\mapsto a_{\varphi}a,
\]where $a_{\varphi}$ denotes some fixed invertible element in $A_{\varphi}$, is an isomorphism of locally convex $B$-modules for each $\varphi\in\widehat{G}$. In particular, each isotypic component $A_{\varphi}$ is a free $A^G$-module.
\end{proposition}

\begin{proof}
\,\,\,An easy calculation shows that $\Psi_{\varphi}$ is a morphism of locally convex $B$-modules, and therefore the assertion follows from the fact that $a_{\varphi}\in A_{\varphi}$ is invertible. 
\end{proof}

%

We finally apply the results of this section to dynamical systems arising from classical geometry. The following theorem may be viewed as a first answer to Remark \ref{remark of free action in geometry}. A more detailed analysis will be given in the following chapters.

\begin{theorem}\label{free action for triples}
Let $P$ be a manifold and $G$ be a compact abelian Lie group. Further, let $(C^{\infty}(P),G,\alpha)$ be a smooth dynamical system. If $\pr:P\rightarrow P/G$ denotes the orbit map corresponding to the action of $G$ on $P$ \emph{(}cf. Proposition \ref{smoothness of the group action on the set of characters}\emph{)}, then the following assertions hold:
\begin{itemize}
\item[\emph{(a)}]
If each isotypic component $C^{\infty}(P)_{\varphi}$ contains an invertible element, then we obtain a principal bundle $(P,P/G,G,\pr,\sigma)$.
\item[\emph{(b)}]
If $(\varphi_i)_{i\in I}$ is a finite set of generators of $\widehat{G}$ and each subspace $C^{\infty}(P)_{\varphi_i}$ contains an invertible element, then we obtain a principal bundle $(P,P/G,G,\pr,\sigma)$.
\end{itemize}
\end{theorem}

\begin{proof}
\,\,\,(a) Since $G$ is compact, the induced action $\sigma$ is automatically proper. Therefore, the first assertion follows from Theorem \ref{free action theorem} and Proposition \ref{freeness for compact abelian groups II}.

(b) The second assertion follows from Proposition \ref{set of gen} (b) and part (a).
\end{proof}

\section{Some Topological Aspects of Free Dynamical Systems}

In this section we discuss some topological aspects of (free) dynamical systems. Our main goal is to provide conditions which ensure that a dynamical system induces a topological principal bundle. Again, all groups are assumed to act continuously by morphisms of algebras. For the following lemma we recall Remark \ref{equicont}.

\begin{lemma}\label{cont. action I}{\bf(Continuity of the evaluation map).}\index{Continuity!of the Evaluation Map}
Let $A$ be a commutative unital locally convex algebra. If $\Gamma_A$ is locally equicontinuous, then the evaluation map
\[\ev_A:\Gamma_A\times A\rightarrow\mathbb{C},\,\,\,(\chi,a)\mapsto\chi(a)
\]is continuous.
\end{lemma}

\begin{proof}
\,\,\,To prove the continuity of $\ev_A$, we pick $(\chi_0,a_0)\in \Gamma_A\times A$, $\epsilon>0$ and choose an equicontinuous neighbourhood $V$ of $\chi_0$ in $\Gamma_A$ such that
\[V\subseteq\big\{\chi\in\Gamma_A:\,\vert(\chi-\chi_0)(a_0)\vert<\frac{\epsilon}{2}\big\}.
\]Further, we choose a neighbourhood $W$ of $a_0$ in $A$ such that $\vert\chi(a-a_0)\vert<\frac{\epsilon}{2}$ for all $a\in W$ and $\chi\in V$. We thus obtain
\[\vert \chi(a)-\chi_0(a_0)\vert\leq\vert \chi(a)-\chi(a_0)\vert+\vert \chi(a_0)-\chi_0(a_0)\vert\leq\frac{\epsilon}{2}+\frac{\epsilon}{2}=\epsilon
\]for all $\chi\in V$ and $a\in W$.
\end{proof}

\begin{proposition}\label{cont. action II}{\bf(Continuity of the induced action map).}\index{Continuity!of the Induced Action Map}
Let $A$ be a commutative unital locally convex algebra, $G$ a topological group and $(A,G,\alpha)$ a dynamical system. If the evaluation map $\ev_A$ is continuous, then the induced action
\[\sigma:\Gamma_A\times G\rightarrow\Gamma_A,\,\,\,(\chi,g)\mapsto\chi\circ\alpha(g)
\]of $G$ on $\Gamma_A$ is continuous.
\end{proposition}

\begin{proof}
\,\,\,The topology (of pointwise convergence) on $\Gamma_A$ implies that the map $\sigma$ is continuous if and only if the maps
\[\sigma_a:\Gamma_A\times G\rightarrow\mathbb{C},\,\,\,(\chi,g)\mapsto\chi(\alpha(g)(a))
\]are continuous for all $a\in A$. Therefore, we fix $a\in A$ and note that $\sigma_a=\ev_A\circ(\id_{\Gamma_A}\times\alpha_a)$,
where 
\[\alpha_a:G\rightarrow A,\,\,\,g\mapsto\alpha(g,a)
\]denotes the continuous orbit map of $a$. In view of the assumption, the map $\sigma_a$ is continuous as a composition of continuous maps. Since $a$ was arbitrary, this proves the proposition.
\end{proof}

\begin{remark}\label{orbit map of group action}
Recall that if $\sigma:X\times G\rightarrow X$ is an action of a topological group $G$ on a topological space $X$, then the orbit map $\pr:X\rightarrow X/G$, $x\mapsto x.G:=\sigma(x,G)$ is surjective, continuous and open. 
\end{remark}

\begin{proposition}\label{cont. action III}
Let $A$ be a commutative unital locally convex algebra, $G$ a compact group and $(A,G,\alpha)$ a dynamical system. If the induced action $\sigma:\Gamma_A\times G\rightarrow\Gamma_A$ is free and continuous, then the following assertions hold:
\begin{itemize}
\item[\emph{(a)}]
For each $\chi\in\Gamma_A$ the map $\sigma_{\chi}:G\rightarrow\Gamma_A,\,\,\,g\mapsto\chi.g:=\sigma(\chi,g)$ is a homeomorphism of $G$ onto the orbit $\mathcal{O}_{\chi}$.
\item[\emph{(b)}]
If $\Gamma_A$ is locally compact, then the orbit space $\Gamma_A/G$ is locally compact and Hausdorff.
\item[\emph{(c)}]
For each pair $(\chi,\chi')\in\Gamma_A\times\Gamma_A$ with $\mathcal{O}_{\chi}=\mathcal{O}_{\chi'}$ there is a unique $\tau(\chi,\chi')\in G$ such that $\chi.\tau(\chi,\chi')=\chi'$, and the map
\[\tau:\Gamma_A\times_{\Gamma_A/G}\Gamma_A:=\{(\chi,\chi')\in\Gamma_A\times\Gamma_A:\,\pr(\chi)=\pr(\chi')\}\rightarrow G
\]is continuous and surjective.
\end{itemize}
\end{proposition}
\begin{proof}
\,\,\,(a) The map $\sigma_{\chi}$ is continuous because $\sigma$. Further, the bijectivitiy of $\sigma_{\chi}$ follows from the freeness of $\sigma$. Since $G$ is compact and $\Gamma_A$ is Hausdorff, a well-known Theorem from topology now implies that $\sigma_{\chi}$ is a homeomorphism of $G$ onto the orbit $\mathcal{O}_{\chi}$.

(b) If $\Gamma_A$ is locally compact, then the orbit space $\Gamma_A/G$ is locally compact because the orbit map is open and continuous. Moreover, the compactness of $G$ implies that the action $\sigma$ is proper. Therefore, the image of the map
\[\Gamma_A\times G\rightarrow\Gamma_A\times\Gamma_A,\,\,\,(\chi,g)\mapsto(\chi,\chi.g)
\]is a closed subset of $\Gamma_A\times\Gamma_A$. Now, the assertion follows from Remark \ref{orbit map of group action} and the more general fact that the target space of a surjective, continuous, open map $f:X\rightarrow Y$ is Hausdorff if and only if the preimage of the diagonal under $f\times f$ is closed. 

(c) Suppose $\chi_i\to\chi$, $\chi'_i\to\chi'$, and $\mathcal{O}_{\chi}=\mathcal{O}_{\chi'}$ so that by definition, $\chi_i.\tau(\chi_i,\chi'_i)=\chi'_i$. Since $G$ is compact, we can assume by passing to a subnet that $\tau(\chi_i,\chi'_i)$ converges to $g$, say. Then we have 
\[\chi'=\lim_i\chi'_i=\lim_i(\chi_i\cdot\tau(\chi_i,\chi'_i))=\chi\cdot g,
\] which implies $\tau(\chi,\chi')=g$ and $\tau(\chi_i,\chi'_i)\to\tau(\chi,\chi)$.
\end{proof}

\begin{remark}{\bf(Topological principal bundles).}\index{Bundles!Topological Principal}
The map $\tau$ in Proposition \ref{cont. action III} (c) is called the \emph{translation map} and is part of the definition of principal bundles in [Hu75]. Note that, if a topological group $G$ acts freely, continuously and satisfies (c), then $G$ automatically acts properly; thus the principal bundles in [Hu75] are by definition the free and proper $G$-spaces. Note further that these principal bundles are, in general not, locally trivial. We call a free and proper $G$-space which is Hausdorff a \emph{topological principal bundle}, if each orbit of the action is homeomorphic to $G$ and the orbit space is Hausdorff. For more informations on topological (locally trivial) principal bundles we refer to [RaWi98], Chapter 4, Section 2.
\end{remark}

\begin{theorem}\label{shit theorem}
Let $A$ be a commutative CIA, $G$ a compact group and $(A,G,\alpha)$ a free dynamical system. Then the induced action $\sigma:\Gamma_A\times G\rightarrow\Gamma_A$ is continuous and we obtain a topological principal bundle
\[(\Gamma_A,\Gamma_A/G,G,\sigma,\pr).
\]
\end{theorem}

\begin{proof}
\,\,\,We first recall from Remark \ref{equicont} that $\Gamma_A$ is equicontinuous. Hence, Lemma \ref{cont. action I} and Proposition \ref{cont. action II} imply that the map $\sigma$ is continuous. Further, we note that the map $\sigma$ is proper. Indeed, this follows from the compactness of $G$. In view of Theorem \ref{freeness of induced action}, we conclude that $\sigma$ is free. Therefore, $\Gamma_A$ is a free and proper $G$-space which is Hausdorff and thus the claim follows from Proposition \ref{cont. action III} (a) and (b). 
\end{proof}

\begin{corollary}
Let $A$ be a commutative CIA and $G$ a compact abelian group. Furthermore, let $(A,G,\alpha)$ be a dynamical system. If each isotypic component $A_{\varphi}$ contains an invertible element, then we obtain a topological principal bundle $(\Gamma_A,\Gamma_A/G,G,\sigma,\pr)$.
\end{corollary}

\begin{proof}
\,\,\,This assertion immediately follows from Proposition \ref{freeness for compact abelian groups II} and Theorem \ref{shit theorem}.
\end{proof}

\section{An Open Problem and Applications to Representation Theory}

This section is dedicated to the following interesting open problem:

\begin{open problem}{\bf (Primitive ideals).}\index{Primitive Ideals}
Theorem \ref{freeness of induced action} may be viewed as a first step towards a geometric approach to noncommutative principal bundles. Nevertheless, in order to get a broader picture, it might be helpful to get rid of the characters. This might be done with the help of primitive ideals, i.e., kernels of irreducible representations $(\rho,W)$ of the (locally convex) algebra $A$, since they can be considered as generalizations of characters (points). To be more precise:

Let $(A,G,\alpha)$ be a dynamical system, consisting of a (not necessarily commutative) unital locally convex algebra $A$, a topological group $G$ and a group homomorphism $\alpha:G\rightarrow\Aut(A)$, which induces a continuous action of $G$ on $A$. Further, let $\Prim(A)$\sindex[n]{$\Prim(A)$} denote the set of primitive ideals of $A$. As already mentioned, note that if $A$ is commutative, then $\Prim(A)\cong\Gamma_A$. Do there exist ``geometrically oriented" conditions which ensure that the corresponding action
\[\sigma:\Prim(A)\times G\rightarrow\Prim(A),\,\,\,(I,g)\mapsto\alpha(g).I
\]of $G$ on the primitive ideals $\Prim(A)$ of $A$ is free?
\end{open problem}

An interesting application to representation theory is given by the ``generalized Effros--Hahn Conjecture":

\begin{theorem}{\bf(The generalized Effros--Hahn conjecture).}\index{Theorem!of Effros--Hahn}
Suppose $G$ is an amenable group, $A$ a separable $C^*$-algebra and $(A,G,\alpha)$ a $C^*$-dynamical system. If $G$ acts freely on $\Prim(A)$,  then there is one and only one primitive ideal of the crossed product $A\rtimes_{\alpha}G$ lying over each hull-kernel quasi-orbit in $\Prim(A)$. In particular, if every orbit is also hull-kernel dense, then $A\rtimes_{\alpha}G$ is simple.
\end{theorem}
 
\begin{proof}
\,\,\,A nice proof of this theorem can be found in [GoRo79], Corollary 3.3.
\end{proof}

\section{Strongly Free Dynamical Systems}

In this final section we introduce a stronger version of freeness for dynamical systems than the one given in Section \ref{section:free dynamical systems} (cf. Definition \ref{free dynamical systems}). In fact, instead of considering arbitrary families $(\pi_j,V_j)_{j\in J}$ of (continuous) point separating representations of a topological group $G$, we restrict our attention to families $(\pi_j,\mathcal{H}_j)_{j\in J}$ of \emph{
unitary irreducible} point separating representations. At this point, we recall that each locally compact group $G$ admits a family of continuous unitary irreducible point separating representations (cf. Theorem \ref{gelfand-raikov}). We show that Theorem \ref{freeness of induced action} and Theorem \ref{characterization of free group actions} stay true in this context of \emph{strongly free} dynamical systems and that Proposition \ref{freeness for compact abelian groups I} actually turns into a definition for strongly free dynamical systems with compact abelian structure group. We close this short section with a nice example.

\begin{definition}\label{free dynamical systems,strong}{\bf (Strongly free dynamical systems).}\index{Dynamical Systems!Strongly Free}
Let $A$ be a commutative unital locally convex algebra and $G$ a topological group. A dynamical system $(A,G,\alpha)$ is called \emph{strongly free} if there exists a family $(\pi_j,\mathcal{H}_j)_{j\in J}$ of unitary irreducible point separating representations of $G$ such that the map
\[\ev^j_{\chi}:=\ev^{\mathcal{H}_j}_{\chi}:\Gamma_A \mathcal{H}_j\rightarrow \mathcal{H}_j,\,\,\,a\otimes v\mapsto\chi(a)\cdot v
\]is surjective for all $j\in J$ and all $\chi\in\Gamma_A$. 
\end{definition}

\begin{proposition}\label{freeness of induced action,strong}{\bf(Freeness of the induced action again).}\index{Freeness!of the Induced Action}
If $(A,G,\alpha)$ is a strongly free dynamical system, then the induced action
\[\sigma:\Gamma_A\times G\rightarrow\Gamma_A,\,\,\,(\chi,g)\mapsto\chi\circ\alpha(g)
\]of $G$ on the spectrum $\Gamma_A$ of $A$ is free. 
\end{proposition}

\begin{proof}
\,\,\,This assertion immediately follows from Theorem \ref{freeness of induced action}, since each strongly free dynamical system is free.
\end{proof}

\begin{proposition}\label{characterization of free group actions,strong}{\bf(Characterization of free group actions again).}\index{Characterization of Free Group Actions}
Let $P$ be a manifold, $G$ a compact Lie group and $(C^{\infty}(P),G,\alpha)$ a smooth dynamical system. Then the following statements are equivalent:
\begin{itemize}
\item[\emph{(a)}]
The smooth dynamical system $(C^{\infty}(P),G,\alpha)$ is strongly free.
\item[\emph{(b)}]
The induced smooth group action $\sigma:P\times G\rightarrow P$ is free.
\end{itemize}
In particular, in this situation the concepts of freeness coincide.
\end{proposition}

\begin{proof}
\,\,\,This assertion can be proved similarly to Corollary \ref{characterization of free group actions} (cf. Theorem \ref{free action theorem}): In fact, given a finite-dimensional representation $(\pi,V)$ of $G$, we just have to note that it is possible to find an inner product on $V$ such that $G$ acts by unitary transformations (``Weyl's trick") and that each unitary finite-dimensional representation of $G$ can be decomposed into the (finite) sum of irreducible representations. 
\end{proof}


\begin{proposition}\label{freeness for compact abelian groups,strong}{\bf(The strong freeness condition for compact abelian groups).}\index{Strong Freeness!Condition for Compact Abelian Groups}
Let $A$ be a commutative unital locally convex algebra and $G$ a compact abelian group. A dynamical system $(A,G,\alpha)$ is strongly free in the sense of Definition \ref{free dynamical systems,strong} if and only if the map
\[\ev^{\varphi}_{\chi}: A_{\varphi}\rightarrow \mathbb{C},\,\,\,a\mapsto\chi(a)
\]is surjective for all $\varphi\in\widehat{G}$ and all $\chi\in\Gamma_A$.
\end{proposition}

\begin{proof}
\,\,\,($``\Leftarrow"$) This direction is obvious, since the characters of the group $G$ induce a family of unitary irreducible representations that separate the points of $G$ (cf. Lemma \ref{G abelian}, Proposition \ref{char sep the points} and Lemma \ref{A_phi=gammaC}).

($``\Rightarrow"$) For the other direction we first note that each unitary irreducible representation $(\pi,\mathcal{H})$ of $G$ is one-dimensional by Schur's Lemma (cf. [Ne09], Theorem 4.2.7), i.e., $\pi(g).v=\varphi(g)\cdot v$ for all $g\in G$, $v\in\mathcal{H}$ and some character $\varphi$ of $\widehat{G}$ (cf. Lemma \ref{G abelian} or Remark \ref{rem G hut}). Thus, if the dynamical system $(A,G,\alpha)$ is strongly free and $(\pi_j,\mathcal{H}_j)_{j\in J}$ is a family of unitary irreducible point separating representations of $G$ satisfying the conditions of Definition \ref{free dynamical systems,strong}, then [HoMo06], Corollary 2.3.3.(i) implies that the corresponding characters $\varphi_j$ generate $\widehat{G}$ and from this we easily conclude that the map
\[\ev^{\varphi}_{\chi}: A_{\varphi}\rightarrow \mathbb{C},\,\,\,a\mapsto\chi(a)
\]is surjective for all $\varphi\in\widehat{G}$ and all $\chi\in\Gamma_A$ (cf. Lemma \ref{A_phi=gammaC}).
\end{proof}

\begin{example}
We now want to use Proposition \ref{freeness for compact abelian groups,strong} to show that the action of the group $C_2:=\{-1,+1\}$ on $\mathbb{R}$ defined by
\[\sigma:\mathbb{R}\times C_2\rightarrow\mathbb{R},\,\,\,r.(-1):=\sigma(r,-1):=-r
\]is not free: Indeed, we first note that the map
\[\Psi:C_2\rightarrow\Hom_{\text{gr}}(C_2,\mathbb{T}),\,\,\,\Psi(-1)(-1):=-1
\]is an isomorphism of abelian groups. From this we easily conclude that the isotypic component of the associated smooth dynamical system $(C^{\infty}(\mathbb{R}),C_2,\alpha)$ (cf. Proposition \ref{smoothness of the group action on the algebra of smooth functions})  corresponding to the generator $-1\in C_2$ is given by
\[C^{\infty}(\mathbb{R})_{-1}=\{f:\mathbb{R}\rightarrow\mathbb{C}:\,(\forall r\in\mathbb{R})\,f(-r)=-f(r)\}.
\]
Since $f(0)=0$ for each $f\in C^{\infty}(\mathbb{R})_{-1}$, the map
\[\ev^{-1}_{0}: C^{\infty}(\mathbb{R})_{-1}\rightarrow \mathbb{C},\,\,\,f\mapsto f(0)=0
\]is not surjective showing that $(C^{\infty}(\mathbb{R}),C_2,\alpha)$ is not strongly free (cf. Proposition \ref{freeness for compact abelian groups,strong}). Therefore, Proposition \ref{characterization of free group actions,strong} implies that the action $\sigma$ is not free.
\end{example}

\chapter{Trivial NCP Torus Bundles}\label{trivial ncp torus bundles}

The main goal of this chapter is to present a new, geometrically oriented approach to the noncommutative geometry of trivial principal $\mathbb{T}^n$-bundles based on dynamical systems of the form $(A,\mathbb{T}^n,\alpha)$. We call a dynamical system $(A,\mathbb{T}^n,\alpha)$ a trivial NCP $\mathbb{T}^n$-bundle if each isotypic component contains an invertible element. It turns out that each trivial noncommutative principal $\mathbb{T}^n$-bundle possesses an underlying algebraic structure of a $\mathbb{Z}^n$-graded unital associative algebra, which might be considered as an algebraic counterpart of a trivial NCP $\mathbb{T}^n$-bundle. We then provide a complete classification of this underlying algebraic structure, i.e., we classify all trivial NCP $\mathbb{T}^n$-bundles up to completion.\sindex[n]{$(A,\mathbb{T}^n,\alpha)$} 

\section{The Concept of Trivial NCP Torus Bundles}

In this section we present a geometrically oriented approach to the noncommutative geometry of trivial principal $\mathbb{T}^n$-bundles based on dynamical systems of the form $(A,\mathbb{T}^n,\alpha)$. We will in particular see that this approach reproduces the classical geometry of trivial principal $\mathbb{T}^n$-bundles: If $A=C^{\infty}(P)$ for some manifold $P$, then we recover a trivial principal $\mathbb{T}^n$-bundle and, conversely, each trivial principal bundle $(P,M,\mathbb{T}^n,q,\sigma)$ gives rise to a trivial NCP $\mathbb{T}^n$-bundle of the form $(C^{\infty}(P),\mathbb{T}^n,\alpha)$.\sindex[n]{$(P,M,\mathbb{T}^n,q,\sigma)$}\sindex[n]{$(C^{\infty}(P),\mathbb{T}^n,\alpha)$} 

\begin{notation} For fixed $n\in\mathbb{N}$ we define $I_n:=\{1,\ldots,n\}$\sindex[n]{$I_n$}. We write ${\bf k}=(k_1,\ldots,k_n)$ for elements of $\mathbb{Z}^n$ and think of them as multi-indices. In particular, we write $e_i=(0,\ldots,1,\ldots,0)$, $i\in I_n$, for the canonical basis of $\mathbb{Z}^n$ and ${\bf 0}=(0,\ldots,0)$ for its unit element. If $A$ is an algebra, $a=(a_1,\ldots,a_n)\in A^n$ and ${\bf k}\in\mathbb{Z}^n$, then we write $$a^{\bf k}:=a_1^{k_1}\cdots a_n^{k_n}.$$
\end{notation}


\begin{lemma}\label{widehat T^n}{\bf(The character group of $\mathbb{T}^n$).}\index{Character Group!of $\mathbb{T}^n$}
If $G=\mathbb{T}^n$, then the map
\[\Psi:\mathbb{Z}^n\rightarrow\widehat{G},\,\,\,\Psi({\bf k})(z)=z^{\bf k}
\]is an isomorphism of abelian groups.
\end{lemma}
\begin{proof}
\,\,\,An easy calculation shows that $\Psi$ is a homomorphism of abelian groups. Moreover, $\Psi({\bf k})(z)=1$ holds for all $z\in\mathbb{T}^n$ if and only if ${\bf k}={\bf 0}$. Hence, $\Psi$ is injective. The surjectivity of $\Psi$ follows from usual covering theory and the fact that each group homomorphism from $\mathbb{R}^n$ to $\mathbb{R}$ is of the form
\[x\mapsto c_1x_1+\cdot+c_nx_n
\]for some $c\in\mathbb{R}^n$.
\end{proof}


\begin{remark}\label{set of gen for widehat T^n}
Note that the previous lemma implies that the elements $\chi_{i}:=\Psi(e_i)$, $i\in I_n$, define a set of generators of the dual group of $\mathbb{T}^n$.
\end{remark}

\begin{notation}
In the following we will identify the character group of $\mathbb{T}^n$ with $\mathbb{Z}^n$ via the isomorphism $\Psi$ of Lemma \ref{widehat T^n}. In particular, if $A$ is a unital locally convex algebra and $(A,\mathbb{T}^n,\alpha)$ a dynamical system, then we write
\[A_{\bf k}:=A_{\Psi({\bf k})}=\{a\in A:\,(\forall z\in\mathbb{T}^n)\,\,\alpha(z).a=z^{\bf k}\cdot a\}
\]for the isotypic component corresponding to ${\bf k}\in\mathbb{Z}^n$. For $i\in I_n$ we write $A_i:=A_{\Psi(e_i)}$ 
\end{notation}

\begin{proposition}\label{freeness for T^n I}{\bf(The freeness condition for $\mathbb{T}^n$).}\index{Freeness!Condition for $\mathbb{T}^n$}
Let $A$ be a commutative unital locally convex algebra. A dynamical system $(A,\mathbb{T}^n,\alpha)$ is free in the sense of Definition \ref{free dynamical systems} if the map
\[\ev^{\bf k}_{\chi}: A_{\bf k}\rightarrow \mathbb{C},\,\,\,a\mapsto\chi(a)
\]is surjective for all ${\bf k}\in\mathbb{Z}^n$ and all $\chi\in\Gamma_A$.
\end{proposition}

\begin{proof}
\,\,\,The assertion is a consequence of Proposition \ref{freeness for compact abelian groups I} and Lemma \ref{widehat T^n}.
\end{proof}

\begin{proposition}\label{freeness for T^n II}{\bf(Invertible elements in isotypic components).}
Let $A$ be a commutative unital locally convex algebra and $(A,\mathbb{T}^n,\alpha)$ a dynamical system. If each isotypic component $A_{\bf k}$ contains an invertible element, then the dynamical system $(A,\mathbb{T}^n,\alpha)$ is free.
\end{proposition}

\begin{proof}
\,\,\,The assertion easily follows from Proposition \ref{freeness for T^n I}. Indeed, if $a_{\bf k}\in A_{\bf k}$ is invertible, then $\chi(a)\neq 0$ for all $\chi\in\Gamma_A$.
\end{proof}

\begin{remark}
Note that it is possible to ask for invertible elements in the isotypic components even if the algebra A is noncommutative.
\end{remark}

\begin{corollary}\label{set of gen again}
Let $A$ be a unital locally convex algebra and $(A,\mathbb{T}^n,\alpha)$ a dynamical system. Then the following two statements are equivalent:
\begin{itemize}
\item[\emph{(a)}]
$A_{\bf k}$ contains invertible elements for all ${\bf k}\in\mathbb{Z}^n$.
\item[\emph{(b)}]
$A_i$ contains invertible elements for all $i\in I_n$.
\end{itemize}
In particular, if $A$ is commutative and one of the statements hold, then the dynamical system $(A,\mathbb{T}^n,\alpha)$ is free.
\end{corollary}

\begin{proof}
\,\,\,(a) $\Rightarrow$ (b): This direction is trivial.

(b) $\Rightarrow$ (a): For each $i\in I_n$ we choose an invertible element $a_i\in A_i$ and put $a:=(a_1,\ldots,a_n)$. Now, if ${\bf k}\in\mathbb{Z}^n$, then $a^{\bf k}\in A_k$ is invertible.

The last assertion is a direct consequence of Proposition \ref{freeness for T^n II}.
\end{proof}

We will now introduce a (reasonable) definition of trivial NCP $\mathbb{T}^n$-bundles:



\begin{definition}{\bf(Trivial NCP $\mathbb{T}^n$-bundles).}\label{trivial noncomm T^n-bundles}\index{Trivial NCP!$\mathbb{T}^n$-Bundles} Let $A$ be a unital locally convex algebra. A (smooth) dynamical system $(A,\mathbb{T}^n,\alpha)$ is called a (\emph{smooth}) \emph{trivial NCP $\mathbb{T}^n$-bundle}, if each isotypic component $A_{\bf k}$ contains an invertible element. 
\end{definition}

\begin{proposition}\label{structure of trivial noncomm T^n-bundles}{\bf(The algebraic skeleton).}\index{Algebraic!Skeleton}
Let $A$ be a complete unital locally convex algebra and $(A,\mathbb{T}^n,\alpha)$ a trivial NCP $\mathbb{T}^n$-bundle. If $B:=A_{\bf 0}$, then the following assertions hold:
\begin{itemize}
\item[\emph{(a)}]
If $a_{\bf k}\in A_{\bf k}$ is an invertible element for each $k\in\mathbb{Z}^n$, then 
\[A_{\emph{\text{fin}}}=\bigoplus_{{\bf k}\in\mathbb{Z}^n}a_{\bf k}B.
\]
\item[\emph{(b)}]
If $a_i\in A_i$ is invertible for each $i\in I_n$ and $a:=(a_1,\ldots,a_n)\in A^n$, then 
\[A_{\emph{\text{fin}}}=\bigoplus_{{\bf k}\in\mathbb{Z}^n}a^{\bf k}B.
\]
\end{itemize}
\end{proposition}

\begin{proof}
\,\,\,(a) The first assertion follows from Definition \ref{trivial noncomm T^n-bundles}, Theorem \ref{big peter} (a) and Proposition \ref{iso of A_1-modules}.

(b) The second assertion is a consequence of Corollary \ref{set of gen again} and part (a) since $a^{\bf k}\in A_{\bf k}$ for each ${\bf k}\in\mathbb{Z}^n$.
\end{proof}

\begin{proposition}\label{transformation triples induce free actions}
If $P$ is a manifold and $(C^{\infty}(P),\mathbb{T}^n,\alpha)$ a smooth trivial NCP $\mathbb{T}^n$-bundle, then the induced smooth action map $\sigma$ of Proposition \ref{smoothness of the group action on the set of characters} is free and proper. In particular, we obtain a principal bundle $(P,P/\mathbb{T}^n,\mathbb{T}^n,\pr,\sigma)$.
\end{proposition}

\begin{proof}
\,\,\,The freeness of the map $\sigma$ follows directly from Definition \ref{trivial noncomm T^n-bundles}, Proposition \ref{freeness for T^n II} and Theorem \ref{freeness of induced action}. Its properness is automatic, since $\mathbb{T}^n$ is compact. The last claim now follows from the Quotient Theorem.
\end{proof}

\begin{lemma}\label{Aut(A)=Diff(A)}
If $P$ is a manifold and $\phi:C^{\infty}(P)\rightarrow C^{\infty}(P)$ an algebra automorphism, then there is a diffeomorphism $\tau:P\rightarrow P$ such that $\phi(f):=f\circ\tau^{-1}$ for all $f\in C^{\infty}(P)$.
\end{lemma}

\begin{proof}
\,\,\,This assertion is a consequence of [Gue08], Lemma 2.95.
\end{proof}

\begin{remark}\label{star-automorphism}{\bf(*-Automorphisms).}\index{*-Automorphisms}
In view of the previous lemma, each algebra automorphism of $C^{\infty}(P)$ is automatically a *-automorphism.
\end{remark}

\begin{proposition}\label{invariance of I*I}
Let $P$ be a manifold and $(C^{\infty}(P),\mathbb{T}^n,\alpha)$ a dynamical system. If ${\bf k}\in\mathbb{Z}^n$ and $f\in C^{\infty}(P)_{\bf k}$ is invertible, then $\vert f\vert$ is invariant under the action of $\mathbb{T}^n$, i.e.,
\[\alpha(z).\vert f\vert=\vert f\vert\,\,\,\text{for all}\,\,\,z\in\mathbb{T}^n.
\]
\end{proposition}

\begin{proof}
\,\,\,If $z\in\mathbb{T}^n$, then $\alpha(z).f=z^{\bf k}\cdot f$. On the other hand, Lemma \ref{Aut(A)=Diff(A)} implies that there exists a diffeomorphism $\tau_z:P\rightarrow P$ such that $\alpha(z).g=g\circ\tau^{-1}_z$ for all $g\in C^{\infty}(P)$. Hence, we obtain
\[\alpha(z).\vert f\vert=\vert\cdot\vert\circ f\circ\tau^{-1}_z=\vert z^{\bf k}\cdot f\vert=\vert f\vert.
\]
\end{proof}

We now come to the main theorem of this section:

\begin{theorem}\label{TNPB for manifold}{\bf(Trivial principal $\mathbb{T}^n$-bundles).}\index{Trivial Principal!$\mathbb{T}^n$-Bundles}
If $P$ is a manifold, then the following assertions hold:
\begin{itemize}
\item[\emph{(a)}]
If $(C^{\infty}(P),\mathbb{T}^n,\alpha)$ is a smooth trivial NCP $\mathbb{T}^n$-bundle, then the corresponding principal bundle $(P,P/\mathbb{T}^n,\mathbb{T}^n,\pr,\sigma)$ of Proposition \ref{transformation triples induce free actions} is trivial.
\item[\emph{(b)}]
Conversely, if $(P,M,\mathbb{T}^n,q,\sigma)$ is a trivial principal $\mathbb{T}^n$-bundle, then the corresponding smooth dynamical system $(C^{\infty}(P),\mathbb{T}^n,\alpha)$ of Remark \ref{induced transformation triples} is a smooth trivial NCP $\mathbb{T}^n$-bundle.
\end{itemize}
\end{theorem}
\begin{proof}
\,\,\,(a) By Corollary \ref{set of gen again} and Definition \ref{trivial noncomm T^n-bundles}, we may choose for each $i\in I_n$ an invertible function $f_i\in C^{\infty}(P)_i$ with $\im(f_i)\subseteq\mathbb{T}$. Indeed, if $f_i\in C^{\infty}(P)_i$ is invertible, then the function
\[g_i:P\rightarrow\mathbb{C},\,\,\,p\mapsto \frac{f_i(p)}{\vert f_i(p)\vert}
\]is invertible and satisfies $\im(g_i)\subseteq\mathbb{T}$. Moreover, we conclude from Proposition \ref{invariance of I*I} that $\alpha(z).g_i=z_i\cdot g_i$ holds for all $z\in\mathbb{T}^n$, i.e., that $g_i\in C^{\infty}(P)_i$. Therefore, the map
\[\varphi:P\rightarrow P/\mathbb{T}^n\times\mathbb{T}^n,\,\,\,p\mapsto(\pr(p),f_1(p),\ldots,f_n(p))
\]
defines an equivalence of principal $\mathbb{T}^n$-bundles over $P/\mathbb{T}^n$. In particular, the principal bundle $(P,P/\mathbb{T}^n,\mathbb{T}^n,\pr,\sigma)$ of Proposition \ref{transformation triples induce free actions} is trivial.

(b) Conversely, let $(P,M,\mathbb{T}^n,q,\sigma)$ be a trivial principal $\mathbb{T}^n$-bundle and
\[\varphi:P\rightarrow M\times\mathbb{T}^n,\,\,\,p\mapsto(q(p),f_1(p),\ldots,f_n(p))
\]be an equivalence of principal $\mathbb{T}^n$-bundles over $M$. We first note that each function $f_i\in C^{\infty}(P)$ is invertible. Furthermore, the $\mathbb{T}^n$-equivariance of $\varphi$ implies that $f_i\in C^{\infty}(P)_i$ for all $i\in I_n$. Therefore, each isotypic component $C^{\infty}(P)_i$ contains invertible elements, and we conclude from Corollary \ref{set of gen again} that $(C^{\infty}(P),\mathbb{T}^n,\alpha)$ is a smooth trivial NCP $\mathbb{T}^n$-bundle.
\end{proof}

\section{Examples of Trivial NCP Torus Bundles}\label{examples of trivial NCP torus bundles}

We present a bunch of examples of trivial NCP $\mathbb{T}^n$-bundles:

\begin{example}{\bf(Noncommutative n-tori).}\label{NC n-tori as NCPTB}\index{Noncommutative!$n$-Tori}
For a real skewsymmetric $n\times n$ matrix $\theta$ with entries $\theta_{ij}$ let $A^n_{\theta}$ be the noncommutative $n$-torus of Definition \ref{appendix E noncommutative n-tori}. We recall that $A^n_{\theta}$ is the universal unital C*-algebra generated by unitaries $U_1,\ldots,U_n$ with
\[U_rU_s=\exp(2\pi i\theta_{rs})U_sU_r\,\,\,\text{for all}\,\,\,1\leq r,s\leq n.
\]We further recall that there is a continuous action $\alpha$ of $\mathbb{T}^n$ on $A^n_{\theta}$ by algebra automorphisms given on generators by
\[\alpha(z).U^{\bf k}:=z.U^{\bf k}:=z^{\bf k}\cdot U^{\bf k}\,\,\,\text{for}\,\,\,{\bf k}\in\mathbb{Z}^n,
\]where $U^{\bf k}:=U^{k_1}_1\cdots U^{k_n}_n$ (cf. Lemma \ref{t^n on a^n}). Hence, $(A^n_{\theta},\mathbb{T}^n,\alpha)$\sindex[n]{$(A^n_{\theta},\mathbb{T}^n,\alpha)$} is a dynamical system. The definition of the action implies that $(A^n_{\theta})_{\bf k}=\mathbb{C}\cdot U^{\bf k}$ holds for each ${\bf k}\in\mathbb{Z}^n$. In particular, each isotypic component $(A^n_{\theta})_{\bf k}$ contains invertible elements and thus the dynamical system 
\[(A^n_{\theta},\mathbb{T}^n,\alpha)
\]is a trivial NCP $\mathbb{T}^n$-bundle.
\end{example}

\begin{example}{\bf(Smooth noncommutative n-tori).}\label{Smooth NC n-tori as NCPTB}\index{Smooth!Noncommutative $n$-Tori}
The smooth noncommutative $n$-torus $\mathbb{T}^n_{\theta}$ is the dense unital $^*$-subalgebra of smooth vectors for the action $\alpha$ of the previous example (cf. Definition \ref{n-dim QT}). Its elements are given by (norm-convergent) sums
\[a=\sum_{{\bf k}\in\mathbb{Z}^n}a_{\bf k}U^{\bf k},\,\,\,\text{with}\,\,\,(a_{\bf k})_{{\bf k}\in\mathbb{Z}^n}\in S(\mathbb{Z}^n).
\]Further, Corollary \ref{automatic smoothness II} implies that $\mathbb{T}^n_{\theta}$ is a CIA and that the induced action of $\mathbb{T}^n$ on $\mathbb{T}^n_{\theta}$ is smooth. Hence, $(\mathbb{T}^n_{\theta},\mathbb{T}^n,\alpha)$\sindex[n]{$(\mathbb{T}^n_{\theta},\mathbb{T}^n,\alpha)$} is a smooth dynamical system. Since $(\mathbb{T}^n_{\theta})_{\bf k}=\mathbb{C}\cdot U^{\bf k}=\mathbb{C}\cdot U^{\bf k}$ holds for each ${\bf k}\in\mathbb{Z}^n$, we conclude that each isotypic component $(\mathbb{T}^n_{\theta})_{\bf k}$ contains invertible elements and thus the smooth dynamical system 
\[(\mathbb{T}^n_{\theta},\mathbb{T}^n,\alpha)
\]is a smooth trivial NCP $\mathbb{T}^n$-bundle.
\end{example}

\begin{remark}\sindex[n]{$C^{\infty}(\mathbb{T}^n)$}
Suppose we are in the situation of Example \ref{NC n-tori as NCPTB}. If $\theta=0$, then we get the smooth trivial NCP $\mathbb{T}^n$-bundle $(C^{\infty}(\mathbb{T}^n),\mathbb{T}^n,\alpha)$\sindex[n]{$(C^{\infty}(\mathbb{T}^n),\mathbb{T}^n,\alpha)$}. The corresponding trivial principal bundle of Theorem \ref{TNPB for manifold} (a) is the (trivial) principal $\mathbb{T}^n$-bundle over a single point $\{\ast\}$, i.e., 
\[(\mathbb{T}^n,\{\ast\},\mathbb{T}^n,q,\sigma_{\mathbb{T}^n})
\]for $q:\mathbb{T}^n\rightarrow\{\ast\},\,\,\,z\mapsto\ast$. Therefore, one should think of noncommutative n-tori as deformations of the trivial principal $\mathbb{T}^n$-bundle over a single point.
\end{remark}

\begin{example}{\bf(The group $C^{*}$-algebra of the discrete Heisenberg group).}\label{group c-star of heisenberg}
The discrete (three-dimensional) \emph{Heisenberg group} $H$ is abstractly defined as the group generated by elements $a$ and $b$ such that the commutator $c = aba^{-1}b^{-1}$ is central. It can be realized as the multiplicative group of upper-triangular matrices
   \[H:=\left\{\left(\begin{matrix}
1 & a & c\\
0 & 1 & b\\
0 & 0 & 1\end{matrix}\right):\,a,b,c\in\mathbb{Z}\right\}.
\]A short observation shows that $H$ is isomorphic (as a group) to the semidirect product $\mathbb{Z}^2\rtimes_S\mathbb{Z}$, where the semidirect product is defined by the homomorphism
\[S:\mathbb{Z}\rightarrow\Aut(\mathbb{Z}^2),\,\,\,S(k).(m,n):=(m,km+n).
\]The associated group $C^{*}$-algebra $C^{*}(H)$\sindex[n]{$C^{*}(H)$} is the universal $C^{*}$-algebra generated by unitaries $U$, $V$ and $W$ satisfying
\[UW=WU,\,\,\,VW=WV\,\,\,\text{and}\,\,\,UV=WVU.
\]A short observation shows that the map
\[\alpha:\mathbb{T}^2\times C^{*}(H)\rightarrow C^{*}(H),\,\,\,\alpha(z,w).U^kV^lW^m:=z^kw^l\cdot U^kV^lW^m
\]for $k,l,m\in\mathbb{Z}$ defines a continuous action of $\mathbb{T}^2$ on $C^{*}(H)$ by algebra automorphisms. Hence, $(C^{*}(H),\mathbb{T}^2,\alpha)$\sindex[n]{$(C^{*}(H),\mathbb{T}^2,\alpha)$} is a dynamical system. The definition of the action implies that the corresponding fixed point algebra $B$ is the centre of $C^{*}(H)$ which is equal to the group $C^{*}$-algebra $C^{*}(Z)$ of the center
   \[Z:=\left\{\left(\begin{matrix}
1 & 0 & c\\
0 & 1 & 0\\
0 & 0 & 1\end{matrix}\right):\,c\in\mathbb{Z}\right\}\cong\mathbb{Z}
\]of $H$. Moreover, the isotypic component $C^{*}(H)_{(k,l)}$ corresponding to $(k,l)\in\mathbb{Z}^2$ is given by $C^{*}(H)_{(k,l)}=B\cdot U^kV^l$. In particular, each isotypic component $C^{*}(H)_{(k,l)}$ contains invertible elements and thus the dynamical system 
\[(C^{*}(H),\mathbb{T}^2,\alpha)
\]is a trivial NCP $\mathbb{T}^2$-bundle. The group $C^{*}$-algebra of the discrete Heisenberg group plays an important role in the nice paper [ENOO09].
\end{example}

\begin{remark}
Suppose we are in the situation of Example \ref{group c-star of heisenberg}. Since $C^{*}(H)^{\mathbb{T}^2}\cong C(\mathbb{T})$, there is a canonical $C^{*}$-algebra bundle structure on $C^{*}(H)$ with base $\mathbb{T}$. The fibre of this bundle at $z=\exp(2\pi i\theta)$ is isomorphic to the noncommutative $2$-torus $A^2_{\theta}$. Moreover, for each $\theta\in\mathbb{R}$ we have a natural  homomorphism $C^{*}(H)\rightarrow A^2_{\theta}$ with $W\mapsto\exp(2\pi i\theta)\cdot{\bf 1}$. For further details we refer to the introduction of [ENOO09].
\end{remark}


\begin{construction}\label{l^1 spaces associated to dynamical systems I}{\bf($\ell^1$-crossed products).}\index{Crossed Products!$\ell^1$-}
Let $(A,\Vert\cdot\Vert,^{*})$ be an involutive Banach algebra and $(A,\mathbb{Z}^n,\alpha)$ a dynamical system. Note that this means that $\mathbb{Z}^n$ acts by isometries of $A$. We write $F(\mathbb{Z}^n,A)$\sindex[n]{$F(\mathbb{Z}^n,A)$} for the vector space of functions $f:\mathbb{Z}^n\rightarrow A$ with finite support and define a multiplication on this space by
\[(f\star g)({\bf k}):=\sum_{{\bf l}\in\mathbb{Z}^n}f({\bf l})\alpha({\bf l},g({\bf k-l})).
\]Moreover, an involution is given by
\[f^{*}({\bf k}):=\alpha({\bf k},(f(-{\bf k}))^{*}).
\]These two operations are continuous for the $L^1$-norm 
\[\Vert f\Vert_1:=\sum_{{\bf k}\in\mathbb{Z}^n}\Vert f({\bf k})\Vert,
\]and the completion of $F(\mathbb{Z}^n,A)$ in this norm is again an involutive Banach algebra denoted by $\ell^1(A\rtimes_{\alpha}\mathbb{Z}^n)$\sindex[n]{$\ell^1(A\rtimes_{\alpha}\mathbb{Z}^n)$}.
\begin{flushright}
$\blacksquare$
\end{flushright}
\end{construction}

\begin{remark}
If $A=\mathbb{C}$, then $\ell^1(A\rtimes_{\alpha}\mathbb{Z}^n)$ is just the algebra $\ell^1(\mathbb{Z}^n)$.
\end{remark}

\begin{lemma}\label{continuous action}
If $(A,\Vert\cdot\Vert,^{*})$ is an involutive Banach algebra and $(A,\mathbb{Z}^n,\alpha)$ a dynamical system, then the map
\[\widehat{\alpha}:\mathbb{T}^n\times\ell^1(A\rtimes_{\alpha}\mathbb{Z}^n)\rightarrow\ell^1(A\rtimes_{\alpha}\mathbb{Z}^n),\,\,\,(\widehat{\alpha}(z,f))({\bf k}):=(z.f)({\bf k}):=z^{\bf k}\cdot f({\bf k})
\]defines a continuous action of $\mathbb{T}^n$ on $\ell^1(A\rtimes_{\alpha}\mathbb{Z}^n)$ by algebra automorphisms.
\end{lemma}

\begin{proof}
\,\,\,Obviously $\widehat{\alpha}$ defines an action. Moreover,
\begin{align}
((z.f)\star(z.g))({\bf k})&=\sum_{{\bf l}\in\mathbb{Z}^n}((z.f)({\bf l}))(\alpha({\bf l},(z.g)({\bf k-l})))=\sum_{{\bf l}\in\mathbb{Z}^n}(z^{\bf l}\cdot f({\bf l}))(z^{\bf k-l}\cdot\alpha({\bf l},g({\bf k-l})))\notag\\
&=z^{\bf k}\cdot\sum_{{\bf l}\in\mathbb{Z}^n}f({\bf l})\alpha({\bf l},g({\bf k-l}))=(z.(f\ast g))({\bf k}),\notag
\end{align}
 \[\Vert z.f\Vert_1=\sum_{{\bf k}\in\mathbb{Z}^n}\Vert (z.f)({\bf k})\Vert=\sum_{{\bf k}\in\mathbb{Z}^n}\Vert z^{\bf k}\cdot f({\bf k})\Vert=\sum_{{\bf k}\in\mathbb{Z}^n}\Vert f({\bf k})\Vert=\Vert f\Vert_1
 \]and 
\[((z.f)^*)({\bf k})=\alpha({\bf k},((z.f)(-{\bf k}))^{*})=z^{\bf k}\cdot\alpha({\bf k},(f(-{\bf k}))^{*})=(z.f^{*})({\bf k})
\]show that each element $z\in\mathbb{T}^n$ acts as an automorphism of $(A,\Vert\cdot\Vert,^{*})$.

To see the continuity of $\widehat{\alpha}$, we choose $f\in F(\mathbb{Z}^n,A)$, $\epsilon>0$ and a neighbourhood $U$ of the unity of $\mathbb{T}^n$ such that $\vert z^{\bf k}-1\vert\leq\epsilon$ for all ${\bf k}\in\supp(f)$ and $z\in U$. We thus obtain
\[\Vert z.f-f\Vert_1=\sum_{{\bf k}\in\mathbb{Z}^n}\vert z^{\bf k}-1\vert\Vert f({\bf k})\Vert=\sum_{{\bf k}\in\supp(f)}\vert z^{\bf k}-1\vert\Vert f({\bf k})\Vert\leq\epsilon\Vert f\Vert_1
\]for all $z\in U$. Since $F(\mathbb{Z}^n,A)$ is dense in $\ell^1(A\rtimes_{\alpha}\mathbb{Z}^n)$, $\widehat{\alpha}$ defines a continuous action of $\mathbb{T}^n$ on $\ell^1(A\rtimes_{\alpha}\mathbb{Z}^n)$ by algebra automorphisms.
\end{proof}

\begin{proposition}
If $(A,\Vert\cdot\Vert,^{*})$ is an involutive Banach algebra and $(A,\mathbb{Z}^n,\alpha)$ a dynamical system, then the triple \[(\ell^1(A\rtimes_{\alpha}\mathbb{Z}^n),\mathbb{T}^n,\widehat{\alpha})
\]defines a dynamical system.\sindex[n]{$(\ell^1(A\rtimes_{\alpha}\mathbb{Z}^n),\mathbb{T}^n,\widehat{\alpha})$}
\end{proposition}

\begin{proof}
\,\,\,The claim is a direct consequence of Lemma \ref{continuous action}.
\end{proof}

\begin{example}\label{l^1 as example}
If $(A,\Vert\cdot\Vert,^{*})$ is an involutive Banach algebra and $(A,\mathbb{Z}^n,\alpha)$ a dynamical system, then the dynamical system $(\ell^1(A\rtimes_{\alpha}\mathbb{Z}^n),\mathbb{T}^n,\widehat{\alpha})$ is a trivial NCP $\mathbb{T}^n$-bundle. Indeed, for ${\bf k}\in\mathbb{Z}^n$ we define
\[\delta_{\bf k}({\bf l}):=
\begin{cases}
1_A &\text{for}\,\,\,{\bf l}={\bf k}\\
0 &\text{otherwise}.
\end{cases}
\]Then 
\[z.\delta_{\bf k}=z^{\bf k}\cdot\delta_{\bf k}\,\,\,\text{and}\,\,\,\delta_{\bf k}\star\delta_{\bf k}^{*}=\delta_{\bf k}^{*}\star\delta_{\bf k}={\bf 1}
\]show that $\delta_{\bf k}$ is an invertible element of $\ell^1(A\rtimes_{\alpha}\mathbb{Z}^n)$ lying in the isotypic component $\ell^1(A\rtimes_{\alpha}\mathbb{Z}^n)_{\bf k}$.
\end{example}

\begin{construction}{\bf(The enveloping $C^{*}$-algebra).}\index{Enveloping $C^{*}$-Algebra}
If $(A,\Vert\cdot\Vert,^{*})$ is an involutive Banach algebra, then any involutive representation $(\pi,\mathcal{H})$ of $A$, for some Hilbert space $\mathcal{H}$, satisfies 
\[\Vert\pi(a)\Vert_{\text{op}}\leq\Vert a\Vert.
\]Indeed, this follows from the fact $\pi$ is norm-decreasing since it shrinks spectra and $B(\mathcal{H})$ is a $C^*$-algebra. The supremum over all such involutive representations $(\pi,\mathcal{H})$ is bounded, i.e.,
\[\Vert a\Vert_{\text{sup}}=\sup_{(\pi,\mathcal{H})}\Vert\pi(a)\Vert_{\text{op}}\leq\Vert a\Vert,
\]and thus defines a seminorm on $A$. If this is not already a norm, we factor $A$ by its kernel to get a normed algebra. Since $\Vert\pi(a^*a)\Vert_{\text{op}}=\Vert\pi(a)\Vert_{\text{op}}^2$ for each $(\pi,\mathcal{H})$, this is a $C^*$-norm. The completion of $A$ in this norm is a $C^*$-algebra and called the \emph{enveloping $C^*$-algebra}.
\begin{flushright}
$\blacksquare$
\end{flushright}
\end{construction}

\begin{definition}{\bf($C^{*}$-crossed products).}\index{Crossed Products!$C^{*}$-}
If $(A,\Vert\cdot\Vert,^{*})$ is an involutive Banach algebra and $(A,\mathbb{Z}^n,\alpha)$ a dynamical system, then the enveloping $C^*$-algebra of $\ell^1(A\rtimes_{\alpha}\mathbb{Z}^n)$ is denoted by $C^*(A\rtimes_{\alpha}\mathbb{Z}^n)$\sindex[n]{$C^*(A\rtimes_{\alpha}\mathbb{Z}^n)$} and is called the \emph{$C^{*}$-crossed product} associated to $(A,\mathbb{Z}^n,\alpha)$.
\end{definition}

\begin{example}\label{crossed product as example}
If $(A,\Vert\cdot\Vert,^{*})$ is an involutive Banach algebra and $(A,\mathbb{Z}^n,\alpha)$ a dynamical system, then the action $\widehat{\alpha}$ of Lemma \ref{continuous action} extends to a continuous action of $\mathbb{T}^n$ on the $C^*$-crossed product $C^*(A\rtimes_{\alpha}\mathbb{Z}^n)$ by algebra automorphisms. For details we refer to the paper [Ta74]. In particular, the corresponding dynamical system 
\[(C^*(A\rtimes_{\alpha}\mathbb{Z}^n),\mathbb{T}^n,\widehat{\alpha})
\]is a trivial NCP $\mathbb{T}^n$-bundle. This follows exactly as in Example \ref{l^1 as example}.\sindex[n]{$(C^*(A\rtimes_{\alpha}\mathbb{Z}^n),\mathbb{T}^n,\widehat{\alpha})$}
\end{example}

\begin{example}{\bf(Topological dynamical systems).}\index{Dynamical Systems!Topological}
To each topological dynamical system $(X,\varphi)$, i.e, to each pair $(X,\varphi)$, consisting of a compact Hausdorff space $X$ and a homeomorphism $\varphi:X\rightarrow X$, one can associate a $C^*$-dynamical system. Indeed, choose $A=C(X)$ and define an action $\alpha$  of $\mathbb{Z}$ on $C(X)$ by 
\[(\alpha(k).f)(x):=f(\varphi^{-k}(x)).
\]By Example \ref{crossed product as example}, the associated $C^*$-crossed product $C^*(C(X)\rtimes_{\alpha}\mathbb{Z})$ is a trivial NCP $\mathbb{T}$-bundle.
\end{example}

\begin{construction}\label{l^1 spaces associated to cocycles I}{\bf($\ell^1$-spaces associated to $2$-cocycles).}
Let $(A,\Vert\cdot\Vert,^{*})$ be an involutive Banach algebra and $\omega\in Z^2(\mathbb{Z}^n,\mathbb{T})$\sindex[n]{$Z^2(\mathbb{Z}^n,\mathbb{T})$}. The involutive Banach algebra $\ell^1(A\times_{\omega}\mathbb{Z}^n)$\sindex[n]{$\ell^1(A\times_{\omega}\mathbb{Z}^n)$} is defined by introducing on $F(\mathbb{Z}^n,A)$ a twisted multiplication 
\[(f\star g)({\bf k}):=\sum_{{\bf l}\in\mathbb{Z}^n}f({\bf l})g({\bf k-l})\omega({\bf l},{\bf k-l})
\]and an involution
\[(f^*)({\bf k}):=\overline{\omega({\bf k},-{\bf k})}\cdot f(-{\bf k})^*.
\]The cocycle property ensures that the multiplication is associative.
\begin{flushright}
$\blacksquare$
\end{flushright}
\end{construction}

\begin{example}\label{l^1 twisted as example}
Let $(A,\Vert\cdot\Vert,^{*})$ be an involutive Banach algebra and $\omega\in Z^2(\mathbb{Z}^n,\mathbb{T})$. Similarly to Lemma \ref{continuous action}, we see that the map
\[\widehat{\alpha}:\mathbb{T}^n\times\ell^1(A\times_{\omega}\mathbb{Z}^n)\rightarrow\ell^1(A\times_{\omega}\mathbb{Z}^n),\,\,\,(\widehat{\alpha}(z,f))({\bf k}):=(z.f)({\bf k}):=z^{\bf k}\cdot f({\bf k})
\]defines a continuous action of $\mathbb{T}^n$ on $\ell^1(A\times_{\omega}\mathbb{Z}^n)$ by algebra automorphisms. Moreover, the corresponding dynamical system
\[(\ell^1(A\times_{\omega}\mathbb{Z}^n),\mathbb{T}^n,\widehat{\alpha})
\]turns out to be a trivial NCP $\mathbb{T}^n$-bundle (cf. Example \ref{l^1 as example}).\sindex[n]{$(\ell^1(A\times_{\omega}\mathbb{Z}^n),\mathbb{T}^n,\widehat{\alpha})$}
\end{example}

\begin{example}\label{a connection to locally compact groups}
Let $\theta\in\Alt^2(\mathbb{Z}^n,\mathbb{R})$ be a skew-symmetric real matrix and consider the $2$-cocycle
\[\omega:\mathbb{Z}^n\times\mathbb{Z}^n\rightarrow\mathbb{T},\,\,\,\omega({\bf k},{\bf l}):=\exp(i\theta({\bf k},{\bf l})).
\]Then we obtain a ($2$-step nilpotent) Lie group $H:=\mathbb{T}\times_{\omega}\mathbb{Z}^n$ which is a central extension of $\mathbb{Z}^n$ by the circle group $\mathbb{T}$. A short observation now shows that
\[L^1(H)\cong\ell^1(L^1(\mathbb{T})\times_{\omega}\mathbb{Z}^n)
\]carries the structure of a trivial NCP $\mathbb{T}^n$-bundle.
\end{example}

\begin{remark}
Suppose that we are in the situation of Example \ref{a connection to locally compact groups}. Then the circle $\mathbb{T}$ acts continuously on the group algebra $L^1(H)$ by translations in the first argument. The corresponding Fourier decomposition leads to
\[L^1(H)_{\text{fin}}=\bigoplus_{{\bf k}\in\mathbb{Z}} L^1(H)_{\bf k}\cong\bigoplus_{{\bf k}\in\mathbb{Z}} \chi_{\bf k}\cdot\ell^1(\mathbb{C}\times_{\omega}\mathbb{Z}^n),
\]where $\ell^1(\mathbb{C}\times_{\omega}\mathbb{Z}^n)$ denotes an $\ell^1$-version of Example \ref{NC n-tori as NCPTB}.
\end{remark}

\begin{example}
The enveloping $C^*$-algebra of $\ell^1(A\times_{\omega}\mathbb{Z}^n)$ is denoted by $C^*(A\times_{\omega}\mathbb{Z}^n)$\sindex[n]{$C^*(A\times_{\omega}\mathbb{Z}^n)$} and is called the \emph{twisted group $C^*$-algebra of $G$ by $\omega$}. The action $\widehat{\alpha}$ of Example \ref{l^1 twisted as example} extends to a continuous action of $\mathbb{T}^n$ on $C^*(A\times_{\omega}\mathbb{Z}^n)$ by algebra automorphisms (cf. Example \ref{crossed product as example}). The corresponding dynamical system
\[(C^*(A\times_{\omega}\mathbb{Z}^n),\mathbb{T}^n,\widehat{\alpha})
\]is a trivial NCP $\mathbb{T}^n$-bundle as well.\sindex[n]{$(C^*(A\times_{\omega}\mathbb{Z}^n),\mathbb{T}^n,\widehat{\alpha})$}
\end{example}

\section{Classification of Trivial NCP Torus Bundles}\label{Classification of NCPT^nB}\index{Classification!of Trivial NCP Torus Bundles}

In classical geometry there exists up to isomorphy only one trivial principal $\mathbb{T}^n$-bundle over a given manifold $M$, namely $(M\times\mathbb{T}^n,M,\mathbb{T}^n,q_M,\sigma_{\mathbb{T}^n})$. The situation completely changes in the noncommutative world. In Proposition \ref{structure of trivial noncomm T^n-bundles} we gave a description of the underlying algebraic structure of a trivial NCP $\mathbb{T}^n$-bundle. This structure may therefore be considered as an algebraic counterpart of Definition \ref{trivial noncomm T^n-bundles}. The purpose of this section is to give a complete classification of algebraically trivial NCP $\mathbb{T}^n$-bundles in terms of a suitable cohomology theory. As a main tool we use the theory of group extensions. For the necessary background on the theory of (Lie-) group extensions we refer to Appendix \ref{some lie group cohomology}.

\subsection{Factor Systems for Trivial NCP Torus Bundles}

In this subsection we introduce a ``cohomology theory" for trivial NCP $\mathbb{T}^n$-bundles, which is inspired by the classical cohomology theory of groups. The corresponding cohomology spaces will be crucial for the classification part of this section.

\begin{definition}\label{cochains}
Let $n\in\mathbb{N}$ and $B$ be a unital algebra. 

(a) We write $C_B:B^{\times}\rightarrow\Aut(B)$ for the \emph{conjugation action} of $B^{\times}$ on $B$.

(b) We call a map $S\in C^1(\mathbb{Z}^n,\Aut(B))$\sindex[n]{$C^1(\mathbb{Z}^n,\Aut(B))$} an \emph{outer action} of $\mathbb{Z}^n$ on $B$ if there exists 
\[\omega\in C^2(\mathbb{Z}^n,B^{\times})\,\,\,\text{with}\,\,\,\delta_S=C_B\circ\omega,
\]\sindex[n]{$C^2(\mathbb{Z}^n,B^{\times})$}where $\delta_S({\bf k},{\bf l}):=S({\bf k})S({\bf l})S({\bf k+l})^{-1}$ for ${\bf k},{\bf l}\in\mathbb{Z}^n$.

(c) On the set of outer actions we define an equivalence relation by
\[S\sim S'\,\,\,\Leftrightarrow\,\,\,(\exists h\in C^1(\mathbb{Z}^n,B^{\times}))\,S'=(C_B\circ h)\cdot S
\]and call the equivalence class $[S]$ of an outer action $S$ a $\mathbb{Z}^n$-\emph{kernel}.\index{$\mathbb{Z}^n$-Kernel}

(d) For $S\in C^1(\mathbb{Z}^n,\Aut(B))$ and $\omega\in C^2(\mathbb{Z}^n,B^{\times})$ let
\[(d_S\omega)({\bf k},{\bf l},{\bf m}):=S({\bf k})(\omega({\bf l},{\bf m}))\omega({\bf k},{\bf l+m})\omega({\bf k+l},{\bf m})^{-1}\omega({\bf k},{\bf l})^{-1}.
\]
\end{definition}

\begin{lemma}\label{action on factor system}
Let $n\in\mathbb{N}$ and $B$ be a unital algebra and consider the group $C^1(\mathbb{Z}^n,B^{\times})$ with respect to pointwise multiplication. This group acts on the set 
\[C^1(\mathbb{Z}^n,\Aut(B))\,\,\,\text{by}\,\,\,h.S:=(C_B\circ h)\cdot S
\]and on the product set
\begin{align}
C^1(\mathbb{Z}^n,\Aut(B))\times C^2(\mathbb{Z}^n,B^{\times})\,\,\,\text{by}\,\,\,h.(S,\omega):=(h.S,h\ast_S\omega)\label{formula 6.1}
\end{align}
for
\[(h\ast_S\omega)({\bf k},{\bf l}):=h({\bf k})S({\bf k})h({\bf l})\omega({\bf k},{\bf l})h({\bf k+l})^{-1}.
\]The stabilizer of $(S,\omega)$ is given by
\[C^1(\mathbb{Z}^n,B^{\times})_{(S,\omega)}=Z^1(\mathbb{Z}^n,Z(B)^{\times})_S
\]\emph{(}cf. Definition \ref{center of N}\emph{)} which depends only on $[S]$, but not on $\omega$. Moreover, the following assertions hold:
\begin{itemize}
\item[\emph{(a)}]
The subset 
\[\{(S,\omega)\in C^1(\mathbb{Z}^n,\Aut(B))\times C^2(\mathbb{Z}^n,B^{\times}):\delta_S=C_B\circ\omega\}
\]is invariant.
\item[\emph{(b)}] 
If $\delta_S=C_B\circ\omega$, then $\im(d_S\omega)\subseteq Z(B)^{\times}$ .
\item[\emph{(c)}]  
If $\delta_S=C_B\circ\omega$ and $h.(S.\omega)=(S',\omega')$, then $d_{S'}\omega'=d_S\omega$.
\end{itemize}
\end{lemma}

\begin{proof}
\,\,\,That (\ref{formula 6.1}) defines a group action follows from the trivial relation $(hh').S=h.(h'.S)$, and
\begin{align}
&(h\ast_{h'.S}(h'\ast_S\omega))({\bf k},{\bf l})\notag\\
&=h({\bf k})(h'.S)({\bf k})(h({\bf l}))(h'\ast_S\omega)({\bf k},{\bf l})h({\bf k+l})^{-1}\notag\\
&=h({\bf k})h'({\bf l})S({\bf k})(h({\bf l}))h'({\bf k})^{-1}h'({\bf k})S({\bf k})(h'({\bf l}))\omega({\bf k},{\bf l})h'({\bf k+l})^{-1}h({\bf k+l})^{-1}\notag\\
&=(hh')({\bf k})S({\bf k})(hh'({\bf l}))\omega({\bf k},{\bf l})(hh')({\bf k+l})^{-1}\notag\\
&=((hh')\ast_S\omega)({\bf k},{\bf l}).\notag
\end{align}
To calculate the stabilizer of the pair $(S,\omega)$, we observe that the condition $h.S=S$ is equivalent to $h\in C^1(\mathbb{Z}^n,Z(B)^{\times})$. Then $h\ast_S\omega=\omega\cdot d_Sh$, and this equals $\omega$ if and only if $h$ is a cocycle. Therefore,
\[C^1(\mathbb{Z}^n,B^{\times})_{(S,\omega)}=Z^1(\mathbb{Z}^n,Z(B)^{\times})_S.
\]
(a) If the pair $(S,\omega)$ satisfies $\delta_S=C_B\circ\omega$, then we obtain for $h\in C^1(\mathbb{Z}^n,B^{\times})$ the relation
\begin{align}
\delta_{h.S}({\bf k},{\bf l})&=\delta_{(C_B\circ h)\cdot S}({\bf k},{\bf l})\notag\\
&=C_B(h({\bf k}))c_{S({\bf k})}(C_B(h({\bf l})))\delta_S({\bf k},{\bf l})C_B(h({\bf k+l})^{-1})\notag\\
&=C_B(h({\bf k}))C_B(S({\bf k})(h({\bf l})))C_B(\omega({\bf k},{\bf l}))C_B(h({\bf k+l})^{-1})\notag\\
&=C_B(h({\bf k})S({\bf k})(h({\bf l}))\omega({\bf k},{\bf l})h({\bf k+l})^{-1})\notag\\
&=(C_B\circ(h\ast_S\omega))({\bf k},{\bf l}).\notag
\end{align}

(b) This follows from $C_B((d_S\omega)({\bf k},{\bf l},{\bf m}))=1_B$.

(c) To verify the relation $d_{S'}\omega'=d_S\omega$ one has to use frequently that $d_S\omega$ and $d_{S'}\omega'$ have values in $Z(B)^{\times}$. For a complete calculation we refer to [Ne07a], Lemma 1.10 (4).
\end{proof}

\begin{definition}\label{factor system}
Let $n\in\mathbb{N}$ and $B$ be a unital algebra. The elements of the set
\[Z^2(\mathbb{Z}^n,B):=\{(S,\omega)\in C^1(\mathbb{Z}^n,\Aut(B))\times C^2(\mathbb{Z}^n,B^{\times}):\delta_S=C_B\circ\omega,\, d_S\omega=1_B\}
\]are called \emph{factor systems}\index{Factor Systems} for the pair $(\mathbb{Z}^n,B)$ (or simply $(n,B)$). By Lemma \ref{action on factor system}, the set 
$Z^2(\mathbb{Z}^n,B)$ is invariant under the action of $C^1(\mathbb{Z}^n,B^{\times})$ and we write
\[H^2(\mathbb{Z}^n,B):=Z^2(\mathbb{Z}^n,B)/C^1(\mathbb{Z}^n,B^{\times})
\]for the corresponding cohomology space.\sindex[n]{$Z^2(\mathbb{Z}^n,B)$}\sindex[n]{$H^2(\mathbb{Z}^n,B)$}
\end{definition}

\subsection{Characteristic Classes for Trivial NCP Torus Bundles}

In this short subsection we provide a construction that associates to each trivial NCP $\mathbb{T}^n$-bundle $(A,\mathbb{T}^n,\alpha)$ a class in $H^2(\mathbb{Z}^n,B)$ for $B=A_{\bf 0}$.

\begin{construction}\label{TNCT^nB ass class}{\bf(Characteristic classes).}\index{Characteristic Classes}
Let $(A,\mathbb{T}^n,\alpha)$ be a trivial NCP $\mathbb{T}^n$-bundle. The set
\[A^{\times}_h:=\bigcup_{{\bf k}\in\mathbb{Z}^n}A^{\times}_{\bf k}
\]of homogeneous units is a subgroup of $A^{\times}$ containing $B^{\times}\cdot 1_A\cong B^{\times}$. We thus obtain an extension
\begin{align}
1\longrightarrow B^{\times}\longrightarrow A^{\times}_h\longrightarrow\mathbb{Z}^n\longrightarrow 1\notag
\end{align}
of groups. Next, it is instructive to see how this can be made more explicit in terms of smooth factor systems. We recall that since $\mathbb{Z}^n$ is discrete, all outer actions are smooth, i.e., the smoothness conditions on $S$ and $\omega$ become vacuous. Thus, $A^{\times}_h$ is equivalent to a crossed product of the form $B^{\times}\times_{(S,\omega)}\mathbb{Z}^n$ for a factor system $(S,\omega)\in Z^2(\mathbb{Z}^n,B)$. 
In this way each trivial NCP $\mathbb{T}^n$-bundle $(A,\mathbb{T}^n,\alpha)$ induces a \emph{characteristic class}
\[\chi\left((A,\mathbb{T}^n,\alpha)\right):=[(S,\omega)]\in H^2(\mathbb{Z}^n,B).
\]
\begin{flushright}
$\blacksquare$
\end{flushright}
\end{construction}

\begin{definition}\label{iso of NCTT^B}{\bf(Isomorphy).}\index{Isomorphy}
We call two trivial NCP $\mathbb{T}^n$-bundles $(A,\mathbb{T}^n,\alpha)$ and $(A',\mathbb{T}^n,\alpha')$ \emph{isomorphic} if there exists a $\mathbb{T}^n$-equivariant $\varphi:A\rightarrow A'$ isomorphism of algebras, i.e., a map $\varphi\in\Iso_{\mathbb{T}^n}(A,A')$.
\end{definition}

\begin{proposition}\label{iso TNCT^nB same class}
Let $(A,\mathbb{T}^n,\alpha)$ and $(A',\mathbb{T}^n,\alpha')$ be two isomorphic trivial NCP $\mathbb{T}^n$-bundles. If $\varphi:A\rightarrow A'$ is a $\mathbb{T}^n$-equivariant algebra isomorphism, then the following assertions hold:
\begin{itemize}
\item[\emph{(a)}]
The restriction $\varphi\mid_B:B\rightarrow B'$ is an isomorphism of algebras. In particular, this map induces an isomorphism $H^2(\varphi\mid_B)$ of the corresponding cohomology spaces $H^2(\mathbb{Z}^n,B)$ and $H^2(\mathbb{Z}^n,B')$.
\item[\emph{(b)}]
For the characteristic classes of $(A,\mathbb{T}^n,\alpha)$ and $(A',\mathbb{T}^n,\alpha')$ the following relation holds:
\[H^2(\varphi\mid_{B})\left(\chi\left((A,\mathbb{T}^n,\alpha)\right)\right)=\chi\left((A',\mathbb{T}^n,\alpha')\right).
\]
\end{itemize}
\end{proposition}

\begin{proof}
\,\,\,(a) The first part follows from an easy calculation and the definition of the cohomology spaces.

(b) For the second assertion we note that the map $\varphi$ induces the following commutative diagram of groups:
\[\xymatrix{ 1 \ar[r]& B^{\times}\ar[r] \ar[d]^{\varphi\mid_{B^{\times}}}& A^{\times}_h\ar[r] \ar[d]^{\varphi\mid_{A^{\times}_h}}& \mathbb{Z}^n\ar[r] \ar@{=}[d]& 1 \\
1 \ar[r]& (B')^{\times}\ar[r] & (A')^{\times}_h\ar[r] & \mathbb{Z}^n \ar[r] & 1.}
\]
Hence, the claim follows from classical extension theory of groups.
\end{proof}

\begin{definition}\label{equivalence of NCTT^B}{\bf(Equivalence).}\index{Equivalence}
Let $n\in\mathbb{N}$. We call two trivial NCP $\mathbb{T}^n$-bundles $(A,\mathbb{T}^n,\alpha)$ and $(A',\mathbb{T}^n,\alpha')$ \emph{equivalent} if the following two conditions are satisfied:
\begin{itemize}
\item[(i)]
$B=B'$.
\item[(ii)]
There exists a $\mathbb{T}^n$-equivariant algebra isomorphism $\varphi:A\rightarrow A'$ with $\varphi_{\mid_B}=\id_B$.
\end{itemize}
If $(A,\mathbb{T}^n,\alpha)$ and $(A',\mathbb{T}^n,\alpha')$ are equivalent trivial NCP $\mathbb{T}^n$-bundles, then we write $[(A,\mathbb{T}^n,\alpha)]$ for the corresponding equivalence class.
\end{definition}

\begin{proposition}\label{equ TNCT^nB same class}
Let $(A,\mathbb{T}^n,\alpha)$ and $(A',\mathbb{T}^n,\alpha')$ be two equivalent trivial NCP $\mathbb{T}^n$-bundles. Then their corresponding characteristic classes coincide, i.e.,
\[\chi\left((A,\mathbb{T}^n,\alpha)\right)=\chi\left((A',\mathbb{T}^n,\alpha')\right)\in H^2(\mathbb{Z}^n,B).
\]
\end{proposition}

\begin{proof}
\,\,\,The claim easily follows from Definition \ref{equivalence of NCTT^B} and Proposition \ref{iso TNCT^nB same class}.
\end{proof}

\subsection{Classification of Algebraically Trivial NCP Torus Bundles}

The main goal of this subsection is to present a complete classification of ``algebraically trivial NCP $\mathbb{T}^n$-bundles". Definition \ref{trivial noncomm T^n-bundles} and Proposition \ref{structure of trivial noncomm T^n-bundles} lead to the following definition:

\begin{definition}\label{algebraic trivial NCP T^n-bundles}{\bf(Algebraically trivial NCP $\mathbb{T}^n$-bundles).}\index{Trivial NCP!$\mathbb{T}^n$-Bundles (Algebraic)}
A $\mathbb{Z}^n$-graded unital associative algebra 
\[A=\bigoplus_{{\bf k}\in\mathbb{Z}^n}A_{\bf k}
\]with $B:=A_{\bf 0}$ is called an \emph{algebraically trivial NCP $\mathbb{T}^n$-bundle with base $B$}, if each grading space $A_{\bf k}$ contains an invertible element.
\end{definition}

\begin{remark}{\bf(Examples).}
We will first focus our attention on the abstract theory of algebraically trivial NCP $\mathbb{T}^n$-bundles. For examples we refer the interested reader to the next section.
\end{remark}

\begin{lemma}{\bf(Characteristic classes).}\index{Characteristic Classes}
Each algebraically trivial NCP $\mathbb{T}^n$-bundle $A$ possesses a characteristic class $\chi(A)\in H^2(\mathbb{Z}^n,B)$.
\end{lemma}

\begin{proof}
\,\,\,Indeed, this assertion immediately follows from Construction \ref{TNCT^nB ass class}.
\end{proof}

\begin{definition}{\bf(Equivalence).}\index{Equivalence}
Two algebraically trivial NCP $\mathbb{T}^n$-bundles $A$ and $A'$ with base $B$ are called \emph{equivalent} if there is an algebra isomorphism
$\varphi:A\rightarrow A'$ satisfying $\varphi(A_{\bf k})=A'_{\bf k}$ for all ${\bf k}\in\mathbb{Z}^n$. If $A$ and $A'$ are equivalent algebraically trivial NCP $\mathbb{T}^n$-bundles, then we write $[A]$ for the corresponding equivalence class.
\end{definition}

\begin{proposition}\label{equ  ATNCT^nB same class}
Let $A$ and $A'$ be two equivalent algebraically trivial NCP $\mathbb{T}^n$-bundles. Then their corresponding characteristic classes coincide, i.e.,
\[\chi(A)=\chi(A')\in H^2(\mathbb{Z}^n,B).
\]
\end{proposition}

\begin{proof}
\,\,\,The claim is a consequence of Proposition \ref{equ TNCT^nB same class}.
\end{proof}

\begin{definition}{\bf(Set of equivalence classes).}
Let $n\in\mathbb{N}$ and $B$ be a unital algebra. We write 
\[\Ext(\mathbb{Z}^n,B)
\]for the set of all equivalence classes of algebraically trivial NCP $\mathbb{T}^n$-bundles with base $B$.\sindex[n]{$\Ext(\mathbb{Z}^n,B)$}
\end{definition}


\begin{lemma}\label{chi well-defined}
Let $B$ be a unital algebra. Then the map
\[\chi:\Ext(\mathbb{Z}^n,B)\rightarrow H^2(\mathbb{Z}^n,B),\,\,\,[A]\mapsto\chi(A)
\]is well-defined.
\end{lemma}

\begin{proof}
\,\,\,The assertion immediately follows from Proposition \ref{equ  ATNCT^nB same class}.
\end{proof}

In the remaining part of this subsection we will show that the map $\chi$ of Lemma \ref{chi well-defined} is actually a bijection.


\begin{construction}\label{realization of TNCT^B from factor systems I}
Let $n\in\mathbb{N}$ and $B$ be a unital algebra. Further, let 
\[A:=\bigoplus_{{\bf k}\in\mathbb{Z}^n}Bv_{\bf k}
\]be a vector space with basis $(v_{\bf k})_{{\bf k}\in\mathbb{Z}^n}$. For a factor system $(S,\omega)\in Z^2(\mathbb{Z}^n,B)$ we
define a multiplication map 
\[m_{(S,\omega)}:A\times A\rightarrow A
\]given on homogeneous elements by
\begin{align}
m_{(S,\omega)}(bv_{\bf k},b'v_{\bf l}):=b(S({\bf k})(b'))\omega({\bf k},{\bf l})v_{{\bf k+l}},\label{multiplication}
\end{align}
and write $A_{(S,\omega)}$ for the vector space $A$ endowed with the multiplication (\ref{multiplication}). A short calculation shows that $A_{(S,\omega)}$ is a $\mathbb{Z}^n$-graded unital associative algebra with $A_{\bf 0}=B$ and unit $v_{\bf 0}$. Moreover, each grading space $A_i$, $i\in I_n$, contains invertible elements with respect to to this multiplication: Indeed, if $i\in I_n$ and $b\in B^{\times}$, then the inverse of $bv_i$ is given by
\[S(i)^{-1}(b^{-1}\omega(i,-i)^{-1})v_{-i}.
\]Thus, $A_{(S,\omega)}$ is an algebraically trivial NCP $\mathbb{T}^n$-bundle with base $B$ and $\chi(A_{(S,\omega)})=[(S,\omega)]$.
\begin{flushright}
$\blacksquare$
\end{flushright}
\end{construction}

We summarize what we have seen in the preceding construction in the following proposition:

\begin{proposition}\label{realization of TNCT^B from factor systems II}
If $n\in\mathbb{N}$ and $B$ is a unital algebra, then each element $[(S,\omega)]\in H^2(\mathbb{Z}^n,B)$ can be realized by an algebraically trivial NCP $\mathbb{T}^n$-bundles $A$ with
\[A_{\bf 0}=B\,\,\,\text{and}\,\,\,\chi(A)=[(S,\omega)].
\]
\end{proposition}

\begin{proof}
\,\,\,This assertion is a consequence of Construction \ref{realization of TNCT^B from factor systems I}: If $[(S,\omega)]\in H^2(\mathbb{Z}^n,B)$, then $A_{(S,\omega)}$ satisfies the requirements of the proposition.
\end{proof}

\begin{proposition}\label{realization of TNCT^B from factor systems III}
Let $n\in\mathbb{N}$ and $B$ be a unital algebra. Further, let $A$ be an algebraically trivial NCP $\mathbb{T}^n$-bundle with $A_{\bf 0}=B$. Then $A$ is equivalent to an algebraically trivial NCP $\mathbb{T}^n$-bundle of the form $A_{(S,\omega)}$ for some factor system $(S,\omega)\in Z^2(\mathbb{Z}^n,B)$.
\end{proposition}

\begin{proof}
\,\,\,Let $A$ be an algebraically trivial NCP $\mathbb{T}^n$-bundle with $A_{\bf 0}=B$. We consider the corresponding short exact sequence of groups
\[1\longrightarrow A^{\times}_{\bf 0}\longrightarrow A^{\times}_h\longrightarrow\mathbb{Z}^n\longrightarrow 1
\]and choose a section $\sigma:\mathbb{Z}^n\rightarrow A^{\times}_h$. Now, a short calculation shows that the map
\[\varphi:\left(A=\bigoplus_{{\bf k}\in\mathbb{Z}^n}A_{\bf k},m_A\right)\rightarrow\left(\bigoplus_{{\bf k}\in\mathbb{Z}^n}Bv_{\bf k},m_{(C_B\circ\sigma,\delta_{\sigma})}\right),
\]given on homogeneous elements by
\[\varphi(a_{\bf k}):=a_{\bf k}\sigma({\bf k})^{-1}v_{\bf k},
\]defines an equivalence of algebraically trivial NCP $\mathbb{T}^n$-bundles. 
\end{proof}

\begin{proposition}\label{realization of TNCT^B from factor systems IV}
Let $n\in\mathbb{N}$ and $B$ be a unital algebra. Further, let $(S,\omega)$ and $(S',\omega')$ be two factor systems in $ Z^2(\mathbb{Z}^n,B)$ with $[(S,\omega)]=[(S',\omega')]$. Then the corresponding algebraically trivial NCP $\mathbb{T}^n$-bundles $A_{(S,\omega)}$ and $A_{(S',\omega')}$ are equivalent.
\end{proposition}

\begin{proof}
\,\,\,We first recall that the condition $[(S',\omega')]=[(S,\omega)]$ is equivalent to the existence of an element $h\in C^1(\mathbb{Z}^n,B^{\times})$ with 
\[h.(S,\omega)=(S',\omega'),
\]Then a short observation shows that the map 
\[\varphi:\left(\bigoplus_{{\bf k}\in\mathbb{Z}^n}Bv_{\bf k},m_{(S',\omega')}\right)\rightarrow\left(\bigoplus_{{\bf k}\in\mathbb{Z}^n}Bv_{\bf k},m_{(S,\omega)}\right),
\]given on homogeneous elements by
\[\varphi(bv_{\bf k})=bh(k)v_{\bf k},
\]is an automorphism of vector spaces leaving the grading spaces invariant. We further have
\begin{align}
m_{(S,\omega)}(\varphi(bv_{\bf k}),\varphi(b'v_{\bf l}))&=m_{(S,\omega)}(bh({\bf k})v_{\bf k},b'h({\bf l})v_{\bf l})=bh({\bf k})S({\bf k})(b'h({\bf l}))\omega({\bf k},{\bf l})v_{\bf k+l}\notag\\
&=b[C_B(h({\bf k}))(S({\bf k})(b'))]h({\bf k})S({\bf k})(h({\bf l}))\omega({\bf k},{\bf l})v_{\bf k+l}\notag\\
&=b(h.S)({\bf k})(b')(h\ast_S\omega)({\bf k},{\bf l})h({\bf k+l})v_{\bf k+l}\notag\\
&=\varphi(b(h.S)({\bf k})(b')(h\ast_S\omega)({\bf k},{\bf l})v_{\bf k+l})=\varphi(m_{(S',\omega')}(bv_{\bf k},b'v_{\bf l})).\notag
\end{align}
Hence, $\varphi$ actually defines an equivalence of algebraically trivial NCP $\mathbb{T}^n$-bundles. 
\end{proof}

We are now ready to state and proof the main theorem of this section:

\begin{theorem}\label{classification on TNCT^B}
Let $n\in\mathbb{N}$ and $B$ be a unital algebra. Then the map
\[\chi:\Ext(\mathbb{Z}^n,B)\rightarrow H^2(\mathbb{Z}^n,B),\,\,\,[A]\mapsto\chi(A)
\]is a well-defined bijection.
\end{theorem}

\begin{proof}
\,\,\,Lemma \ref{chi well-defined} implies that the map $\chi$ is well-defined. The surjectivity of $\chi$ follows from Proposition \ref{realization of TNCT^B from factor systems II}. Hence, it remains to show that $\chi$ is injective: Therefore, we choose two algebraically trivial NCP $\mathbb{T}^n$-bundles $A$ and $A'$ with $A_{\bf 0}=A'_{\bf 0}=B$ and $\chi(A)=\chi(A')$. By Proposition \ref{realization of TNCT^B from factor systems III}, we may assume that $A=A_{(S,\omega)}$ and $A'=A_{(S',\omega')}$ for two factor systems $(S,\omega)$ and $(S',\omega')$ in $ Z^2(\mathbb{Z}^n,B)$ with $[(S',\omega')]=[(S,\omega)]$. Therefore, the claim follows from Proposition \ref{realization of TNCT^B from factor systems IV}.
\end{proof}

\subsection{$\mathbb{Z}^n$-Kernels} 

In the previous subsection we saw that the set of all equivalence classes of algebraically trivial NCP $\mathbb{T}^n$-bundles with a prescribed  fixed point algebra $B$ is classified by the cohomology space $H^2(\mathbb{Z}^n,B)$. Moreover, Proposition \ref{equ  ATNCT^nB same class} in particular implies that equivalent algebraically trivial NCP $\mathbb{T}^n$-bundles correspond to the same $\mathbb{Z}^n$-kernel. This leads to the following definition:

\begin{definition}{\bf(Equivalence classes of $\mathbb{Z}^n$-kernels).}
Let $n\in\mathbb{N}$ and $B$ be a unital algebra. We write
\[\Ext(\mathbb{Z}^n,B)_{[S]}
\]for the set of equivalence classes of algebraically trivial NCP $\mathbb{T}^n$-bundles with base $B$ corresponding to the $\mathbb{Z}^n$-kernel $[S]$.\sindex[n]{$\Ext(\mathbb{Z}^n,B)_{[S]}$}
\end{definition}

Note that the set $\Ext(\mathbb{Z}^n,B)_{[S]}$ may be empty. The aim of this section is to classify this set and give conditions for its non-emptiness. We start with the following theorem:

\begin{theorem}\label{class of Z-kernels I}
Let $n\in\mathbb{N}$ and $B$ be a unital algebra. Further, let $S$ be an outer action of $\mathbb{Z}^n$ on $B$ with $\Ext(\mathbb{Z}^n,B)_{[S]}\neq\emptyset$. Then the following assertions hold:
\begin{itemize}
\item[\emph{(a)}]
Each extension class in $\Ext(\mathbb{Z}^n,B)_{[S]}$ can be represented by an algebraically trivial NCP $\mathbb{T}^n$-bundle of the form $A_{(S,\omega)}$.
\item[\emph{(b)}]
Any other algebraically trivial NCP $\mathbb{T}^n$-bundle $A_{(S,\omega')}$ representing an element of $\Ext(\mathbb{Z}^n,B)_{[S]}$ satisfies
\[\omega'\cdot\omega^{-1}\in Z^2(\mathbb{Z}^n,Z(B)^{\times})_{[S]}.
\]
\item[\emph{(c)}]
Two algebraically trivial NCP $\mathbb{T}^n$-bundles $A_{(S,\omega)}$ and $A_{(S,\omega')}$ are equivalent if and only if
\[\omega'\cdot\omega^{-1}\in B^2(\mathbb{Z}^n,Z(B)^{\times})_{[S]}.
\]
\end{itemize}
\end{theorem}

\begin{proof}
\,\,\,(a) From Proposition \ref{realization of TNCT^B from factor systems II} we know that each algebraically trivial NCP $\mathbb{T}^n$-bundle $A$ is equivalent to one of the form $A_{(S',\omega')}$. If $[S]=[S']$ and $h\in C^1(\mathbb{Z}^n,B^{\times})$ satisfies $h.S=S'$, then $h^{-1}.(S',\omega')=(S,h^{-1}\ast_{S'}\omega')$, so that $\omega'':=h^{-1}\ast_{S'}\omega'$ implies that $A_{(S',\omega')}$ and $A_{(S,\omega'')}$ are equivalent. This means that each extension class in $\Ext(\mathbb{Z}^n,B)_{[S]}$ can be represented by an algebraically trivial noncommutative principal $\mathbb{T}^n$-bundle of the form $A_{(S,\omega'')}$.

(b) If $(S,\omega)$ and $(S,\omega')$ are two factor systems for the pair $(\mathbb{Z}^n,B)$, then $C_B\circ\omega=\delta_S=C_B\circ\omega'$ implies $\beta:=\omega'\cdot\omega^{-1}\in C^2(\mathbb{Z}^n,Z(B)^{\times})$. Further $d_S\omega'=d_S\omega=1_B$, so that 
\[1_B=d_S\omega'=d_S\omega\cdot d_S\beta=d_S\beta
\]implies $\beta\in Z^2(\mathbb{Z}^n,Z(B)^{\times})_{[S]}$. This means that $\omega'\in\omega\cdot Z^2(\mathbb{Z}^n,Z(B)^{\times})_{[S]}$.

(c) ($``\Rightarrow"$) In view of Proposition \ref{equ  ATNCT^nB same class} and Proposition \ref{realization of TNCT^B from factor systems IV}, the equivalence of the two algebraically trivial NCP $\mathbb{T}^n$-bundles $A_{(S,\omega)}$ and $A_{(S,\omega')}$is equivalent to the existence of an element $h\in C^1(\mathbb{Z}^n,B^{\times})$ with
\[S=h.S=(C_B\circ h)\cdot S\,\,\,\text{and}\,\,\,\omega'=h\ast_S\omega.
\]Then $C_B\circ h=\id_B$ implies that $h\in C^1(\mathbb{Z}^n,Z(B)^{\times})$. This leads to $\omega'=h\ast_S\omega=(d_Sh)\cdot\omega$, i.e., 
\[\omega'\omega^{-1}\in B^2(\mathbb{Z}^n,Z(B)^{\times})_{[S]}.
\]

($``\Leftarrow"$) If conversely $\omega'\omega^{-1}=d_Sh$ for some $h\in C^1(\mathbb{Z}^n,Z(B)^{\times})$, then we easily conclude that $h.S=S$ and $\omega'=h\ast_S\omega$.
\end{proof}

\begin{corollary}\label{class of Z-kernels II}
Let $n\in\mathbb{N}$ and $B$ be a unital algebra. Further, let $[S]$ be a $\mathbb{Z}^n$-kernel with $\Ext(\mathbb{Z}^n,B)_{[S]}\neq\emptyset$. Then the map
\[H^2(\mathbb{Z}^n,Z(B)^{\times})_{[S]}\times\Ext(\mathbb{Z}^n,B)_{[S]}\rightarrow\Ext(\mathbb{Z}^n,B)_{[S]}
\]given by
\[([\beta],[A_{(S,\omega)}])\mapsto[A_{(S,\omega\cdot\beta)}]
\]is a well-defined simply transitive action.
\end{corollary}

\begin{proof}
\,\,\,This follows directly form Theorem \ref{class of Z-kernels I}.
\end{proof}

\begin{remark}\label{B abelian}{\bf(Commutative fixed point algebras).}\index{Commutative fixed point algebras} 
(a) Let $n\in\mathbb{N}$ and suppose that $B$ is a commutative unital algebra. Then the adjoint representation of $B$ is trivial and a factor system $(S,\omega)$ for $(\mathbb{Z}^n,B)$ consists of a module structure $S:\mathbb{Z}^n\rightarrow \Aut(B)$ and an element $\omega\in C^2(\mathbb{Z}^n,B^{\times})$. Thus, $(S,\omega)$ defines an algebraically trivial NCP $\mathbb{T}^n$-bundle $A_{(S,\omega)}$ if and only if $d_S\omega=1_B$, i.e., $\omega\in Z^2(\mathbb{Z}^n,B^{\times})$. In this case we write 
$A_{\omega}$ for this algebraically trivial NCP $\mathbb{T}^n$-bundle. 

(b) Further $S\sim S'$ if and only if $S=S'$. Hence, a $\mathbb{Z}^n$-kernel $[S]$ is the same as a $\mathbb{Z}^n$-module structure $S$ on $B$ and $\Ext(\mathbb{Z}^n,B)_S:=\Ext(\mathbb{Z}^n,B)_{[S]}$ is the class of all algebraically trivial NCP $\mathbb{T}^n$-bundles with fixed point algebra $B$ corresponding to the $\mathbb{Z}^n$-module structure on $B$ given by $S$.

(c) According to Corollary \ref{class of Z-kernels II}, these equivalence classes correspond to cohomology classes of cocycles, so that the map
\[H^2(\mathbb{Z}^n,B^{\times})_S\rightarrow\Ext(\mathbb{Z}^n,B)_S,\,\,\,[\omega]\mapsto[A_{\omega}]
\]is a well-defined bijection.
\end{remark}


\begin{proposition}\label{cohomology class of d_sw}
Let $n\in\mathbb{N}$ and $B$ be a unital algebra. Then the following assertions hold:
\begin{itemize}
\item[\emph{(a)}]
If $(S,\omega)\in C^1(\mathbb{Z}^n,\Aut(B))\times C^2(\mathbb{Z}^n,B^{\times})$ with $\delta_S=C_B\circ\omega$, then $d_S\omega\in Z^3(\mathbb{Z}^n,Z(B)^{\times})_S$.
\item[\emph{(b)}]
For an outer action $S$ of $\mathbb{Z}^n$ on $B$, the cohomology class $[d_S\omega]\in H^3(\mathbb{Z}^n,Z(B)^{\times})_S$ depends only on the equivalence class $[S]$.
\end{itemize}
\end{proposition}

\begin{proof}
\,\,\,These two assertions follow from a more general calculation, which can be found in [Ne07a], Lemma 1.10 (5), (6).
\end{proof}

\begin{definition}\label{char. class d_Sw}
Let $n\in\mathbb{N}$ and $B$ be a unital algebra. Further, let $S$ be an outer action of $\mathbb{Z}^n$ on $B$ and choose $\omega\in C^2(\mathbb{Z}^n,B^{\times})$ with $\delta_S=C_B\circ\omega$. In view of Proposition \ref{cohomology class of d_sw}, the cohomology class
\[\nu(S):=[d_S\omega]\in H^3(\mathbb{Z}^n,Z(B)^{\times})_S
\]does not depend on the choice of $\omega$ and is constant on the equivalence class of $S$, so that we may also write $\nu([S]):=\nu(S)$. We call $\nu(S)$ the \emph{characteristic class} of $S$.
\end{definition}


\begin{theorem}\label{Ext(Z^n,B)_S non-empty}
Let $n\in\mathbb{N}$ and $B$ be a unital algebra. If $[S]$ is a $\mathbb{Z}^n$-kernel, then
\[\nu([S])={\bf 1}\,\,\,\Leftrightarrow\,\,\,\Ext(\mathbb{Z}^n,B)_{[S]}\neq\emptyset.
\]
\end{theorem}

\begin{proof}
\,\,\,($``\Leftarrow"$) If there exists an algebraically trivial NCP $\mathbb{T}^n$-bundle $A$ corresponding to $[S]$, then Proposition \ref{realization of TNCT^B from factor systems III} implies that we may without loss of generality assume that it is of the form $A_{(S,\omega)}$. In particular, this means that $d_S\omega=1_B$. Hence, we obtain $\nu([S])=[d_S\omega]=\textbf{1}$.

($``\Rightarrow"$) Suppose, conversely, that $\nu([S])=[d_S\omega]=\textbf{1}$. Then there exists $\omega\in C^2(\mathbb{Z}^n,B^{\times})$ with $\delta_S=C_B\circ\omega$ and some $h\in C^2(\mathbb{Z}^n,Z(B)^{\times})$ with $d_S\omega=d_Sh^{-1}$. Therefore, the element $\omega':=\omega\cdot h\in C^2(\mathbb{Z}^n,B^{\times})$ satisfies 
\[d_S\omega'=d_S\omega\cdot d_Sh=1_B\,\,\,\text{and}\,\,\,\delta_S=C_B\circ\omega'.
\]Hence, $(S,\omega')$ is a factor system for $(\mathbb{Z}^n,B)$ and Proposition \ref{realization of TNCT^B from factor systems II} implies the existence of an algebraically trivial NCP $\mathbb{T}^n$-bundle $A_{(S,\omega')}$ corresponding to the $\mathbb{Z}^n$-kernel $[S]$.
\end{proof}

\subsection{Some Useful Results on the Second Group Cohomology of $\mathbb{Z}^n$}

Let $n\in\mathbb{N}$ and $B$ be a commutative unital algebra. In view of the previous two subsections, we now present some useful results on the second cohomology groups $H^2(\mathbb{Z}^n,B^{\times})$\sindex[n]{$H^2(\mathbb{Z}^n,B^{\times})$}. We start with some general facts on the set of equivalence classes $\Ext(\mathbb{Z}^n,B^{\times})\cong H^2(\mathbb{Z}^n,B^{\times})$ of central extensions of $\mathbb{Z}^n$ by the abelian group $B^{\times}$. We recall that since $\mathbb{Z}^n$ is free abelian, the set $\Ext_{\text{ab}}(\mathbb{Z}^n,B^{\times})$\sindex[n]{$\Ext_{\text{ab}}(\mathbb{Z}^n,B^{\times})$} of equivalence classes of central extensions of $\mathbb{Z}^n$ by the abelian group $B^{\times}$ which are abelian consists of a single element. Further, we write $\Alt^2(\mathbb{Z}^n,B^{\times})$\sindex[n]{$\Alt^2(\mathbb{Z}^n,B^{\times})$} for the set of biadditive alternating maps from $\mathbb{Z}^n$ to $B^{\times}$ and define for each $[\omega]\in H^2(\mathbb{Z}^n,B^{\times})$ a map $f_{\omega}\in\Alt^2(\mathbb{Z}^n,B^{\times})$ by 
\[f_{\omega}({\bf k},{\bf l}):=\omega({\bf k},{\bf l})(\omega({\bf l},{\bf k}))^{-1}.
\]

\begin{proposition}\label{H^2(Z^n,B)}
Let $n\in\mathbb{N}$ and $B$ be a commutative unital algebra. Moreover, let $B^{\times}$ be a trivial $\mathbb{Z}^n$-module. Then the following assertions hold:
\begin{itemize}
\item[\emph{(a)}]
The map
\[\Phi:H^2(\mathbb{Z}^n,B^{\times})\rightarrow\Alt^2(\mathbb{Z}^n,B^{\times}),\,\,\,[\omega]\mapsto f_{\omega}
\]is an isomorphism of abelian groups.
\item[\emph{(b)}]
Moreover, the map
\[\Alt^2(\mathbb{Z}^n,B^{\times})\rightarrow(B^{\times})^{\frac{1}{2}n(n-1)},\,\,\,f\mapsto(f(e_i,e_j))_{1\leq i<j\leq n}
\]is an isomorphism of abelian groups.
\end{itemize}
\end{proposition}

\begin{proof}
\,\,\,This assertion is a special case of [Ne07b], Proposition II.4.
\end{proof}

Next, we consider the general case of abelian extensions of $\mathbb{Z}^2$ by the unit group of a commutative unital algebra $B$. We recall that if $(B^{\times},S)$ is a $\mathbb{Z}^2$-module, then $\Ext(\mathbb{Z}^2,B^{\times})_S\cong H^2(\mathbb{Z}^2,B^{\times})_S$ denotes the set of equivalence classes of all $B^{\times}$-extensions of $\mathbb{Z}^2$ for which the associated $\mathbb{Z}^2$-module structure on $B^{\times}$ is $S$. Moreover, we write 
\[\sigma_{\omega}:\mathbb{Z}^2\rightarrow B^{\times}\times_{\omega}\mathbb{Z}^2,\,\,\,({\bf k},{\bf l})\mapsto (1_B,({\bf k},{\bf l}))
\]for the section of the extension of $\mathbb{Z}^2$ by $B^{\times}$ corresponding to the cocycle $\omega\in Z^2(\mathbb{Z}^2,B^{\times})_S$.

\begin{proposition}\label{H^2(Z^n,B)_S}
Let $B$ be a commutative unital algebra and $(B^{\times},S)$ a $\mathbb{Z}^2$-module. Further, let ${\bf e}$ and ${\bf e'}$ be generators of $\mathbb{Z}^2$. Then the map
\[\beta:H^2(\mathbb{Z}^2,B^{\times})_S\rightarrow B^{\times}/C,\,\,\,[\omega]\mapsto b_{\omega}({\bf e},{\bf e'})+C,
\]is a well-defined isomorphism of abelian groups, where 
\[b_{\omega}({\bf e},{\bf e'}):=\sigma_{\omega}({\bf e})\sigma_{\omega}({\bf e'})\sigma_{\omega}({\bf e})^{-1}\sigma_{\omega}({\bf e'})^{-1}
\]and
\[C:=\langle b'(S({\bf e})b')^{-1}b^{-1}S({\bf e'})b:\,b,b'\in B^{\times}\rangle.
\]
\end{proposition}

\begin{proof}
\,\,\,We divide the proof of this proposition into four parts:

(i) We first show that the map $\beta$ is well defined: Therefore, we choose two cocycles $\omega$ and $\omega'$ with $[\omega]=[\omega']$. Then Theorem \ref{Orbit of factor systems} implies that there exists an element $h\in C^1(\mathbb{Z}^2,B^{\times})$ such that the map
\[\varphi:B^{\times}\times_{\omega'}\mathbb{Z}^2\rightarrow B^{\times}\times_{\omega}\mathbb{Z}^2,\,\,\,(b,{\bf k})\mapsto (bh({\bf k}),{\bf k})
\]is an equivalence of extensions. In particular we get $\varphi\circ\sigma_{\omega'}=h\cdot\sigma_{\omega}$, which leads to
\[b_{\omega'}({\bf e},{\bf e'})=b_{\omega}({\bf e},{\bf e'})h({\bf e'})(S({\bf e})h({\bf e'}))^{-1}(h({\bf e}))^{-1}S({\bf e'})h({\bf e}).
\]

(ii) Next, let $[\omega],[\omega']\in H^2(\mathbb{Z}^2,B^{\times})_S$. Then a short calculation leads to \[\beta([\omega+\omega'])=\beta([\omega])\beta([\omega']).
\]Hence, $\beta$ is a homomorphism of abelian groups.

(iii) Now, we pick $[\omega]\in H^2(\mathbb{Z}^2,B^{\times})_S$ with $\beta([\omega])=0$, i.e., 
\[b_{\omega}({\bf e},{\bf e'})=b'(S({\bf e})b')^{-1}b^{-1}S({\bf e'})b
\]for some $b,b'\in B^{\times}$, and consider the section $\sigma:\mathbb{Z}^2\rightarrow B^{\times}\times_{\omega}\mathbb{Z}^2$
given for $k,l\in\mathbb{Z}$ by 
\[\sigma({\bf e}^k+({\bf e'})^l):=(b^{-1}\sigma_{\omega}({\bf e}))^k ((b')^{-1}\sigma_{\omega}({\bf e'}))^l.
\]Then 
\[\sigma({\bf e})\sigma({\bf e'})=\sigma({\bf e'})\sigma({\bf e})
\]implies that $\sigma$ is a homomorphism of groups. In particular the extension
\[q:B^{\times}\times_{\omega}\mathbb{Z}^2,\rightarrow\mathbb{Z}^2,\,\,\,(b,(k,l))\mapsto (k,l)
\]splits. Thus, $[\omega]=0\in H^2(\mathbb{Z}^2,B^{\times})_S$, which means that the map $\beta$ is injective.

(iv) To show that $\beta$ is surjective, we associate to each $b_0\in B^{\times}$ a group extension $G_{b_0}$ of $\mathbb{Z}^2$ by $B^{\times}$ in the following way: We take $G_{b_0}$ to be the set of all symbols
\[b\diamond mu\diamond nv,\,\,\,\,b\in B^{\times},\,\,\,\,m,n\in\mathbb{Z},
\]with multiplication given by
\[u\diamond b={\bf e}.b\diamond u,\,\,\,\,v\diamond b={\bf e'}.b\diamond v,\,\,\,\,v\diamond u=b_0\diamond u\diamond v.
\]It is an easy exercise to show that this multiplication is always associative and makes the set of symbols a group. This completes the proof.
\end{proof}

\begin{remark}
Note that if $B$ is a commutative unital algebra and $B^{\times}$ is a trivial $\mathbb{Z}^2$-module, then the result of Proposition \ref{H^2(Z^n,B)} specializes to Proposition \ref{H^2(Z^n,B)_S}.
\end{remark}

\begin{remark}\label{vanishing of H^2}
Let $B$ be a commutative unital algebra.

(a) If $n\in\mathbb{N}$ and $B^{\times}$ is any $\mathbb{Z}^n$-module given by the group homomorphism 
\[S:\mathbb{Z}^n\rightarrow\Aut(B^{\times}),
\]then one can show that 
\[H^r(\mathbb{Z}^n,B^{\times})_S=0
\]for $r\in\mathbb{N}$ with $r>n$. Indeed, for this statement we refer to [Ma95], Chapter VI.6.

(b) Let $S$ be a $\mathbb{Z}$-module structure on $B$ given by the group homomorphism 
\[S:\mathbb{Z}\rightarrow\Aut(B).
\]According to Remark \ref{B abelian}, there exists up to equivalence only one algebraically trivial noncommutative principal $\mathbb{T}$-bundle with fixed point algebra $B$ corresponding to the $\mathbb{Z}$-module structure on $B$ given by $S$.

\end{remark}

\section{Examples of Algebraically Trivial NCP Torus Bundles}

We finally present a bunch of examples of algebraically trivial NCP $\mathbb{T}^n$-bundles:

\begin{example}{\bf(The classical picture).}
If $B=C^{\infty}(M)$ and $S={\bf 1}$, then there exists up to isomorphy only one algebraically trivial NCP $\mathbb{T}^n$-bundle for which $A$ is commutative. A possible realization is given by 
\[A:=\bigoplus_{{\bf k}\in\mathbb{Z}^n}C^{\infty}(M)\chi_{\bf k},
\]where $\chi_{\bf k}$ denotes the character of $\mathbb{T}^n$ corresponding to ${\bf k}\in\mathbb{Z}^n$. It is the algebraic skeleton of the smooth trivial NCP $\mathbb{T}^n$-bundle $(C^{\infty}(M\times\mathbb{T}^n),\mathbb{T}^n,\alpha)$, which is induced from the the trivial principal $\mathbb{T}^n$-bundle 
\[(M\times\mathbb{T}^n,M,\mathbb{T}^n,q_M,\sigma_{\mathbb{T}^n}).
\]This example perfectly reflects the commutative world, in which there exists up to isomorphy only one trivial principal $\mathbb{T}^n$-bundle over a given manifold $M$. On the other hand, it shows that the situation completely changes in the noncommutative world, where there exist plenty of algebraically trivial NCP $\mathbb{T}^n$-bundles with base $C^{\infty}(M)$ (cf. Proposition \ref{H^2(Z^n,B)}).
\end{example}

\begin{definition}\label{}{\bf(The algebraic noncommutative $n$-torus).}\index{Algebraic!Noncommutative $n$-Torus}
A $\mathbb{Z}^n$-graded unital associative algebra 
\[A=\bigoplus_{{\bf k}\in\mathbb{Z}^n}A_{\bf k}
\]is called an \emph{algebraic noncommutative $n$-torus}, if each grading space $A_{\bf k}$ is one-dimensional and each nonzero element of $A_{\bf k}$ is invertible. In [OP95], these algebras are called \emph{twisted group algebras}. We recall that algebraic noncommutative $n$-tori are the algebraic counterpart of Example \ref{NC n-tori as NCPTB}.
\end{definition}

\begin{example}
Let $n\in\mathbb{N}$ and $B=\mathbb{C}$. The algebraic noncommutative $n$-tori are exactly the algebraically trivial NCP $\mathbb{T}^n$-bundles with base $\mathbb{C}$ corresponding to the trivial $\mathbb{Z}^n$-module structure on $\mathbb{C}$. Indeed, if $A$ is an algebraic noncommutative $n$-torus, then a short observation shows that the short exact sequence
\[1\longrightarrow\mathbb{C}^{\times} \longrightarrow A^{\times}_h\longrightarrow\mathbb{Z}^n\longrightarrow 1
\]of groups is central. By Remark \ref{B abelian} and Proposition \ref{H^2(Z^n,B)}, these algebras are classified by
\[H^2(\mathbb{Z}^n,\mathbb{C}^{\times})\cong(\mathbb{C}^{\times})^{\frac{1}{2}n(n-1)}.
\]In particular, if $A^n_{\theta}$ is the noncommutative $n$-torus from Definition \ref{appendix E noncommutative n-tori} and $(A^n_{\theta},\mathbb{T}^n,\alpha)$ the corresponding trivial NCP $\mathbb{T}^n$-bundle from Example \ref{NC n-tori as NCPTB}, then the related characteristic class in $H^2(\mathbb{Z}^n,\mathbb{C}^{\times})$ is induced by the cocycle
\[\omega:\mathbb{Z}^n\times\mathbb{Z}^n\rightarrow\mathbb{C}^{\times},\,\,\,(k,l)\mapsto e^{-i\langle k,\theta l\rangle}.
\]
\end{example}

\begin{lemma}\label{aut ot quantumtori}
Let $A$ be an algebraic noncommutative $n$-torus. Then the following assertions hold:
\begin{itemize}
\item[\emph{(a)}]
Each unit of $A$ is graded, i.e., $A^{\times}=A^{\times}_h$.
\item[\emph{(b)}]
Each automorphism of $A$ is graded, i.e., $\Aut(A)=\Aut_{\emph{\text{gr}}}(A)$.
\end{itemize}
\end{lemma}

\begin{proof}
\,\,\,The proofs of these statements can be found in [Ne07b], Appendix A.
\end{proof}

\begin{proposition}
Let $n,m\in\mathbb{N}$ and $B:=B_f$ be the algebraic noncommutative $n$-torus corresponding to the cocycle $f\in Z^2(\mathbb{Z}^n,\mathbb{C}^{\times})$. Further, let $(S,\omega)\in Z^2(\mathbb{Z}^m,B)$ be a factor system. Then $A_{(S,\omega)}$ is an algebraic noncommutative $(m+n)$-torus if and only if $S$ leaves the grading spaces of $B$ invariant and the cocycle $\omega$ takes values in $B^{\times}_{\bf 0}\cong\mathbb{C}^{\times}$.
\end{proposition}

\begin{proof}
\,\,\,Let 
\[\bigoplus_{{\bf k}\in\mathbb{Z}^n}\mathbb{C}v_{\bf k},\,\,\,\text{resp.},\,\,\,\bigoplus_{{\bf l}\in\mathbb{Z}^m}B w_{\bf l}
\]be the underlying vector space of $B$, resp., of $A$. 

($``\Leftarrow"$) We first write $c_{{\bf l},{\bf k}}\in\mathbb{C}^{\times}$ for the constant satisfying $S({\bf l})v_{\bf k}=c_{{\bf l},{\bf k}}v_{\bf k}$. Since the cocycle $\omega$ takes values in $\mathbb{C}^{\times}$, the map $S$ is a group homomorphism and thus the calculation
\begin{align}
(v_{\bf k}w_{\bf l})(v_{\bf k'}w_{\bf l'})&=v_{\bf k}S({\bf l})v_{{\bf k'}}\omega({\bf l},{\bf l'})w_{{\bf l+l'}}=c_{{\bf l},{\bf k'}}\omega({\bf l},{\bf l'})v_{\bf k}v_{\bf k'}w_{\bf l+l'}\notag\\
&=c_{{\bf l},{\bf k'}}\omega({\bf l},{\bf l'})f({\bf k},{\bf k'})v_{\bf k+k'}w_{{\bf l+l'}}\notag
\end{align}
shows that $A_{(S,\omega)}$ is an algebraic noncommutative $(m+n)$-torus corresponding to the cocycle
\[f':\mathbb{Z}^{m+n}\times\mathbb{Z}^{m+n}\rightarrow\mathbb{C}^{\times},\,\,\,(({\bf k},{\bf l}),({\bf k'},{\bf l'}))\mapsto c_{{\bf l},{\bf k'}}\omega({\bf l},{\bf l'})f({\bf k},{\bf k'}).
\]

($``\Rightarrow"$) Conversely, if $A_{(S,\omega)}$ is an algebraic noncommutative $(m+n)$-torus, then
\[(v_{\bf k}w_{\bf l})w_{\bf l'}=v_{\bf k}\omega({\bf l},{\bf l'})w_{\bf l+l'}\in A_{{\bf k},{\bf l+l'}}
\]and Lemma \ref{aut ot quantumtori} (a) imply that $\omega({\bf l},{\bf l'})\in B^{\times}_{\bf 0}\cong\mathbb{C}^{\times}$ for all ${\bf l},{\bf l'}\in\mathbb{Z}^m$. Moreover,
\[w_{\bf l'}(v_{\bf k}w_{\bf l})=S({\bf l'})v_{\bf k}\omega({\bf l'},{\bf l})w_{\bf l+l'}=\omega({\bf l'},{\bf l})S({\bf l'})v_{\bf k}w_{\bf l+l'}\in A_{{\bf k},{\bf l+l'}}
\]and Lemma \ref{aut ot quantumtori} (b) imply that the map $S$ must leave each grading spaces of $B$ invariant.
\end{proof}

\begin{lemma}\label{aut=inn}
Each automorphism of the matrix algebra $\M_n(\mathbb{C})$ is inner.
\end{lemma}

\begin{proof}
\,\,\,This is a corollary of the well-known Skolem-Noether Theorem.
\end{proof}

\begin{example}
Let $B=\M_m(\mathbb{C})$. In view of Lemma \ref{aut=inn}, each outer action of $\mathbb{Z}^n$ on $\M_m(\mathbb{C})$ is equivalent to $S={\bf 1}$. In particular, $\Ext(\mathbb{Z}^n,B)_{[S]}\neq\emptyset$ and Corollary \ref{class of Z-kernels II} implies that the equivalence classes of algebraically trivial NCP $\mathbb{T}^n$-bundles with base $\M_m(\mathbb{C})$ are classified by 
\[H^2(\mathbb{Z}^n,\mathbb{C}^{\times})\cong(\mathbb{C}^{\times})^{\frac{1}{2}n(n-1)}.
\]
\end{example}

\begin{example}{\bf(Direct sums).}
Let $A$ and $A'$ be two algebraically trivial NCP $\mathbb{T}^n$-bundles with base $B$ and $B'$, respectively. Then the direct sum $A\oplus A'$ is an algebraically trivial NCP $\mathbb{T}^n$-bundle with base $B\oplus B'$.
\begin{flushright}
$\blacksquare$
\end{flushright}
\end{example}

\begin{example}{\bf(Tensor products).}
Let $A$ be an algebraically trivial NCP $\mathbb{T}^n$-bundle with base $B$ and $A'$ an algebraically trivial NCP $\mathbb{T}^m$-bundle with base $B'$.
Then their tensor product $A\otimes A'$ is an algebraically trivial NCP $\mathbb{T}^{n+m}$-bundle with base $B\otimes B'$.
\begin{flushright}
$\blacksquare$
\end{flushright}
\end{example}

\chapter{Trivial NCP Cyclic Bundles}\label{trivial ncp cyclic bundles}

For each $n\in\mathbb{N}$ we write 
\[C_n:=\{z\in\mathbb{C}^{\times}:\,z^n=1\}=\{\zeta^k:\,\zeta:=\exp(\frac{2\pi i}{n}),\,k=0,1,\ldots,n-1\}
\]\sindex[n]{$C_n$}for the cyclic subgroup of $\mathbb{T}$ of $n$-th roots of unity. The goal of this chapter is to present a new, geometrically oriented approach to the noncommutative geometry of trivial principal $C_n$-bundles based on dynamical systems of the form $(A,C_n,\alpha)$. Our main idea is to characterize the functions appearing in a trivialization of a trivial principal $C_n$-bundle $(P,M,C_n,q,\sigma)$ as elements in $C^{\infty}(P)$. This approach leads to the following concept: Given a unital locally convex algebra $A$, a dynamical system $(A,C_n,\alpha)$ is called a trivial NCP $C_n$-bundle if the isotypic component corresponding to $\zeta$ contains an invertible element $a$ with the additional property that there exists an element $b$ in the corresponding fixed point algebra with $b^n=a^n$. Moreover, we also treat the case of products of finite cyclic groups (i.e. finite abelian groups) and present some examples.\sindex[n]{$(A,C_n,\alpha)$}\sindex[n]{$(P,M,C_n,q,\sigma)$}

\section{The Concept of a Trivial NCP Cyclic Bundle}

In this section we present a geometrically oriented approach to the noncommutative geometry of trivial principal $C_n$-bundles based on dynamical systems of the form $(A,C_n,\alpha)$. We will in particular see that this approach reproduces the classical geometry of trivial principal $C_n$-bundles: If $A=C^{\infty}(P)$ for some manifold $P$, then we recover a trivial principal $C_n$-bundle and, conversely, each trivial principal bundle $(P,M,C_n,q,\sigma)$ gives rise to a trivial NCP $C_n$-bundle of the form $(C^{\infty}(P),C_n,\alpha)$. 

\begin{notation} 
We recall that the map
\[\Psi:C_n\rightarrow\Hom(C_n,\mathbb{T}),\,\,\,\Psi(\zeta^k)(\zeta):=\zeta^k
\]is an isomorphism of abelian groups. In the following we will identify the character group of $C_n$ with $C_n$ via the isomorphism $\Psi$. In particular, if $A$ is a unital locally convex algebra and $(A,C_n,\alpha)$ a dynamical system, then we write
\[A_k:=A_{\Psi(\zeta^k)}=\{a\in A:\alpha(\zeta).a=\Psi(\zeta^k)(\zeta)\cdot a=\zeta^k\cdot a\}
\]for the isotypic component corresponding to $\zeta^k$, $k=0,1,\ldots,n-1$.
\end{notation}

\begin{proposition}\label{C_n I}{\bf(The freeness condition for $C_n$).}\index{Freeness!Condition for $C_n$}
Let $A$ be a commutative unital locally convex algebra. A dynamical system $(A,C_n,\alpha)$ is free in the sense of Definition \ref{free dynamical systems} if the maps
\[\ev^{1}_{\chi}: A_1\rightarrow \mathbb{C},\,\,\,a\mapsto\chi(a)
\]is surjective for all $\chi\in\Gamma_A$.
\end{proposition}

\begin{proof}
\,\,\,Since $C_n$ is cyclic and $\zeta$ is a generator, the assertion is a consequence of Corollary \ref{freeness for compact abelian groups I,5} (b).
\end{proof}

\begin{proposition}\label{C_n II}{\bf(Invertible elements in isotypic components).}
Let $A$ be a commutative unital locally convex algebra and $(A,C_n,\alpha)$ a dynamical system. If the isotypic component $A_1$ contains an invertible element, then the dynamical system $(A,C_n,\alpha)$ is free.
\end{proposition}

\begin{proof}
\,\,\,The assertion easily follows from Proposition \ref{C_n I}. Indeed, if $a\in A_1$ is invertible, then $\chi(a)\neq 0$ for all $\chi\in\Gamma_A$.
\end{proof}


We will now introduce a (reasonable) definition of trivial NCP $C_n$-bundles:

\begin{definition}{\bf(Trivial NCP $C_n$-bundles).}\label{def C_n-bundles}\index{Trivial NCP!$C_n$-Bundles}
Let $A$ be a unital locally convex algebra. A (smooth) dynamical system $(A,C_n,\alpha)$ is called a (\emph{smooth}) \emph{trivial NCP $C_n$-bundle}, if the following two conditions are satisfied:
\begin{itemize}
\item[(C1)]
The isotypic component $A_1$ contains an invertible element (let's say) $a$.
\item[(C2)]
There exists an element $b$ in the fixed point algebra $B=A^{C_n}$ with $b^n=a^n$ (note that $a^n\in B$).
\end{itemize} 
\end{definition}

\begin{lemma}\label{C_n III}
If $P$ is a manifold and $(C^{\infty}(P),C_n,\alpha)$ a smooth trivial NCP $C_n$-bundle, then the induced smooth action map $\sigma$ of Proposition \ref{smoothness of the group action on the set of characters} is free and proper. In particular, we obtain a principal bundle $(P,P/C_n,C_n,\pr,\sigma)$.
\end{lemma}

\begin{proof}
\,\,\,The freeness of the map $\sigma$ follows directly from Definition \ref{def C_n-bundles}, Proposition \ref{C_n II} and Theorem \ref{freeness of induced action}. Its properness is automatic, since $C_n$ is compact. The last claim now follows from the Quotient Theorem.
\end{proof}

Note that we didn't use condition (C2) yet. We will now see that this condition guarantees that the principal $C_n$-bundle of Lemma \ref{C_n III} becomes trivial. 



\begin{theorem}\label{C_n IV}{\bf(Trivial principal $C_n$-bundles).}\index{Trivial Principal!$C_n$-Bundles}
If $P$ is a manifold, then the following assertions hold:
\begin{itemize}
\item[\emph{(a)}]
If $(C^{\infty}(P),C_n,\alpha)$ is a smooth trivial NCP $C_n$-bundle, then the corresponding principal bundle $(P,P/C_n,C_n,\pr,\sigma)$ of Lemma \ref{C_n III} is trivial.
\item[\emph{(b)}]
Conversely, if $(P,M,C_n,q,\sigma)$ is a trivial principal $C_n$-bundle, then the corresponding smooth dynamical system $(C^{\infty}(P),C_n,\alpha)$ of Remark \ref{induced transformation triples} is a smooth trivial NCP $C_n$-bundle.
\end{itemize}
\end{theorem}

\begin{proof}
\,\,\,(a) We first choose an invertible element $f\in C^{\infty}(P)_1$ and an element $g\in C^{\infty}(P)_0$ such that $g^n=f^n$. Since $f$ is invertible, so is $g$ and a short observation shows that $h:=\frac{f}{g}$ lies in $C^{\infty}(P)_1$ and has image $C_n$. Indeed, by construction we obtain $h\in C^{\infty}(P)_1$. Moreover, this and $h^n=1$ implies that $\im(h)=C_n$. In particular, the map
\[\varphi:P\rightarrow P/C_n\times C_n,\,\,\,p\mapsto(\pr(p),h(p))
\]defines an equivalence of principal $C_n$-bundles over $P/C_n$. We thus conclude that the principal bundle $(P,P/C_n,C_n,\pr,\sigma)$ of Lemma \ref{C_n III} is trivial.

(b) Conversely, let $(P,M,C_n,q,\sigma)$ be a trivial principal $C_n$-bundle and
\[\varphi:P\rightarrow M\times C_n,\,\,\,p\mapsto(q(p),f(p))
\]be an equivalence of principal $C_n$-bundles over $M$. We first note that the function $f\in C^{\infty}(P)$ is invertible. Furthermore, the $C_n$-equivariance of $\varphi$ implies that $f\in C^{\infty}(P)_1$. Thus, $f\in C^{\infty}(P)_1$ is invertible with $f^n=1$. From this we conclude that $(C^{\infty}(P),C_n,\alpha)$ is a smooth trivial NCP $C_n$-bundle.
\end{proof}

\section{(Two-Fold) Products of Cyclic Groups}

In this section we present a geometrically oriented approach to the noncommutative geometry of trivial principal $C_n\times C_m$-bundles based on dynamical systems of the form $(A,C_n\times C_m,\alpha)$. Again, we will see that this approach reproduces the classical geometry of trivial principal $C_n\times C_m$-bundles: If $A=C^{\infty}(P)$ for some manifold $P$, then we recover a trivial principal $C_n\times C_m$-bundle and, conversely, each trivial principal bundle $(P,M,C_n\times C_m,q,\sigma)$ gives rise to a trivial NCP $C_n\times C_m$-bundle of the form $(C^{\infty}(P),C_n\times C_m,\alpha)$.\sindex[n]{$(A,C_n\times C_m,\alpha)$}\sindex[n]{$(P,M,C_n\times C_m,q,\sigma)$}

\begin{notation} 
For $n,m\in\mathbb{N}$ we write 
\[C_n:=\{z\in\mathbb{C}^{\times}:\,z^n=1\}=\{\zeta_1^k:\,\zeta_1:=\exp(\frac{2\pi i}{n}),\,k=0,1,\ldots,n-1\}
\]and
\[C_m:=\{z\in\mathbb{C}^{\times}:\,z^m=1\}=\{\zeta_2^l:\,\zeta_2:=\exp(\frac{2\pi i}{m}),\,l=0,1,\ldots,m-1\}.
\]We recall that the map
\begin{align}
\Psi:C_n\times C_m\rightarrow\Hom(C_n\times C_m,\mathbb{T}),\,\,\,\Psi(\zeta_1^k,\zeta_2^l)(\zeta_1,1):=\zeta_1^k,\,\,\,\Psi(\zeta_1^k,\zeta_2^l)(1,\zeta_2)=\zeta_2^l\label{formula C_nC_m}
\end{align}
is an isomorphism of abelian groups. In the following we will identify the character group of $C_n\times C_m$ with $C_n\times C_m$ via the isomorphism $\Psi$. In particular, if $A$ is a unital locally convex algebra and $(A,C_n\times C_m,\alpha)$ a dynamical system, then we write
\[A_{(k,l)}:=A_{\Psi(\zeta_1^k,\zeta_2^l)}=\{a\in A:\alpha(\zeta_1,1).a=\zeta_1^k\cdot a,\,\,\,\alpha(1,\zeta_2).a=\zeta_2^l\cdot a\}
\]for the isotypic component corresponding to the element $(\zeta^k,\xi^l)$ for $k=0,1,\ldots,n-1$ and $l=0,1,\ldots,m-1$.
\end{notation}

\begin{proposition}\label{C_nC_m I}{\bf(The freeness condition for $C_n\times C_m$).}\index{Freeness!Condition for $C_n\times C_m$}
Let $A$ be a commutative unital locally convex algebra. A dynamical system $(A,C_n\times C_m,\alpha)$ is free in the sense of Definition \ref{free dynamical systems} if the maps
\[\ev^{(1,0)}_{\chi}: A_{(1,0)}\rightarrow \mathbb{C},\,\,\,a\mapsto\chi(a)
\]and
\[\ev^{(0,1)}_{\chi}: A_{(0,1)}\rightarrow \mathbb{C},\,\,\,a\mapsto\chi(a)
\]are surjective for all $\chi\in\Gamma_A$.
\end{proposition}

\begin{proof}
\,\,\,Since $C_n$ and $C_m$ are both cyclic groups with generators $\zeta_1$ and $\zeta_2$, the assertion is a consequence of Corollary \ref{freeness for compact abelian groups I,5} (b).
\end{proof}

\begin{proposition}\label{C_nC_m II}{\bf(Invertible elements in isotypic components).}
Let $A$ be a commutative unital locally convex algebra and $(A,C_n\times C_m,\alpha)$ a dynamical system. If the isotypic components $A_{(1,0)}$ and $A_{(0,1)}$ contain invertible elements, then the dynamical system $(A,C_n\times C_m,\alpha)$ is free.
\end{proposition}

\begin{proof}
\,\,\,The assertion easily follows from Proposition \ref{C_nC_m I}. Indeed, if $a\in A_{(1,0)}$ is invertible, then $\chi(a)\neq 0$ for all $\chi\in\Gamma_A$ and similarly for $A_{(0,1)}$.
\end{proof}

We will now introduce a (reasonable) definition of trivial NCP $C_n\times C_m$-bundles:

\begin{definition}{\bf(Trivial NCP $C_n\times C_m$-bundles).}\label{def C_nC_m-bundles}\index{Trivial NCP!$C_n\times C_m$-Bundles}
Let $A$ be a unital locally convex algebra. A (smooth) dynamical system $(A,C_n\times C_m,\alpha)$ is called a (\emph{smooth}) \emph{trivial NCP $C_n\times C_m$-bundle}, if the following conditions are satisfied:
\begin{itemize}
\item[(C1)]
There exists an invertible element $a\in A_{(1,0)}$ and an invertible element $a'\in A_{(0,1)}$.
\item[(C2)]
There exist elements $b$ and $b'$ in the fixed point algebra $B=A^{C_n\times C_m}$ such that $b^n=a^n$ and $(b')^m=(a')^m$.
\end{itemize}
\end{definition}

\begin{lemma}\label{C_nC_m III}
If $P$ is a manifold and $(C^{\infty}(P),C_n\times C_m,\alpha)$ a smooth trivial NCP $C_n\times C_m$-bundle, then the induced smooth action map $\sigma$ of Proposition \ref{smoothness of the group action on the set of characters} is free and proper. In particular, we obtain a principal bundle $(P,P/(C_n\times C_m),C_n\times C_m,\pr,\sigma)$.
\end{lemma}

\begin{proof}
\,\,\,The freeness of the map $\sigma$ follows directly from Definition \ref{def C_nC_m-bundles}, Proposition \ref{C_nC_m II} and Theorem \ref{freeness of induced action}. Its properness is automatic, since $C_n\times C_m$ is compact. The last claim now follows from the Quotient Theorem.
\end{proof}

Note that we didn't use condition $(C2)$ yet. We will now see that this condition guarantees that the principal $C_n\times C_m$-bundle of Lemma \ref{C_nC_m III} becomes trivial. 

\begin{theorem}\label{C_nC_m IV}{\bf(Trivial principal $C_n\times C_m$-bundles).}\index{Trivial Principal!$C_n\times C_m$-Bundles}
If $P$ is a manifold, then the following assertions hold:
\begin{itemize}
\item[\emph{(a)}]
If $(C^{\infty}(P),C_n\times C_m,\alpha)$ is a smooth trivial NCP $C_n\times C_m$-bundle, then the corresponding principal bundle $(P,P/(C_n\times C_m),C_n\times C_m,\pr,\sigma)$ of Lemma \ref{C_nC_m III} is trivial.
\item[\emph{(b)}]
Conversely, if $(P,M,C_n\times C_m,q,\sigma)$ is a trivial principal $C_n\times C_m$-bundle, then the corresponding smooth dynamical system $(C^{\infty}(P),C_n\times C_m,\alpha)$ of Remark \ref{induced transformation triples} is a smooth trivial NCP $C_n\times C_m$-bundle.
\end{itemize}
\end{theorem}

\begin{proof}
\,\,\,(a) We first choose an invertible element $f\in C^{\infty}(P)_{(1,0)}$ and an invertible element $f'\in C^{\infty}(P)_{(0,1)}$. We further choose two elements $g$ and $g'$ in $C^{\infty}(P)_{(0,0)}$ such that $g^n=f^n$ and $(g')^m=(f')^m$. Since $f$ and $f'$ are invertible, so are $g$ and $g'$. Now, a short observation shows that $h:=\frac{f}{g}$ lies in $C^{\infty}(P)_{(1,0)}$ and has image $C_n$. Moreover, $h':=\frac{f'}{g'}$ lies in $C^{\infty}(P)_{(0,1)}$ and has image $C_m$. Indeed, by construction we obtain $h\in C^{\infty}(P)_{(1,0)}$. This and $h^n=1$ implies that $\im(h)=C_n$. A similar argument works for $h'$. In particular, the map
\[\varphi:P\rightarrow P/(C_n\times C_m)\times (C_n\times C_m),\,\,\,p\mapsto(\pr(p),h(p),h'(p))
\]defines an equivalence of principal $C_n\times C_m$-bundles over $P/(C_n\times C_m)$. We thus conclude that the principal bundle $(P,P/(C_n\times C_m),C_n\times C_m,\pr,\sigma)$ of Lemma \ref{C_nC_m III} is trivial.

(b) Conversely, let $(P,M,C_n\times C_m,q,\sigma)$ be a trivial principal $C_n\times C_m$-bundle and
\[\varphi:P\rightarrow M\times(C_n\times C_m),\,\,\,p\mapsto(q(p),f(p),f'(p))
\]be an equivalence of principal $C_n\times C_m$-bundles over $M$. We first note that both functions $f$ and $f'$ are invertible elements of $C^{\infty}(P)$. Furthermore, the $C_n\times C_m$-equivariance of $\varphi$ implies that $f\in C^{\infty}(P)_{(1,0)}$ and $f'\in C^{\infty}(P)_{(0,1)}$. Thus, $f\in C^{\infty}(P)_{(1,0)}$ is invertible with $f^n=1$ and $f'\in C^{\infty}(P)_{(0,1)}$ is invertible with $(f')^m=1$. From this we conclude that $(C^{\infty}(P),C_n\times C_m,\alpha)$ is a smooth trivial NCP $C_n\times C_m$-bundle.
\end{proof}

\subsection*{Finite Abelian Groups}

In view of the previous constructions, it is clear how to define (smooth) trivial NCP $\Lambda$-bundles for a finite abelian group $\Lambda$. In fact, the Structure Theorem for finitely generated abelian groups implies that $\Lambda$ is isomorphic to a finite product of cyclic groups. Thus, up to an isomorphism, we may assume that $\Lambda=C_{n_1}\times\cdots\times C_{n_k}$. Here, $n_i\in\mathbb{N}$ and
\[C_{n_i}:=\{z\in\mathbb{C}^{\times}:\,z^{n_i}=1\}=\{\zeta_i^l:\,\zeta_i:=\exp(\frac{2\pi i}{n_i}),\,l=0,1,\ldots,n_i-1\}.
\]for each $i=1,\ldots,k$. A short observation shows (right as before) that the character group of $\Lambda$ is canonically isomorphic to $\Lambda$ (as abelian group) and we write $\Psi$ for the corresponding isomorphism (cf. (\ref{formula C_nC_m})). 
Moreover, we write $A_i$ for the isotypic component corresponding to the generator $\zeta_i$, i.e., $A_i:=A_{\Psi(\zeta_i)}$ understood in the appropriate sense. 

\begin{definition}\label{C_nC_m V}{\bf(Trivial NCP $\Lambda$-bundles).}\index{Trivial NCP!$\Lambda$-Bundles}
Let $A$ be a unital locally convex algebra and $\Lambda$ a finite abelian group. A (smooth) dynamical system $(A,\Lambda,\alpha)$ is called a (\emph{smooth}) \emph{trivial NCP $\Lambda$-bundle}, if the following conditions are satisfied:
\begin{itemize}
\item[(C1)]
For each $i=1,\ldots,k$ there exists an invertible element $a_i\in A_i$. 
\item[(C2)]
For each $i=1,\ldots,k$ there exists an elements $b_i$ in the fixed point algebra $B$ such that $b_i^{n_i}=a_i^{n_i}$.
\end{itemize}
\end{definition}

\begin{remark}\label{C_nC_m VI}{\bf(Trivial principal $\Lambda$-bundles).}\index{Trivial Principal!$\Lambda$-Bundles}
A theorem characterizing the trivial principal $\Lambda$-bundles can be obtained similarly to Theorem \ref{C_nC_m IV}.
\end{remark}

\section{Some Examples}

In this section we present a bunch of examples of trivial NCP $\Lambda$-bundles for a finite abelian group $\Lambda$. In fact, all constructions of Section \ref{examples of trivial NCP torus bundles} can be applied to the group $\Lambda$. For this reason we will leave out most of the proofs. For detailed proofs we refer to the corresponding results in Section \ref{examples of trivial NCP torus bundles}.

\begin{construction}\label{l^1 spaces associated to dynamical systems again I}{\bf($\ell^1$-crossed products again).}\index{Crossed Products!$\ell^1$-}
Let $(A,\Vert\cdot\Vert,^{*})$ be an involutive Banach algebra and $(A,\Lambda,\alpha)$ a dynamical system. We recall that this means that $\Lambda$ acts by isometries of $A$. We write $F(\Lambda,A)$\sindex[n]{$F(\Lambda,A)$} for the vector space of functions $f:\Lambda\rightarrow A$ and define a multiplication, an involution and a norm on this space as in Construction \ref{l^1 spaces associated to dynamical systems I}. For consistency we write $\ell^1(A\rtimes_{\alpha} \Lambda)$.\sindex[n]{$\ell^1(A\rtimes_{\alpha} \Lambda)$}
for the corresponding involutive Banach algebra.
\begin{flushright}
$\blacksquare$
\end{flushright}
\end{construction}

\begin{lemma}\label{continuous action again}
If $(A,\Vert\cdot\Vert,^{*})$ is an involutive Banach algebra and $(A,\Lambda,\alpha)$ a dynamical system, then the map
\[\widehat{\alpha}:\Lambda\times\ell^1(A\rtimes_{\alpha} \Lambda)\rightarrow\ell^1(A\rtimes_{\alpha} \Lambda),\,\,\,(\widehat{\alpha}(g',f))(g):=(g'.f)(g):=(\Psi(g).(g'))\cdot f(g)
\]defines a continuous action of $\Lambda$ on $\ell^1(A\rtimes_{\alpha} \Lambda)$ by algebra automorphisms.
\begin{flushright}
$\blacksquare$
\end{flushright}
\end{lemma}

\begin{proposition}
If $(A,\Vert\cdot\Vert,^{*})$ is an involutive Banach algebra and $(A,\Lambda,\alpha)$ a dynamical system, then the triple \[(\ell^1(A\rtimes_{\alpha} \Lambda),\Lambda,\widehat{\alpha})
\]defines a dynamical system.\sindex[n]{$(\ell^1(A\rtimes_{\alpha} \Lambda),\Lambda,\widehat{\alpha})$}
\end{proposition}

\begin{proof}
\,\,\,The claim is a direct consequence of Lemma \ref{continuous action again}.
\end{proof}

\begin{example}\label{l^1 as example again}
If $(A,\Vert\cdot\Vert,^{*})$ is an involutive Banach algebra and $(A,\Lambda,\alpha)$ a dynamical system, then the dynamical system $(\ell^1(A\rtimes_{\alpha} \Lambda),\Lambda,\widehat{\alpha})$ is a trivial NCP $\Lambda$-bundle. Indeed, for $g\in \Lambda$ we define
\[\delta_g(h):=
\begin{cases}
1_A &\text{for}\,\,\,h=g\\
0 &\text{otherwise}.
\end{cases}
\]In particular, we write 
\[\delta_i(h):=
\begin{cases}
1_A &\text{for}\,\,\,h=\zeta_i\\
0 &\text{otherwise}
\end{cases}.
\]Then 
\[g'.\delta_i=(\Psi(\zeta_i).(g'))\cdot\delta_i,\,\,\text{and}\,\,\,\delta_i\star\delta_i^{*}=\delta_i^{*}\star\delta_i={\bf 1}
\]show that $\delta_i$ is an invertible element of $\ell^1(A\rtimes_{\alpha} \Lambda)$ lying in the isotypic component $\ell^1(A\rtimes_{\alpha} \Lambda)_i$. From this observation and $\delta_i^n={\bf 1}$ we conclude that $\ell^1(A\rtimes_{\alpha}\Lambda)$ is a trivial NCP $\Lambda$-bundle.
\end{example}

\begin{definition}{\bf($C^{*}$-crossed products again).}\index{Crossed Products!$C^{*}$-}
If $(A,\Vert\cdot\Vert,^{*})$ is an involutive Banach algebra and $(A,\Lambda,\alpha)$ a dynamical system, then the enveloping $C^*$-algebra of $\ell^1(A\rtimes_{\alpha} \Lambda)$ is denoted by $C^*(A\rtimes_{\alpha} \Lambda)$\sindex[n]{$C^*(A\rtimes_{\alpha} \Lambda)$} and is called the \emph{$C^{*}$-crossed product} associated to $(A,\Lambda,\alpha)$.
\end{definition}

\begin{example}\label{crossed product as example again}
If $(A,\Vert\cdot\Vert,^{*})$ is an involutive Banach algebra and $(A,\Lambda,\alpha)$ a dynamical system, then the action $\widehat{\alpha}$ of Lemma \ref{continuous action again} extends to a continuous action of $\Lambda$ on the $C^*$-crossed product $C^*(A\rtimes_{\alpha} \Lambda)$ by algebra automorphisms. For details we refer to the paper [Ta74]. In particular, the corresponding dynamical system 
\[(C^*(A\rtimes_{\alpha} \Lambda),\Lambda,\widehat{\alpha})
\]is a trivial NCP $\Lambda$-bundle. This follows exactly as in Example \ref{l^1 as example again}.\sindex[n]{$(C^*(A\rtimes_{\alpha} \Lambda),\Lambda,\widehat{\alpha})$}
\end{example}

\begin{construction}\label{l^1 spaces associated to cocycles again I}{\bf($\ell^1$-spaces associated to $2$-cocycles again).}
Let $(A,\Vert\cdot\Vert,^{*})$ be an involutive Banach algebra and $\omega\in Z^2(\Lambda,\mathbb{T})$. The involutive Banach algebra $\ell^1(A\times_{\omega}\Lambda)$ is defined exactly as in Construction \ref{l^1 spaces associated to cocycles I}.\sindex[n]{$\ell^1(A\times_{\omega}\Lambda)$}.
\begin{flushright}
$\blacksquare$
\end{flushright}
\end{construction}

\begin{example}\label{l^1 twisted as example again}
Let $(A,\Vert\cdot\Vert,^{*})$ be an involutive Banach algebra and $\omega\in Z^2(\Lambda,\mathbb{T})$.\sindex[n]{$Z^2(\Lambda,\mathbb{T})$} Similarly to Lemma \ref{continuous action again}, we see that the map
\[\widehat{\alpha}:\Lambda\times\ell^1(A\times_{\omega}\Lambda)\rightarrow\ell^1(A\times_{\omega}\Lambda),\,\,\,(\widehat{\alpha}(g',f))(g):=(g'.f)(g):=(\Psi(g).(g'))\cdot f(g)
\]defines a continuous action of $\Lambda$ on $\ell^1(A\times_{\omega} \Lambda)$ by algebra automorphisms. Moreover, the corresponding dynamical system
\[(\ell^1(A\times_{\omega} \Lambda),\Lambda,\widehat{\alpha})
\]\sindex[n]{$(\ell^1(A\times_{\omega} \Lambda),\Lambda,\widehat{\alpha})$}turns out to be a trivial NCP $\Lambda$-bundle (cf. Example \ref{l^1 as example again}).
\end{example}

\begin{example}
The enveloping $C^*$-algebra of $\ell^1(A\times_{\omega} \Lambda)$ is denoted by $C^*(A\times_{\omega} \Lambda)$\sindex[n]{$C^*(A\times_{\omega} \Lambda)$} and is called the \emph{twisted group $C^*$-algebra of $\Lambda$ by $\omega$}. The action $\widehat{\alpha}$ of Example \ref{l^1 twisted as example again} extends to a continuous action of $\Lambda$ on $C^*(A\times_{\omega} \Lambda)$ by algebra automorphisms (cf. Example \ref{crossed product as example again}). The corresponding dynamical system
\[(C^*(A\times_{\omega} \Lambda),\Lambda,\widehat{\alpha})
\]\sindex[n]{$(C^*(A\times_{\omega} \Lambda),\Lambda,\widehat{\alpha})$}is a trivial NCP $\Lambda$-bundle as well.
\end{example}

\begin{example}{\bf(The matrix algebra).}\label{the matrix algebra}\index{Algebra!Matrix}
For $m\in\mathbb{N}$ and $\zeta:=\exp(\frac{2\pi i}{m})$ we define
$$
R:=\begin{pmatrix}
 1      &  &  &  &    \\
         & \zeta &  &  &  \\
         &  & \zeta^2  &  &  \\
         &  &  & \ddots &  \\
         &  &  &  &  \zeta^{m-1} 
\end{pmatrix}
$$

and

$$
S:=\begin{pmatrix}
 0      & \ldots & \ldots & 0 & 1   \\
 1      & 0 & \ldots & 0 &  0\\
         & \ddots & \ddots & \vdots & \vdots \\
         &  & \ddots& 0 & \vdots \\
 0      &  &  & 1 & 0 
\end{pmatrix}.
$$
We note that the matrices $R$ and $S$ are unitary and satisfy the relations $R^m={\bf 1}$, $S^m={\bf 1}$ and $RS=\zeta\cdot SR$. Further, they generate a $C^*$-subalgebra which clearly commutes only with the multiples of the identity, so it has to be the full matrix algebra. The map
\[\alpha:(C_m\times C_m)\times M_m(\mathbb{C})\rightarrow M_m(\mathbb{C}),\,\,\, (\zeta^k,\zeta^l).A:=R^lS^kAS^{-k}R^{-l}
\]for $k,l=0,1,\ldots,m-1$ defines a smooth group action of $C_m\times C_m$ on $M_m(\mathbb{C})$ with fixed point algebra $\mathbb{C}$. Since $R^*=R^{-1}\in M_m(\mathbb{C})_{(1,0)}$ and $S\in M_m(\mathbb{C})_{(0,1)}$, we conclude that the triple $(M_m(\mathbb{C}),C_m\times C_m,\alpha)$ is a smooth trivial NCP $C_m\times C_m$-bundle.
\end{example}

\section{A Remark on the Classifications of Trivial NCP Cyclic Bundles}

In Section \ref{Classification of NCPT^nB} we provided a classification of NCP $\mathbb{T}^n$-bundles based on the extension theory of groups. The same framework applied to a $k$-fold product $\Lambda=C_{n_1}\times\cdots\times C_{n_k}$ of cyclic groups leads to a classification of $\Lambda$-graded unital associative algebras for which each grading space contains an invertible element. Unfortunately, in order to classify NCP $\Lambda$-bundles, we still have to handle Condition (C2) of Remark \ref{C_nC_m V}. This gives rise to the following open problem:

\begin{open problem}{\bf(Automatic existence of roots).}\label{automatic existence of roots}\index{Automatic!Existence of Roots}
Suppose we are in the situation of Remark \ref{C_nC_m V}. Is it (always) possible to choose an invertible element $a_i\in A_i$ such that Condition (C2) is automatic? If not, for which classes of algebras is Condition (C2) automatic? We shall point out that Condition (C2) is actually a condition on the fixed point algebra $B$.
\end{open problem}

\begin{remark}
If the answer to the Open Problem \ref{automatic existence of roots} is yes, or if (at least) there exists an interesting class of algebras for which Condition (C2) is automatic, then we can apply the framework of Section \ref{Classification of NCPT^nB} to give a complete classification of the algebraic counterpart of trivial NCP $\Lambda$-bundles, i.e., of $\Lambda$-graded unital associative algebras for which each grading space contains an invertible element.
\end{remark}

\begin{example}{\bf($B=\mathbb{C}$).}
In the algebra $\mathbb{C}$ it is always possible to find (arbitrary) roots. Thus, the set $\Ext(\Lambda,\mathbb{C})$\sindex[n]{$\Ext(\Lambda,\mathbb{C})$} of all equivalence classes of algebraically trivial NCP $\Lambda$-bundles with base $\mathbb{C}$ is in one-to-one correspondence to the second group cohomology space $H^2(\Lambda,\mathbb{C}^{\times})$. Indeed, this follows from Remark \ref{B abelian} and the fact that each automorphism of $\mathbb{C}$ is trivial. For its computation we refer to [Ma95], Chapter IV or [Ne07b], Proposition II.4.
\end{example}

\begin{remark}{\bf(The holomorphic functional calculus).}\index{Holomorphic Functional Calculus}
In a CIA which is Mackey complete, there exists a holomorphic functional calculus based on integration along smooth contours. This has been worked out by [Gl02] and can be used to obtain roots of elements, whose spectra is contained in a set of the form $\mathbb{C}\backslash L$ for $L:=\mathbb{R}_{+}\cdot z_0$, $z_0\in\mathbb{C}^{\times}$. Indeed, suppose that it is possible to pick for each $i=1,\ldots,k$ an invertible element $a_i\in A_i$ such that $\spec(a^{n_i}_i)$ is contained in a ``sliced" plane of the form $\mathbb{C}\backslash L_i$. Then we can choose a branch $f_i$ of the complex $n_i$-th root function on $\mathbb{C}\backslash L_i$ and use the holomorphic functional calculus to get an element $b_i$ with $b_i^{n_i}=a_i^{n_i}$.
\end{remark}

\section{A Possible Approach to Noncommutative Finite Coverings}\index{Noncommutative!Finite Coverings}

In this small section we present some general facts on principal bundles with finite structure group and deduce some algebraic properties of their corresponding smooth function algebras. This discussion in particularly leads to a possible approach to ``noncommutative finite coverings''. One notable aspect is given by the following observation: If $q:P\rightarrow M$ is a finite covering, i.e., a principal bundle $(P,M,G,q,\sigma)$ with finite structure group $G$, then $C^{\infty}(P)$ is a finitely generated projective $C^{\infty}(M)$-module. We start with the following lemma:

\begin{lemma}\label{nice lemma, very nice lemma}
Let $(P,M,G,q,\sigma)$ be a principal bundle with compact structure group $G$. Further, let $(C^{\infty}(P),G,\alpha)$ be the corresponding smooth dynamical system of Proposition \ref{smoothness of the group action on the algebra of smooth functions}, $(\pi,V)$ a finite-dimensional representation of $G$ and $\langle\cdot,\cdot\rangle$ a $G$-invariant scalar product on $V$. If $V^*$ denotes the dual space of $V$ and $(\pi^*,V^*)$ the natural induced representation of $G$ on $V^*$ given for all $g\in G$ and $\lambda\in V^*$ by 
\[\pi^*(g).\lambda:=\lambda\circ\pi(g^{-1}),
\]then the following assertions hold:
\begin{itemize}
\item[\emph{(a)}]
The space $\Hom_G(V,C^{\infty}(P))$\sindex[n]{$\Hom_G(V,C^{\infty}(P))$} of all interwining operators carries the structure of a $C^{\infty}(M)$-module given for all $f\in C^{\infty}(M)$, $\varphi\in \Hom_G(V,C^{\infty}(P))$ and $v\in V$ by
\[(f.\varphi)(v):=f\cdot\varphi(v)
\]with the usual $C^{\infty}(M)$-module structure on $C^{\infty}(P)$ on the right-hand side.
\item[\emph{(b)}]
If $V$ is a real vector space, then the map
\[A_V:V\rightarrow V^*,\,\,\,v\mapsto\langle v,\cdot\rangle
\]defines an equivalence of $G$-representations.
\item[\emph{(c)}]
If $C^{\infty}(P,V^*)^G$ denotes the space of equivariant smooth $V^*$-valued functions on $P$ \emph{(}cf. Proposition \ref{sections of an associated vector bundle}\emph{)}, then the map
\[\Psi:C^{\infty}(P,V^*)^G\rightarrow\Hom_G(V,C^{\infty}(P)),\,\,\,(\Psi(f).v)(p):=f(p).v
\]defines an isomorphism of $C^{\infty}(M)$-modules. 
\item[\emph{(d)}]
The space $\Hom_G(V,C^{\infty}(P))$ is a finitely generated projective $C^{\infty}(M)$-module.
\end{itemize}
\end{lemma}

\begin{proof}
\,\,\,(a) The first assertion of the lemma just consits of a simple calculation.

(b) The second assertion directly follows from $A_V\circ\pi(g)=\pi^*(g)\circ A_V$ for all $g\in G$ and the Theorem of Fr\'echet-Riesz.

(c) We first observe that $\Psi$ is well-defined and $C^{\infty}(M)$-linear. Here, the $C^{\infty}(M)$-module structure on $\Hom_G(V,C^{\infty}(P))$ is given by part (a) of the Lemma. If $f$ is an element in $C^{\infty}(P,V^*)^G$ such that $(\Psi(f).v)(p):=f(p).v=0$ holds for all $v\in V$ and $p\in P$, then $f=0$. Hence, the map $\Psi$ is injective. To verify its surjectivity we choose $\varphi\in\Hom_G(V,C^{\infty}(P))$ and define a map 
\[f_{\varphi}:P\rightarrow V^*,\,\,\,p\mapsto\delta_p\circ\varphi,
\]where $\delta_p:C^{\infty}(P)\rightarrow\mathbb{C}$ denotes the evalutation map in $p\in P$. We claim that $f_{\varphi}\in C^{\infty}(P,V^*)^G$: The smoothness of $f_{\varphi}$ is a consequence of the smoothness of the map
\[F_{\varphi}:P\times V\rightarrow\mathbb{C},\,\,\,(p,v)\mapsto(\delta_p\circ\varphi)(v).
\]In fact, according to [NeWa07], Proposition I.2, the evaluation map
\[\ev_P:P\times C^{\infty}(P)\rightarrow \mathbb{C},\,\,\,(p,f)\mapsto f(p)
\]is smooth. Therefore, $F_{\varphi}=\ev_P\circ(\id_P\times\varphi)$ implies that the map $F_{\varphi}$ is smooth as a composition of smooth maps. Since $\alpha(g)\circ\varphi=\varphi\circ\pi(g)$ holds for all $g\in G$, a short calculation shows that
\[f_{\varphi}(p.g)=\pi^*(g^{-1}).f_{\varphi}(p)
\]is valid for all $p\in P$ and $g\in G$. In particular, we arrive at $f_{\varphi}\in C^{\infty}(P,V^*)^G$ as desired. We finally conclude from $\Psi(f_{\varphi})=\varphi$ that the map $\Psi$ is an isomorphism of $C^{\infty}(M)$-modules.

(d) The last assertion is a direct consequence of part (c), Proposition \ref{sections of an associated vector bundle} and Theorem \ref{Serre-Swan}.
\end{proof}

\begin{proposition}\label{structure if G finite}
If $q:P\rightarrow M$ is a finite covering, i.e., a principal bundle $(P,M,G,q,\sigma)$ with finite structure group $G$, then $C^{\infty}(P)$ is a finitely generated projective $C^{\infty}(M)$-module.
\end{proposition}

\begin{proof}
\,\,\,The proof of this assertion is divided into three parts:

(i) Let $(C^{\infty}(P),G,\alpha)$ be the smooth dynamical system corresponding to the principal bundle $(P,M,G,q,\sigma)$ (cf. Proposition \ref{smoothness of the group action on the algebra of smooth functions}). We can decompose $C^{\infty}(P)$ as a finite direct sum
\[C^{\infty}(P)=\bigoplus_{\varepsilon\in\widehat{G}}C_{\varepsilon}^{\infty}(P),
\]where $\widehat{G}$ denotes the set of all equivalence classes of finite-dimensional irreducible representations of $G$ and $C_{\varepsilon}^{\infty}(P)$ is the $\varepsilon$-isotypic component, i.e., the sum of all irreducible subrepresentations of $\alpha$ belonging to the class $\varepsilon\in\widehat{G}$ (cf. Theorem \ref{str thm of G-mod}). In particular, we observe that $C^{\infty}(P)^G$ is the isotypic component of the trivial representation. Therefore, each space $C_{\varepsilon}^{\infty}(P)$ is a $C^{\infty}(M)$-module.

(ii) Next, [HoMo06], Theorem 4.19 implies that the map
\[\Phi_V:=\Phi_{(\pi,V)}:\Hom_G(V,C^{\infty}(P))\otimes V\rightarrow C_{\varepsilon}^{\infty}(P),\,\,\,\varphi\otimes v\mapsto\varphi(v)
\]is an isomorphism of locally convex $C^{\infty}(M)$-modules for each irreducible representation $(\pi,V)$ of $G$ belonging to the class $\varepsilon\in\widehat{G}$. Here, the assignment
\[f.(\varphi\otimes v):=(f.\varphi)\otimes v, 
\]where the $C^{\infty}(M)$-module structure on $\Hom_G(V,C^{\infty}(P))$ is described in Lemma \ref{nice lemma, very nice lemma} (a),
turns $\Hom_G(V,C^{\infty}(P))\otimes V$ into a $C^{\infty}(M)$-module (cf. Proposition \ref{EoF as B-module}). 


(iii) As a consequence of Lemma \ref{nice lemma, very nice lemma} (d), the fact that each irreducible representation of $G$ is finite-dimensional and part (ii), we conclude that each isotypic component $C_{\varepsilon}^{\infty}(P)$ is a finitely generated projective $C^{\infty}(M)$-module. Hence, the assertion follows from part (i), because $C^{\infty}(P)$ is a finite direct sum of finitely generated projective $C^{\infty}(M)$-modules.
\end{proof}

\begin{example}
Let $n\in\mathbb{N}$. If $p_n:\mathbb{T}\rightarrow\mathbb{T}$, $z\mapsto z^n$ is the $n$-fold covering of the one-dimensional torus $\mathbb{T}$, then the corresponding fixed point algebra of the induced action of
\[C_n:=\{z\in\mathbb{C}^{\times}:\,z^n=1\}=\{\zeta^k:\,\zeta:=\exp(\frac{2\pi i}{n}),\,k=0,1,\ldots,n-1\}
\]on $C^{\infty}(\mathbb{T})$ is given by
\[B_n:=C^{\infty}(\mathbb{T})^{C_n}:=\{f:\mathbb{T}\rightarrow\mathbb{C}:\,(\forall k=0,1,\ldots,n-1)\,f(z\zeta^k)=f(z)\}.
\]Moreover, a short observation shows that the decomposition of $C^{\infty}(\mathbb{T})$ into isotypic components (with respect to $C_n$) is given by
\[C^{\infty}(\mathbb{T})=\bigoplus_{i=0}^{n-1}z^ i\cdot B_n.
\]In particular, $C^{\infty}(\mathbb{T})$ is a free $B_n$-module of rank $n$. We thus conclude from Theorem \ref{Serre-Swan} that the set $C^{\infty}(\mathbb{T})$, viewed as a finitely generated projective $B_n\cong C^{\infty}(\mathbb{T})$-module, corresponds geometrically to the trivial vector bundle 
\[\pr_{\mathbb{T}}:\mathbb{T}\times\mathbb{C}^n,\,\,\,(z,(z_1,\ldots,z_n))\mapsto z.
\]
\end{example}

\begin{remark}
The preceeding example shows that even if the finite covering $q:P\rightarrow M$ is non-trivial, then the corresponding set $C^{\infty}(P)$ of smooth functions on $P$, viewed as a finitely generated $C^{\infty}(M)$-module, can geometrically correspond to a trivial vector bundle.
\end{remark}

The following discussion is inspired from the definition of Galois extensions of commutative rings in [AG60]:

\begin{remark}\label{remark on finite coverings}
Let $G$ be a finite group acting smoothly from the right on a manifold $P$. Further, let $M:=P/G$ be the corresponding orbit space and $\pr:P\rightarrow M$, $p\mapsto p.G$ the corresponding orbit map. If
\[P\times_M P:=\{(p,p')\in P\times P:\,\pr(p)=\pr(p')\}
\]\sindex[n]{$P\times_M P$}is the fibred product manifold of $P$ and $P$ over $M$, then the canonical smooth map
\[\varphi:P\times G\rightarrow P\times_M P,\,\,\,(p,g)\mapsto(p,p.g)
\]is always surjective, and and it is injective if and only if $G$ acts freely. In particular, the map $\varphi$ is a diffeomorphsim if and only if $\pr:P\rightarrow M$ is a covering map. 
\end{remark}

\begin{lemma}\label{proposition if G finite}
Let $P$ be a compact manifold and $G$ a finite group. If $q:P\rightarrow M$ is a covering map with structure group $G$, then 
the map
\[\Phi:C^{\infty}(P)\otimes_{C^{\infty}(M)}C^{\infty}(P)\rightarrow C^{\infty}(P\times_M P),\,\,\,f\otimes g\mapsto (f\circ\pr_1)\cdot(g\circ\pr_2)
\]is an isomorphism of unital Fr\'echet algebras.\sindex[n]{$C^{\infty}(P)\otimes_{C^{\infty}(M)}C^{\infty}(P)$}\sindex[n]{$C^{\infty}(P\times_M P)$}
\end{lemma}

\begin{proof}
\,\,\,Since $C^{\infty}(P)$ is a finitely generated projective $C^{\infty}(M)$-module (cf. Proposition \ref{structure if G finite}), so is $C^{\infty}(P)\otimes_{C^{\infty}(M)}C^{\infty}(P)$. In particular,
\[C^{\infty}(P)\otimes_{C^{\infty}(M)}C^{\infty}(P)=C^{\infty}(P)\widehat{\otimes}_{C^{\infty}(M)}C^{\infty}(P).
\]holds a a consequence of Proposition \ref{modules of frechet algebras} (b). Therefore, the statement follows from Remark \ref{fibres products of manifolds}.
\end{proof}

The next theorem provides an algebraic characterization for the freeness of a group action in the case of a finite group:

\begin{theorem}{\bf(An algebraic characterization for the freeness of a group action).}\label{theorem if G finite}\index{Algebraic!Characterization for the Freeness of a Group Action}
Let $G$ be a finite group acting smoothly from the right on a compact manifold $P$. If $M:=P/G$ is the corresponding orbit space and $\pr:P\rightarrow M$, $p\mapsto p.G$ the corresponding orbit map, then the following conditions are equivalent:
\begin{itemize}
\item[\emph{(a)}]
The map $\pr:P\rightarrow M$ is a covering map.
\item[\emph{(b)}]
The map
\[\varphi:P\times G\rightarrow P\times_M P,\,\,\,(p,g)\mapsto(p,p.g)
\]is a diffeomorphism.
\item[\emph{(c)}]
The map
\[\phi:C^{\infty}(P)\otimes_{C^{\infty}(M)}C^{\infty}(P)\rightarrow\prod_{g\in G}C^{\infty}(P),\,\,\,f_1\otimes f_2\mapsto (f_1\cdot (g.f_2))_{g\in G}
\]is an isomorphism of unital \emph{(}Fr\'echet\emph{)} algebras.
\end{itemize}
\end{theorem}

\begin{proof}
\,\,\,(a) $\Leftrightarrow$ (b): We have already discussed this in Remark \ref{remark on finite coverings}.

(b) $\Rightarrow$ (c): For this direction we have to dualize the map $\varphi$, use Lemma \ref{proposition if G finite} and recall that the map
\[C^{\infty}(P\times G)\rightarrow\prod_{g\in G}C^{\infty}(P),\,\,\,F\mapsto (F(\cdot,g))_{g\in G}
\]is an isomorphism of unital (Fr\'echet) algebras.

(c) $\Rightarrow$ (b): This direction can be obtained by dualizing the map $\phi$, a short calculation and using Lemma \ref{spec as manifold}.
\end{proof}

\begin{remark}\label{remark if G finite}{\bf(A possible approach to NCP $G$-bundles).}
(a) The previous discussion leads to the following approach: Let $G$ be a finite group and $A$ be a unital locally convex algebra. A dynamical system $(A,G,\alpha)$ with fixed point algebra $B$ is called a \emph{NCP $G$-bundle} (with base $B$), if the map
\[\phi:A\otimes_B A\rightarrow\prod_{g\in G} A,\,\,\,a_1\otimes a_2\mapsto (a_1\cdot (\alpha(g).a_2))_{g\in G}
\]is an isomorphism of unital (commutative) algebras.

(b) Note that the algebra $A$ has to be commutative in order to guarantee that the ``source space" $A\otimes_B A$ carries the structure of an algebra. In fact, if $A$ is noncommutative, then $A\otimes_B A$ carries in general ``only" the structure of a right locally convex $A$-module (cf. Corollary \ref{mod and ringhom}). A possible way to avoid this problem would be to assume that the fixed point algebra $B$ is contained in the center $C_A$ of $A$.

(c) We want to point out that the previous discussion is also valid in the case of a compact Lie group $G$ (even for a general Lie group $G$, if $G$ acts properly). Of course, we slightly have to reformulate Theorem \ref{theorem if G finite} (c) in the following way: The map
\[\phi:C^{\infty}(P)\widehat{\otimes}_{C^{\infty}(M)}C^{\infty}(P)\rightarrow C^{\infty}(G,C^{\infty}(P)),\,\,\,f_1\otimes f_2\mapsto (g\mapsto f_1\cdot(\alpha(g).f_2))
\]is an isomorphism of unital Fr\'echet algebras.
\end{remark}

\chapter{NCP Bundles with Compact Abelian Structure Group}\label{a geometric approach to NCPB}

Let $G$ be a compact abelian Lie group. The main goal of this chapter is to present a reasonable approach to NCP $G$-bundles. Since, for example, the isotypic components of a dynamical system $(A,\mathbb{T}^n,\alpha)$ do in general not contain invertible elements, we have to come up with algebraic conditions on a dynamical system $(A,G,\alpha)$ that still ensure the freeness of the induced (right-) action
\[\sigma:\Gamma_A\times G\rightarrow\Gamma_A,\,\,\,(\chi,g)\mapsto \chi\circ\alpha(g).
\]of $G$ on the spectrum $\Gamma_A$ of $A$. Because the freeness of a group action is a local condition (cf. Remark \ref{free=locally free}), our main idea is inspired by the classical setting: Loosely speaking, a dynamical system $(A,G,\alpha)$ is called a NCP $G$-bundle, if it is ``locally" a trivial NCP $G$-bundle, i.e., once the concept of a trivial NCP $G$-bundle is known for a certain (compact) group $G$, Section \ref{ACNCPT^nB} enters the picture. In view of Chapter \ref{trivial ncp torus bundles} and Chapter \ref{trivial ncp cyclic bundles}, we restrict our attention to compact abelian Lie groups $G$. In fact, [HoMo06], Proposition 2.42 implies that each compact abelian Lie group $G$ is isomorphic to $\mathbb{T}^n\times \Lambda$ for some natural number $n$ and a finite abelian group $\Lambda$. We prove that this approach extends the classical theory of principal bundles and present some noncommutative examples like sections of algebra bundles with trivial NCP $G$-bundle as fibre, sections of algebra bundles which are pull-backs of 
principal $G$-bundles and sections of trivial equivariant algebra bundles.

\section{The Concept of NCP $G$-Bundles}

Let $G$ be a compact abelian Lie group. The purpose of this section is to present a reasonable approach to NCP $G$-bundles. 
In particular, we show that our definition extends the classical theory of principal $G$-bundles. 

\begin{notation}
Given a dynamical system $(A,G,\alpha)$, we (again) write $Z$ for the fixed point algebra of the induced action of $G$ on the center $C_A$ of $A$, i.e., $Z:=C_A^G$. In particular, if $P$ is a manifold, $G$ a Lie group group and $(C^{\infty}(P),G,\alpha)$ a (smooth) dynamical system, then $Z=C^{\infty}(P)^G$.
\end{notation}

\begin{definition}{\bf(Trivial NCP $G$-bundles).}\label{trivial NCPT^nB again}\index{Trivial NCP!$G$-Bundles}
Let $G$ be a compact abelian Lie group, i.e., $G\cong\mathbb{T}^n\times\Lambda$ for some natural number $n$ and a finite abelian group $\Lambda$ (cf. [HoMo06], Proposition 2.42). Then the corresponding character group $\widehat{G}$ is isomorphic to $\mathbb{Z}^n\times \Lambda$. We call a (smooth) dynamical system $(A,G,\alpha)$ a \emph{\emph{(}smooth\emph{)} trivial NCP $G$-bundle} if the following conditions are satisfied:
\begin{itemize}
\item[(G1)]
The isotypic components belonging to the $\mathbb{Z}^n$-part contain invertible elements.
\item[(G2)]
The isotypic components belonging to generators of the $\Lambda$-part satisfy the conditions of Definition \ref{C_nC_m V}.
\end{itemize}
\end{definition}

\begin{remark}\label{blabla}{\bf(Trivial principal $G$-bundles).}\index{Trivial Principal!$G$-Bundles}
A theorem characterizing the trivial principal $G$-bundles can be obtained similarly to Theorem \ref{TNPB for manifold} and Theorem \ref{C_nC_m IV}.
\end{remark}

\begin{definition}{\bf(NCP $G$-bundles).}\label{NCPT^nB again}\index{NCP $G$-Bundles}
Let $G$ be a compact abelian Lie group. We call a (smooth) dynamical system $(A,G,\alpha)$ a \emph{\emph{(}smooth\emph{)} NCP $G$-bundle} if for each $\chi\in\Gamma_Z$ there exists an element $z\in Z$ with $\chi(z)\neq 0$ such that the corresponding (smooth) localized dynamical system $(A_{\{z\}},G,\alpha_{\{z\}})$\sindex[n]{$(A_{\{z\}},G,\alpha_{\{z\}})$} of Proposition \ref{T^n action on spec again} is a (smooth) trivial NCP $G$-bundle (cf. Definition \ref{trivial NCPT^nB again}). 
\end{definition}

\begin{remark}
The previous definition of NCP $G$-bundles is inspired by the classical setting. In fact, we will see in Theorem \ref{NCT^nB for manifold again} that it actually agrees with the classical definition of principal bundles (cf. Definition \ref{principal bundles I}) if we identify open subsets $U$ of the base manifold with the corresponding $U$-defining functions and use Corollary \ref{C(U)=C(M)_f_U}.
\end{remark}

\begin{remark}\label{remark on NCPT^nB again}{\bf(An equivalent point of view).}
Given $z\in Z$, we recall that 
\[D(z):=\{\chi\in\Gamma_Z:\,\chi(z)\neq 0\}.
\]Then a short observation shows that a (smooth) dynamical system $(A,G,\alpha)$ is a smooth NCP $G$-bundle if and only if there exists a family of elements $(z_i)_{i\in I}\subseteq Z$ such that the following two conditions are satisfied:
\begin{itemize}
\item[(i)]
The family $(D(z_i))_{i\in I}$ is an open covering of $\Gamma_Z$.
\item[(ii)]
The (smooth) localized dynamical systems $(A_{\{z_i\}},G,\alpha_{\{z_i\}})$ of Proposition \ref{T^n action on spec again} are (smooth) trivial NCP $G$-bundles.
\end{itemize}
\end{remark}


\begin{lemma}\label{T^n action on spec III again}
Let $(A,G,\alpha)$ be a NCP $G$-bundle. Further, let $\chi\in\Gamma_Z$ and $z\in Z$ with $\chi(z)\neq 0$ such that the corresponding \emph{(}smooth\emph{)} localized dynamical system $$(A_{\{z\}},G,\alpha_{\{z\}})$$ is a trivial NCP $G$-bundle. Then the $G$-action map 
\[\sigma_{\{z\}}:\Gamma^{\emph{cont}}_{A_{\{z\}}}\times G\rightarrow\Gamma^{\emph{cont}}_{A_{\{z\}}},\,\,\,(\chi,g)\mapsto \chi\circ\alpha_{\{z\}}(g)
\]of Lemma \ref{T^n action on spec II again} is free.
\end{lemma}

\begin{proof}
\,\,\,The assertion follows from fact that each trivial NCP $G$-bundle defines a free action of $G$ on its spectrum (cf. for example Proposition \ref{freeness for T^n II}).
\end{proof}

\subsection*{Reconstruction of Principal $G$-bundles}

In this part of the section we show how to recover the classical definition of principal $G$-bundles. 

\begin{lemma}\label{extending lemma for dynamical systems}
Let $P$ be a compact manifold, $G$ a Lie group and $(C^{\infty}(P),G,\alpha)$ a smooth dynamical system. Then each character $\chi:Z\rightarrow\mathbb{C}$ is the restriction of an evaluation homomorphism 
\[\delta_p:C^{\infty}(P)\rightarrow\mathbb{C},\,\,\,f\mapsto f(p)\,\,\,\text{for some}\,\,\,p\in P.
\]
\end{lemma}

\begin{proof}
\,\,\,Let $\chi:Z\rightarrow\mathbb{C}$ be a character. 
If all functions of $I:=\ker\chi$ vanish in the point $p$ in $P$, then the maximality of $I$ implies that $I=(\ker\delta_p)\cap Z$, i.e., that $\chi={\delta_p}_{\mid Z}$. So we have to show that such a point exists. Let us assume that this is not the case. From that we shall derive the contradiction $I=Z$:

(i) If such a point does not exist, then for each $p\in P$ there exists a function $f_p\in I$ with $f_p(p)\neq 0$. The family $(f_p^{-1}(\mathbb{C}^{\times}))_{p\in P}$ is an open cover of $P$, so that there exist $p_1,\ldots,p_n\in P$ and a smooth $G$-invariant function $f_P:=\sum_{i=1}^n \vert f_{p_i}\vert^2>0$ on $P$. Here, the last statement follows from the fact that each algebra automorphism of $C^{\infty}(P)$ is automatically a *-automorphism (cf. Remark \ref{star-automorphism}).

(ii) In view of part (i) we obtain a function $f_P\in I\cap Z$ which is nowhere zero. Thus, $f_P$ is invertible in $C^{\infty}(P)$. We claim that its inverse $h_P:=f_P^{-1}=\frac{1}{f_P}$ is also $G$-invariant: For this we choose an arbitrary element $g\in G$ and note that $f_P\cdot h_P=1$ implies
\[\alpha(g)(h_P)\cdot f_P=\alpha(g)(h_P\cdot f_P)=1.
\]In particular, $\alpha(g)(h_P)=h_P$. Since $g\in G$ was arbitrary, we conclude that $h_P\in Z$.

(iii) Finally, part (ii) leads to the contradiction $f_P\in I\cap Z^{\times}$, i.e., $I=Z$ as desired.
\end{proof}

\begin{proposition}\label{extension of char on fixed point algebras VI}
Let $P$ be a compact manifold, $G$ a compact Lie group and $(C^{\infty}(P),G,\alpha)$ a smooth dynamical system. Then the map
\[\Phi:P/G\rightarrow\Gamma_Z,\,\,\,p.G\mapsto \delta_p
\]is a homeomorphism.
\end{proposition}

\begin{proof}
\,\,\,We divide the proof of this proposition into four parts:

(i) We first note that the quotient $P/G$ is a compact Hausdorff space. Further, the map $\Phi$ is well-defined since each function $f\in Z$ is $G$-invariant. 

(ii) Proposition \ref{spec of C(M) top} and the definition of the quotient topology on $P/G$ implies that $\Phi$ is continuous.

(iii) The surjectivity of $\Phi$ is a consequence of Lemma \ref{extending lemma for dynamical systems} (alternatively we can use Corollary \ref{extension of char on fixed point algebras V}). To show that $\Phi$ is injective, we choose elements $p,p'\in P$ with $p.G\neq p'.G$. Since $P$ is a manifold, there exist disjoint $G$-invariant open subsets $U$ and $V$ of $P$ containing the compact subsets $p.G$ and $p'.G$, respectively. Next, we choose a positive smooth function $f$ on $P$ which is equal to $1$ on $p.G$ and vanishes outside $U$. Then the function 
\[F:P\rightarrow\mathbb{C},\,\,\,F(q):=\int_Gf(q.g)\,dg,
\]where $dg$ denotes the normalized Haar measure on $G$, defines a smooth $G$-invariant function satisfying 
\[F(p)=1\neq 0=F(p').
\]Here, we have used that the smooth action map $\alpha$ induces a smooth action of $G$ on $P$ (cf. Proposition \ref{smoothness of the group action on the set of characters}). Hence,
\[\delta_p(F)=F(p)\neq F(p')=\delta_{p'}(F)
\]implies that $\delta_p\neq\delta_{p'}$, i.e., $\Phi$ is injective. 

(iv) So far we have seen that $\Phi$ is a bijective continuous map from a compact Hausdorff space onto a Hausdorff space. Thus, the claim follows from a well-known theorem from topology guaranteeing that $\Phi$ is a homeomorphism.
\end{proof}

For a compact manifold $P$, a compact Lie group $G$ and a dynamical system $(C^{\infty}(P),G,\alpha)$ we write 
\[\pr:P\rightarrow P/G,\,\,\,p\mapsto p.G
\]for the canonical quotient map. Moreover, we identify $D(z)$ with an open subset of $P/G$ via the homeomorphism of Proposition \ref{extension of char on fixed point algebras VI} and define
\[P_{D(z)}:=\pr^{-1}(D(z)).
\]For the following lemma we recall that $D_A(z):=\{\chi\in\Gamma_A:\,\chi(z)\neq 0\}$:

\begin{lemma}\label{C(P) localized II again}
Let $P$ be a compact manifold, $G$ a compact Lie group and $(C^{\infty}(P),G,\alpha)$ a dynamical system. If $A=C^{\infty}(P)$ and $z\in Z$, then the map
\[\Phi_z:P_{D(z)}\rightarrow D_A(z),\,\,\,p\mapsto\delta_p
\]is a homeomorphism.
\end{lemma}

\begin{proof}
\,\,\,This lemma is an easy consequence of Proposition \ref{spec of C(M) top}.
\end{proof}

\begin{proposition}\label{C(P) localized III again}
Suppose we are in the situation of Lemma \ref{C(P) localized II again}. Then the following assertions hold:
\begin{itemize}
\item[\emph{(a)}]
The map 
\[\Psi_{\{z\}}:P_{D(z)}\rightarrow\Gamma_{C^{\infty}(P)_{\{z\}}},\,\,\,\Psi_{\{z\}}(p)([F]):=F\left(\frac{1}{z(p)},p\right)
\]is a homeomorphism.
\item[\emph{(b)}]
There is a smooth structure on $\Gamma_{C^{\infty}(P)_{\{z\}}}$ for which the map $\Psi_{\{z\}}$ becomes a diffeomorphism.
\end{itemize}
\end{proposition}

\begin{proof}
\,\,\,(a) We first note that $C^{\infty}(P)_{\{z\}}$ is a Fr\'echet CIA (cf. Lemma \ref{quotient algebras}). Hence, each character of $C^{\infty}(P)_{\{z\}}$ is continuous by Lemma \ref{cont of char of CIA}. Since $\Psi_{\{z\}}=\Phi_{\{z\}}\circ\Phi_z$ (cf. Remark \ref{spec A_z}), the claim follows from Lemma \ref{C(P) localized II again}.

(b) The second assertion follows from the fact that $P_{D(z)}$ is an open subset of the manifold $P$ and is therefore a manifold in its own right. Thus, $\Psi_{\{z\}}$ induces a smooth structure on $\Gamma_{C^{\infty}(P)_{\{z\}}}$ for which $\Psi_{\{z\}}$ becomes a diffeomorphism. 
\end{proof}

\begin{remark}\label{free=locally free}{\bf (Freeness is a local condition).}\index{Freeness!is a Local Condition}
Let $\sigma: X\times G\rightarrow X$ be an action of a topological group $G$ on a topological space $X$ and let $\pr:X\rightarrow X/G$ be the corresponding quotient map. Then a short observation shows that the action $\sigma$ is free if and only if there exists an open cover $(U_i)_{i\in I}$ of the quotient space $X/G$ (with respect to the quotient topology) such that each restriction map
\[\sigma_i:X_{U_i}\times G\rightarrow X_{U_i},\,\,\,\sigma_i(x,g):=\sigma(x,g),
\]where $X_{U_i}:=\pr^{-1}(U_i)$, is free.
\end{remark}

\begin{theorem}\label{NCT^nB for manifold again}{\bf(Reconstruction Theorem).}\index{Theorem!The Reconstruction}
For a manifold $P$, the following assertions hold:
\begin{itemize}
\item[\emph{(a)}]
If $P$ is compact and $(C^{\infty}(P),G,\alpha)$ is a smooth NCP $G$-bundle, then we obtain a principal $G$-bundle $(P,P/G,G,\pr,\sigma)$.
\item[\emph{(b)}]
Conversely, if $(P,M,G,q,\sigma)$ is a principal $G$-bundle, then the corresponding smooth dynamical system $(C^{\infty}(P),G,\alpha)$ is a smooth NCP $G$-bundle.
\end{itemize}
\end{theorem}

\begin{proof}
\,\,\,(a) Since $G$ is compact, the induced smooth action map $\sigma$ of Proposition \ref{smoothness of the group action on the set of characters} is proper. In view of the Quotient Theorem (cf. Remark \ref{principal bundles II} (c)), it remains to verify the freeness of $\sigma$: For this, we first note that the map $\sigma$ is free if and only if there exists an open cover of $\Gamma_Z$ of the form $(D(z_i))_{i\in I}$ such that each restriction map
\[\sigma_i:P_{D(z_i)}\times G\rightarrow P_{D(z_i)},\,\,\,\sigma_i(\delta_p,g):=\sigma(\delta_p,g)
\]is free (cf. Remark \ref{free=locally free}). Moreover, Proposition \ref{C(P) localized III again} (a) implies that
\[\Psi_{\{z\}}(\sigma(\delta_p,g))=\Psi_{\{z\}}(\delta_p\circ\alpha(g))=\Psi_{\{z\}}(\delta_p)\circ\alpha_{\{z\}}(g)=\sigma_{\{z\}}(\Psi_{\{z\}}(\delta_p),g)
\]holds for each $z\in Z$, $p\in P$ and $g\in G$. Therefore, the freeness of $\sigma$ is a consequence of Remark \ref{remark on NCPT^nB again} (a) and Lemma \ref{T^n action on spec III again}.

(b) Conversely, we have to show that the condition of Definition \ref{NCPT^nB again} is satisfied: This will be done in the following three steps:

(i) We first note that $Z\cong C^{\infty}(M)$ (cf. Proposition \ref{fixed point algebra of principal bundles}). Hence, $\Gamma_Z$ is homeomorphic to $M$ by Proposition \ref{spec of C(M) top}.

(ii) Next, we choose an open cover $(U_i)_{i\in I}$ of $M$ such that each $P_{U_i}$ is a trivial principal $G$-bundle over $U_i$, i.e., such that $P_{U_i}\cong U_i\times G$ holds for each $i\in I$. Further, we choose for all $i\in I$ a $U_i$-defining function $f_i$ and note that each function $h_i:=f_i\circ q$ is a (smooth) $P_{U_i}$-defining function.

(iii) For $p\in P$ we choose $i\in I$ with $q(p)\in U_i$. Then $h_i$ is an element in $Z$ with $h_i(p)\neq 0$ and we conclude from Corollary \ref{C(U)=C(M)_f_U} that the map
\[R_i:C^{\infty}(P)_{\{h_i\}}\rightarrow C^{\infty}(P_{U_i}),\,\,\,[F]\mapsto\left(p\mapsto F\left(\frac{1}{h_i(p)},p\right)\right)
\]is an isomorphism of unital Fr\'{e}chet algebras. Moreover, the map $R_i$ is $G$-equivariant. In fact, we have
\[(R_i\circ\alpha_{\{h_i\}})(g,[F])=\alpha(g,R_i([F]))
\]for all $g\in G$ and $[F]\in C^{\infty}(P)_{\{h_i\}}$. In particular, the corresponding dynamical system 
\[(C^{\infty}(P)_{\{h_i\}},G,\alpha_{\{h_i\}})
\]of Proposition \ref{T^n action on spec again} carries the structure of a smooth trivial NCP $G$-bundle.
\end{proof}

\section{Examples of NCP Bundles}\label{examples of NCP T^n-bundles}

We now present a bunch of examples. We first show that each trivial NCP $G$-bundle carries the structure of a NCP $G$-bundle in its own right. We further show that examples of NCP $G$-bundles are provided by sections of algebra bundles with trivial NCP $G$-bundle as fibre, sections of algebra bundles which are pull-backs of principal $G$-bundles and sections of trivial equivariant algebra bundles. At the end of this chapter we present a very concrete example.  

\subsection*{Example 1: Trivial NCP Bundles}\index{Trivial NCP!Bundles}

We show that each trivial NCP $G$-bundle carries the structure of a NCP $G$-bundle:

\begin{theorem}{\bf(Trivial NCP $G$-bundles).}\label{TNCTB are NCTB}
Each trivial NCP $G$-bundle $(A,G,\alpha)$ carries the structure of a NCP $G$-bundle.
\end{theorem}

\begin{proof}
\,\,\,If $(A,G,\alpha)$ is a trivial NCP $G$-bundle, then $1_A\in Z$ and $\chi(1_A)=1\neq 0$ holds for each $\chi\in\Gamma_Z$. In particular, Corollary \ref{inverting 1 iso} implies that $A_{\{1_A\}}\cong A$ holds as unital locally convex algebras. Therefore, $(A,G,\alpha)$ is a NCP $G$-bundle in its own right.
\end{proof}

\subsection*{Example 2: Sections of Algebra Bundles with a Trivial NCP $G$-Bundle as Fibre}\label{examples of NCP T^n-bundles II}

We show that if $A$ is a unital Fr\'echet algebra and $(A,G,\alpha)$ a smooth trivial NCP $G$-bundle such that $Z=C_A^G$ is isomorphic to $\mathbb{C}$, then the algebra of sections of each algebra bundle with ``fibre" $(A,G,\alpha)$ is a NCP $G$-bundle. We start with the following lemma:

\begin{lemma}\label{dynamical system and trivial algebra bundles}
Let $(A,G,\alpha)$ be a \emph{(}smooth\emph{)} dynamical system and $M$ a manifold. Then the map
\[\beta:G\times C^{\infty}(M,A)\rightarrow C^{\infty}(M,A),\,\,\,(g,f)\mapsto\alpha(g)\circ f
\]defines a \emph{(}smooth\emph{)} continuous action of $G$ on $C^{\infty}(M,A)$ by algebra automorphisms. In particular, the triple $(C^{\infty}(M,A),G,\beta)$\sindex[n]{$(C^{\infty}(M,A),G,\beta)$} is a \emph{(}smooth\emph{)} dynamical system.
\end{lemma}

\begin{proof}
\,\,\,We first note that the map
\[\ev_M:C^{\infty}(M,A)\times M\rightarrow A,\,\,\,(f,m)\mapsto f(m)
\]is smooth (cf. [NeWa07], Proposition I.2). Then Lemma \ref{smooth exp law} implies that the map $\beta$ is smooth if and only if the map
\[\beta^{\wedge}:G\times C^{\infty}(M,A)\times M\rightarrow A,\,\,\,(g,m,f)\mapsto\alpha(g,f(m))
\]is smooth. Since $\beta^{\wedge}=\alpha\circ(\id_G\times\ev_M)$, we conclude that $\beta^{\wedge}$ is smooth as a composition of smooth maps. Clearly, $\beta$ defines an action of $G$ on $C^{\infty}(M,A)$ by algebra automorphisms.
\end{proof}

\begin{proposition}\label{trivial NCP T^n-bundles from trivial algebra bundles}
If $(A,G,\alpha)$ is a \emph{(}smooth\emph{)} trivial NCP $G$-bundle and $M$ a manifold, then the triple $(C^{\infty}(M,A),G,\beta)$ of Lemma \ref{dynamical system and trivial algebra bundles} is also a \emph{(}smooth\emph{)} trivial NCP $G$-bundle.
\end{proposition}

\begin{proof}
\,\,\,The claim directly follows from the fact that the algebra $A$ is naturally embedded in $C^{\infty}(M,A)$ through the constant maps. In fact, if $\varphi\in\widehat{G}$, then the corresponding isotypic component $C^{\infty}(M,A)_{\varphi}$ is equal to $C^{\infty}(M,A_{\varphi})$ and thus contains invertible elements.
\end{proof}

\begin{definition}\label{automorphism of dynamical systems}{\bf(The automorphism group of a dynamical system).}
Let $(A,G,\alpha)$ be a dynamical system. The group
\[\Aut_G(A):=\{\varphi\in\Aut(A):\,(\forall g\in G)\,\alpha(g)\circ\varphi=\varphi\circ\alpha(g)\}
\]is called the \emph{automorphism group of} the dynamical system $(A,G,\alpha)$.\sindex[n]{$\Aut_G(A)$}\index{Dynamical Systems!Automorphism Group of}
\end{definition}

\begin{example}
If $A=C^{\infty}(M\times G)$, then [Ne08b], Proposition 1.4.8 implies that 
\[\Aut_G(A)=C^{\infty}(M,G)\rtimes_{\gamma}\Diff(M),
\]where $\gamma:\Diff(M)\rightarrow\Aut(C^{\infty}(M,G))$, $\gamma(\varphi).f:=f\circ\varphi^{-1}$. In particular, if $A=C^{\infty}(G)$, then 
\[\Aut_G(A)\cong G.
\]
\end{example}

\begin{example}\label{G-automorphism group of quantumtori}
If $\mathbb{T}^n_{\theta}$ is the smooth noncommutative $n$-torus from Definition \ref{n-dim QT} and $(\mathbb{T}^n_{\theta},\mathbb{T}^n,\alpha)$ the corresponding smooth dynamical system of Example \ref{Smooth NC n-tori as NCPTB}, then
\[\Aut_{\mathbb{T}^n}(\mathbb{T}^n_{\theta})\cong \mathbb{T}^n.
\]Indeed, we first recall that $\mathbb{T}^n_{\theta}$ is generated by unitaries $U_1,\ldots, U_n$. If now $\varphi\in\Aut_{\mathbb{T}^n}(\mathbb{T}^n_{\theta})$, $z\in\mathbb{T}^n$ and $U_r$ ($1\leq r\leq n$) is such a 
unitary, then 
\[z.\varphi(U_r)=\varphi(z.U_r)=\varphi(z_r\cdot U_r)=z_r\cdot\varphi(U_r)
\]leads to $\varphi(U_r)=\lambda_r\cdot U_r$ for some $\lambda_r\in\mathbb{C}^{\times}$. Since $U_r$ is a unitary and $\varphi$ is a $^*$-automorphism ($\mathbb{T}^n_{\theta}$ is a $^*$-algebra), we conclude that $\lambda_r\in\mathbb{T}$. In particular, each element in $\Aut_{\mathbb{T}^n}(\mathbb{T}^n_{\theta})$ corresponds to an element in $\mathbb{T}^n$ and vice versa. 
\end{example}

\begin{proposition}\label{cocycle description of bundles}{\bf(Cocycle description of bundles).}\index{Bundles!Cocycle Description of}
Let $(A,G,\alpha)$ be a smooth dynamical system and $M$ a manifold. Further, let $(U_i)_{i \in I}$ be an open cover of $M$ and $U_{ij}:=U_i\cap U_j$ for $i,j\in I$. If $(g_{ij})_{i,j \in I}$ is a collection of functions $g_{ij}\in C^{\infty}(U_{ij},\Aut_G(A))$ satisfying 
\[g_{ii}={\bf 1}\,\,\,\text{and}\,\,\,g_{ij}g_{jk}=g_{ik}\,\,\,\text{on}\,\,\,U_{ijk}:=U_i\cap U_j\cap U_k,
\]then the following assertions hold:
\begin{itemize}
\item[\emph{(a)}]
There exists an algebra bundle $(\mathbb{A},M,A,q)$ and bundle charts $\varphi_i:U_i\times A\rightarrow\mathbb{A}_{U_i}$ such that
\[(\varphi_i^{-1}\circ\varphi_j)(x,a)=(x,g_{ij}(x).a).
\]
\item[\emph{(b)}]
The map
\[\sigma:G\times\mathbb{A}\rightarrow\mathbb{A},\,\,\,(g,a)\mapsto\varphi_i(x,\alpha(g).a_0),
\]for $i\in I$ with $x=q(a)\in U_i$ and $a_0\in A$ with $\varphi_i(x,a_0)=a$, defines a smooth action of $G$ on $\mathbb{A}$ by fibrewise algebra automorphisms.
\end{itemize}
\end{proposition}

\begin{proof}
\,\,\,(a) A proof of the first statement can be found in [St51], Part I, Section 3,  Theorem 3.2.

(b) The crucial point is to show that the map $\sigma$ is well-defined: For this let $i,j\in I$ with $x=q(a)\in U_{ij}$, $a_0\in A$ with $\varphi_i(x,a_0)=a$ and $a'_0\in A$ with $\varphi_j(x,a'_0)=a$. Then $a_0=g_{ij}(x).a'_0$ leads to
\begin{align}
\varphi_i(x,\alpha(g).a_0)&=\varphi_i(x,\alpha(g)(g_{ij}(x).a'_0))=\varphi_i(x,g_{ij}(x)(\alpha(g).a'_0))=\varphi_j(x,\alpha(g).a'_0).\notag
\end{align}
Further, a short calculation shows that $\sigma$ defines an action of $G$ on $\mathbb{A}$ by fibrewise algebra automorphisms. Its smoothness follows from the local description by a smooth function.
\end{proof}

\begin{proposition}\label{dynamical system and algebra bundles}
Suppose we are in the situation of Proposition \ref{cocycle description of bundles}. Then the map
\[\beta:G\times\Gamma\mathbb{A}\rightarrow\Gamma\mathbb{A},\,\,\,\beta(g,s)(m):=\sigma(g,s(m))
\]defines a smooth action of $G$ on $\Gamma\mathbb{A}$ by algebra automorphisms. In particular, the triple $(\Gamma\mathbb{A},G,\beta)$ is a smooth dynamical system.
\end{proposition}

\begin{proof}
\,\,\,That the map $\beta$ defines an action of $G$ on $\Gamma\mathbb{A}$ by algebra automorphisms is a consequence of Proposition \ref{cocycle description of bundles} (b). To verify the smoothness of $\beta$, we first choose a bundle atlas $(\varphi_i,U_i)_{i\in I}$ of $(\mathbb{A},M,A,q)$ and use the definition of the smooth structure on $\Gamma\mathbb{A}$: In fact, Remark \ref{smooth structure on sections} implies that the map $\beta$ is smooth if and only if each map 
\[\Phi_i\circ\beta:G\times\Gamma\mathbb{A}\rightarrow C^{\infty}(U_i,A),\,\,\,(g,s)\mapsto\alpha(g)\circ s_i
\]is smooth. Next, we recall that each map
\[\alpha_i:G\times C^{\infty}(U_i,A)\rightarrow C^{\infty}(U_i,A),\,\,\,(g,f)\mapsto\alpha(g)\circ f
\]is smooth by Lemma \ref{dynamical system and trivial algebra bundles}. Since $\Phi_i\circ\beta=\alpha_i\circ(\id_G\times\Phi_i)$ for each $i\in I$, each map $\Phi_i\circ\beta$ is smooth as a composition of smooth maps.
\end{proof}


\begin{remark}
For the next theorem we recall that the smooth dynamical systems $(\mathbb{T}^n_{\theta},\mathbb{T}^n,\alpha)$ of Example \ref{Smooth NC n-tori as NCPTB} provide a class of examples of trivial NCP $\mathbb{T}^n$-bundles for which $C_{\mathbb{T}^n_{\theta}}^{\mathbb{T}^n}$ is isomorphic to $\mathbb{C}$. Another class with this property is given by the trivial NCP $C_m\times C_m$-bundles $(\M_m(\mathbb{C}),C_m\times C_m,\alpha)$ of Example \ref{the matrix algebra}.
\end{remark}

\begin{theorem}\label{NCP G_bundle with fibre trivial NCP G bundle}
Let $A$ be a unital Fr\'echet algebra and $(A,G,\alpha)$ a smooth trivial NCP $G$-bundle such that $C_A^G$ is isomorphic to $\mathbb{C}$. Further, let $M$ be a manifold, $(U_i)_{i \in I}$ an open cover of $M$ and $U_{ij}:=U_i\cap U_j$ for $i,j\in I$. If $(g_{ij})_{i,j \in I}$ is a collection of functions $g_{ij}\in C^{\infty}(U_{ij},\Aut_G(A))$ satisfying 
\[g_{ii}={\bf 1}\,\,\,\text{and}\,\,\,g_{ij}g_{jk}=g_{ik}\,\,\,\text{on}\,\,\,U_{ijk}:=U_i\cap U_j\cap U_k,
\]then the smooth dynamical system $(\Gamma\mathbb{A},G,\beta)$\sindex[n]{$(\Gamma\mathbb{A},G,\beta)$} of Proposition \ref{dynamical system and algebra bundles} is a NCP $G$-bundle.
\end{theorem}

\begin{proof}
\,\,\,
 (i) We first note that $Z:=C_{\Gamma\mathbb{A}}^G\cong C^{\infty}(M)$. In particular, the spectrum $\Gamma_Z$ is homeomorphic to $M$.

(ii) Next, we choose for each $i\in I$ a $U_i$-defining function $f_i$. Further, for $m\in M$, we choose $i\in I$ with $m\in U_i$. Then $f_i$ is an element in $C^{\infty}(M)$ with $f_i(m)\neq 0$ and we conclude from Corollary \ref{algebra section 3} that the map
\[\phi_{U_i}:\Gamma\mathbb{A}_{\{f_i\}}\rightarrow\Gamma\mathbb{A}_{U_i},\,\,\,[F]\mapsto F\circ\left(\frac{1}{f_i}\times\id_{U_i}\right)
\]is an isomorphism of unital Fr\'echet algebras. Moreover, the map $\phi_{U_i}$ is $G$-equivariant. In fact, we have
\[(\phi_{U_i}\circ\beta_{\{f_i\}})(g,[F])=\beta(g,\phi_{U_i}([F]))
\]for all $g\in G$ and $[F]\in\Gamma\mathbb{A}_{\{f_i\}}$. Since the natural isomorphism between the space $\Gamma\mathbb{A}_{U_i}$ and $C^{\infty}(U_i,A)$ is also a $G$-equivariant isomorphism of unital Fr\'echet algebras (cf. Proposition \ref{cocycle description of bundles} (a)), it follows from Proposition \ref{trivial NCP T^n-bundles from trivial algebra bundles} that the corresponding dynamical system 
\[(\Gamma\mathbb{A}_{\{f_i\}},G,\beta_{\{f_i\}})
\]of Proposition \ref{T^n action on spec again} carries the structure of a smooth trivial NCP $G$-bundle.
\end{proof}

\begin{example}\label{non-triviality of the previous construction}{\bf(Non-triviality of the previous construction).}
In this example we show that the previous construction actually leads to non-trivial examples. For this we apply Theorem \ref{NCP G_bundle with fibre trivial NCP G bundle} to the trivial NCP $\mathbb{T}^n$-bundle $(\mathbb{T}^n_{\theta},\mathbb{T}^n,\alpha)$ (cf. Example \ref{Smooth NC n-tori as NCPTB}). In view of Example \ref{G-automorphism group of quantumtori} we have 
\begin{align}
\Aut_{\mathbb{T}^n}(\mathbb{T}^n_{\theta})\cong \mathbb{T}^n.\label{referenz G-automorphism group of quantumtori}
\end{align}
In particular, a similar argument as in Proposition \ref{vector bundles are associated to principal bundles} implies that there is a one-to-one correspondence between the algebra bundles arising from Proposition \ref{cocycle description of bundles} and principal $\mathbb{T}^n$-bundles. Thus, if $(\mathbb{A},M,\mathbb{T}^n_{\theta},q)$ is such an algebra bundle which corresponds to a non-trivial principal $\mathbb{T}^n$-bundle, then also $(\mathbb{A},M,\mathbb{T}^n_{\theta},q)$ is non-trivial as algebra bundle. We claim that the associated smooth dynamical system $(\Gamma\mathbb{A},\mathbb{T}^n,\beta)$ of Proposition \ref{dynamical system and algebra bundles} is a non-trivial NCP $\mathbb{T}^n$-bundle. To prove this claim we assume the converse, i.e., that $(\Gamma\mathbb{A},\mathbb{T}^n,\beta)$ is a trivial NCP $\mathbb{T}^n$-bundle. For this, we rename the open subsets $U_i$ of $M$ of Theorem \ref{NCP G_bundle with fibre trivial NCP G bundle} to $O_i$ in order to avoid an abuse of notation and proceed as follows:

(i) We first recall that $\mathbb{T}^n_{\theta}$ is generated by unitaries $U_1,\ldots,U_n$ and that its elements are given by (norm-convergent) sums
\[a=\sum_{{\bf k}\in\mathbb{Z}^n}a_{\bf k}U^{\bf k},\,\,\,\text{with}\,\,\,(a_{\bf k})_{{\bf k}\in\mathbb{Z}^n}\in S(\mathbb{Z}^n).
\]Here, 
\[U^{\bf k}:=U^{k_1}_1\cdots U^{k_n}_n.
\]

(ii) For each $1\leq r\leq n$ let $(\Gamma\mathbb{A})_r$ be the isotypic component corresponding to the canonical basis element $e_r=(0,\ldots,1,\ldots,0)$ of $\mathbb{Z}^n$. Then the definition of the action $\beta$ implies that
\begin{align}
(\Gamma\mathbb{A})_r&=\{s\in\Gamma\mathbb{A}:\,(\forall z\in\mathbb{T}^n)\,z.s=z_r\cdot s\}\notag\\
&=\{s\in\Gamma\mathbb{A}:\,(\forall m\in M)\,s(m)\in(\mathbb{T}^n_{\theta})_{m,r}\},\notag
\end{align}
where $(\mathbb{T}^n_{\theta})_{m,r}$ denotes the isotypic component of the fibre $(\mathbb{T}^n_{\theta})_m$ corresponding to $e_r$.
Since $(\Gamma\mathbb{A},\mathbb{T}^n,\beta)$ is assumed to be a trivial NCP $\mathbb{T}^n$-bundle, we may choose in each space $(\Gamma\mathbb{A})_r$ an invertible element $s_r:M\rightarrow\Gamma\mathbb{A}$.

(iii) If $s_{r,i}:=\pr_{\mathbb{T}^n_{\theta}}\circ\varphi_i^{-1}\circ {s_r}_{\mid O_i}:O_i\rightarrow \mathbb{T}^n_{\theta}$ (cf. Construction \ref{top on space of sections}), then $s_{r,i}=\lambda_{r,i}\cdot U_r$ for some smooth function $\lambda_{r,i}:O_i\rightarrow\mathbb{C}^{\times}$. Moreover, the compatibility property for $(s_{r,i})_{i\in I}$ and (\ref{referenz G-automorphism group of quantumtori}) imply that $\vert\lambda_{r,j}\vert=\vert\lambda_{r,i}\vert$ for all $i,j\in I$. Hence, we may choose a section $s_r:M\rightarrow\Gamma\mathbb{A}$ which is locally given by $\lambda_{r,i}\cdot U_r$ for some smooth function $\lambda_{r,i}:O_i\rightarrow\mathbb{T}$.

(iv) We now show that the map
\[\varphi:M\times\mathbb{T}^n_{\theta}\rightarrow\mathbb{A},\,\,\,\left(m,a=\sum_{{\bf k}\in\mathbb{Z}^n}a_{\bf k}U^{\bf k}\right)\mapsto\sum_{{\bf k}\in\mathbb{Z}^n}a_{\bf k}s_1(m)^{k_1}\cdots s_n(m)^{k_n}
\]is an equivalence of algebra bundles over $M$. Indeed, $\varphi$ is bijective and fibrewise an algebra automorphism. Moreover, the map $\varphi$ is smooth if and only if the map 
\[\psi_i:=\pr_{\mathbb{T}^n_{\theta}}\circ\varphi_i^{-1}\circ\varphi_{\mid O_i\times \mathbb{T}^n_{\theta}}:O_i\times \mathbb{T}^n_{\theta}\rightarrow \mathbb{T}^n_{\theta},\,\,\,\left(x,a=\sum_{{\bf k}\in\mathbb{Z}^n}a_{\bf k}U^{\bf k}\right)\mapsto\sum_{{\bf k}\in\mathbb{Z}^n}a_{\bf k}s_{1,i}(x)^{k_1}\cdots s_{n,i}(x)^{k_n}
\]is smooth for each $i\in I$. Since 
\[\psi_i(x,a)=\sum_{{\bf k}\in\mathbb{Z}^n}a_{\bf k}\lambda_{1,i}(x)^{k_1}\cdots\lambda_{n,i}(x)^{k_n}U^{\bf k},
\]we conclude that $\psi_i=\alpha\circ((\lambda_{1,i},\ldots,\lambda_{n,i})\times\id_{\mathbb{T}^n_{\theta}})$, i.e., that $\psi_i$ is smooth as a composition of smooth maps. A similar argument shows the smoothness of the inverse map.

(v) We finally achieve the desired contradiction: In view of part (iv), $\mathbb{A}$ is a trivial algebra bundle contradicting the construction of $\mathbb{A}$, i.e., that $\mathbb{A}$ is non-trivial as algebra bundle. This proves the claim.
\end{example}

\begin{remark}\label{nontriviality of the bundle corresponding to A^2}
In this remark we want to point out that there exist non-trivial algebra bundles which are trivial as NCP $G$-bundles: In fact, [GVF01], Corollary 12.7 implies that the $2$-tori $A^2_{\theta}$ are mutually non-isomorphic for $0\leq\theta\leq\frac{1}{2}$.
Further, if $\theta$ is rational, $\theta=\frac{n}{m}$, $n\in\mathbb{Z}$, $m\in\mathbb{N}$ relatively prime, then $A^2_{\theta}$ is isomorphic to the algebra of continuous sections of an algebra bundle over $\mathbb{T}^2$ with fibre $M_m(\mathbb{C})$ (cf. [GVF01], Proposition 12.2 or Proposition \ref{rational quantumtori} for the smooth case). Therefore, the (continuous) bundle corresponding to such a rational quantum torus $A^2_{\theta}$, $\theta$ rational with $0<\theta\leq\frac{1}{2}$, is non-trivial as algebra bundle. Since the smooth noncommutative $2$-torus $\mathbb{T}^2_{\theta}$ is a dense subalgebra of $A^2_{\theta}$, the same conclusion holds for the (smooth) bundle corresponding to $\mathbb{T}^2_{\theta}$. Nevertheless, the associated dynamical systems $(A^2_{\theta},\mathbb{T}^n,\alpha)$ and $(\mathbb{T}^2_{\theta},\mathbb{T}^n,\alpha)$ of Example \ref{NC n-tori as NCPTB} and Example \ref{Smooth NC n-tori as NCPTB} are trivial NCP $\mathbb{T}^2$-bundles.
\end{remark}

\subsection*{Example 3: Sections of Algebra Bundles which are Pull-Backs of Principal $G$-Bundles}\label{examples of NCP T^n-bundles III}

We show that if $A$ is a unital Fr\'echet algebra with trivial center, $(\mathbb{A},M,A,q)$ an algebra bundle and $(P,M,G,\pi,\sigma)$ a principal $G$-bundle, then the algebra of sections of the pull-back bundle 
\[\pi^{*}(\mathbb{A}):=\{(p,a)\in P\times\mathbb{A}:\,\pi(p)=q(a)\}
\]\sindex[n]{$\pi^{*}(\mathbb{A})$} is a NCP $G$-bundle. We start with the following lemma:

\begin{lemma}\label{pull-back 0}
If $A$ is a unital locally convex algebra and $M$ a manifold, then the map
\[\alpha:G\times C^{\infty}(M\times G,A)\rightarrow C^{\infty}(M\times G,A),\,\,\,(g,f)\mapsto (g.f)(m,h):=f(m,gh)
\]defines a smooth action of $G$ on $C^{\infty}(M\times G,A)$ by algebra automorphisms. In particular, the triple $(C^{\infty}(M\times G,A),G,\alpha)$\sindex[n]{$(C^{\infty}(M\times G,A),G,\alpha)$} is a smooth trivial NCP $G$-bundle.
\end{lemma}

\begin{proof}
\,\,\,For the proof we just have to note that the algebra $C^{\infty}(M\times G)$ is naturally embedded in $C^{\infty}(M\times G,A)$ through the unit element of $A$. The smoothness of the map $\alpha$ can be proved similarly to Lemma \ref{dynamical system and trivial algebra bundles}.
\end{proof}

\begin{lemma}\label{pull-back I}
If $(\mathbb{A},M,A,q)$ is an algebra bundle, $(P,M,G,\pi,\sigma)$ a principal bundle and $\pi^{*}(\mathbb{A})$ the pull-back bundle over $P$, then the map
\[\sigma:\pi^{*}(\mathbb{A})\times G\rightarrow\pi^{*}(\mathbb{A}),\,\,\,((p,a),g)\mapsto(p.g,a)
\]defines a smooth action of $G$ on $\pi^{*}(\mathbb{A})$.
\end{lemma}

\begin{proof}
\,\,\,We first note that the map $\sigma$ is well-defined. Its smoothness follows from local considerations and we leave the details to the reader.
\end{proof}

\begin{proposition}\label{pull-back II}
Suppose we are in the situation of Lemma \ref{pull-back I}. If $\mathcal{A}:=\Gamma\pi^{*}(\mathbb{A})$, then the map
\[\alpha:G\times\mathcal{A}\rightarrow\mathcal{A},\,\,\,\alpha(g,s)(p):=\sigma(s(p.g),g^{-1})
\]defines a smooth action of $G$ on $\mathcal{A}$ by algebra automorphisms. In particular, the triple $(\mathcal{A},G,\alpha)$ is a smooth dynamical system.
\end{proposition}

\begin{proof}
\,\,\,An easy calculation shows that the map $\alpha$ is a well-defined action of $G$ on $\mathcal{A}$ by algebra automorphisms. Next, we choose a bundle atlas $(\varphi_i,U_i)_{i\in I}$ of $(\mathbb{A},M,A,q)$ with the additional property that each $V_i:=\pi^{-1}(U_i)$ is trivial. Then $(\pi^{*}(\varphi_i),V_i)_{i\in I}$ is a bundle atlas of the pull-back bundle $\pi^{*}(\mathbb{A})$ over $P$ and we can use the definition of the smooth structure on $\mathcal{A}$ to verify the smoothness of the map $\alpha$: In fact, Remark \ref{smooth structure on sections} implies that $\alpha$ is smooth if and only if each map 
\[\Phi_i\circ\alpha:G\times\mathcal{A}\rightarrow C^{\infty}(V_i,A),\,\,\,(g,s)\mapsto s_i\circ(\sigma_g)_{\mid V_i}
\]is smooth. For this we note that each map
\[\alpha_i:G\times C^{\infty}(V_i,A)\rightarrow C^{\infty}(V_i,A),\,\,\,(g,f)\mapsto f\circ(\sigma_g)_{\mid V_i}
\]is smooth and further that $\Phi_i\circ\alpha=\alpha_i\circ(\id_G\times\Phi_i)$ holds for each $i\in I$. Thus, each map $\Phi_i\circ\beta$ is smooth as a composition of smooth maps.
\end{proof}

\begin{theorem}\label{pull-back IV}
Let $A$ be a unital Fr\'echet algebra with trivial center, $(\mathbb{A},M,A,q)$ an algebra bundle and $(P,M,G,\pi,\sigma)$ a principal bundle. If $\pi^{*}(\mathbb{A})$ is the pull-back bundle over $P$ and $\mathcal{A}:=\Gamma\pi^{*}(\mathbb{A})$\sindex[n]{$\mathcal{A}$}, then the smooth dynamical system $(\mathcal{A},G,\alpha)$\sindex[n]{$(\mathcal{A},G,\alpha)$} of Proposition \ref{pull-back II} is a smooth NCP $G$-bundle.
\end{theorem}

\begin{proof}
\,\,\,
(i) We first note that $C_{\mathcal{A}}\cong C^{\infty}(P)$ and therefore that $Z=C_{\mathcal{A}}^G\cong C^{\infty}(M)$ (cf. Proposition \ref{fixed point algebra of principal bundles}). In particular, the spectrum $\Gamma_Z$ is homeomorphic to $M$.

(ii) Next, we choose an open cover $(U_i)_{i\in I}$ of $M$ such that $\mathbb{A}_{U_i}\cong U_i\times A$ and $V_i:=P_{U_i}\cong U_i\times G$ holds for each $i\in I$. Further, we choose for each $i\in I$ a $U_i$-defining function $f_i$ and note that each function $h_i:=f_i\circ q$ is a (smooth) $V_i$-defining function $f_i$.

(iii) For $p\in P$ we choose $i\in I$ with $q(p)\in U_i$. Then $h_i$ is an element in $Z$ with $h_i(p)\neq 0$ and we conclude from Corollary \ref{algebra section 3} that the map
\[\phi_{V_i}:\mathcal{A}_{\{h_i\}}\rightarrow\mathcal{A}_{V_i},\,\,\,[F]\mapsto F\circ\left(\frac{1}{h_i}\times\id_{V_i}\right)
\]is an isomorphism of unital Fr\'echet algebras. Moreover, the map $\phi_{V_i}$ is $G$-equivariant. In fact, we have
\[(\phi_{V_i}\circ\alpha_{\{h_i\}})(g,[F])=\alpha(g,\phi_{V_i}([F]))
\]for all $g\in G$ and $[F]\in\mathcal{A}_{\{h_i\}}$. Since the natural isomorphism between the space $\mathcal{A}_{V_i}$ and $C^{\infty}(U_i\times G,A)$ is also a $G$-equivariant isomorphism of unital Fr\'echet algebras, the dynamical system 
\[(\mathcal{A}_{\{h_i\}},G,\alpha_{\{h_i\}})
\]carries the structure of a smooth trivial NCP $G$-bundle (cf. Lemma \ref{pull-back 0}).
\end{proof}

\begin{example}\label{non-triviality of the previous construction for example 3}{\bf(Non-triviality of the previous construction).}
In this example we show that the previous construction actually leads to non-trivial examples. Therefore we choose $n\in\mathbb{N}$ with $n>1$ and write 
\[C_n:=\{z\in\mathbb{C}^{\times}:\,z^n=1\}=\{\zeta^k:\,\zeta:=\exp(\frac{2\pi i}{n}),\,k=0,1,\ldots,n-1\}
\]for the cyclic subgroup of $\mathbb{T}$ of $n$-th roots of unity (cf. Chapter \ref{trivial ncp cyclic bundles}). We further choose $m\in\mathbb{N}$ such that $\zeta^m=\zeta$ (e.g. $m=n+1$). Then Remark \ref{nontriviality of the bundle corresponding to A^2} implies that $\mathbb{T}^2_{\frac{1}{m}}$ is isomorphic to the space of sections of a non-trivial algebra bundle $\mathbb{A}$ over $\mathbb{T}^2$ with fibre $M_m(\mathbb{C})$. Therefore, the pull-back along the non-trivial principal bundle (covering) defined by the natural action of $C_n\times C_n$ on $\mathbb{T}^2$, i.e., by 
\[(z_1,z_2).(\zeta^k,\zeta^l):=(\zeta^k\cdot z_1,\zeta^l\cdot z_2)
\]for $k,l=0,1,\ldots,n-1$, leads to a non-trivial algebra bundle $\pi^{*}(\mathbb{A})$ over $\mathbb{T}^2$ with fibre $M_m(\mathbb{C})$. We claim that the associated smooth dynamical system $(\mathcal{A},C_n\times C_n,\alpha)$ of Proposition \ref{pull-back II} is a non-trivial NCP $C_n\times C_n$-bundle. To prove this claim we assume the converse, i.e., that $(\mathcal{A},C_n\times C_n,\alpha)$ is a trivial NCP $C_n\times C_n$-bundle and proceed as follows:

(i) In the following let $(\varphi_i,U_i)_{i\in I}$ be a bundle atlas of the pull-back bundle $\pi^{*}(\mathbb{A})$ over $\mathbb{T}^2$. For a section $s\in\mathcal{A}$ and $i\in I$ we write 
\[s_i:=\pr_{M_m(\mathbb{C})}\circ\varphi_i^{-1}\circ s_{\mid U_i}
\]for the corresponding function in $C^{\infty}(U_i,M_m(\mathbb{C}))$ (cf. Construction \ref{top on space of sections}). We recall that $s_j(z)=g_{ji}(z)\cdot s_i(z)$ holds for all $i,j\in I$ and $z\in U_i\cap U_j$, where 
\[g_{ji}:(U_i\cap U_j)\times M_m(\mathbb{C})\rightarrow M_m(\mathbb{C})
\]denotes the smooth map defined by the transition function $\varphi_j^{-1}\circ\varphi_i$.

(ii) Let $s\in\mathcal{A}$. We show that the map
\[\Det(s):\mathbb{T}^2\rightarrow\mathbb{C},\,\,\,\Det(s)(z):=\det(s_i(z)),
\]for $i\in I$ with $z\in U_i$, is well-defined and smooth: In fact, the crucial point is to show that the map $\Det(s)$ is well-defined: For this let $i,j\in I$ with $z\in U_i\cap U_j$. Since each automorphism of the matrix algebra $\M_m(\mathbb{C})$ is inner (cf. Lemma \ref{aut=inn}), we easily conclude that
\[\det(s_j(z))=\det(g_{ji}(z)\cdot s_i(z))=\det(s_i(z)).
\]The smoothness of the map $\Det(s)$ follows from the local description by a smooth function.

(iii) Since $(\mathcal{A},C_n\times C_n,\alpha)$ is assumed to be a trivial NCP $C_n\times C_n$-bundle, there exist an invertible elements $F\in\mathcal{A}_{(1,0)}$ and an invertible element $F'\in\mathcal{A}_{(0,1)}$. Further, there exist two elements $G$ and $G'$ in $\mathcal{A}_{(0,0)}$ such that $G^n=F^n$ and $(G')^n=(F')^n$. 

(iv) Part (ii) now implies that $f:=\Det(F)\in C^{\infty}(\mathbb{T}^2)$. Since $F$ is invertible, so is $f$, i.e., $f$ takes values in $\mathbb{C}^{\times}$. Moreover, the function $f$ satisfies
\[(\zeta^k,\zeta^l).f=(\zeta^k,\zeta^l).\Det(F)=\Det(\zeta^k\cdot F)=(\zeta^k)^m\cdot\Det(F)=\zeta^k\cdot f,
\]for all $k,l=0,1,\ldots,n-1$. Thus, $f$ is an invertible element in $C^{\infty}(\mathbb{T}^2)_{(1,0)}$ (here we have used that the action of $C_n\times C_n$ on $\mathcal{A}$ restricts to an action on $C_{\mathcal{A}}\cong C^{\infty}(\mathbb{T}^2)$). Likewise, the function $g:=\Det(G)$ defines a smooth $C_n\times C_n$-invariant function, i.e., an element in $C^{\infty}(\mathbb{T}^2)_{(0,0)}$. We further conclude that $g^n=f^n$. In particular, $g$ is invertible. The same construction applied to $F'$ and $G'$ gives invertible elements $f'\in C^{\infty}(\mathbb{T}^2)_{(0,1)}$ and $g'\in C^{\infty}(\mathbb{T}^2)_{(0,0)}$ which satisfy $(g')^n=(f')^n$.

(v) Finally, part (iv) leads to the desired contradiction: Indeed, we conclude just as in the proof of Theorem \ref{C_nC_m IV} (a) that the smooth functions $h:=\frac{f}{g}$ and $h':=\frac{f'}{g'}$ have image $C_n$ and define an equivalence of principal $C_n\times C_n$-bundles over 
\[\mathbb{T}^2/(C_n\times C_n)\cong\mathbb{T}^2\]
given through the map
\[\varphi:\mathbb{T}^2\rightarrow \mathbb{T}^2/(C_n\times C_n)\times (C_n\times C_n),\,\,\,z\mapsto(\pr(z),h(z),h'(z)).
\]This is not possible, since $\mathbb{T}^2$ is connected.
\end{example}

\subsection*{Example 4: Sections of Trivial Equivariant Algebra Bundles}\label{examples of NCP T^n-bundles V}

We now consider again a principal bundle $(P,M,G,q,\sigma)$ and, in addition, a unital locally convex algebra $A$. If $\pi:G\times A\rightarrow A$ defines a smooth action of $G$ on $A$ by algebra automorphisms, then 
\[(p,a).g:=(p.g,\pi(g^{-1}).a)
\]defines a (free) action of $G$ on $P\times A$ and one easily verifies that the trivial algebra bundle $(P\times A,P,A,q_P)$ is $G$-equivariant. Moreover, a short observation shows that the map
\[\alpha:G\times C^{\infty}(P,A)\rightarrow C^{\infty}(P,A),\,\,\,(g.f)(p):=\pi(g).f(p.g)
\]defines a smooth action of $G$ on $C^{\infty}(P,A)$ by algebra automorphisms. We recall that the corresponding fixed point algebra 
is isomorphic (as $C^{\infty}(M)$-module) to the space of sections of the associated algebra bundle
\[\mathbb{A}:=P\times_{\pi}A:=P\times_GA:=(P\times A)/G
\]over $M$. In fact, we refer to Construction \ref{associated vector bundle} if $A$ is finite-dimensional. If $A$ is not finite-dimensional, then one can use that bundle charts for $(P,M,G,q,\sigma)$ induces bundle charts for the associated algebra bundle.

\begin{lemma}\label{V.1}
The above situation applied to the trivial principal bundle $(M\times G,M,G,q_M,\sigma_G)$ leads to a smooth trivial NCP $G$-bundle $(C^{\infty}(M\times G,A),G,\alpha)$ with fixed point algebra $C^{\infty}(M,A)$.
\end{lemma}

\begin{proof}
\,\,\,For the proof we again just have to note that the algebra $C^{\infty}(M\times G)$ is naturally embedded in $C^{\infty}(M\times G,A)$ through the unit element of $A$. The smoothness of the map $\alpha$ can be proved similarly to Lemma \ref{dynamical system and trivial algebra bundles}.
\end{proof}

\begin{theorem}\label{V.2}
If $(P,M,G,\pi,\sigma)$ is a principal bundle, $A$ a unital Fr\'echet algebra with trivial center and $\pi:G\times A\rightarrow A$ a smooth action of $G$ on $A$, then the smooth dynamical system $(C^{\infty}(P,A),G,\alpha)$\sindex[n]{$(C^{\infty}(P,A),G,\alpha)$} is a smooth NCP $G$-bundle.
\end{theorem}

\begin{proof}
\,\,\,
(i) We first note that $Z\cong C^{\infty}(M)$ (cf. Proposition \ref{fixed point algebra of principal bundles}). Hence, $\Gamma_Z$ is homeomorphic to $M$ by Proposition \ref{spec of C(M) top}.

(ii) Next, we choose an open cover $(U_i)_{i\in I}$ of $M$ such that each $P_i:=P_{U_i}$ is a trivial principal $G$-bundle over $U_i$, i.e., such that $P_i\cong U_i\times G$ holds for each $i\in I$. Further, we choose for all $i\in I$ a $U_i$-defining function $f_i$ and note that each function $h_i:=f_i\circ q$ is a (smooth) $P_i$-defining function $f_i$.

(iii) For $p\in P$ we choose $i\in I$ with $q(p)\in U_i$. Then $h_i$ is an element in $Z$ with $h_i(p)\neq 0$ and we conclude from Corollary \ref{algebra section 3} that the map
\[\phi_{P_i}:C^{\infty}(P,A)_{\{h_i\}}\rightarrow C^{\infty}(P_i,A),\,\,\,[F]\mapsto F\circ\left(\frac{1}{h_i}\times\id_{P_i}\right)
\]is an isomorphism of unital Fr\'echet algebras. Moreover, the map $\phi_{P_i}$ is $G$-equivariant. In fact, we have
\[(\phi_{P_i}\circ\alpha_{\{h_i\}})(g,[F])=\alpha(g,\phi_{P_i}([F]))
\]for all $g\in G$ and $[F]\in C^{\infty}(P,A)_{\{h_i\}}$. Since the natural isomorphism between $C^{\infty}(P_i,A)$ and $C^{\infty}(U_i\times G,A)$ is also a $G$-equivariant isomorphism of unital Fr\'echet algebras, the dynamical system 
\[(C^{\infty}(P_i,A)_{\{h_i\}},G,\alpha_{\{h_i\}})
\]carries the structure of a smooth trivial NCP $G$-bundle (cf. Lemma \ref{V.1}).
\end{proof}

\begin{example}\label{V.3}
We want to apply Theorem \ref{V.2} to the the non-trivial principal bundle (covering) defined by the natural action of $C_m\times C_m$ on $\mathbb{T}^2$, i.e., by 
\[(z_1,z_2).(\zeta^k,\zeta^l):=(\zeta^k\cdot z_1,\zeta^l\cdot z_2)
\]for $k,l=0,1,\ldots,m-1$, the algebra $M_m(\mathbb{C})$ and the action of $C_m\times C_m$ on $M_m(\mathbb{C})$ defined by 
\[(\zeta^k,\zeta^l).A:=R^lS^kAS^{-k}R^{-l}
\]for $k,l=0,1,\ldots,m-1$. Here, $R$ and $S$ are defined as in Proposition \ref{rational quantumtori} ($\theta=\frac{1}{m}$). The corresponding smooth dynamical system $(C^{\infty}(\mathbb{T}^2,M_m(\mathbb{C})),C_m\times C_m,\alpha)$ is a NCP $C_m\times C_m$-bundle with fixed point algebra the rational quantum torus $\mathbb{T}^2_{\frac{1}{m}}$ (cf. Proposition \ref{rational quantumtori}). According to Example \ref{the matrix algebra}, the algebra $M_m(\mathbb{C})$ carries the structure of a trivial NCP $C_m\times C_m$-bundle. Therefore, it turns out that the same holds for $(C^{\infty}(\mathbb{T}^2,M_m(\mathbb{C})),C_m\times C_m,\alpha)$ since $M_m(\mathbb{C})$ is naturally embedded in $C^{\infty}(\mathbb{T}^2,M_m(\mathbb{C}))$ through the constant maps.
\end{example}

\begin{remark}\label{non-triviality of the previous construction for example 4}{\bf(Non-triviality of the previous construction).}
Non-trivial examples can be constructed similarly as in Example \ref{non-triviality of the previous construction for example 3}.
\end{remark}

\subsection*{Example 5: A Very Concrete Example of a NCP Torus Bundle}\label{examples of NCP T^n-bundles IV}

Let $\pi:\mathbb{R}\rightarrow\mathbb{T}$ be the universal covering of $\mathbb{T}$, $A=C^{\infty}(\mathbb{T}^2)$ and note that the map
\[\gamma:\mathbb{Z}\rightarrow\Aut(A),\,\,\,(\gamma(n).f)(z_1,z_2)=f(z_1,z_1^nz_2)
\]defines a (smooth) action of $\mathbb{Z}$ on $A$. We thus can form the associated algebra bundle
\begin{align}
q:\mathbb{A}:=\mathbb{R}\times_{\gamma} A\rightarrow\mathbb{T},\,\,\,[r,f]\mapsto\pi(r),\label{associated algebra bundle}
\end{align}
with fibre $A$.

\begin{proposition}
The space $\Gamma\mathbb{A}$ of section of the bundle \emph{(}\ref{associated algebra bundle}\emph{)} carries the structure of a non-trivial NCP $\mathbb{T}$-bundle.
\end{proposition}

\begin{proof}
\,\,\,The proof of this claim is divided into the following five steps:

(i) We first note that the map
\[\Psi:C^{\infty}(\mathbb{R},A)^{\mathbb{Z}}\rightarrow\Gamma\mathbb{A},\,\,\,\Psi(f)(\pi(r)):=[r,f(r)],
\]is an isomorphism of unital Fr\'echet algebras. Indeed, both $C^{\infty}(\mathbb{R},A)^{\mathbb{Z}}$ and $\Gamma\mathbb{A}$ are unital Fr\'echet algebras. An easy calculation shows that the map $\Phi$ is an isomorphism of unital algebras, and its continuity is a consequence of the topology on $\Gamma\mathbb{A}$ (cf. Definition \ref{top on space of sections}). Finally, Proposition \ref{open mapping theorem} implies that $\Psi$ is open.

(ii) From Lemma \ref{smooth exp law}, we conclude that
\[C^{\infty}(\mathbb{R},A)^{\mathbb{Z}}=\{f:\mathbb{R}\rightarrow A:\,(\forall r\in\mathbb{R},n\in\mathbb{Z})\,f(r+n)=\gamma(-n).f(r)\}
\]is isomorphic (as a unital Fr\'echet algebra) to
\begin{align}
\{f^{\wedge}:\mathbb{R}\times\mathbb{T}^2\rightarrow\mathbb{C}:\,(\forall r\in\mathbb{R},(z_1,z_2)\in\mathbb{T}^2,n\in\mathbb{Z})\,f^{\wedge}(r+n,z_1,z_2)=f^{\wedge}(r,z_1,z_1^{-n}z_2)\}.\label{equation example}
\end{align}

(iii) The action of $\mathbb{Z}$ on $\mathbb{R}\times\mathbb{T}^2$ corresponding to (\ref{equation example}) is given by 
\[\sigma:\mathbb{R}\times\mathbb{T}^2\times\mathbb{Z}\rightarrow\mathbb{R}\times\mathbb{T}^2,\,\,\,((r,z_1,z_2),n)\mapsto(r+n,z_1,z_1^nz_2).
\]Moreover, this action is free. To verify its properness, we take two compact subsets $K$ and $L$ of $\mathbb{R}\times\mathbb{T}^2$. Then there exists $n_0\in\mathbb{N}$ such that the projection of $K$ onto $\mathbb{R}$ is contained in the interval $[-n_0,n_0]$. This implies that the set $\{n\in\mathbb{Z}:\,K\cap L.n\neq\emptyset\}$ is finite and thus that $\sigma$ is proper (cf. [Ne08b], Example 1.2.2 (a)). In view of the Quotient Theorem (cf. Remark \ref{principal bundles II} (c)), we get a $\mathbb{Z}$-principal bundle
\[(\mathbb{R}\times\mathbb{T}^2,M,\mathbb{Z},\pr,\sigma),
\]where $M:=(\mathbb{R}\times\mathbb{T}^2)/\mathbb{Z}$ and $\pr:\mathbb{R}\times\mathbb{T}^2\rightarrow M$ denotes the canonical quotient map. In particular, we conclude from (i), (ii) and the previous discussion that $\Gamma\mathbb{A}$ is canonically isomorphic to $C^{\infty}(M)$. 


(iv) Next, one easily verifies that the map
\[\sigma':M\times\mathbb{T}\rightarrow M,\,\,\,([(r,z_1,z_2)],z)\mapsto [(r,z_1,z_2z)]
\]defines a smooth action of $\mathbb{T}$ on $M$, which is is free and proper. We thus get a $\mathbb{T}$-principal bundle
\[(M,M/\mathbb{T},\mathbb{T},\pr',\sigma'),
\]where $\pr':M\rightarrow M/\mathbb{T}$ denotes the canonical quotient map. In particular, Theorem \ref{NCT^nB for manifold again} (b) implies that the triple $(C^{\infty}(M),\mathbb{T},\alpha)$ is a smooth NCP $\mathbb{T}$-bundle. 

(v) It remains to show that $M$ is non-trivial as principal $\mathbb{T}$-bundle over $M/\mathbb{T}\cong\mathbb{T}^2$. For this it is enough to show that $\pi_1(M)\ncong\mathbb{Z}^3$. Indeed, if $M\cong\mathbb{T}^3$, then $\pi_1(M)\cong\mathbb{Z}^3$. We proceed as follows: Let $\mathbb{R}\ltimes_S\mathbb{R}^2$ be the (left-) semidirect product of $\mathbb{R}$ and $\mathbb{R}^2$ defined by the homomorphism
\[S:\mathbb{R}\rightarrow\Aut(\mathbb{R}^2),\,\,\,S(r)(x,y):=(x,y+rx).
\]Then the assignment
\[(r,x,y).n:=(n,0,0)(r,x,y)=(r+n,x,y+rx),\,\,\,r,x,y\in\mathbb{R}, n\in\mathbb{Z},
\]is the lifting of the action $\sigma$ of part (iii) to (the universal covering) $\mathbb{R}^3$ ($\cong\mathbb{R}\ltimes_S\mathbb{R}^2$ as manifolds). In particular, the homogeneous space defined by the discrete subgroup $\mathbb{Z}\ltimes_S\mathbb{Z}^2$ of $\mathbb{R}\ltimes_S\mathbb{R}^2$ is diffeomorphic to $M$, i.e.,
\[(\mathbb{R}\ltimes_S\mathbb{R}^2)/(\mathbb{Z}\ltimes_S\mathbb{Z}^2)\cong M.
\]We thus conclude from the corresponding exact sequence of homotopy groups (cf. [Br93], Theorem VII.6.7 or [Ne08b], Theorem 6.3.17) that the map
\[\delta_1:\pi_1(M)\rightarrow\pi_0(\mathbb{Z}\ltimes_S\mathbb{Z}^2)\cong\mathbb{Z}\ltimes_S\mathbb{Z}^2,\,\,\,\delta_1([\gamma]):=(n,m,m'),
\]where $(n,m,m')\in\mathbb{Z}\ltimes_S\mathbb{Z}^2$ is such that $\widetilde{\gamma}(1)=\widetilde{\gamma}(0).(n,m,m')$ holds for a continuous lift $\widetilde{\gamma}:[0,1]\rightarrow\mathbb{R}\ltimes_S\mathbb{R}^2$ of the loop $\gamma$ in $M$ with $\widetilde{\gamma}(0)=(0,0,0)$, is an isomorphism of groups. This proves the claim.
\end{proof}

\chapter{Characteristic Classes of Lie Algebra Extensions}\label{chapter characteristic classes of lie algebra extensions}

Characteristic classes are topological invariants of principal bundles and vector bundles associated to principal bundles.
The theory of characteristic classes was started in the 1930s by Stiefel and Whitney. Stiefel studied certain homology classes of the tangent bundle $TM$ of a smooth manifold $M$, while Whitney considered the case of an arbitrary sphere bundle and introduced the concept of a characteristic cohomology class. In the next decade, Pontryagin constructed important new characteristic classes by studying the homology of real Grassmann manifolds and Chern defined characteristic classes for complex vector bundles. Nowadays, characteristic classes are an important tool for both mathematics and modern physics. For example, they provide a way to measure the non-triviality of a principal bundle respectively the non-triviality of an associated vector bundle. The Chern--Weil homomorphism of a principal bundle is an algebra homomorphism from the algebra of polynomials invariant under the adjoint action of a Lie group $G$ on the corresponding Lie algebra $\mathfrak{g}$ into the even de Rham cohomology $H_{\text{dR}}^{2\bullet}(M,\mathbb{K})$\sindex[n]{$H_{\text{dR}}^{2\bullet}(M,\mathbb{K})$} of the base space $M$ of a principal bundle $P$. This map is achieved by evaluating an invariant polynomial $f$ of degree $k$ on the curvature $\Omega$ of a connection $\omega$ on $P$ and thus obtaining a closed form on the base. A nice reference for these statements can be found in [KoNo69], Chapter XII.

Around 1970, another set of characteristic classes called the \emph{secondary characteristic classes} have been discovered. The secondary characteristic classes are also topological invariants of principal bundles which are derived from the curvature of adequate connections. They appear for example in the Lagrangian formulation of modern quantum field theories. The best known of these classes are the Chern--Simons classes. For a detailed background and examples of secondary characteristic classes arising in geometry we refer to [MaMa92].

In 1985, P. Lecomte described a general cohomological construction which generalizes the classical Chern--Weil homomorphism. This construction associates characteristic classes to every extension of Lie algebras. The classical construction of Chern and Weil arises in this context from an extension of Lie algebras commonly known as the \emph{Atiyah sequence} (cf. Example \ref{Attia sequence}).

The aim of this chapter is to define secondary characteristic classes in the setting of Lie algebra extensions. This will also provide a new proof of Lecomte's construction.

\section{Covariant Derivatives and Curvature}

In this section we discuss the algebraic fundamentals of the geometric concepts of covariant derivative and curvature.

\begin{definition}\label{covariat differential}{\bf(The covariant differential).}\index{Covariant Differential}
Let $\mathfrak{g}$ be a Lie algebra and $V$ a vector space, considered as a trivial $\mathfrak{g}$-module, so that we have the corresponding Chevalley-Eilenberg complex $(C(\mathfrak{g},V),d_{\mathfrak{g}})$\sindex[n]{$(C(\mathfrak{g},V),d_{\mathfrak{g}})$}. We now twist the differential in this complex with a linear map
\[S:\mathfrak{g}\rightarrow\End(V),
\]i.e., an element $S\in C^1(\mathfrak{g},\End(V))$. For this we first note that the bilinear evaluation map 
\[\ev:\End(V)\times V\rightarrow V,\,\,\,(\varphi,v)\mapsto\varphi(v)
\]leads to a linear operator
\[S_{\wedge}:C^p(\mathfrak{g},V)\rightarrow C^{p+1}(\mathfrak{g},V),\,\,\,\alpha\mapsto S\wedge_{\ev}\alpha.
\]The corresponding \emph{covariant differential on} $C(\mathfrak{g},V)$ is defined by
\[d_S:=S_{\wedge}+d_{\mathfrak{g}}:C^p(\mathfrak{g},V)\rightarrow C^{p+1}(\mathfrak{g},V)\,\,\,p\in\mathbb{N}_0.
\]This can also be written as
\begin{align}
d_S\alpha(x_0,\ldots,x_p)&:=\sum^p_{j=0}(-1)^j S(x_j).\alpha(x_0,\ldots,\widehat{x}_j,\ldots,x_p)\notag\\
&+\sum_{i<j}(-1)^{i+j}\alpha([x_i,x_j],x_0,\ldots,\widehat{x}_i,\ldots,\widehat{x}_j,\ldots,x_p).\notag
\end{align}
\end{definition}

For the following propositions we recall Example \ref{lie algebra bracket} for the wedge product $[\cdot,\cdot]$ defined by the Lie bracket of a Lie algebra:

\begin{proposition}\label{5.14}
\,\,\,Let $R_S(x,y)=[S(x),S(y)]-S([x,y])$ for $x,y\in\mathfrak{g}$. Then
\[R_S=d_{\mathfrak{g}}S+\frac{1}{2}[S,S]\in C^2(\mathfrak{g},\End(V)),
\]and for $\alpha\in C^p(\mathfrak{g},V)$ we have
\begin{align}
d^2_S\alpha=R_S\wedge_{\ev}\alpha.\label{eqn.3}
\end{align}
In particular $d^2_S=0$ if and only if $S$ is a homomorphism of Lie algebras, i.e., $R_S=0$.
\end{proposition}
\begin{proof}
\,\,\,We divide the proof into two parts. For the definition of the wedge product $\wedge_C$ and the corresponding associativity properties we refer to Example \ref{wedgeproduct for composition}:

(i) For $\alpha\in C^p(\mathfrak{g},V)$ we get
\begin{align}
d^2_S\alpha &=d_S(S\wedge_{\ev}\alpha +d_{\mathfrak{g}}\alpha)\notag\\
&=(S\wedge_{\ev}(S\wedge_{\ev}\alpha))+S\wedge_{\ev}d_{\mathfrak{g}}\alpha +d_{\mathfrak{g}}(S\wedge_{\ev}\alpha)+d^2_{\mathfrak{g}}\alpha\notag\\
&=(S\wedge_C S)\wedge_{\ev}\alpha +S\wedge_{\ev}d_{\mathfrak{g}}\alpha +(d_{\mathfrak{g}}S\wedge_{\ev}\alpha -S\wedge_{\ev}d_{\mathfrak{g}}\alpha)\notag\\
&=(S\wedge_C S)\wedge_{\ev}\alpha +d_{\mathfrak{g}}S\wedge_{\ev}\alpha = (S\wedge_C S+d_{\mathfrak{g}}S)\wedge_{\ev}\alpha.\notag
\end{align}
Now,
\[(S\wedge_C S)(x,y)=S(x)S(y)-S(y)S(x)=[S(x),S(y)]=\frac{1}{2}[S,S](x,y)
\]proves (\ref{eqn.3}).

(ii) For $v\in V\cong C^0(\mathfrak{g},V)$ we obtain in particular $(d^2_Sv)(x,y)=R_S(x,y).v$,
showing that $d^2_S=0$ on $C^{\bullet}(\mathfrak{g},V)$ is equivalent to $R_S=0$, which means that $S:\mathfrak{g}\rightarrow(\End(V),[\cdot,\cdot])$ is a homomorphism of Lie algebras.
\end{proof}

The following proposition provides an abstract algebraic version of identities originating from the context of differential forms.

\begin{proposition}\label{abstract bianchi identity}
\,\,\,Suppose that $V$ is a Lie algebra, considered as a trivial $\mathfrak{g}$-module. Further, let $\sigma\in C^1(\mathfrak{g},V)$, $S:=\ad\circ\sigma$ and
\[R_{\sigma}:=d_{\mathfrak{g}}\sigma+\frac{1}{2}[\sigma,\sigma]\in C^2(\mathfrak{g},V),\,\,\,\text{i.e.},\,\,\,R_{\sigma}(x,y)=[\sigma(x),\sigma(y)]-\sigma([x,y]).
\]Then the following assertions hold:
\begin{itemize}
\item[\emph{(a)}]
$d^2_S\alpha=[R_{\sigma},\alpha]$ for $\alpha\in C^p(\mathfrak{g},V)$.
\item[\emph{(b)}]
$R_{\sigma}$ satisfies the abstract Bianchi identity\index{Bianchi Identity} $d_S R_{\sigma}=0$, i.e.,
\[\sum_{\text{cyc.}}[\sigma(x),R_{\sigma}(y,z)]-R_{\sigma}([x,y],z)=0.
\]
\item[\emph{(c)}]
For $\gamma\in C^1(\mathfrak{g},V)$ we have
\[R_{\sigma+\gamma}=R_{\sigma}+R_{\gamma}+[\sigma,\gamma]=R_{\sigma}+d_S\gamma+\frac{1}{2}[\gamma,\gamma].
\]
\end{itemize}
\end{proposition}

\begin{proof}\,\,\,(a) Since $\ad:V\rightarrow\End(V)$ is a homomorphism of Lie algebras, the definition of $R_{\sigma}$ and Proposition \ref{5.14} leads for $\alpha\in C^p(\mathfrak{g},V)$ to
\[d^2_S\alpha=R_S\wedge_{ev}\alpha=(\ad\circ R_{\sigma})\wedge_{ev}\alpha=[R_{\sigma},\alpha].
\]

(b) With Proposition \ref{wedge-product and differential} we get 
\[d_{\mathfrak{g}}[\sigma,\sigma]=[d_{\mathfrak{g}}\sigma,\sigma]-[\sigma,d_{\mathfrak{g}}\sigma]=2[d_{\mathfrak{g}}\sigma,\sigma].
\]Now the abstract Bianchi identity follows from
\begin{align}
d_SR_{\sigma}&=(d_{\mathfrak{g}}+S_{\wedge})R_{\sigma}=d_{\mathfrak{g}}^2\sigma+\frac{1}{2}d_{\mathfrak{g}}[\sigma,\sigma]+S\wedge R_{\sigma}=[d_{\mathfrak{g}}\sigma,\sigma]+[\sigma,R_{\sigma}]\notag\\
&=[d_{\mathfrak{g}}\sigma,\sigma]-[R_{\sigma},\sigma]=-\frac{1}{2}[[\sigma,\sigma],\sigma]=0.\notag
\end{align}

(c) This is an easy calculation.
\end{proof}

We finish this section with the following useful proposition:

\begin{proposition}\label{wedgeproduct and covariant differential}
Suppose that $m:V\times V\rightarrow V$ is a bilinear map and that $S:\mathfrak{g}\rightarrow\der(V,m)$ is linear. Then we have for $\alpha\in C^p(\mathfrak{g},V)$ and $\beta\in C^q(\mathfrak{g},V)$ the relation
\begin{align}
d_S(\alpha\wedge_m\beta)=d_S\alpha\wedge_m\beta+(-1)^p\alpha\wedge_m d_S\beta.
\end{align}
\end{proposition}

\begin{proof}
\,\,\,Since $d_S=d_{\mathfrak{g}}+S_{\wedge}$, Proposition \ref{wedge-product and differential} implies the assertion for $S=0$. It therefore remains to show that
\[S\wedge(\alpha\wedge_m\beta)=(S\wedge\alpha)\wedge_m\beta+(-1)^p\alpha\wedge_m(S\wedge\beta).
\]
We recall that
\[S\wedge(\alpha\wedge_m\beta)=\frac{1}{p!q!}\Alt(S\cdot(\alpha\cdot_m\beta))
\](cf. Remark \ref{definition of alt})and note that $S(\mathfrak{g})\subseteq\der(V,m)$ implies that
\[S\cdot(\alpha\cdot_m\beta)=(S\cdot\alpha)\cdot_m\beta+(\alpha\cdot_m(S\cdot\beta))^{\sigma},
\]where $\sigma=(1 2\ldots p+1)\in S_{p+q+1}$ is a cycle of length $p+1$. Mow the proposition follows from $\sgn(\sigma)=(-1)^p$.
\end{proof}

\section{Lecomte's Generalization of the Chern--Weil Map}
In this section we give a short overview over Lecomte's generalization of the Chern--Weil map. First of all we have to explain how to multiply $k$-tuples of Lie algebra cochains with $k$-linear maps. In the following let $\mathfrak{n}$ and $\mathfrak{g}$ be Lie algebras and $V$ be a vector space.

\begin{definition}\label{k-fold wedge product}
For a $k$-linear map $f:\mathfrak{n}^k\rightarrow V$ and $\varphi_1,\ldots,\varphi_k\in C^2(\mathfrak{g},\mathfrak{n})$ we define
\[f_{\varphi_1,\ldots,\varphi_k}(x_1,\ldots,x_{2k}):=\sum_{\sigma_{2i-1}<\sigma_{2i}}\sgn(\sigma)f(\varphi_1(x_{\sigma_1},x_{\sigma_2}),\ldots,\varphi_k(x_{\sigma_{2k-1}},x_{\sigma_{2k}})).
\]This means that
\[f_{\varphi_1,\ldots,\varphi_k}=\frac{1}{2^k}\Alt(f\circ(\varphi_1,\ldots,\varphi_k)),
\]\sindex[n]{$f_{\varphi_1,\ldots,\varphi_k}$}so that we may consider $f_{\varphi_1,\ldots,\varphi_k}$ as a $k$-fold wedge product of the $\varphi_i$, defined by $f$.
\end{definition}

\begin{remark}
(a) If $f\in\Mult^k(\mathfrak{n},V)$ is not symmetric there is always a symmetric map 
\[f_s\in\Sym^k(\mathfrak{n},V)\,\,\,\text{with}\,\,\,f_{\varphi_1,\ldots,\varphi_k}=(f_s)_{\varphi_1,\ldots,\varphi_k}.
\]Indeed, if $f\in\Mult^k(\mathfrak{n},V)$ is not symmetric and $f_s:=\frac{1}{k!}\sum_{\sigma\in S_k}f^{\sigma}$ (cf. Remark \ref{definition of alt}), then we have $f^{\sigma}_{\varphi_1,\ldots,\varphi_k}=f_{\varphi_1,\ldots,\varphi_k}$ because for any permutation $\sigma\in S_k$ the relation 
\[f\circ(\varphi_{\sigma(1)},\ldots,\varphi_{\sigma(k)})=(f\circ(\varphi_1,\ldots,\varphi_k))^{\tilde{\sigma}}
\]holds for an even permutation $\tilde{\sigma}\in S_{2k}$. In the following we will therefore just be concerned with symmetric maps.

(b) Now, we may also consider $f\in\Sym^k(\mathfrak{n},V)$ as a linear map on the \emph{symmetric tensor product}\index{Symmetric Tensor Product} (cf. [La02])
\[\tilde{f}:S^k(\mathfrak{n})\rightarrow V\,\,\,\text{with}\,\,\,\tilde{f}(x_1\otimes_s\cdots\otimes_s x_k)=f(x_1,\ldots,x_k).
\]Writing $\wedge_{\otimes_s}$ for the wedge products
\[C^p(\mathfrak{g},S^k(\mathfrak{n}))\times C^q(\mathfrak{g},S^m(\mathfrak{n}))\rightarrow C^{p+q}(\mathfrak{g},S^{k+m}(\mathfrak{n}))
\]defined by the canonical multiplication $S^k(\mathfrak{n})\times S^m(\mathfrak{n})\rightarrow S^{k+m}(\mathfrak{n})$, $(x,y)\mapsto x\otimes_s y$, we find the formula
\[f_{\varphi_1,\ldots,\varphi_k}=\tilde{f}\circ(\varphi_1\wedge_{\otimes_s}\cdots\wedge_{\otimes_s}\varphi_k),
\]expressing $f_{\varphi_1,\ldots,\varphi_k}$ simply as a composition of a linear map with an iterated wedge product.
\end{remark}
%
%
We now come to the central definition of this chapter:
\begin{definition}\label{g-invariant symmetric k-linear maps}
(a) Let 
\begin{align}
0\longrightarrow\mathfrak{n}{\longrightarrow}\widehat{\mathfrak{g}}\stackrel{q}{\longrightarrow}\mathfrak{g}\longrightarrow0\notag
\end{align}
be an extension of Lie algebras and $V$ be a $\mathfrak{g}$-module which we also consider as a $\widehat{\mathfrak{g}}$-module with respect to the action $x.v:=q(x).v$ for $x\in\widehat{\mathfrak{g}}$. Further, let $\sigma:\mathfrak{g}\rightarrow\widehat{\mathfrak{g}}$ be a linear section and define $R_{\sigma}\in C^2(\mathfrak{g},\mathfrak{n})$ by
\[R_{\sigma}(x,y)=[\sigma(x),\sigma(y)]-\sigma([x,y]).
\]For $f\in\Sym^k(\mathfrak{n},V)$ we define
\[f_{\sigma}:=f_{R_{\sigma},\ldots,R_{\sigma}}\in C^{2k}(\mathfrak{g},V)
\]\sindex[n]{$f_{\sigma}$}in the sense of Definition \ref{k-fold wedge product}.

(b) For $S(x):=\ad(\sigma(x))$ we write
\[\Sym^k(\mathfrak{n},V)^{\widehat{\mathfrak{g}}}:=
\{f\in\Sym^k(\mathfrak{n},V):\,x.f(y_1,\ldots,y_k)=\sum^k_{i=1}
f(y_1,\ldots,S(x)y_i,\ldots,y_k)\,\,\,\text{for all}\,\,\,x\in\mathfrak{g}\}.
\]for the set of $\widehat{\mathfrak{g}}$-equivariant symmetric $k$-linear maps.\sindex[n]{$\Sym^k(\mathfrak{n},V)^{\widehat{\mathfrak{g}}}$}
\end{definition}


\begin{theorem}\label{Lecomtes's Chern-Weil map}{\bf(Lecomte).}\index{Theorem!of Lecomte}
With the notation of Definition \ref{g-invariant symmetric k-linear maps} the following assertions hold:
\begin{itemize}
\item[\emph{(a)}]
For each $k\in\mathbb{N}_0$ we have a natural map
\[C_k:\Sym^k(\mathfrak{n},V)^{\widehat{\mathfrak{g}}}\rightarrow H^{2k}(\mathfrak{g},V),\,\,\,f\mapsto\frac{1}{k!}[f_{\sigma}],
\]which is independent of the choice of the section $\sigma$.
\item[\emph{(b)}]
Suppose, in addition, that $m_V:V\times V\rightarrow V$ is an associative multiplication and that $\mathfrak{g}$ acts on $V$ by derivations, i.e., $m_V$ is $\mathfrak{g}$-invariant. Then $(C^{\bullet}(\mathfrak{g},V),\wedge_{m_V})$ is an associative algebra, inducing an algebra structure on $H^{\bullet}(\mathfrak{g},V)$. Further, $m_V$ endows $\Sym^{\bullet}(\mathfrak{n},V)$ with the structure of an associative algebra such that the maps $(C_k)_{k\in\mathbb{N}_0}$ combine to an algebra homomorphism
\[C:\Sym^{\bullet}(\mathfrak{n},V)^{\widehat{\mathfrak{g}}}\rightarrow H^{2\bullet}(\mathfrak{g},V).
\]
\end{itemize}
\end{theorem}
\begin{proof}
\,\,\,For the proof of this theorem we refer to the original work of Lecomte [Lec85] or [Ne08b], Theorem 3.5.7.
\end{proof}

\section{Secondary Characteristic Classes of Lie Algebra Extensions}\label{SCCLAE}

In this section we present the theory of secondary characteristic classes in a purely algebraic way. This will allow us to associate secondary characteristic classes to every extension of Lie algebras. As in the classical theory of secondary characteristic classes our construction will provide another proof for Lecomte's generalization of the Chern--Weil map. We start with the following remark:

\begin{remark}\label{transformation rule}
If $\Delta_n$\sindex[n]{$\Delta_n$} denotes the $n$-simplex\index{$n$-Simplex} in $\mathbb{R}^{n+1}$ and $E_0:=\{x\in\mathbb{R}^{n+1}:\,x_0=0\}$, then we write
\[p:\Delta_n\rightarrow E_0,\,\,\,(t_0,\ldots,t_n)\mapsto (0,t_1,\ldots,t_n)
\]for the vertical projection on $E_0$. We further write $D_n$\sindex[n]{$D_n$} for the image of $p$, i.e., $D_n:=p(\Delta_n)$. Then
\[\Phi:D_n\rightarrow\Delta_n,\,\,\,(0,t_1,\ldots,t_n)\mapsto\left(1-\sum^n_{i=1}t_i,t_1,\ldots,t_n\right)
\]is a parametrization of $\Delta_n$ and the Gram determinant\index{Gramsche Determinant } 
\[g_{\Phi}:=\sqrt{\det({d\Phi}^T\cdot d\Phi)}
\]is constant. If $c_n$ denotes its value, then we obtain
\[\int_{\Delta_n}FdO=c_n\cdot\int_{D_n}(F\circ \Phi)d\lambda_n.
\]for every $F\in C^{\infty}(\Delta_n)$\sindex[n]{$C^{\infty}(\Delta_n)$}.
\end{remark}

\begin{construction}
Let 
\begin{align}
0\longrightarrow\mathfrak{n}{\longrightarrow}\widehat{\mathfrak{g}}\stackrel{q}{\longrightarrow}\mathfrak{g}\longrightarrow0\notag
\end{align}
be an extension of Lie algebras and $V$ be a $\mathfrak{g}$-module which we also consider as a $\widehat{\mathfrak{g}}$-module with respect to the action $x.v:=q(x).v$ for $x\in\widehat{\mathfrak{g}}$. Moreover, let $\sigma_0,\ldots,\sigma_n\in C^1(\mathfrak{g},\widehat{\mathfrak{g}})$ be linear sections of $q$. For each $t\in\Delta_n$ we define a linear section $\sigma_t\in C^1(\mathfrak{g},\widehat{\mathfrak{g}})$ of $q$ by
\[\sigma_t(x):=\sum^n_{i=0}t_i\cdot\sigma_i(x)\,\,\,\text{for}\,\,\,x\in\mathfrak{g}
\]and write $R_t:=R_{\sigma_t}:=d_{\mathfrak{g}}\sigma_t+\frac{1}{2}[\sigma_t,\sigma_t]\in C^2(\mathfrak{g},\mathfrak{n})$ for the corresponding curvature.
\end{construction}

\begin{definition}\label{SCC integral}{\bf(The Bott-Lecomte homomorphism).}\index{Bott-Lecomte Homomorphism}
For $f\in\Sym^k(\mathfrak{n},V)^{\widehat{\mathfrak{g}}}$ and $n\leq k$ we define 
\begin{align}
\Delta_f(\sigma_0,\ldots,\sigma_n):=\frac{1}{c_n}\cdot\int_{\Delta_n}f_{\alpha_1,\ldots,\alpha_n,R_t,\ldots,R_t}dO\in C^{2k-n}(\mathfrak{g},V),
\end{align}
\sindex[n]{$\Delta_f(\sigma_0,\ldots,\sigma_n)$}where $\alpha_i:=\sigma_i-\sigma_0\in C^1(\mathfrak{g},\mathfrak{n})$. We thus obtain a map
\[\Delta(\sigma_0,\ldots,\sigma_n):\Sym^k(\mathfrak{n},V)^{\widehat{\mathfrak{g}}}\rightarrow C^{2k-n}(\mathfrak{g},V)
\]defined by
\[f\mapsto\Delta_f(\sigma_0,\ldots,\sigma_n),
\]which extends to a map
\[\Delta(\sigma_0,\ldots,\sigma_n):\Sym^{\bullet}(\mathfrak{n},V)^{\widehat{\mathfrak{g}}}\rightarrow C^{\bullet}(\mathfrak{g},V).
\]\sindex[n]{$\Delta(\sigma_0,\ldots,\sigma_n)$}It is called the \emph{Bott-Lecomte homomorphism} relative to $(\sigma_0,\ldots,\sigma_n)$. 
\end{definition}

As in the classical case the following theorem gives an important relation between the forms $\Delta_f$:

\begin{theorem}\label{SCC main theorem}{\bf(The Main Theorem).}\index{Theorem!The Main}
For $f\in\Sym^k(\mathfrak{n},V)^{\widehat{\mathfrak{g}}}$ we have
\[(k-n+1)\cdot d_{\mathfrak{g}}(\Delta_f(\sigma_0,\ldots,\sigma_n))=\sum^n_{i=0}(-1)^i\Delta_f(\sigma_0,\ldots,\widehat{\sigma}_i,\ldots,\sigma_n),
\]where $\widehat{\sigma}_i$ means that $\sigma_i$ is omitted.
\end{theorem}

\begin{proof}
\,\,\,We have divided the proof of this theorem into several parts:

(i) For $t\in\Delta_n$ and $x\in\mathfrak{g}$ let $S_t(x):=\ad(\sigma_t(x))$. Then we have
\begin{align}
&(k-n+1)\cdot d_{\mathfrak{g}}(\Delta_f(\sigma_0,\ldots,\sigma_n))=\frac{k-n+1}{c_n}\cdot\int_{\Delta_n}d_{\mathfrak{g}}f_{\alpha_1,\ldots,\alpha_n,R_t,\ldots,R_t}dO\notag\\
&=\frac{k-n+1}{c_n}\cdot\int_{\Delta_n}d_{\mathfrak{g}}(\tilde{f}\circ(\alpha_1\wedge_{\otimes_s}\cdots\wedge_{\otimes_s}\alpha_n\wedge_{\otimes_s}R_t\wedge_{\otimes_s}\cdots\wedge_{\otimes_s} R_t))dO\notag\\
&=\frac{k-n+1}{c_n}\cdot\int_{\Delta_n}\tilde{f}\circ(d_{S_t}(\alpha_1\wedge_{\otimes_s}\cdots\wedge_{\otimes_s}\alpha_n\wedge_{\otimes_s}R_t\wedge_{\otimes_s}\cdots\wedge_{\otimes_s} R_t))dO.\notag
\end{align}

(ii) We note that $d_{S_t}\alpha_i=d_{\mathfrak{g}}\alpha_i+[\alpha_i,\sigma_t]=\frac{d}{d t_i}R_t$. Hence, the abstract Bianchi identity $d_{S_t} R_t=0$ and Proposition \ref{wedgeproduct and covariant differential} lead to
\begin{align}
&d_{S_t}(\alpha_1\wedge_{\otimes_s}\cdots\wedge_{\otimes_s}\alpha_n\wedge_{\otimes_s}R_t\wedge_{\otimes_s}\cdots\wedge_{\otimes_s} R_t)\notag\\
&=\sum^n_{i=1}(-1)^{i-1}\alpha_1\wedge_{\otimes_s}\cdots\wedge_{\otimes_s}\frac{d}{d t_i}R_t\wedge_{\otimes_s}\ldots\wedge_{\otimes_s}\alpha_n\wedge_{\otimes_s}R_t\wedge_{\otimes_s}\cdots\wedge_{\otimes_s} R_t.\notag
\end{align}
From this and the commutativity property of $\wedge_{\otimes_s}$ (cf. Remark \ref{commutativity properties}) we conclude that
\begin{align}
&\frac{k-n+1}{c_n}\cdot\int_{\Delta_n}\tilde{f}\circ(d_{S_t}(\alpha_1\wedge_{\otimes_s}\cdots\wedge_{\otimes_s}\alpha_n\wedge_{\otimes_s}R_t\wedge_{\otimes_s}\cdots\wedge_{\otimes_s} R_t))dO.\notag\\
&=\sum^n_{i=1}(-1)^{i-1}\frac{1}{c_n}\cdot\int_{\Delta_n}\frac{d}{d t_i}f_{\alpha_1,\ldots,\widehat{\alpha}_i,\ldots,\alpha_n,R_t,\ldots,R_t}dO.\notag
\end{align}

(iii) For $t=(t_0,\ldots,0,\ldots,t_n)\in\Delta^i_n$ ($\Delta^i_n$ is the $i$-th face of $\Delta_n$) we define 
\[(\sigma_t)_i:=\sum_{j\in\{0,\ldots,\widehat{i},\ldots,n\}}t_j\cdot\sigma_j\in C^1(\mathfrak{g},\widehat{\mathfrak{g}})
\]and let $(R_t)_i$ be the corresponding curvature. We further define 
\[s_i:=\sum_{j\in\{1,\ldots,\widehat{i},\ldots,n\}}t_j.
\]
Then we have $R_{t:t_i=0}=(R_t)_i$, $R_{t:t_i=1-s_i}=(R_t)_0$ for each $1\leq i\leq n$, and Remark \ref{transformation rule} shows that
\begin{align}
&\frac{1}{c_n}\cdot\int_{\Delta_n}\frac{d}{d t_i}f_{\alpha_1,\ldots,\widehat{\alpha}_i,\ldots,\alpha_n,R_t,\ldots,R_t}dO=\int_{D_n}\frac{d}{d t_i}f_{\alpha_1,\ldots,\widehat{\alpha}_i,\ldots,\alpha_n,R_t,\ldots,R_t}d\lambda_n\notag\\
&=\int^1_0\cdots\int^{1-s_i}_0\frac{d}{d t_i}f_{\alpha_1,\ldots,\widehat{\alpha}_i,\ldots,\alpha_n,R_t,\ldots,R_t}dt_i\cdot dt_1\ldots d\widehat{t}_i\ldots dt_n\notag\\
&=\int^1_0\ldots\int^{1-s_i+t_1}_0[f_{\alpha_1,\ldots,\widehat{\alpha}_i,\ldots,\alpha_n,R_t,\ldots,R_t}]^{1-s_i}_0dt_1\ldots d\widehat{t}_i\ldots dt_n\notag\\
&=\frac{1}{c_{n-1}}\cdot\int^1_0\ldots\int^{1-s_i+t_1}_0c_{n-1}(f_{\alpha_1,\ldots,\widehat{\alpha}_i,\ldots,\alpha_n,(R_t)_0,\ldots,(R_t)_0}-f_{\alpha_1,\ldots,\widehat{\alpha}_i,\ldots,\alpha_n,(R_t)_i,\ldots,(R_t)_i})dt_1\ldots d\widehat{t}_i\ldots dt_n\notag\\
&=\frac{1}{c_{n-1}}\cdot\int_{\Delta^i_n}f_{\alpha_1,\ldots,\widehat{\alpha}_i,\ldots,\alpha_n,(R_t)_0,\ldots,(R_t)_0}dO-\frac{1}{c_{n-1}}\cdot\int_{\Delta^i_n}f_{\alpha_1,\ldots,\widehat{\alpha}_i,\ldots,\alpha_n,(R_t)_i,\ldots,(R_t)_i}dO\notag\\
&=\frac{1}{c_{n-1}}\cdot\int_{\Delta^i_n}f_{\alpha_1,\ldots,\widehat{\alpha}_i,\ldots,\alpha_n,(R_t)_0,\ldots,(R_t)_0}dO-\Delta_f(\sigma_0,\ldots,\widehat{\sigma}_i,\ldots,\sigma_n).\notag
\end{align}

(iv) Summing up our results so far we get
\begin{align}
&(k-n+1)\cdot d_{\mathfrak{g}}(\Delta_f(\sigma_0,\ldots,\sigma_n))\notag\\
&=\sum^n_{i=1}(-1)^{i-1}\frac{1}{c_{n-1}}\cdot\int_{\Delta^i_n}f_{\alpha_1,\ldots,\widehat{\alpha}_i,\ldots,\alpha_n,(R_t)_0,\ldots,(R_t)_0}dO+\sum^n_{i=1}(-1)^i\Delta_f(\sigma_0,\ldots,\widehat{\sigma}_i,\ldots,\sigma_n).\notag
\end{align}

(v) It remains to show that 
\[\Delta_f(\sigma_1,\ldots,\sigma_n)=\sum^n_{i=1}(-1)^{i-1}\frac{1}{c_{n-1}}\cdot\int_{\Delta^i_n}f_{\alpha_1,\ldots,\widehat{\alpha}_i,\ldots,\alpha_n,(R_t)_0,\ldots,(R_t)_0}dO,
\]where
\[\Delta_f(\sigma_1,\ldots,\sigma_n)=\frac{1}{c_{n-1}}\cdot\int_{\underbrace{\Delta_{n-1}}_{\cong\Delta^i_n}}f_{\alpha_2-\alpha_1,\ldots,\alpha_n-\alpha_1,(R_t)_0,\ldots,(R_t)_0}dO.
\]For this we recall that
\begin{align}
&f_{\alpha_2-\alpha_1,\ldots,\alpha_n-\alpha_1,(R_t)_0,\ldots,(R_t)_0}\notag\\
&=\tilde{f}\circ((\alpha_2-\alpha_1)\wedge_{\otimes_s}\cdots\wedge_{\otimes_s}(\alpha_n-\alpha_1)\wedge_{\otimes_s}(R_t)_0\wedge_{\otimes_s}\cdots\wedge_{\otimes_s}(R_t)_0)\notag
\end{align}
and note that $\alpha\wedge_{\otimes_s}\alpha=0$ holds for every $\alpha\in C^1(\mathfrak{g},\mathfrak{n})$ by the symmetry of the canonical multiplication 
\[S^k(\mathfrak{n})\times S^m(\mathfrak{n})\rightarrow S^{k+m}(\mathfrak{n}),\,\,\,(x,y)\mapsto x\otimes_s y
\]and Remark \ref{commutativity properties}. Therefore, a short calculation shows
\begin{align}
&(\alpha_2-\alpha_1)\wedge_{\otimes_s}\cdots\wedge_{\otimes_s}(\alpha_n-\alpha_1)\wedge_{\otimes_s}(R_t)_0\wedge_{\otimes_s}\cdots\wedge_{\otimes_s}(R_t)_0\notag\\
&=\sum^n_{i=1}(-1)^{i+1}\alpha_1\wedge_{\otimes_s}\cdots\wedge_{\otimes_s}\widehat{\alpha}_i\wedge_{\otimes_s}\ldots\wedge_{\otimes_s}\alpha_n\wedge_{\otimes_s}(R_t)_0\cdots\wedge_{\otimes_s}(R_t)_0\notag
\end{align}

and hence
\[f_{\alpha_2-\alpha_1,\ldots,\alpha_n-\alpha_1,(R_t)_0,\ldots,(R_t)_0}=\sum^n_{i=1}(-1)^{i-1}f_{\alpha_1,\ldots,\widehat{\alpha}_i,\ldots,\alpha_n,(R_t)_0,\ldots,(R_t)_0}.
\]
This completes the proof.
\end{proof}

Applying the above theorem to the case of a $0$-simplex and a $1$-simplex we obtain the following corollaries:

\begin{corollary}\label{main corollary 1}
In the case $n=0$, we have 
\[d_{\mathfrak{g}}(\Delta_f(\sigma))=0.
\]Hence, $\Delta_f(\sigma)$ is a closed form and defines an element of the cohomology space $H^{2k}(\mathfrak{g},V)$. The Bott-Lecomte homomorphism induces a map
\[[\Delta(\sigma)]:\Sym^k(\mathfrak{n},V)^{\widehat{\mathfrak{g}}}\rightarrow H^{2k}(\mathfrak{g},V)
\]defined by $f\mapsto [\Delta_f(\sigma)]$.
\end{corollary}

\begin{corollary}\label{main corollary 2}
In the case $n=1$, we have 
\[k\cdot d_{\mathfrak{g}}(\Delta_f(\sigma_0,\sigma_1))=\Delta_f(\sigma_0)-\Delta_f(\sigma_1)\in C^{2k}(\mathfrak{g},V).
\]Thus 
\[[\Delta_f(\sigma_0)]=[\Delta_f(\sigma_1)]\in H^{2k}(\mathfrak{g},V),
\]and hence the map $[\Delta(\sigma)]$ of Corollary \ref{main corollary 1} is independent of the connection $\sigma_0$ and is, in fact, the homomorphism of Theorem \ref{SCC main theorem} .
\end{corollary}

With the help of Corollary \ref{main corollary 2} we are now ready to define secondary characteristic classes in the following way:

\begin{construction}
For linear sections $\sigma,\sigma_1,\sigma_2\in C^1(\mathfrak{g},\widehat{\mathfrak{g}})$ of the Lie algebra extension
\begin{align}
0\longrightarrow\mathfrak{n}{\longrightarrow}\widehat{\mathfrak{g}}\stackrel{q}{\longrightarrow}\mathfrak{g}\longrightarrow0\notag
\end{align}
we define
\[\Sym^k(\mathfrak{n},V)^{\widehat{\mathfrak{g}}}_{(\sigma)}:=\{f\in\Sym^k(\mathfrak{n},V)^{\widehat{\mathfrak{g}}}:\,f_{\sigma}=0\},
\]\sindex[n]{$\Sym^k(\mathfrak{n},V)^{\widehat{\mathfrak{g}}}_{(\sigma)}$}and further
\[\Sym^k(\mathfrak{n},V)^{\widehat{\mathfrak{g}}}_{(\sigma_1,\sigma_2)}:=\Sym^k(\mathfrak{n},V)^{\widehat{\mathfrak{g}}}_{(\sigma_1)}\cap \Sym^k(\mathfrak{n},V)^{\widehat{\mathfrak{g}}}_{(\sigma_2)}.
\]\sindex[n]{$\Sym^k(\mathfrak{n},V)^{\widehat{\mathfrak{g}}}_{(\sigma_1,\sigma_2)}$}
\end{construction}

\begin{definition}\label{secondary characteristic class}{\bf (Secondary characteristic classes).}\index{Secondary Characteristic Classes}\index{Characteristic Classes!Secondary}
Let $f\in\Sym^k(\mathfrak{n},V)^{\widehat{\mathfrak{g}}}_{(\sigma_1,\sigma_2)}$. Then Corollary \ref{main corollary 2} implies that $d_{\mathfrak{g}}(\Delta_f(\sigma_0,\sigma_1))=0$. Hence, $\Delta_f(\sigma_0,\sigma_1)$ defines a cohomology class in $H^{2k-1}(\mathfrak{g},V)$. We call the class
\[[\Delta_f(\sigma_0,\sigma_1)]\in H^{2k-1}(\mathfrak{g},V)
\]a \emph{simple secondary characteristic class} of the Lie algebra extension 
\begin{align}
0\longrightarrow\mathfrak{n}{\longrightarrow}\widehat{\mathfrak{g}}\stackrel{q}{\longrightarrow}\mathfrak{g}\longrightarrow0\notag
\end{align}
 with respect to the tuple $(\sigma_1,\sigma_2)$.
\end{definition}

\begin{open problem}{\bf(Producing characteristic classes).}\label{producing secondary characteristic classes}
Find interesting examples of extensions of Lie algebras and produce (secondary) characteristic classes with the help of the previous constructions: A natural class of such examples is given by semidirect products of Lie algebras. Among these examples, the class of principal bundles $(P,M,G,q,\sigma)$ for which the associated Atiyah sequence (cf. Example \ref{Attia sequence}) splits is of particular interest. A deeper analysis shows that these principal bundles are essentially the natural bundles in the sense of Thurston--Epstein (cf. [EpTh79]). A nice reference for this topic is [Lec85], Section 4. Moreover, it might also be interesting to produce secondary characteristic classes for flat principal bundles.
\end{open problem}






\chapter{Outlook: More Aspects, Ideas and Problems on NCP Bundles}\label{outlook}

This final chapter is devoted to discussing more aspects of our approach to NCP bundles. In particular, it contains additional ideas and open problems which might serve as a basis for further studies on this topic. In the first section we present the problem of embedding our approach to NCP bundles with compact abelian structure group in a theory of NCP bundles with compact structure group. In the second section we suggest a possible approach to a classification theory for (non-trivial) NCP $\mathbb{T}^n$-bundles. The third section is dedicated to a theory of ``infinitesimal" objects on noncommutative spaces, i.e., we discuss possible ``geometric" invariants for our approach to NCP  bundles. In particular, this discussion makes use of Lecomte's Chern--Weil homomorphism of Chapter \ref{chapter characteristic classes of lie algebra extensions}. Finally, we show how to associate characteristic classes to finitely generated projective modules of CIA's.


\section{An Approach to NCP Bundles with Compact Structure Group}

So far, we have developed a convenient geometric approach to the noncommutative geometry of principal bundles with compact abelian structure group (cf. Chapter \ref{a geometric approach to NCPB}). The natural next step is to try to embed our approach to NCP bundles with compact abelian structure group into a theory of NCP bundles with compact structure group. Of course, for this it is enough to investigate the algebraic structure of the trivial geometric objects, i.e., to find a good concept of trivial NCP bundles with compact structure group $G$. In fact, once the concept of trivial NCP bundles with compact structure group is known, Section \ref{ACNCPT^nB} enters the picture:

\begin{open problem}{\bf(NCP bundles with compact structure group).}
Develop a convenient geometric approach to the noncommutative geometry of principal bundles with compact structure group in which our approach (cf. Chapter \ref{a geometric approach to NCPB}) naturally embeds. We suggest the following strategy: 

First of all it is important to study and to understand the structure of dynamical systems $(A,G,\alpha)$ with compact structure group $G$. As before, our 
point of view is to study $(A,G,\alpha)$ with the help of the representation theory of $G$. In particular, Theorem \ref{str thm of G-mod} is of major relevance. Also, the \emph{Theorem of Cartan--Weyl}\index{Theorem!of Cartan--Weyl} (cf. [Kna02], Theorem 5.110) and the \emph{Theorem of Borel--Weil}\index{Theorem!of Borel--Weil} (cf. [DK00], Theorem 4.12.5) might be useful in this context. The Theorem of Cartan-Weyl, for example, basically says that the equivalence classes of unitary irreducible representations are parametrized by characters of a maximal torus. Further, it is worth to study Definition \ref{free dynamical systems} more explicitly in the case of compact groups $G$. Here, two questions arise immediately:
\begin{itemize}
\item[(i)]
Do there exist natural algebraic conditions on the algebra $A$ which ensure the freeness of the corresponding dynamical system $(A,G,\alpha)$ (cf. Proposition \ref{freeness for compact abelian groups II} for the case of a compact abelian group)?
\item[(ii)]
Given a (compact) manifold $P$, a compact Lie group $G$ and a smooth dynamical system $(C^{\infty}(P),G,\alpha)$, do there exist natural algebraic conditions such that $(C^{\infty}(P),G,\alpha)$ defines a trivial principal bundle $P$ over $P/G$ (cf. Remark \ref{blabla})?
\end{itemize}
At this stage we want to point out explicitly that we are looking for an algebraic condition that is weaker than to assume the existence of a $G$-invariant algebra homomorphisms from the algebra $C^{\infty}(G)$ of smooth functions on $G$ into the algebra $A$. 
A good example to start from is the \emph{special unitary group of rank $2$}: 
\[\SU_2(\mathbb{C}):=\Big\{\begin{pmatrix}
 z &  \overline{w}  \\
 -w & \overline{z} 
\end{pmatrix}\in\M_2(\mathbb{C}):\,\vert z\vert^2+\vert w\vert^2=1\Big\}.
\]\sindex[n]{$\SU_2(\mathbb{C})$}We recall that $C^{\infty}(\SU_2(\mathbb{C}))$ is generated  by functions $f_z$ and $f_w$ satisfying $f_zf_z^*+f_wf_w^*=1$.
\end{open problem}


\section{Classification of NCP Torus Bundles}

In Chapter \ref{a geometric approach to NCPB} we introduced the concept of (non-trivial) NCP $\mathbb{T}^n$-bundles. Since we gave a complete classification of trivial NCP $\mathbb{T}^n$-bundles in Section \ref{Classification of NCPT^nB}, it is, of course, a natural ambition to work out a classification theory for (non-trivial) NCP $\mathbb{T}^n$-bundles. To be more precise, the problem is the following:

\begin{open problem}\label{open problem classification of NCP torus bundles}{\bf(Classification of NCP torus bundles).}\index{Classification!of NCP Torus Bundles}
Let $B$ be a unital locally convex algebra. Work out a good classification theory for NCP $\mathbb{T}^n$-bundles $(A,\mathbb{T}^n,\alpha)$ for which $A^{\mathbb{T}^n}=B$, i.e., classify all NCP $\mathbb{T}^n$-bundles $(A,\mathbb{T}^n,\alpha)$ for which $A^{\mathbb{T}^n}=B$.
\end{open problem}

As is well-know from classical differential geometry, the relation between locally and globally defined objects is important for many constructions and applications. For example, a (non-trivial) principal bundle $(P,M,G,q,\sigma)$ can be considered as a geometric object that is glued together from local pieces which are trivial, i.e., which are of the form $U\times G$ for some (arbitrary small) open subset $U$ of $M$. This approach immediately leads to the concept of $G$-valued cocycles and therefore to a cohomology theory, called the \v Cech cohomology of the pair $(M,G)$. This cohomology theory gives a complete classification of principal bundles with structure group $G$ and base manifold $M$ (cf. [To00], Kapitel IX.5).



\begin{open problem}{\bf(A noncommutative \v Cech cohomology).}\index{Noncommutative!\v Cech Cohomology}
In view of the previous discussion, a possible approach to the Open Problem \ref{open problem classification of NCP torus bundles} is to work out a \emph{noncommutative version} of the \v Cech cohomology. For this, it is worth to study the paper [Du96]. Moreover, ideas of the classification methods for loop algebras of [ABP04] might be of particular interest. Also, a possible point of view for the relation between locally and globally defined objects in the noncommutative world can be found in [CaMa99]; their subspaces are described by ideals and one glues algebras along ideals, using a pull-back construction. 
\end{open problem}



\section{``Infinitesimal" Objects on Noncommutative Spaces.}\label{infinitesimal objects on noncommutative spaces}

In this section we discuss possible ``geometric" invariants for our approach to NCP  bundles: We recall that if $(A,G,\alpha)$ is a smooth dynamical system, then $\alpha$ induces a continuous action of the corresponding Lie algebra $\mathfrak{g}$ on $A$ by continuous derivations. Hence, the Lie group complex $(C^{\bullet}(G,A),d_G)$ and the Lie algebra complex $(C^{\bullet}(\mathfrak{g},A),d_{\mathfrak{g}})$ are at our disposal. In particular, the derivation-based differential calculus described by Dubois--Violette in [DV88] fits into our approach. Moreover, we recall that the notion of connections (cf. Section \ref{conpb}) plays a key role in the differential theory of principal bundles. They are, for example, used in the development of curvature and hence are also important in the Chern--Weil theory of characteristic classes of principal bundles. We therefore present a natural generalization of the concept of connections in the noncommutative setting; the corresponding curvature induced by such a connection turns out to be a certain Lie algebra 2-cocycle and allows us, with the help of Lecomte's Chern--Weil homomorphism (cf. Theorem \ref{Lecomtes's Chern-Weil map}) to define noncommutative characteristic classes that generalize the classical de Rham cohomology. 

\subsection{Generalized Connections}

For the following discussion we keep Section \ref{conpb} in mind. We start with a reminder on the theory of generalized connections: In the following let $B$ and $M$ be manifolds. Further, let $q:B\rightarrow M$ be a fibre bundle over $M$ with fibre $F$. We write $V(B)$\sindex[n]{$V(B)$} for the (vertical) subbundle of the tangent bundle $TB$. The corresponding fibre $V_b(B)$ at $b\in B$ is the tangent space to the fibre $B_{q(b)}$ of $B$ passing through $b$. If $q^*(TM)$\sindex[n]{$q^*(TM)$} is the pull-back bundle of the tangent bundle, then we have the following short exact sequence of vector bundles over $B$:
\begin{align}
0\longrightarrow V(B){\longrightarrow}TB\stackrel{\pi}{\longrightarrow}q^*(TM)\longrightarrow 0\label{generalized connections 1}.
\end{align}
Here, the map $\pi$ is defined by 
\[\pi(v):=(b,T_b(q)v)\,\,\,\text{for}\,\,\,v\in T_b(B).
\]

\begin{definition}\label{generalized connections 2}{\bf(Generalized connections).}\index{Connections!Generalized}
A \emph{generalized connection on} $B$ is a splitting of the exact sequence (\ref{generalized connections 1}), i.e., a vector bundle morphism
\[\phi:q^*(TM)\rightarrow TB
\]such that $\pi\circ\phi=\id_{q^*(TM)}$.
\end{definition}


The next remark helps to clarify the role of the structure group in the usual definition of connection 1-forms in a principal bundle.

\begin{remark}\label{generalized connections 4}{\bf(Recovering the usual definition).}
If we consider a principal bundle $(P,M,G,q,\sigma)$, then the action of $G$ on $P$ extends to all the vector bundles in the exact sequence (\ref{generalized connections 1}). In particular, it induces the following exact sequence of vector bundles over $M$:
\[0\longrightarrow \Ad(P){\longrightarrow}TP/G\stackrel{\overline{\pi}}{\longrightarrow}TM\longrightarrow 0.
\]The previous sequence of vector bundles over $M$ can be rewritten as a short exact sequence of Lie algebras
\begin{align}
0\longrightarrow\mathfrak{gau}(P)\longrightarrow\mathfrak{aut}(P)\stackrel{q_*}{\longrightarrow}\mathcal{V}(M)\longrightarrow 0 \label{generalized connections 4,5}
\end{align}
which is commonly known as the Atiyah sequence associated to the principal bundle $(P,M,G,q,\sigma)$ (cf. Example \ref{Attia sequence}). 

We now show that each splitting of the exact sequence (\ref{generalized connections 1}) induces a splitting of the exact sequence (\ref{generalized connections 4,5}), i.e., a $C^{\infty}(M)$-linear cross section of $q_*$, and vice versa. We proceed as follows:


(i) Of course, each splitting of the exact sequence (\ref{generalized connections 1}) induces a splitting of the exact sequence (\ref{generalized connections 4,5}), i.e., a $C^{\infty}(M)$-linear cross section of $q_*$.

(ii) Conversely, let $\tau$ be a $C^{\infty}(M)$-linear cross sections of $q_*$. By Proposition \ref{morphism of VB-morphism of f.g.p. modules} (b), the map $\tau$ induces an injective morphism of vector bundles $\overline{\phi}:TM\rightarrow TP/G$ which satisfies 
\[\overline{\pi}\circ\overline{\phi}=\id_{TM}.
\]We further conclude from Corollary \ref{injective} that $I:=\im\overline{\phi}$ is a vector subbundle of $TP/G$. Thus, the preimage $H:=\pr^{-1}(I)$ with respect to the canonical projection $\pr:TP\rightarrow TP/G$ is a vector subbundle of $TP$. Since the canonical projection $\pr$ collapses orbits, we get
\[T(\sigma_g)H_p(P)=H_{p.g}(P)
\]for all $p\in P$ and $g\in G$. Finally, since $\overline{\phi}$ is a right-inverse to $\overline{\pi}$, it follows that $TP/G=\Ad(P)\oplus I$. Hence, we obtain the desired splitting $TP=V(P)\oplus H$.
\end{remark}

\subsection{Connections associated to dynamical systems}

Inspired by the previous considerations we now change to the noncommutative setting. We recall that for a general associative algebra $A$ it is well known that the infinitesimal version of automorphisms are the derivations $\der(A)$\sindex[n]{$\der(A)$} of $A$ and that $\der(A)$ carries the structure of a Lie algebra when endowed with the commutator bracket. 


\begin{lemma}\label{generalized connections 5}
Let $(A,G,\alpha)$ be a dynamical system. Then the following assertions hold:
\begin{itemize}
\item[\emph{(a)}]
If $C_A$ denotes the center of $A$, then the map 
\[\rho:C_A\times\der(A)\rightarrow\der(A),\,\,\,c.D:=\rho(c,D):=cD
\]defines the structure of a \emph{(}left\emph{)} $C_A$-module on $\der(A)$.
\item[\emph{(b)}]
The map 
\[\mu:G\times\der(A)\rightarrow\der(A),\,\,\,g.D:=\mu(g,D):=\alpha(g)\circ D\circ\alpha(g)^{-1}
\]defines a \emph{(}left\emph{)} group action of $G$ on the $\der(A)$ by Lie algebra automorphisms. In particular, the corresponding fixed point space $(\der(A))^G$ is a Lie subalgebra of $\der(A)$.
\end{itemize}
\end{lemma}

\begin{proof}
\,\,\,Both parts of the proof just consist of simple calculations. 
\end{proof}

\begin{definition}\label{generalized connections 6}{\bf($G$-invariant derivations).}\index{$G$-invariant derivations}
If $(A,G,\alpha)$ is a dynamical system, then we write
\[\der_G(A):=\{D\in\der(A):\,(\forall g\in G)\,\alpha(g)\circ D=D\circ\alpha(g)\}
\]\sindex[n]{$\der_G(A)$}for the Lie subalgebra of $\der(A)$ consisting of $G$-invariant derivations. In fact, $\der_G(A)$ is exactly the fixed point Lie algebra of the action map $\mu$ of Lemma \ref{generalized connections 5} (b), i.e., $\der_G(A)=(\der(A))^G$.
\end{definition}

\begin{lemma}\label{generalized connections 7}
Let $(A,G,\alpha)$ be a dynamical system. If $Z:=C_A^G$ denotes the fixed point algebra of the induced action of $G$ on the center $C_A$ of $A$, then the map
\[\overline{\rho}:Z\times\der_G(A)\rightarrow\der_G(A),\,\,\,z.D:=\overline{\rho}(z,D):=zD
\]defines the structure of a \emph{(}left\emph{)} $Z$-module on $\der_G(A)$.
\end{lemma}

\begin{proof}
\,\,\,This is again a simple calculation.
\end{proof}

\begin{proposition}\label{generalized connections 9}
If $(A,G,\alpha)$ is a dynamical system and $B$ the corresponding fixed point algebra, then the map
\[q_*:\der_G(A)\rightarrow\der(B),\,\,\,D\mapsto D_{\mid B}
\]\sindex[n]{$\der(B)$}defines a homomorphism of Lie algebras.
\end{proposition}

\begin{proof}
\,\,\,For the proof we just have to note that each element $D\in\der_G(A)$ leaves the fixed point algebra $B$ invariant. Indeed, for $b\in B$ and $g\in G$ we have
\[\alpha(g)(D.b)=D(\alpha(g).b)=D.b,
\]i.e., $D(B)\subseteq B$.
\end{proof}

\begin{definition}\label{generalized connections 10}
Let $(A,G,\alpha)$ be a dynamical system and $B$ the corresponding fixed point algebra. We write 
\[\mathfrak{gau}(A):=\{D\in\der_G(A):\,D_{\mid B}=0\}
\]\sindex[n]{$\mathfrak{gau}(A)$}for the Lie subalgebra of $\der_G(A)$ consisting of ``vertical" derivations. In fact, $\mathfrak{gau}(A)$ is exactly the kernel of the Lie algebra homomorphism $q_*$, i.e., $\mathfrak{gau}(A)=\ker(q_*)$.
\end{definition}

\begin{definition}\label{generalized connections 11}{\bf(Connections associated to dynamical systems).}\index{Connections!Associated to Dynamical Systems}
Let $(A,G,\alpha)$ be a dynamical system for which the Lie algebra homomorphism $q_*$ of Proposition \ref{generalized connections 9} is surjective. Then we obtain a short exact sequence
\begin{align}
0\longrightarrow\mathfrak{gau}(A)\longrightarrow\mathfrak\der_G(A)\stackrel{q_*}{\longrightarrow}\der(B)\longrightarrow 0\label{generalized connections 11,5}
\end{align}
of Lie algebras. If $Z:=C_A^G$ denotes the fixed point algebra of the induced action of $G$ on the center $C_A$ of $A$, then a $Z$-linear section 
\[\sigma:\der(B)\rightarrow\der_G(A)
\]of $q_*$ is called a \emph{connection of the dynamical system} $(A,G,\alpha)$. Here, we have used that $\der(B)$ is a $C_B$-module (cf. Lemma \ref{generalized connections 5} (a)).
\end{definition}

\begin{remark}
If $(A,G,\alpha)$ is a dynamical system for which the Lie algebra homomorphism $q_*$ of Proposition \ref{generalized connections 9} is surjective, then we can use Lecomte's Chern--Weil homomorphism (cf. Theorem \ref{Lecomtes's Chern-Weil map}) to produce (noncommutative) characteristic classes. For example, since $B$ is a $\der(B)$-module, we get for each $k\in\mathbb{N}_0$ a natural map
\[C_k:\Sym^k(\mathfrak{gau}(A),B)^{\der_G(A)}\rightarrow H^{2k}(\der(B),B),\,\,\,f\mapsto\frac{1}{k!}[f_{\sigma}],
\]which is independent of the choice of the section $\sigma$.\sindex[n]{$\Sym^k(\mathfrak{gau}(A),B)^{\der_G(A)}$}
\end{remark}

\begin{example}\label{generalized connections 8}
If $(P,M,G,q,\sigma)$ is a principal bundle and $(C^{\infty}(P),G,\alpha)$ the corresponding smooth dynamical system of Remark \ref{induced transformation triples}, then the previous construction reproduces the classical setting of connections (cf. Section \ref{conpb}). In fact, we obtain
\[\mathfrak{gau}(C^{\infty}(P))\cong\mathfrak{gau}(P)\,\,\,\text{and}\,\,\,\der_G(C^{\infty}(P))\cong\mathfrak{aut}(P).
\]Moreover, the short exact sequence (\ref{generalized connections 11,5}) is ``isomorphic" to the Atiyah sequence of Example \ref{Attia sequence} and therefore each $C^{\infty}(M)$-linear section produces characteristic classes in the even de Rham cohomology $H_{\text{dR}}^{2\bullet}(M,\mathbb{K})$ (cf. [KoNo69], Chapter XII).
\end{example}

\begin{open problem}{\bf(Dynamical systems for which $q_*$ is surjective).}\label{open problem connections}
Find interesting classes of dynamical systems $(A,G,\alpha)$ for which the Lie algebra homomorphism $q_*$ of Proposition \ref{generalized connections 9} is surjective and produce characteristic classes with the help of Lecomte's Chern--Weil homomorphism of Theorem \ref{Lecomtes's Chern-Weil map}. 

(i) As a first step it would be instructive to study Example \ref{V.3}. For this we recall that if
\[A:=C^{\infty}(M,M_m(\mathbb{C}))
\]is the algebra of smooth $M_m(\mathbb{C})$-valued functions on a manifold $M$, then [Mas08], Proposition 3.26 implies that
\[\der(A)\cong\mathcal{V}(M)\oplus C^{\infty}(M,\mathfrak{sl}_n(\mathbb{C}))
\]holds as Lie algebras and $C^{\infty}(M)$-modules. We further recall that
\[\der(\mathbb{T}^2_{\frac{1}{m}})\cong\Inn(\der)(\mathbb{T}^2_{\frac{1}{m}})\rtimes\Out(\der)(\mathbb{T}^2_{\frac{1}{m}})
\]and that the outer derivations correspond exactly to the derivations of the center, i.e., to $\mathcal{V}(\mathbb{T}^2)$ (cf. Proposition \ref{derivations of the quantumtorus}). Thus, it remains to determine the Lie algebra 
\[\der_{C_m\times C_m}(C^{\infty}(\mathbb{T}^2,M_m(\mathbb{C})))
\]of $C_m\times C_m$-invariant derivations.

(ii) For a deeper investigation on this open problem, the paper [BCER84] is of particular interest. Here, the authors consider the problem of deciding when a given derivation on a fixed point algebra of a $C^*$-dynamical systems $(A,G,\alpha)$ with compact abelian structure group $G$ extends to a derivation on $A$ which commutes with the group action.
\end{open problem}

\subsection{Connections on Modules and the Associated Curvature}

This subsection is devoted to the algebraic background of the geometric concept of connections on vector bundles. In classical differential geometry the concept of a connection on a vector bundle gives us the possibility to define directional derivatives of sections along vector fields on the base manifold. With this directional derivatives it is possible to compare elements in two fibres by parallel transport. In particular, if $(P,M,G,q,\sigma)$ is a principal bundle and $(\pi,V)$ a smooth representation of $G$ defining the associated bundle $\mathbb{V}:=P\times_{\pi} V$ over $M$, then each connection 1-form $\theta\in\mathcal{C}(P)$ induces a connection on $\mathbb{V}:=P\times_{\pi} V$. In the following we discuss a possible generalization to our noncommutative setting. We start with the following reminder:

\begin{definition}\label{connection on a vector bundle}{\bf(Connections on vector bundles).}\index{Connections!on Vector Bundles}
A \emph{connection on a vector bundle} $(\mathbb{V},M,V,q)$ is a map
\[\nabla:\mathcal{V}(M)\rightarrow\End(\Gamma\mathbb{V})
\]such that the following equations hold for all $f\in C^{\infty}(M)$, $X,Y\in\mathcal{V}(M)$ and $s\in\Gamma \mathbb{V}$:
\begin{align}
\nabla_{fX+Y}&=f\nabla_{X}+\nabla_Y.\\
\nabla_X(f s)&=(X.f)s+f\nabla_X s.
\end{align}
\end{definition}

\begin{proposition}\label{Connections on Modules and the associated Curvature I}
Let $(P,M,G,q,\sigma)$ be a principal bundle and $(\pi,V)$ a smooth representation of $G$ defining the associated bundle $\mathbb{V}:=P\times_{\pi} V$ over $M$ \emph{(}cf. Construction \ref{associated vector bundle}\emph{)}. Then each connection 1-form $\theta\in\mathcal{C}(P)$ induces a connection on $(\mathbb{V},M,V,q_{\mathbb{V}})$.
\end{proposition}

\begin{proof}
\,\,\,A proof of this proposition can be found in [KoNo63], Chapter III, Proposition 1.2.
\end{proof}

Our next goal is to present a noncommutative analogue of the previous proposition. For that we need the following construction: 

\begin{construction}\label{Connections on Modules and the associated Curvature II}
Let $(A,G,\alpha)$ be a dynamical system and $\mu$ the action map of Lemma \ref{generalized connections 5}.
Further, let $E$ be a locally convex space such that
\begin{itemize}
\item[$\bullet$]
$(E,\beta)$ defines a locally convex $A$-module,
\item[$\bullet$]
$(E,\gamma)$ defines a representation of $G$ and 
\item[$\bullet$]
$(E,\delta)$ defines a representation of the Lie algebra $\der(A)$.
\end{itemize}
If these structures are compatible in the sense that
\begin{align}
\gamma(g,\beta(e,a))=\beta(\gamma(g,e),\alpha(g,a)),\label{Connections on Modules and the associated Curvature II.1}
\end{align}
\begin{align}
\gamma(g,\delta(D,e))=\delta(\mu(g,D),\gamma(g,e))\label{Connections on Modules and the associated Curvature II.2}
\end{align}
and
\begin{align}
\delta(D,\beta(e,a))=\beta(e,D.a)+\beta(\delta(D,e),a)\label{Connections on Modules and the associated Curvature II.3}
\end{align}
hold for all $a\in A$, $g\in G$, $e\in E$ and $D\in\der(A)$, then the map
\begin{align}
\overline{\beta}:E^G\times B\rightarrow E^G,\,\,\,\overline{\beta}(e,b):=\beta(e,b)\label{Connections on Modules and the associated Curvature II.4}
\end{align}
defines the structure of a $B$-module on the fixed point space $E^G$ and the map
\begin{align}
\overline{\delta}:\der_G(A)\times E^G\rightarrow E^G,\,\,\,\overline{\delta}(D,e):=\beta(D,e)\label{Connections on Modules and the associated Curvature II.5}
\end{align}
defines a representation of the Lie algebra $\der_G(A)$ on the fixed point space $E^G$. Moreover, these maps are compatible in the sense that
\begin{align}
\overline{\delta}(D,\overline{\beta}(e,b))=\overline{\beta}(e,D.b)+\overline{\beta}(\overline{\delta}(D,e),b)\label{Connections on Modules and the associated Curvature II.6}
\end{align}
hold for all $b\in B$, $e\in M^G$ and $D\in\der_G(A)$.
\end{construction}

We now provide two classes of examples that satisfy the conditions of the previous construction:

\begin{example}\label{Connections on Modules and the associated Curvature III}
(a) Let $(P,M,G,q,\sigma)$ be a principal bundle and $(C^{\infty}(P),G,\alpha)$ the corresponding smooth dynamical system of Remark \ref{induced transformation triples}. Further, let $(\pi,V)$ be a smooth representation of $G$ defining the associated bundle $\mathbb{V}:=P\times_{\pi} V$ over $M$ (cf. Construction \ref{associated vector bundle}). Then $E:=C^{\infty}(P,V)$ is a locally convex space and the assignment
\[\beta: E\times C^{\infty}(P)\rightarrow E,\,\,\,\beta(F,f)(p):=fF
\]turns $E$ into a locally convex $C^{\infty}(P)$-module. Moreover, the map
\[\gamma:G\times E\rightarrow E,\,\,\,\gamma(g,f)(p):=\pi(g).f(p.g)
\]defines a representation of $G$ on $E$ and the map
\[\delta:\mathcal{V}(P)\times E\rightarrow E,\,\,\,\delta(X,f):=X.f
\]defines a representation of the Lie algebra $\mathcal{V}(P)$ (of smooth vector fields on $P$) on $E$. As is well-known from basic calculus, these maps are compatible in the sense of Construction \ref{Connections on Modules and the associated Curvature II}. In particular, equation (\ref{Connections on Modules and the associated Curvature II.3}) corresponds to the Leibniz rule for derivatives, i.e.,
\[X.(fF)=(X.f) F+f(X.F)
\]for $f\in C^{\infty}(P)$, $F\in E$ and $X\in\mathcal{V}(P)$. Here, the fixed point space $E^G=C^{\infty}(P,V)^G$ is isomorphic to $\Gamma\mathbb{V}$ (cf. Proposition \ref{sections of an associated vector bundle}).

(b) Another example which is a noncommutative analogue of part (a) is given by a dynamical system $(A,G,\alpha)$ and a representation $(\pi,V)$ of $G$. Then $E:=A\otimes V$ is a locally convex space and the assignment
\[\beta: E\times A\rightarrow E,\,\,\,\beta(a\otimes v,a_0):=aa_0\otimes v
\]turns $E$ into a locally convex $A$-module. Moreover, the map
\[\gamma:G\times E\rightarrow E,\,\,\,\gamma(g,a\otimes v):=(\alpha(g).a)\otimes(\pi(g).v)
\]defines a representation of $G$ on $E$ and the map
\[\delta:\der(A)\times E\rightarrow E,\,\,\,\delta(D,a\otimes v):=(D.a)\otimes v
\]defines a representation of the Lie algebra $\der(A)$ on $E$.
A short calculation now shows that these maps are compatible in the sense of Construction \ref{Connections on Modules and the associated Curvature II}. The corresponding fixed point space $E^G$ is exactly the space $\Gamma_A V$ of Definition \ref{sections again}.
\end{example}

\begin{definition}\label{Connections on Modules and the associated Curvature IV}{\bf(Connections and covariant derivatives on modules).}\index{Connections on Modules}
Let $(A,G,\alpha)$ be a dynamical system for which the Lie algebra homomorphism $q_*$ of Proposition \ref{generalized connections 9} is surjective (we shall keep the Open Problem \ref{open problem connections} in mind). Further, suppose that we are in the situation of Construction \ref{Connections on Modules and the associated Curvature II}. Then each $Z$-linear section 
\[\sigma:\der(B)\rightarrow\der_G(A)
\]of $q_*$ induces a map
\[\nabla^{\sigma}:\der(B)\rightarrow\End(E^G),\,\,\,\nabla^{\sigma}(D).e:=\sigma(D).e
\]\sindex[n]{$\nabla^{\sigma}$}which is well-defined in view of (\ref{Connections on Modules and the associated Curvature II.5}). The map $\nabla^{\sigma}$ is called a \emph{connection on the locally convex $B$-module} $E^G$ (\emph{associated to the section $\sigma$}). The map
\[\nabla^{\sigma}_D:E^G\rightarrow E^G,\,\,\,\nabla^{\sigma}_D.e:=\nabla^{\sigma}(D).e
\]is called \emph{covariant derivative on} $E^G$ (\emph{with respect to $D$}).
\end{definition}

The following proposition justifies the previous notations:

\begin{proposition}\label{Connections on Modules and the associated Curvature VI}
Suppose we are in the situation of Definition \ref{Connections on Modules and the associated Curvature IV}. Then the following equations hold for all $b\in B$, $z\in Z$, $D,D'\in\der(B)$ and $e\in E^G$:
\begin{align}
\nabla^{\sigma}_{zD+D'}&=z\nabla^{\sigma}_D+\nabla^{\sigma}_{D'}.\label{Connections on Modules and the associated Curvature VI.1}\\
\nabla^{\sigma}_D(e.b)&=e.(\sigma(D).b)+(\nabla^{\sigma}_D.e).b.\label{Connections on Modules and the associated Curvature VI.2}
\end{align}
\end{proposition}

\begin{proof}
\,\,\,Equation (\ref{Connections on Modules and the associated Curvature VI.1}) easily follows from the definition of $\nabla^{\sigma}$. Moreover, Equation (\ref{Connections on Modules and the associated Curvature VI.2}) is a consequence of (\ref{Connections on Modules and the associated Curvature II.6}).
\end{proof}

In the case of noncommutative differential geometry based on the universal differential calculus (cf. for example Chapter 8 of [GVF01]) one has the following very similar definition:

\begin{remark}\label{universal connection}\index{Connection!Universal}{\bf(The universal connection).}
Let $A$ be a unital algebra and $E$ be a (right) $A$-module. Consider the $A$-module $E\otimes\Omega^1A$, where $\Omega^1A$ is the universal $A$-bimodule as described in Definition \ref{DSHA1} (or any other suitable module of 1-forms). A \emph{universal connection} or just connection on $E$ is a linear map
\[\nabla:E\rightarrow E\otimes_A\Omega^1A,
\]that satisfies for each $a\in A$ and $e\in E$ the Leibniz rule:
\[\nabla(s.a)=s\otimes d_Aa+(\nabla s).a.
\]
\end{remark}

Is is gratifying that only projective modules admit universal connections:

\begin{theorem}{\bf(Cuntz--Quillen).}\label{Cuntz-Quillen}\index{Theorem!of Cuntz--Quillen}
A right $A$-module $E$ admits a universal connection if and only if it it is projective.
\end{theorem}

\begin{proof}
\,\,\,We only give a sketch of the proof: First of all we define a right $A$-module homomorphism
\[E\otimes_A\Omega^1 A\stackrel{j}{\longrightarrow}E\otimes_{\mathbb{C}}A\stackrel{m}{\longrightarrow}E\longrightarrow 0
\]by $j(sd_Aa):=s\otimes a-s.a\otimes 1$ and $m(s\otimes a):=s.a$. This yields a short exact sequence of right $A$-modules (we think here of $E\otimes_{\mathbb{C}}A$ as a free $A$-module generated by a vector space basis of $E$). Any linear map $\nabla:E\rightarrow E\otimes_A\Omega^1A$ gives rise to a linear section of $m$ by $f(s):=s\otimes 1+j(\nabla s)$. Then 
\[f(s.a)-f(s)a=j(\nabla(s.a)-(\nabla s)a-s\otimes d_Aa),
\] so $f$ is an $A$-module homomorphism precisely when $\nabla$ satisfies the Leibniz rule. If that happens, $f$ splits the exact sequence and embeds $E$ as a direct summand of the free $A$-module $E\otimes_{\mathbb{C}}A$. By Proposition \ref{Appendix B4} $E$ is projective.
\end{proof}

\begin{open problem}
Does there exist a statement which similar to Theorem \ref{Cuntz-Quillen} in the context of our definition of connections on (locally convex) modules (cf. Definition \ref{Connections on Modules and the associated Curvature IV}). For this it might be helpful to study (and try to copy) the proof of Theorem \ref{Cuntz-Quillen}. Of course, in the commutative setting this is always true, since sections of vector bundles are finitely generated projective modules over the base manifold (cf. Theorem \ref{Serre-Swan}).
\end{open problem}

We finally come to the concept of curvature. We start with the classical definition:

\begin{definition}\label{the curvature of a connection I}{\bf(The classical curvature of a connection).}\index{Curvature!Classical}
The \emph{curvature} $R:=R_{\nabla}$ of a connection $\nabla$ on a vector bundle $(\mathbb{V},M,V,q)$ is a $\End(\mathbb{V})$-valued 2-form on $M$, i.e., an element $R$ in $\Omega^2(M,\End(\mathbb{V}))$ defined for all $X,Y,\in\mathcal{V}(M)$ and $s\in\Gamma\mathbb{V}$ by
\[R(X,Y)s:=[\nabla_X,\nabla_Y]s-\nabla_{[X,Y]}s.
\]
\end{definition}

In the noncommutative world we do not have to change the definition at all:
 
\begin{definition}\label{the curvature of a connection II}{\bf(The noncommutative curvature of a connection).}\index{Curvature!Noncommutative}
Suppose we are again in the situation of Definition \ref{Connections on Modules and the associated Curvature IV}. The \emph{curvature} $R:=R_{\sigma}:=R_{\nabla^{\sigma}}$\sindex[n]{$R_{\sigma}$}\sindex[n]{$R_{\nabla^{\sigma}}$} of a connection $\nabla^{\sigma}$ (on the locally convex $B$-module $E^G$) is defined for all $D,D'\in\der(B)$ by
\[R(D,D'):E^G\rightarrow E^G,\,\,\,R(D,D').e:=[\nabla^{\sigma}_D,\nabla^{\sigma}_D]e-\nabla^{\sigma}_{[D,D']}e.
\]We recall that the curvature $R$ may also be viewed as Lie algebra 2-cochain with values in $\mathfrak{gau}(A)$:
\[R:\der(B)\times\der(B)\rightarrow\mathfrak{gau}(A),\,\,\,R(D,D'):=[\sigma(D),\sigma(D')]-\sigma([D,D']).
\]
\end{definition}

\begin{remark}\label{Connections on Modules and the associated Curvature VII}
We want to point out that all of the previous constructions extends the classical notion of connections, covariant derivatives and curvature on vector bundles.
\end{remark}

\subsection{The Globalization Problem}

\begin{definition}\label{continuous lie algebra cohomology}{\bf(Continuous Lie algebra cohomology).}\index{Lie!Algebra Cohomology, Continuous}\index{Continuous Lie Algebra Cohomology}
Let $\mathfrak{g}$ be a topological Lie algebra and $V$ be a topological $\mathfrak{g}$-module, i.e., a pair $(V,\rho)$ of a topological space $V$ and a continuous $\mathfrak{g}$-action $\mathfrak{g}\times V\rightarrow V$ given by the Lie algebra homomorphism $\rho:\mathfrak{g}\rightarrow \End(V)$. We write $C^p_c(\mathfrak{g},V)$\sindex[n]{$C^p_c(\mathfrak{g},V)$} for the set of all alternating continuous $p$-linear maps $\mathfrak{g}\rightarrow V$. Then the complex 
\[(C^{\bullet}_c(\mathfrak{g},V),d_{\mathfrak{g}})
\]\sindex[n]{$(C^{\bullet}_c(\mathfrak{g},V),d_{\mathfrak{g}})$}is a graded subalgebra of the Chevalley-Eilenberg complex $(C^{\bullet}(\mathfrak{g},V),d_{\mathfrak{g}})$ (cf. Subsection \ref{The Chevalley-Eilenberg Complex}).
\end{definition}

\begin{definition}\label{derivation map}{\bf(The derivation map).}\index{Derivation Map}
Let $G$ be a Lie group and $A$ be a topological algebra. If $A$ is a smooth $G$-module given by the group homomorphism
\[S:G\rightarrow\Aut(A),
\]then $S$ induces a continuous action of the Lie algebra $\mathfrak{g}$ on $A$ which is given by the Lie algebra homomorphism 
\[s:\mathfrak{g}\rightarrow\der(A),\,\,\,s(x).a:=(\textbf{L}(S).x).a
\]
In this case we also have maps
\[D:C^p_s(G,A)\rightarrow C^p_c(\mathfrak{g},A),\,\,\,f\mapsto Df
\]given by
\[(Df)(x_1,\ldots,x_p)=\sum_{\sigma\in S_p}\sgn(\sigma)(d^pf)(1_G,\ldots,1_G)(x_{\sigma(1)},\ldots,x_{\sigma(p)}),
\]where
\[\mathfrak{g}^p\rightarrow A,\,\,\,x\mapsto d^pf(1_G,\ldots,1_G)(x_1,\ldots,x_p)
\]is the Taylor polynomial of order $p$ of $f$ in $(1_G,\ldots,1_G)$. All terms of lower order vanishes because $f$ vanishes on all tuples of the form $(g_1,\ldots,g_{j-1},1_G,g_{j+1},\ldots,g_p)$ (cf. [Ne04], Lemma B.7). According to [Ne04], Lemma F.3, the map $D$ even defines a morphism
\[D:(C^{\bullet}_s(G,A), d_S)\rightarrow (C^{\bullet}_c(\mathfrak{g},A),d_{\mathfrak{g}})
\]of differential graded algebras. The map $D$ induces in particular linear maps
\[D_p:H_s^p(G,A)\rightarrow H^p_c(\mathfrak{g},A),\,\,\,[f]\mapsto [Df].
\]
\end{definition}

\begin{open problem}\label{globalization problem}{\bf(The globalization problem).}\index{Globalization Problem}
The \emph{globalization problem} for cohomology classes consists in characterizing the image of $D_p$ in the Lie algebra cohomology space $H^n_c(\mathfrak{g},A)$. The globalization problem contains already a lot of interesting geometry in the classical case, where $A=C^{\infty}(M,\mathbb{R})$, $G=\Diff(M)$ and $\mathfrak{g}=\mathcal{V}(M)$. In fact, we now give a short insight in the cases $p=1,2$ and 3 below. 
\end{open problem}

Let $M$ be a manifold and $\mathfrak{g}:=\mathcal{V}(M)$ the corresponding Lie algebra of vector fields on $M$. Then $A:=C^{\infty}(M,\mathbb{R})$ is a $\mathfrak{g}$-module and the de Rham complex $(\Omega^{\bullet}(M,\mathbb{R}),d)$ is a subcomplex of the Chevalley-Eilenberg complex $(C^{\bullet}(\mathfrak{g},A),d_{\mathfrak{g}})$ (cf. Definition \ref{Chevalley-Eilenberg complex}). On the other hand $A$ is a module of the diffeomorphism group $G:=\Diff(M)^{\text{op}}$ and it is an interesting problem to ask for a global picture of the subcomplex $(\Omega^{\bullet}(M,\mathbb{R}),d)$ in the (Lie) group cohomology $(C^{\bullet}(G,A),d_G)$ (cf. Section \ref{some lie group cohomology}). This is of particular interest in degrees 1,2 and 3.

If $M$ is compact, then $G$ is a Lie group and each closed $p$-form $\alpha$ on $M$ defines a continuous Lie algebra $p$-cocycle $\mathfrak{g}^p\rightarrow A$. Let $\alpha^{\text{eq}}\in \Omega^p(G,A)$\sindex[n]{$\alpha^{\text{eq}}$} denote the corresponding equivariant $A$-valued $p$-form on $G$, which we consider as an infinitesimal version of a group $p$-cocycle, having in mind that $p$-cocycles can be viewed as equivariant functions $G^p\rightarrow A$ (cf. [Ne04], Definition B.4).

A first obstruction to the integrability of the Lie algebra cocycle $\alpha$ to a group cocycle $\alpha$ to a group cocycle is given by the period map
\[\per_{\alpha}:\pi_p(\Diff(M))\rightarrow A=C^{\infty}(M,\mathbb{R}),\,\,\,[\sigma]\mapsto\int_{\sigma}\alpha^{\text{eq}}
\]\sindex[n]{$\per_{\alpha}$}\sindex[n]{$\pi_p(\Diff(M))$}\sindex[n]{$\int_{\sigma}\alpha^{\text{eq}}$}If $p=2$, then [Ne04], Lemma 4.2 implies that the image of the period map always lies in $A^G\cong\mathbb{R}$ of constant functions.

For $m_0\in M$ let
\[\ev_{m_0}:\Diff(M)\rightarrow M,\,\,\,\varphi\mapsto\varphi(m_0),
\]\sindex[n]{$\ev_{m_0}$}be the corresponding evaluation map. Then it is possible to derive from the special structure of $\alpha$ that the period map factors as
\[\per_{\alpha}=\per^M_{\alpha}\circ\pi_p(\ev_{m_0}),
\]where
\[\per^M_{\alpha}:\pi_p(M,m_0)\rightarrow \mathbb{R}
\]\sindex[n]{$\per^M_{\alpha}$}is the period map of the closed $p$-form $\alpha$ on $M$. For $\alpha$ to correspond to a global group cocycle with values in a group of the form $\mathbb{R}/\Gamma$ it is necessary that the image of the period map lies in a discrete subgroup $\Gamma$ of $\mathbb{R}$ (cf. [Ne04], Theorem 5.4 and Theorem 6.7 for $p$=2).

Now, the main problem is to understand if these group cocycles naturally correspond to geometrical objects:

\subsubsection*{p=1}

In this case we have a closed 1-form $\alpha$ on $M$. If $\alpha=df$ is exact, then
\[\alpha_G:\Diff(M)\rightarrow A,\,\,\,\varphi\mapsto\varphi.f-f
\]is a 1-cocycle integrating $\alpha$. If it is not exact, then we consider the universal covering manifold of $q_M:\widetilde{M}\rightarrow M$ (assuming that $M$ is connected). We then obtain an exact sequence of groups

\[1\rightarrow \pi_1(M,m_0)\rightarrow\widehat{\Diff}(M)\stackrel{p_M}{\longrightarrow}\Diff(M)\rightarrow 1,
\]where $q_M\circ\varphi=p_M(\varphi)\circ q_M$. The period map $\per^M_{\alpha}:\pi_1(M,m_0)\rightarrow\mathbb{R}$ can now be used
to obtain an associated flat central extension
\[1\rightarrow \mathbb{R}\rightarrow\widetilde{\Diff}(M):=(\widehat{\Diff}(M)\times\mathbb{R})/\Gamma(\per^M_{\alpha})\longrightarrow\Diff(M)\rightarrow 1,
\]to which the 1-cocycle $\alpha$ integrates. Here,
\[\Gamma(\per^M_{\alpha}):=\{(d,\per^M_{\alpha}(d))|\,d\in\pi_1(M,m_0)\}
\]is the graph of $\per^M_{\alpha}$. From this point of view, 1-cocycles correspond to flat bundles which are associated to the universal covering of the manifold $M$.


\subsubsection*{p=2}

This situation is more involved, although still classical. If the periods of the closed $2$-form $\alpha$ lie in a discrete subgroup $\Gamma$ of $\mathbb{R}$, then $\alpha$ is the curvature of a principal bundle $q:P\rightarrow M$ with structure group $Z:=\mathbb{R}/\Gamma$ (which can be $\mathbb{R}$ or isomorphic to $\mathbb{T}$). Then we obtain the exact sequence
\begin{align}
1\longrightarrow\Gau(P)\cong C^{\infty}(M,Z)\longrightarrow\Aut(P)\longrightarrow\Diff(M)_{[\alpha]}\longrightarrow 1\notag
\end{align}
(cf. Example \ref{Aut(P) as LGE}). In this sense, the bundle $P$, respectively, its automorphism group $\Aut(P)$ is an `` integrated form" of the Lie algebra cocycle $\alpha$ (cf. [Bry93] and [Ko70]). To obtain a more concrete cocycle, one needs a section
\[\sigma:\Diff(M)_{[\alpha]}\rightarrow\Aut(P)
\]so that
\[\alpha_G(\varphi,\phi):=\sigma(\varphi)\sigma(\phi)\sigma(\varphi\phi)^{-1}.
\]is a group cocycle integrating $\alpha$, respectively, its cohomology class. From this point of view, 2-cocycles somehow correspond abelian principal bundles.

\subsubsection*{p=3}

Let $q:P\rightarrow M$ be a principal bundle with structure group $K$ over the compact base manifold $M$ and
\begin{align}
1\longrightarrow\Gau(P)\longrightarrow\Aut(P)\longrightarrow\Diff(M)_{[P]}\longrightarrow 1\label{outlook equation1}
\end{align}
the associated extension of Lie groups from Example \ref{Aut(P) as LGE}. Additionally, let 
\begin{align}
1\longrightarrow Z\longrightarrow\widehat{K}\longrightarrow K\longrightarrow 1\label{outlook equation2}
\end{align}
be a central extension of Lie groups. From (\ref{outlook equation2}) we obtain another central extension of Lie groups given by
\begin{align}
1\longrightarrow C^{\infty}(M,Z)\longrightarrow\widehat{\Gau}(P)\longrightarrow\Gau(P)\longrightarrow 1, \label{outlook equation3}
\end{align}
where
\[\widehat{\Gau}(P):=\{f\in C^{\infty}(P,\widehat{K}):\,(\forall p\in P)(\forall k\in K) f(pk)=\widehat{c}_k.f(p)\}
\]and $\widehat{c}_k$ is the canonical lift of the conjugation map $c_k(x):=kxk^{-1}$ to the central extension $\widehat{K}$ of $K$. 
Then (\ref{outlook equation1}) and (\ref{outlook equation3}) lead to a so-called \emph{crossed module}\index{Crossed Module}
\begin{align}
S:1\rightarrow C^{\infty}(M,Z)\rightarrow\widehat{\Gau}(P)\rightarrow\Aut(P)\rightarrow\Diff(M)_{[P]}\rightarrow 1\label{outlook equation4}
\end{align}
(cf. [Ne07a]). The characteristic class $\chi(S)$ corresponding to this crossed module is an element in the Lie group cohomology
$H^3_{ss}(\Diff(M)_{[P]},C^{\infty}(M,Z))$\sindex[n]{$H^3_{ss}(\Diff(M)_{[P]},C^{\infty}(M,Z))$} and the derivation map
\[D_3:H^3_{ss}(\Diff(M)_{[P]},C^{\infty}(M,Z))\rightarrow H^3_c(\mathcal{V}(M),C^{\infty}(M,\mathbb{R}))
\]turns this element into an element of the Lie algebra cohomology $H^3_c(\mathcal{V}(M),C^{\infty}(M,\mathbb{R}))$\sindex[n]{$H^3_c(\mathcal{V}(M),C^{\infty}(M,\mathbb{R}))$}. 

\begin{open problem}
It is a natural and interesting problem to give a description of those closed 3-forms $\alpha$ on $M$ which arise from crossed modules $S$ of the form (\ref{outlook equation4}).
\end{open problem}

\section{Nonabelian Lie Group Extensions in Noncommutative Geometry}\label{NLGENCG} 

In Example \ref{Aut(P) as LGE} we have seen that every principal bundle $(P,M,G,q,\sigma)$ with compact base manifold $M$ induces a short exact sequence of Lie groups:
\begin{align}
1\longrightarrow\Gau(P)\longrightarrow\Aut(P)\longrightarrow\Diff(M)_{[P]}\longrightarrow 1.\label{NLGENGF1}
\end{align}
From Remark \ref{remark on vector bundles are associated to principal bundles} we conclude that the automorphism group of a vector bundle $(\mathbb{V},M,V,q)$ is isomorphic to the automorphism group of its corresponding frame bundle 
\[(\Fr(\mathbb{V}),M,\GL(V),q_{\Fr(\mathbb{V})}),
\]\sindex[n]{$(\Fr(\mathbb{V}),M,\GL(V),q_{\Fr(\mathbb{V})})$}i.e.,
\begin{align}
\Aut(\mathbb{V})\cong\Aut(\Fr(\mathbb{V}))\label{framebundle=vectorbundle}
\end{align}
\sindex[n]{$\Aut(\mathbb{V})$} and that the gauge group $\Gau(\mathbb{V})$\sindex[n]{$\Gau(\mathbb{V})$} is isomorphic to the unit group of the space of smooth sections of the vector bundle $\End(\mathbb{V})$, i.e.,
\begin{align}
\Gau(\mathbb{V})\cong\GL_{C^{\infty}(M)}(\Gamma\mathbb{V}).\label{gauge group of vector bundle}
\end{align}
From (\ref{NLGENGF1}, (\ref{framebundle=vectorbundle}) and (\ref{gauge group of vector bundle}) we obtain the following short exact sequence of Lie groups associated to the vector bundle $(\mathbb{V},M,V,q)$:
\begin{align}
1\longrightarrow\GL_{C^{\infty}(M)}(\Gamma\mathbb{V})\longrightarrow\Aut(\mathbb{V})\longrightarrow\Diff(M)_{[P]}\longrightarrow 1.\label{Aut(V) as LGE}
\end{align}
Now, let $A$ be a CIA. In view of the Theorem of Serre and Swan (Theorem \ref{Serre-Swan}), finitely generated projective $A$-modules are the noncommutative analogue of classical vector bundles. Thus, it is a natural ambition to try to associate to every finitely generated projective $A$-module $E$ a short exact sequence of Lie groups in a similar way as in (\ref{Aut(V) as LGE}). Unfortunately, automorphism groups of CIA's do not always carry a natural Lie group structure. Nevertheless, it was shown in [Ne08a] that, for a (possibly infinite-dimensional) connected Lie group $G$ and a smooth dynamical system $(A,G,\alpha)$, it is possible to associate to every finitely generated projective $A$-module $E$ a natural extension of Lie groups which is a noncommutative version of (\ref{Aut(V) as LGE}). In particular, this Lie group extension induces a short exact sequence of the corresponding Lie algebras. Thus, Lecomte's Chern--Weil homomorphism of Theorem \ref{Lecomtes's Chern-Weil map} is at our disposal. 

\begin{open problem}{\bf(Characteristic classes for finitely generated projective modules).}\index{Characteristic Classes!for Finitely Generated Projective Modules}
Let $G$ be a (possibly infinite-dimensional) connected Lie group and $(A,G,\alpha)$ a smooth dynamical system. Further, let $E$ be a finitely generated projective $A$-module. Produce characteristic classes for $E$ with the help of the previous construction and Lecomte's Chern--Weil homomorphism of Theorem \ref{Lecomtes's Chern-Weil map}. Of particular interest are the finitely generated projective modules over the smooth noncommutative 2-torus $\mathbb{T}^2_{\theta}$. These are classified by the Grothendieck group $K_0(\mathbb{T}^2_{\theta})\cong\mathbb{Z}^2$ (cf. Proposition \ref{K-groups of noncommutative n-torus}). A very nice description of these modules can be found in [Va06].
\end{open problem}

\begin{appendix}

\chapter{Hopf--Galois Extensions and Differential Structures}\label{appendix A}
 
The noncommutative geometry of principal bundles is, so far, not sufficiently well understood, but there is a well-developed purely algebraic approach using the theory of bialgebras and Hopf algebras commonly known as Hopf--Galois extensions. In this appendix we report on the basics of Hopf--Galois extensions. The following discussion is very much inspired by [BrzWi08]:

\section{Hopf--Galois Extensions}


\begin{definition}{\bf(Coalgebras).}\index{Algebra!Co-}
A vector space $H$ with $\mathbb{K}$-linear maps $\Delta_H:H\rightarrow H\otimes H$\sindex[n]{$\Delta_H$} and $\epsilon_H:H\rightarrow\mathbb{K}$\sindex[n]{$\epsilon_H$} is called a \emph{coalgebra} if the following conditions are satisfied: 
\begin{itemize}
\item[(C1)]
$(\Delta_H\otimes\id_H)\circ\Delta_H=(\id_H\otimes\Delta_H)\circ\Delta_H$.
\item[(C2)]
$(\id_H\otimes\epsilon_H)\circ\Delta_H=(\epsilon_H\circ\id_H)\circ\Delta_H=\id_H$.
\end{itemize}
The map $\Delta_H$ is called a \emph{comultiplication} and $\epsilon_H$ is called a \emph{counit}.
\end{definition}

\begin{remark}{\bf(Sweedler sigma convention).}\index{Sweedler Sigma Convention}
For a coalgebra $H$ we use an explicit expression $\Delta_H(h)=h_0\otimes h_1$, where the summation is implied according to the \emph{Sweedler sigma convention}, i.e. $h_0\otimes h_1=\sum_{i\in I}h^{i}_0\otimes h^{i}_1$ for a finite index set $I$.
\end{remark}

\begin{definition}{\bf(Bialgebras).}\index{Algebra!Bi-}
A vector space $H$ is called a \emph{bialgebra} if the following conditions are satisfied:
\begin{itemize}
\item[(B1)]
$H$ is a unital algebra with multiplication $m=m_H$ and unit $1_H$.
\item[(B2)]
$H$ is a coalgebra with comultiplication $\Delta_H$ and counit $\epsilon_H$.
\item[(B3)]
$\Delta_H$ and $\epsilon_H$ are algebra morphisms.
\end{itemize}
\end{definition}

\begin{example}{\bf(The group algebra I).}\label{the group algebra}\index{Group Algebra}
If $G$ is a group, $\mathbb{K}G$\sindex[n]{$\mathbb{K}G$} denotes the group algebra of $G$ over $\mathbb{K}$, that is, the vector space over $\mathbb{K}$ with $G$ as basis; its product is defined by extending linearly the group multiplication of $G$, so the unit in $\mathbb{K}G$ is the identity element of $G$.
The group algebra $\mathbb{K}G$ is a bialgebra where the coproduct and counit are the respective linear extensions of the diagonal map $\Delta(g):=g\otimes g$ and the constant map $\epsilon(g):=1_{\mathbb{K}}$.
\end{example}

\begin{definition}{\bf(Comodules).}\index{Comodule}
Let $H$ be a bialgebra.

(a) A pair $(V,\rho_V)$\sindex[n]{$(V,\rho_V)$}, consisting of a vector space $V$ and a linear map
$\rho_V:V\rightarrow V\otimes H$, is called a \emph{right $H$-comodule} if the following two conditions are satisfied:
\begin{itemize}
\item[(M1)]
$(\rho_V\otimes\id_H)\circ\rho_V=(\id_V\otimes\Delta_H)\circ\rho_V$.
\item[(M2)]
$(\id_V\otimes\epsilon_H)\circ\rho_V=\id_V$.
\end{itemize}
The map $\rho_V$ is called a \emph{coaction}.

(b) If $A$ is a unital algebra, then a right $H$-comodule $(A,\rho_A)$\sindex[n]{$(A,\rho_A)$} is called a \emph{right $H$-comodule algebra} if $\rho_A$ is a homomorphism of algebras.
\end{definition}

\begin{remark}
Given a right $H$-comodule algebra $(A,\rho_A)$, the algebra structure of $A\otimes H$ is that of a tensor product algebra. For a coaction $\rho_A$ we use an explicit notation $\rho_A(a)=a_0\otimes a_1$, where the summation is also implied. We recall that $a_0\in A$ and $a_1\in H$.
\end{remark}
 
\begin{example}{\bf(Graded algebras I).}\label{g graded algebra I}\index{Algebra!Graded}
If $G$ is a group with neutral element $e$ and $H:=\mathbb{K}G$, then $A$ is a $H$-comodule algebra if and only if $A$ is a $G$-graded algebra, i.e., if
\[A=\bigoplus_{g\in G}A_g,\,\,\,\text{with}\,\,\,A_gA_{g'}\subseteq A_{gg'},\,\,\text{and}\,\,\,1_A\in A_e.
\]Indeed, for $a\in A_g$ we define $\rho_A(a)=a\otimes g$. Since $1_A\in A_e$, we obtain $\rho_A(1_A)=1_A\otimes e=1_A\otimes 1_H$. Moreover, if $a\in A_g$ and $a'\in A_{g'}$, then $aa'\in A_{gg'}$ and thus $\rho_A(aa')=aa'\otimes gg'$ as needed.
\end{example}

\begin{example}{\bf(Regular comodules I).}\label{regular comodules I}\index{Comodule!Regular}
Since the comultiplication $\Delta_H$ in a bialgebra $H$ is an algebra morphism, the pair $(H,\Delta_H)$ carries the structure of a right comodule algebra and is often called a (\emph{right}) \emph{regular comodule algebra}.
\end{example} 

\begin{definition}{\bf(Coinvariants).}\label{coinvariants}\index{Coinvariants}
If $H$ is a bialgebra and $(A,\rho_A)$ a right $H$-comodule algebra, then the set of \emph{coinvariants} is defined by
\[A^H:=\{a\in A:\,\rho_A(a)=a\otimes 1_H\}.
\]\sindex[n]{$A^H$}The set of coinvariants is a subalgebra of $A$, since $\rho_A$ is a homomorphism of algebras.
\end{definition}

\begin{example}{\bf(Regular comodules II).}\label{regular comodules II}\index{Comodule!Regular}
If $(H,\Delta_H)$ is a regular comodule, then a short calculation shows that $H^H=\mathbb{K}\cdot 1_H$. 
\end{example} 

\begin{definition}{\bf(Hopf--Galois extensions).}\label{hopf-galois extension}\index{Hopf--Galois Extensions}
Let $H$ be a bialgebra. A right $H$-comodule algebra $(A,\rho_A)$ is called a \emph{Hopf--Galois extension} (of $B:=A^H$) or \emph{quantum principal bundle} if the \emph{canonical map}
\begin{align}
\chi: A\otimes_B A\rightarrow A\otimes H,\,\,\,a\otimes a'\mapsto a\rho_A(a')=aa'_0\otimes a'_1\label{hopf galois}
\end{align}
is bijective, i.e., an isomorphism of left $A$-modules and right $H$-comodules.
\end{definition}

\begin{remark}{\bf(Principal bundles).}\label{hopf-galois extension=principal bundles}
Definition \ref{hopf-galois extension} reproduces the classical algebraic situation (but in a dual language) in which a group $G$ acts freely on a space $X$. In fact, if 
\[X\times_{X/G} X:=\{(x,x')\in X\times X:\,x.G=x'.G\}
\]then the freeness of the action of $G$ on $X$ means that the map
\[X\times G\rightarrow X\times_{X/G} X,\,\,\,(x,g)\mapsto (x,xg)
\]is bijective. For this reason Hopf--Galois extensions are sometimes interpreted as ``noncommutative principal bundles".
\end{remark}

\begin{example}{\bf(Graded algebras II).}\label{g graded algebra III}\index{Algebra!Graded}
If $G$ is a group with neutral element $e$, $H:=\mathbb{K}G$ and $A=\bigoplus_{g\in G}A_g$, then $A$ is a Hopf--Galois extension (of $A_e$) if and only if $A$ is strongly graded, i.e., $A_gA_{g'}=A_{gg'}$ for all $g,g'\in G$. In this case we have for all $a\in A_g$ and $a'\in A_{g'}$
\[\chi(a\otimes_Ba')=aa'\otimes g'\,\,\,\text{and}\,\,\,\chi^{-1}(a\otimes g)=\sum_{i}aa'_i\otimes_Ba_i,
\]where $a_i\in A_g$ and $a'_i\in A_{g^{-1}}$ are such that $\sum_{i}a'_i a_i=1_A$.
\end{example}

\subsection*{Hopf Algebras}

\begin{definition}{\bf(Antipodes).}\label{antipodes}\index{Anitpode}
If $H$ is a bialgebra, then an anti-algebra homomorphism $S:H\rightarrow H$ satisfying 
\[m\circ(id_H\otimes S)=m\circ(S\otimes\id_H)=\epsilon_H1_H.
\]is called an \emph{antipode}.
\end{definition}

\begin{definition}{\bf(Hopf algebras).}\label{hopf algebra}\index{Algebra!Hopf}
A bialgebra $H$ with an antipode $S$ is called a \emph{Hopf algebra}.
\end{definition}

\begin{remark}
One should consider a Hopf algebra as a noncommutative generalization of the algebra of representative functions on a compact group ([JoSt91]). In this case $\Delta$ corresponds to the group multiplication, $\epsilon$ corresponds to the unit in a group and the antipode $S$ corresponds to the  inversion, written in a dual form. For this reason Hopf algebras are also called \emph{quantum groups}.
\end{remark}

\begin{example}{\bf(The group algebra II).}\label{the group algebra II}\index{Group Algebra}
The group algebra $\mathbb{K}G$ is a Hopf algebra where the antipode is the respective linear extension of the inverse map $S(g):=g^{-1}$.
For a finite group, it is easy to check that the algebra of representative functions on $G$ is the dual bialgebra of the group algebra $\mathbb{K}G$.
\end{example}

\begin{example}{\bf(Regular comodules III).}\label{regular comodules III}\index{Comodule!Regular}
A regular comodule $(H,\Delta_H)$ is a Hopf--Galois extension (of $\mathbb{K}\cdot 1_H$ cf. \ref{regular comodules II}) if and only if $H$ is a Hopf algebra. Indeed, if $H$ is a Hopf algebra, then the inverse of the canonical map is given by
\[\chi^{-1}(h'\otimes h)=h'S(h_0)\otimes h_1.
\]Conversely, if $\chi^{-1}$ exists, then the map 
\[S:=(\id_H\otimes\epsilon_H)\circ\chi^{-1}\circ(1_H\otimes \id_H)
\]has the required properties.
\end{example} 

\subsection*{Cleft Extensions}

Roughly speaking the concept of cleft extensions may be considered as an algebraic formulation of trivial principal bundles in the context of Hopf--Galois extensions:

\begin{construction}
If $B$ is a unital algebra and $H$ a Hopf algebra, then $B\otimes H$ carries the structure of a right $H$-comodule algebra with coaction given by $\id_B\otimes\Delta_H$. We further have $(B\otimes H)^H\cong B$. The canonical map is given by
\[\chi:B\otimes H\otimes H\rightarrow B\otimes H\otimes H,\,\,\,b\otimes h'\otimes h\mapsto b\otimes h'h_0\otimes h_1
\]and thus $\chi$ is bijective with inverse
\[\chi^{-1}:B\otimes H\otimes H\rightarrow B\otimes H\otimes H,\,\,\,b\otimes h'\otimes h\mapsto b\otimes h'S(h_0)\otimes h_1.
\]In particular, $(B\otimes H,\id_B\otimes\Delta_H)$ is a Hopf--Galois extension (of $B$).
\end{construction}

\begin{definition}{\bf(Normal basis property).}\label{normal base property}\index{Normal Base Property}
If $H$ is a bialgebra and $(A,\rho_A)$ a Hopf--Galois extension of $B(=A^H)$, then $(A,\rho_A)$ is said to have a \emph{normal basis property} if $A\cong B\otimes H$ as a left $B$-module and right $H$-comodule.
\end{definition}

\begin{definition}{\bf(Convolution product).}\index{Convolution Product}
Let $A$ be a unital algebra, $H$ a Hopf algebra and $f,g:H\rightarrow A$ linear maps. The \emph{convolution product} of $f$ and $g$ is a linear map $f\ast g:H\rightarrow A$ given by $(f\ast g)(h)=f(h_0)g(h_1)$ for any $h\in H$. With respect to the convolution product, the set of all linear maps $H\rightarrow A$ forms an associative algebra with the unit $1_A\cdot\epsilon_H$. We say that a linear map $f:H\rightarrow A$ is \emph{convolution invertible} if there is a map $g:H\rightarrow A$ such that $f\ast g=g\ast f=1_A\cdot\epsilon_H$. The set of all convolution invertible maps $H\rightarrow A$ forms a multiplicative group and the antipode $S$ of $H$ is convolution invertible to the identity $\id_H$. 
\end{definition}

\begin{definition}{\bf(Cleft extension).}\label{cleft extension}\index{Cleft Extension}
A right $H$-comodule algebra $(A,\rho_A)$ such that there is a convolution invertible right $H$-comodule map $j:H\rightarrow A$ is called a \emph{cleft extension}. The map $j$ is called a \emph{cleaving map}.
\end{definition}

\begin{proposition}\label{cleft extension prop}
If $(A,\rho_A)$ is a right $H$-comodule algebra and $B:=A^H$, then the following statements are equivalent:
\begin{itemize}
\item[\emph{(a)}]
$(A,\rho_A)$ is a Hopf--Galois extension with a normal basis property.
\item[\emph{(b)}]
There exists a convolution invertible right $H$-comodule map $j:H\rightarrow A$.
\end{itemize}
\end{proposition}

\begin{proof}
\,\,\,For the proof of this proposition we refer to [BrWi08], Proposition 2.22.
\end{proof}

\begin{remark}{\bf(The NC Hopf fibration).}
A very nice class of non-cleft Hopf Galois extensions can be found in [LaSu04]. In fact, the authors construct non-cleft Hopf Galois extensions which are deformations of the classical $\SU_2(\mathbb{C})$ Hopf fibration.
\end{remark}

\section{Differential Structures on Hopf Algebras}\label{differential structures}

The notion of connections plays a key role in the differential theory of principal bundles. In this part of the appendix we discuss a similar notion in the algebraic context of Hopf--Galois extensions. All proofs of the following statements can be found in [BrWi08]:

\begin{definition}{\bf(The Universal differential structure).}\label{DSHA1}\index{Universal!Differential Structure}
Let $A$ be a unital algebra over a field $\mathbb{K}$. We write $\Omega^1A$\sindex[n]{$\Omega^1A$} for the $A$-bimodule $\ker m$, where $m:A\otimes A\rightarrow A$ denotes the multiplication map on $A$. Further, we write $d_A:A\rightarrow\Omega^1A$ for the linear map given by
\[d_Aa:=1\otimes a-a\otimes 1.
\]
The pair $(\Omega^1A, d_A)$\sindex[n]{$(\Omega^1A, d_A)$} is called the \emph{universal differential structure on} $A$. The space $\Omega^1A$ may be understood as a bimodule of 1-forms. For more details on differential structures we refer to ([Brz97]).
\end{definition}

\begin{lemma}\label{DSHA2}
If $(A,\rho_A)$ is a right $H$-comodule algebra, then $\Omega^1A$ is a right $H$-comodule by
\[\rho_{\Omega^1A}:\Omega^1A\rightarrow\Omega^1A\otimes H,\,\,\,\sum_ia_i\otimes a'_i\mapsto \sum_i(a_i)_0\otimes(a'_i)_0\otimes(a_i)_1(a'_i)_1.
\]Furthermore, the map $d$ is a right $H$-comodule map.
\begin{flushright}
$\blacksquare$
\end{flushright}
\end{lemma}

\begin{definition}{\bf(Horizontal 1-forms).}\label{DSHA3}\index{Horizontal!1-Forms}
Let $(A,\rho_A)$ be a right $H$-comodule algebra and $B:=A^H$. The $A$-subbimodule $\Omega_{\text{hor}}^1A$ of $\Omega^1A$ generated by all $d_B(b)$, $b\in B$, i.e., the $A$-subbimodule
\[\Omega_{\text{hor}}^1A:=A(\Omega^1B)A
\]is called the module of \emph{horizontal one-forms}.\sindex[n]{$\Omega_{\text{hor}}^1A$}
\end{definition}

\begin{definition}{\bf(Connections for Hopf--Galois extensions).}\label{DSHA4}\index{Hopf--Galois Extensions!Connections for}
A \emph{connection} for a Hopf--Galois Extension $(A,\rho_A)$ is a left $A$-linear map $\Pi:\Omega^1A\rightarrow\Omega^1A$ satisfying
\begin{itemize}
\item[(i)]
$\Pi\circ\Pi=\Pi$,
\item[(ii)]
$\ker\Pi=\Omega_{\text{hor}}^1A$,
\item[(iii)]
$(\Pi\otimes\id_H)\circ\rho_{\Omega^1A}=\rho_{\Omega^1A}\circ\Pi$.
\end{itemize}
In other words, a connection is an $H$-covariant splitting of $\Omega^1A$ into the horizontal and ``vertical" parts.
\end{definition}

\begin{definition}{\bf(Vertical lifts for Hopf--Galois extensions).}\label{DSHA5}\index{Hopf--Galois Extensions!Vertical Lifts for}
Let $(A,\rho_A)$ be a right $H$-comodule algebra. For $B:=A^H$ and $H^{+}:=\ker\epsilon_H$ we define
\[\ver:\Omega^1A\rightarrow A\otimes H^{+},\,\,\,\sum_ia_i\otimes a'_i\mapsto\sum_ia'_i\rho_A(a_i)
\]and call $\ver$ a \emph{vertical lift}.\sindex[n]{$H^{+}$}\sindex[n]{$\ver$}
\end{definition}

\begin{definition}{\bf(Connection forms for Hopf--Galois extension).}\label{DSHA6}\index{Hopf--Galois Extensions!Connection Forms for}
A \emph{connection form} for a Hopf--Galois Extension $(A,\rho_A)$ is a linear map $\omega:H^{+}\rightarrow\Omega^1A$ satisfying
\begin{itemize}
\item[(i)]
$\rho_{\Omega^1A}\circ\omega=(\omega\otimes\id_H)\circ\Ad$,
\item[(ii)]
$\ver\circ\omega=1_A\otimes\id_{H^{+}}$.
\end{itemize}
\end{definition}

\begin{theorem}\label{DSHA7}
Connections for a Hopf--Galois Extension $(A,\rho_A)$ are in bijective correspondence with connection forms.
\begin{flushright}
$\blacksquare$
\end{flushright}
\end{theorem}

\begin{remark}\label{DSHA8}
(a) We recall that if $B$ is a unital algebra and $E$ a left $B$-module, then a \emph{connection} for $E$ is a linear map
\[\nabla:E\rightarrow \Omega^1B\otimes_B E
\]satisfying $\nabla(be)=d_B(b)\otimes_Be+b\nabla(e)$ for all $b\in B$ and $e\in E$.

(b) A connection for $E$ exists if and only if $E$ is a projective $B$-module. In fact, this statement is due to a theorem of Cuntz and Quillen (cf. Theorem \ref{Cuntz-Quillen}).

(c) A connection for a Hopf--Galois extension $(A,\rho_A)$ with $B:=A^H$ does in general not induce a connection in the left $B$-module $A$. Only connections which are related to a more restrictive notion of horizontal forms yield connections in modules.
\end{remark}

\begin{definition}{\bf(Strong connections for Hopf--Galois extension).}\label{DSHA9}\index{Hopf--Galois Extension!Strong Connections for}
If $\Pi$ is a connection for a Hopf--Galois Extension $(A,\rho_A)$, then the right $H$-comodule map
\[D:A\rightarrow\Omega_{\text{hor}}^1A,\,\,\,D:=d_A-\Pi\circ d_A
\]is called a \emph{covariant derivative} corresponding to $\Pi$. The connection $\Pi$ is called a \emph{strong connection} if $D(A)\subseteq(\Omega^1B)A$.
\end{definition}

\begin{lemma}\label{DSHA10}
If $D$ is a covariant derivative corresponding to a strong connection for a Hopf--Galois Extension $(A,\rho_A)$, then $D$ is a connection for the left $B$-module $A$.
\begin{flushright}
$\blacksquare$
\end{flushright}
\end{lemma}

\begin{definition}{\bf(Strong connection forms for Hopf--Galois extension).}\label{DSHA11}\index{Hopf--Galois Extensions!Strong Connection Forms for}
A connection form $\omega$ for a Hopf--Galois Extension $(A,\rho_A)$ such that its associated connection is a strong connection is called a \emph{strong connection form}.
\end{definition}

\begin{theorem}\label{DSHA12}
If $H$ is a Hopf algebra and $(A,\rho_A)$ a Hopf--Galois extension with $B:=A^H$, then a strong connection exists if and only if $A$ is $H$-equivariantly projective as a left $B$-module, i.e., if and only if there exists a left $B$-module, right $H$-comodule section of the multiplication map $B\otimes A\rightarrow A$.
\begin{flushright}
$\blacksquare$
\end{flushright}
\end{theorem}

\begin{corollary}\label{DSHA13}
If $(A,\rho_A)$ is a Hopf--Galois extension \emph{(}with $B:=A^H$\emph{)} admitting a strong connection, then the following assertions hold:
\begin{itemize}
\item[\emph{(a)}]
$A$ is projective as a left $B$-module.
\item[\emph{(b)}]
$B$ is a direct summand of $A$ as a left $B$-module.
\item[\emph{(c)}]
$A$ is faithfully flat as a left $B$-module.
\end{itemize}
\begin{flushright}
$\blacksquare$
\end{flushright}
\end{corollary}

\begin{theorem}\label{DSHA14}
If a Hopf algebra $H$ has a bijective antipode $S$, then strong connections of a Hopf Galois extension $(A,\rho_A)$ \emph{(}with $B:=A^H$\emph{)} are in bijective correspondence with linear maps $l:H\rightarrow A\otimes A$ satisfying
\begin{itemize}
\item[\emph{(i)}]
$l(1_H)=1_A\otimes 1_A$,
\item[\emph{(ii)}]
$\chi\circ l=1_A\otimes\id_H$,
\item[\emph{(iii)}]
$(l\otimes\id_H)\circ\Delta_H=(\id_A\otimes\rho_A)\circ l$,
\item[\emph{(iv)}]
$(\id_H\otimes l)\circ\Delta_H=(\rho_A\otimes\id_A)\circ l$.
\end{itemize}
Such a map $l$ is also called a \emph{strong connection}.
\begin{flushright}
$\blacksquare$
\end{flushright}
\end{theorem}

\begin{theorem}\label{DSHA15}
Let $H$ be a Hopf algebra with bijective antipode. If $(A,\rho_A)$ is a right $H$-comodule algebra, then the following statements are equivalent:
\begin{itemize}
\item[\emph{(a)}]
There exists a linear map $l:H\rightarrow A\otimes A$ satisfying the properties \emph{(}i\emph{)}-\emph{(}iv\emph{)} of Theorem \ref{DSHA14}.
\item[\emph{(b)}]
$A$ is a faithfully flat \emph{(}as a left and right $B$-module\emph{)} Hopf--Galois extension.
\end{itemize}
\begin{flushright}
$\blacksquare$
\end{flushright}
\end{theorem}

\begin{definition}{\bf(Principal comodule algebras).}\label{DSHA16}\index{Algebra!Principal Comodule}
Let $H$ be a Hopf algebra with bijective antipode. A right $H$-comodule algebra $(A,\rho_A)$ satisfying the conditions of Theorem \ref{DSHA15} is called a \emph{principal comodule algebra}.
\end{definition}

\begin{remark}
Principal comodule algebras are a noncommutative version of principal bundles which retains a lot of the features of the classical differential geometry objects.
\end{remark}


\chapter{Some Background on Lie Theory}\label{Lie-Shit}

Nowadays Lie theory is useful for many branches of mathematics, like differential geometry, harmonic analysis, and representation theory. 
In this appendix we provide the basic definitions and concepts of Lie theory which appear throughout this thesis.

\section{Manifolds, Lie Groups and Lie Algebras}

\begin{definition}\label{infinite-dimensional manifolds}{\bf (Manifolds).}\index{Manifolds}
Let $E$ and $F$ be locally convex spaces, $U\subseteq E$ open and $f:U\rightarrow F$ a map. Then the \emph{derivative of} $f$ \emph{at} $x$ \emph{in the direction of} $h$ is defined as
\[df(x)(h):=\lim_{t\rightarrow 0}\frac{1}{t}(f(x+th)-f(x))
\]whenever the limit exists. The function $f$ is called \emph{differentiable at} $x$ if $df(x)(h)$ exists for all $h\in E$. It is called \emph{continuously differentiable} or $C^1$ if it is continuous, differentiable at all points of $U$, and
\[df:U\times E\rightarrow F,\,\,\,(x,h)\mapsto df(x)(h)
\]is a continuous map. It is called a $C^n$-map if $f$ is $C^1$ and $df$ is a $C^{n-1}$-map, and $C^{\infty}$ or smooth if it is $C^n$ for all $n\in\mathbb{N}$. This is the notion of differentiability used in [Mil83] and [Gl01]. Since we have a chain rule for $C^1$-maps between locally convex spaces, we can define smooth manifolds as in the finite-dimensional case.
\end{definition}

\begin{definition}\label{lie groups}{\bf (Lie groups).}\index{Lie!Groups}
 A \emph{Lie group} $G$ is a smooth manifold modeled on a locally convex space for which the group multiplication and the inversion are smooth maps. For basic results on Lie groups in this context we refer to [Mil83] f. We write $1_G\in G$ for the identity,
\[\lambda_{g}(x):=gx,\,\,\,\rho_{g}(x):=xg\,\,\,\text{and}\,\,\,c_{g}(x):=gxg^{-1}
\]for the left and right translations, and conjugation on $G$, respectively. In particular, we write $C_G:G\rightarrow\Aut(G)$ for the conjugation action of $G$. For a normal subgroup $N$ of $G$ we write $C_{N}:G\rightarrow\Aut(N)$ with $C_{N}(g):=C_{G}(g)_{\mid N}$.
\end{definition}

\begin{definition}\label{lie subgroups}{\bf (Lie subgroups).}\index{Lie!Subgroups}
A subgroup $H$ of a Lie group $G$ is called an \emph{initial Lie subgroup} if it carries a Lie group structure for which the inclusion map $\iota:H\hookrightarrow G$ is an \emph{immersion} in the sense that the tangent map $T\iota:TH\rightarrow TG$ is fiberwise injective and for each smooth map $f:M\rightarrow G$ from a smooth manifold $M$ to $G$ with $\im(f)\subseteq H$ the corresponding map $\iota^{-1}\circ f:M\rightarrow H$ is smooth. The latter property implies in particular that the initial Lie group structure on $H$ is unique, because if $\iota':H'\hookrightarrow G$ also is a smooth injective immersion of a Lie group with $\iota'(H')=H$ and such that for each smooth map $f:M\rightarrow G$ with $\im(f)\subseteq H$ the map $i'\circ f:M\rightarrow H'$ is smooth, then in particular the maps $\iota^{-1}\circ\iota':H'\rightarrow H$ and $\iota^{'-1}\circ \iota:H\rightarrow H'$ are smooth morphisms of Lie groups. Since these maps are each others inverse, the Lie group $H$ and $H'$ are isomorphic.
\end{definition}

\begin{definition}\label{lie algebras of lie groups}{\bf (Lie algebras of Lie groups).}\index{Lie!Algebras}
Each element $x\in T_{1_G}(G)$ corresponds to a unique left invariant vector field $x_{l}$ with $g.x:=x_{l}(g):=d\lambda_{g}(1).x$, $g\in G$. The space of left invariant vector fields is closed under the Lie bracket of vector fields. Hence it inherits a Lie algebra structure. In this sense we obtain a continuous Lie bracket on the locally convex space $T_{1_G}(G)\cong \mathfrak{g}$ which is uniquely determined by the relation $[x,y]_{l}=[x_{l},y_{l}]$. We call a Lie algebra $\mathfrak{g}$ a $\emph{topological Lie algebra}$ if $\mathfrak{g}$ is a topological vector space and the Lie bracket is continuous. In this sense the Lie algebra $\textbf{L}(G):=(T_1(G),[\cdot,\cdot])$ of a Lie group $G$ is a locally convex Lie algebra.\sindex[n]{$\mathfrak{g}$}
\end{definition}

%
%
%
%
%
%
%
\section{Some Lie Group Cohomology}\label{some lie group cohomology}


In this part of the appendix we introduce the concept of smooth cochains and provide some calculations that will be crucial for the forthcoming section on the extension theory of Lie groups. Moreover, we describe a natural cohomology for Lie groups.
 
\begin{definition}\label{smooth cochains}\index{Smooth!Cochains}
(a) If $G$ and $N$ are Lie groups, then we put $C^0_s(G,N):=N$ and for $p\in\mathbb{N}$ we call a map $f:G^p\rightarrow N$ \emph{locally smooth} if there exists an open identity neighbourhood $U\subseteq G^p$ such that $f_{\mid U}$ is smooth. We say that $f$ is \emph{normalized} if
\[(\exists j)\,g_j= 1_G\,\,\,\Rightarrow\,\,\,f(g_1,\ldots,g_p)=1_N.
\]We write $C^p_s(G,N)$\sindex[n]{$C^p_s(G,N)$} for the space of all normalized locally smooth maps $G^p\rightarrow N$, the so called (locally smooth) p-$\emph{cochains}$.

(b) For $p=2$ and non-connected Lie groups, we sometimes have to require additional smoothness: We write $C^2_{ss}(G,N)$\sindex[n]{$C^2_{ss}(G,N)$} for the set of all elements $\omega\in C^2_{s}(G,N)$ with the additional property that for each $g\in G$ the map
\[\omega_g:G\rightarrow N,\,\,\,x\mapsto\omega(g,x)\omega(gxg^{-1},g)^{-1}
\]is smooth in an identity neighbourhood of $G$. Note that $\omega_g(1_G)=1_N$.

(c) For $f\in C^1_s(G,N)$ we define
\[\delta_f:G\times G\rightarrow N,\,\,\,\delta_f(g,g'):=f(g)f(g')f(gg')^{-1}
\]and observe that
\begin{align}
(\delta_f)_g(x)&=\delta_f(g,x)\delta_f(gxg^{-1},g)^{-1}=f(g)f(x)f(gx)^{-1}f(gx)f(g)^{-1}f(gxg^{-1})^{-1}\notag\\
&=f(g)f(x)f(g^{-1})f(gxg^{-1})^{-1}\notag
\end{align}
is smooth in an identity neighbourhood of $G$, so that $\delta_f\in C^2_{ss}(G,N)$.

(d) If $G$ and $N$ are abstract groups, we write $C^p(G,N)$ for the set of all functions $G^p\rightarrow N$.
\end{definition}

Now, for cochains with values in abelian groups or smooth modules, we can define a natural Lie group cohomology:

\begin{definition}\label{lie group cohomology}\index{Lie!Group Cohomology}
Let $G$ be a Lie group and $A$ a \emph{smooth $G$-module}\index{Smooth!$G$-Module}, i.e., a pair $(A,S)$ of an abelian Lie group $A$ and a smooth $G$-action $G\times A\rightarrow A$ given by the group homomorphism $S:G\rightarrow \Aut(A)$. We call the elements of $C^p_s(G,A)$ (\emph{locally smooth normalized $p$-cochains}) and write
\[d_S:C^p_s(G,A)\rightarrow C^{p+1}_s(G,A)
\]for the group differential given by
\[(d_Sf)(g_0,\ldots,g_p):=
\]
\[S(g_0)(f(g_1,\ldots,g_p))+\sum^{p}_{j=1}(-1)^jf(g_0,\ldots ,g_{j-1}g_j,\ldots,g_p)+(-1)^{p+1}f(g_0,\ldots,g_{p-1}).
\]It is easy to verify that 
\[d_S(C^p_s(G,A))\subseteq C^{p+1}_s(G,A).
\]We thus obtain a sub-complex of the standard group cohomology complex $(C^{\bullet}(G,A),d_S)$\sindex[n]{$(C^{\bullet}(G,A),d_S)$}. In view of $d^2_S=0$, the space \[Z^p_s(G,A)_S:=\ker ({d_S}_{\mid C^p_s(G,A)})
\]of \emph{$p$-cocycles}\sindex[n]{$Z^p_s(G,A)_S$} contains the space $B^p_s(G,A)_S:=d_S(C^{p-1}_s(G,A))$ of $p$-\emph{coboundaries}\sindex[n]{$B^p_s(G,A)_S$}.
The quotient
\[H^p_s(G,A)_S:=Z^p_s(G,A)_S/B^p_s(G,A)_S
\]is the $p^\emph{th}$ \emph{locally smooth cohomology group of $G$ with values in the $G$-module $A$}\sindex[n]{$H^p_s(G,A)_S$}. We write $[f]\in H^p_s(G,A)_S$ for the cohomology class of a cocycle $f\in Z^p_s(G,A)_S$.\index{Cocycles}\index{Coboundaries}
\end{definition}

In the extension theory of Lie groups and for constructing characteristic classes we have to introduce the concept of a smooth outer action of $G$ on $N$:

\begin{definition}\label{smooth outer actions}\index{Smooth!Outer Action}
(a) Let $G$ and $N$ be Lie groups. We define $C^1_s(G,\Aut(N))$ as the set of all maps $S:G\rightarrow \Aut(N)$ with $S(1_G)=\id_N$ and for which there exists an open identity neighbourhood $U\subseteq G$ such that the map
\[U\times N\rightarrow N,\,\,\,(g,n)\mapsto S(g)(n)
\]is smooth.

(b) We call a map $S\in C^1_s(G,\Aut(N))$ a \emph{smooth outer action of $G$ on $N$} if there exists an element $\omega\in C^2_s(G,N)$ with $\delta_S=C_N\circ\omega$. It is called \emph{strongly smooth}\index{Smooth!Strongly-} if $\omega$ can be chosen in the smaller set $C^2_{ss}(G,N)$.

(c) On the set of smooth outer action we define an equivalence relation by
\[S\sim S'\,\,\,\leftrightarrow\,\,\,(\exists h\in C^1_s(G,N))S'=(C_N\circ h)\cdot S.
\]We write $[S]$\sindex[n]{$[S]$} for the equivalence class of $S$ and call $[S]$ a \emph{smooth $G$-kernel}\index{Smooth!$G$-Kernel}.
\end{definition}

\begin{remark}\label{remarks on smooth outer actions}
(a) A smooth outer action $S:G\rightarrow \Aut(N)$ need not be a group homomorphism, but $\delta_S=C_N\circ\omega$ implies that $S$ induces a group homomorphism 
\[s:=Q_N\circ S:G\rightarrow \Out(N),
\]where $Q_N:\Aut(N)\rightarrow \Out(N)$ denotes the quotient homomorphism. 

(b) If $G$ is discrete, then for each homomorphism $s:G\rightarrow \Out(N)$ there exists a map 
\[S:G\rightarrow \Aut(N)
\]satisfying $S(1_G)=\id_N$ and $Q_N\circ S=s$, and further a map $\omega:G\times G\rightarrow N$ with $\delta_S=C_N\circ\omega$. In this case all outer actions are smooth, the smoothness conditions on $S$ and $\omega$ are vacuous.

(c) If $S\in C^1_s(G,\Aut(N))$ is a homomorphism of groups , then we may choose $\omega=1_N$ and $S$ defines a smooth action of $G$ on $N$. In fact, the action is smooth on a set of the form $U\times N$, so that the assertion follows from the fact that all automorphisms $S(g)$ are smooth.
\end{remark}

\begin{example}
 We consider the Lie group 
\[N:=\mathbb{C}^\mathbb{N}\times\mathbb{R}^\mathbb{N}\,\,\,\text{with}\,\,\,((z_n),(t_n))\ast((z'_n),(t'_n)):=((z_n+e^{t_n}z'_n),(t_n+t'_n)),
\]where the locally convex structure on $\mathbb{C}^\mathbb{N}$ and $\mathbb{R}^\mathbb{N}$ is the product topology. We observe that $N$ is an infinite topological product of groups isomorphic to $\mathbb{C}\rtimes\mathbb{R}$ endowed with the multiplication
\[(z,t)\ast(z',t')=(z+e^tz',t+t').
\]Therefore, the center of $N$ is 
\[Z(N)=\{0\}\times(2\pi \mathbb{Z})^{\mathbb{N}}\cong \mathbb{Z}^{\mathbb{N}},
\]and this group is totally disconnected, but not discrete and hence not a Lie group. Therefore, the center of a Lie group need not be a Lie group.
\end{example}

\begin{definition}\label{center of N}
If $S$ is a smooth outer action of the Lie group $G$ on the Lie group $N$, then we define
\[C^p_s(G,Z(N)):=\{f\in C^p_s(G,N):\im(f)\subseteq Z(N)\}\,\,\,\text{and}\,\,\,C^2_{ss}(G,Z(N)):=C^2_{ss}(G,N)\cap C^2_s(G,Z(N)).
\]\sindex[n]{$C^p_s(G,Z(N))$}\sindex[n]{$C^2_{ss}(G,Z(N))$}Moreover, the $G$-module structure on $Z(N)$ produces a group differential $d_S$ on $C^p_s(G,Z(N))$ given in additive notation by
\begin{align}
&(d_Sf)(g_0,\ldots,g_p)\notag\\
&:=S(g_0)(f(g_1,\ldots,g_p))+\sum^{p}_{j=1}(-1)^jf(g_0,\ldots g_{j-1}g_j,\ldots,g_p)+(-1)^{p+1}f(g_0,\ldots,g_{p-1}).\notag
\end{align}
We thus obtain a differential complex $(C^{\bullet}_s(G,Z(N)),d_S)$\sindex[n]{$(C^{\bullet}_s(G,Z(N)),d_S)$} and put
\[Z^p_s(G,Z(N))_S:=\ker({d_S}_{\mid C^p_s(G,Z(N))})
\]and $Z^2_{ss}(G,Z(N))_S:=Z^2_s(G,Z(N))_S\cap C^2_{ss}(G,Z(N))$, respectively. The corresponding cohomology groups are denoted by
\[H^p_s(G,Z(N))_S:=Z^p_s(G,Z(N))_S/d_S(C^{p-1}_s(G,Z(N))).
\]\sindex[n]{$Z^p_s(G,Z(N))_S$}\sindex[n]{$Z^2_{ss}(G,Z(N))_S$}\sindex[n]{$H^p_s(G,Z(N))_S$}
For $\alpha\in C^1_s(G,Z(N)))$ and $g\in G$ we obtain
\begin{align}
(d_S\alpha)_g(x)&=(d_S\alpha)(g,x)-(d_S\alpha)(gxg^{-1},g)\notag\\
&=g.\alpha(x)-\alpha(gx)+\alpha(g)-(gxg^{-1}.\alpha(g)+\alpha(gx)-\alpha(gxg^{-1})\notag\\
& =g.\alpha(x)+\alpha(g)-gxg^{-1}.\alpha(g)-\alpha(gxg^{-1}),\notag
\end{align}
which is smooth in an identity neighbourhood of $G$. Hence,
\[B^2_s(G,Z(N))_S:=d_S(C^1_s(G,Z(N)))\subseteq Z^2_{ss}(G,Z(N))_S
\]and we define
\[H^2_{ss}(G,Z(N))_S:=Z^2_{ss}(G,Z(N))_S/B^2_s(G,Z(N))_S
\]and
\[H^3_{ss}(G,Z(N))_S:=Z^3_s(G,Z(N))_S/d_S(C^2_{ss}(G,Z(N))).
\]We observe that the action of $G$ on $Z(N)$ defined by a smooth outer action $S$ only depends on the equivalence class $[S]$. In this sense we also write
\[Z^p_s(G,Z(N))_{[S]}:=Z^p(G,Z(N))_S,\,\,\,H^p_s(G,Z(N))_{[S]}:=H^p_s(G,Z(N))_S\,\,\,\text{etc}.
\]\sindex[n]{$B^2_s(G,Z(N))_S$}\sindex[n]{$H^2_{ss}(G,Z(N))_S$}\sindex[n]{$H^3_{ss}(G,Z(N))_S$}\sindex[n]{$Z^p_s(G,Z(N))_{[S]}$}
\end{definition}

\begin{remark}
If $Z(N)$ is an initial Lie subgroup of the Lie group $N$, then it is possible to show that it carries the structure of a smooth $G$-module, and the definitions above are consistent with  Definition \ref{lie group cohomology}. But in general, we do not want to assume that $Z(N)$ is an initial Lie subgroup and it is not necessary for the following results.
\end{remark}

\section{Extensions of Lie Groups}
It is a natural ambition to transfer the theory of abstract group extensions to the case of Lie group extensions. Note that non-abelian extensions of Lie groups occur quite naturally in the context of smooth principal bundles over compact manifolds. We will now discuss the basic definitions related to extensions of Lie groups. Since every discrete group can be viewed as a Lie group, our discussion includes in particular the algebraic context of group extensions.


\begin{definition}\label{split lie groups}{\bf (Split Lie groups).}\index{Lie! Split Groups}
Let $G$ be a Lie group. A subgroup $H$ of $G$ is called a $\emph{split Lie subgroup}$ if it carries a Lie group structure for which the inclusion map $i_H:H\hookrightarrow G$ is a morphism of Lie groups and the right action of $H$ on $G$ defined by restricting the multiplication map of $G$ to a map $G\times H\rightarrow G$ defines a smooth $H$-principal bundle. This means that the coset space $G/H$ is a smooth manifold and that the quotient map $\pr:G\rightarrow G/H$ has smooth local sections.
\end{definition}


\begin{proposition}\label{split-initial} 
Each split Lie subgroup of $G$ is initial.
\end{proposition}

\begin{proof}\,\,\,The condition that $H$ is a split Lie subgroup implies that there exists an open subset $U$ of some locally convex space $V$ and a smooth map $\sigma:U\rightarrow G$ such that the map
\[U\times H\rightarrow G,\,\,\,\,\,\,\,\,(g,h)\mapsto \sigma(x)h
\]is a diffeomorphism onto an open subset of $G$. We conclude that for every manifold $M$ a map $f:M\rightarrow H$ is smooth whenever it is smooth as a map to $G$, i.e., to the open subset $\sigma(U)H$ of $G$.
\end{proof}

\begin{definition}\label{extension of lie groups}{\bf (Extensions of Lie groups).}\index{Extensions! of Lie Groups}
An $\emph{extension of Lie groups}$ is a short exact sequence
\[1\longrightarrow N\stackrel{i}{\longrightarrow}\widehat{G}\stackrel{q}{\longrightarrow}G\longrightarrow1
\]of Lie group morphisms, for which $N\cong \ker(q)$ is a split Lie subgroup. This means that $\widehat{G}$ is a smooth $N$-principal bundle over $G$ and $G\cong \widehat{G}/N$. The extension is called \emph{abelian}, resp., \emph{central} if $N$ is abelian, resp., central in $\widehat{G}$.
\end{definition}

\begin{definition}\label{equivalent extensions}{\bf (Equivalence of Lie group extensions).}\index{Extensions!Equivalence of Lie Group}
(a) We call two extensions $N\hookrightarrow\widehat{G}_1\rightarrow G$ and $N\hookrightarrow\widehat{G}_2\rightarrow G$ of the Lie group $G$ by the Lie group $N$ $\emph{equivalent}$ if there exists a Lie group morphism $\varphi :\widehat{G_1}\rightarrow\widehat{G_2}$ such that the following diagram commutes:
\[\xymatrix{ 1 \ar[r]& N\ar[r] \ar@{=}[d]& \widehat{G}_1\ar[r] \ar[d]^{\varphi}& G\ar[r] \ar@{=}[d]& 1 \\
1 \ar[r]& N\ar[r] & \widehat{G}_2\ar[r] & G \ar[r] & 1.}
\]It is easy to see that any such $\varphi$ is in particular an isomorphism of Lie groups.

(b) We call an extension $q:\widehat{G}\rightarrow G$ with $\ker(q)=N$ $\emph{split}$ if there exists a Lie group morphism $\sigma:G\rightarrow\widehat{G}$ with $q\circ\sigma=\id_G$. In this case the map $N\rtimes_S G\rightarrow\widehat{G}, (n,g)\mapsto n\sigma(g)$ is an isomorphism, where the semidirect product is defined by the homomorphism
\[S:=C_N\circ\sigma:G\rightarrow \Aut(N)
\]and $C_N:\widehat{G}\rightarrow \Aut(N)$ denotes the conjugation action of $\widehat{G}$ on $N$. 
\end{definition}

As already mentioned above, we have the following important example:

\begin{example}\label{Aut(P) as LGE}
If $(P,M,G,q,\sigma)$ is a principal bundle with compact base manifold $M$, then the group $\Aut(P)$ of bundle automorphism is an extension of Lie groups:
\begin{align}
1\longrightarrow\Gau(P)\longrightarrow\Aut(P)\longrightarrow\Diff(M)_{[P]}\longrightarrow 1.\notag
\end{align}
Here,
\[\Gau(P)=\{\varphi\in\Aut(P):\, q\circ \varphi=\varphi\}
\]is the gauge group of $P$ and $\Diff(M)_{[P]}$ is the open subgroup of $\Diff(M)$ consisting of all diffeomorphism preserving the bundle class $[P]$  under pull backs, i.e.,
\[\Diff(M)_{[P]}:=\{\varphi\in\Diff(M):\,\varphi^*(P)\cong P\}.
\]If $G$ is abelian, then $\Gau(P)\cong C^{\infty}(M,G)$ and we have an abelian extension of Lie groups. A nice reference for the previous statements is [Wo07].
\end{example}

\subsection*{Smooth Factor Systems}

In this subsection we describe how to produce Lie group extensions in terms of data associated to given Lie groups $G$ and $N$. In particular we will classify all possible extensions related to $G$ and $N$. This will be done within the framework of smooth factor systems.

\begin{construction}\label{group extension to smooth factor system}
We give a description of Lie group extensions $N\hookrightarrow\widehat{G}\rightarrow G$ in terms of data associated to $G$ and $N$. Let $\widehat{G}$ be a Lie group extension of $G$ by $N$. By assumption, the extension has a smooth local section. Hence, there exists a locally smooth global section $\sigma:G\rightarrow\widehat{G}$ which is \emph{normalized} in the sense that $\sigma(1_G)=1_{\widehat{G}}$, i.e., $\sigma\in C^1_s(G,\widehat{G})$. Then the map
\[\Phi:N\times G\rightarrow \widehat{G},\,\,\,(n,g)\mapsto n\sigma(g)
\]
is a bijection which restricts to a local diffeomorphism on an identity neighbourhood. In general $\Phi$ is not continuous, but we may nevertheless use it to identify $\widehat{G}$ with the product set $N\times G$ endowed with the multiplication
\begin{align}
(n,g)(n',g')=(nS(g)(n')\omega(g,g'),gg'),\notag
\end{align}
where
\[S:=C_N\circ\sigma:G\rightarrow \Aut(N)\,\,\,\text{for}\,\,\,C_N:\widehat{G}\rightarrow \Aut(N),\,\,\,C_{N}(\widehat{g}):=C_{\widehat{G}}(\widehat{g})_{\mid N}.
\]and
\[\omega:G\times G\rightarrow N,\,\,\,(g,g')\mapsto \sigma(g)\sigma(g')\sigma(gg')^{-1}.
\]Note that $\sigma$ is smooth in an identity neighbourhood and the map $G\times N\rightarrow N, (g,n)\mapsto S(g).n$ is smooth on a set of the form $U\times N$, where $U$ is an identity neighbourhood of $G$. The maps $S$ and $\sigma$ satisfy the relations
\[\sigma(g)\sigma(g')=\omega(g,g')\sigma(gg'),
\]
\[S(g)S(g')=C_N(\omega(g,g'))S(gg'),\,\,\,\text{i.e.,}\,\,\,\delta_S=C_N\circ\omega,
\]and
\[\omega(g,g')\omega(gg',g'')=S(g)(\omega(g',g''))\omega(g,g'g'').
\]
\begin{flushright}
$\blacksquare$
\end{flushright}
\end{construction}

\begin{definition}
For $S\in C^1(G,\Aut(N))$ and $\omega\in C^2(G,N)$ let
\[(d_S\omega)(g,g',g''):=S(g)(\omega(g',g''))\omega(g,g'g'')\omega(gg',g'')^{-1}\omega(g,g{'})^{-1}.
\]
\end{definition}

\begin{construction}\label{new group}
Let $S\in C^1(G,\Aut(N))$ and $\omega\in C^2(G,N)$ with $\delta_S=C_N\circ\omega$. We define a product on $N\times G$ by
\begin{align}
(n,g)(n',g')=(nS(g)(n')\omega(g,g'),gg').\label{group multiplication}
\end{align}
Then $N\times G$ is a group if and only if $d_S\omega=1_N$. Inversion in this group is given by
\begin{align}
(n,g)^{-1}&=(S(g)^{-1}(n^{-1}\omega(g,g^{-1})^{-1}),g^{-1})\notag\\
&=(\omega(g^{-1},g)^{-1}S(g^{-1})(n^{-1}),g^{-1})
\end{align}
and in particular
\[(1_N,g)^{-1}=(S(g)^{-1}.(\omega(g,g^{-1})^{-1}),g^{-1})=(\omega(g^{-1},g)^{-1},g^{-1}).
\]Conjugation is given by
\[(n,g)(n',g')(n,g)^{-1}=(nS(g)(n')\omega_g(g')S(gg'g^{-1})^{-1}(n^{-1}),gg'g^{-1}),
\]so that 
\[(n,1_G)(n',g')(n,1_G)^{-1}=(nn'S(g')(n^{-1}),g')
\]and
\[(1_N,g)(n',g')(1_N,g)^{-1}=(S(g)(n')\omega_g(g'),gg'g^{-1}).
\]
\begin{flushright}
$\blacksquare$
\end{flushright}
\end{construction}

\begin{definition}\label{smooth factor system}{\bf (Smooth factor systems).}\index{Smooth!Factor System}
Let $G$ and $N$ be Lie groups. The elements of the set
\[Z^2_{ss}(G,N):=\{(S,\omega)\in C^1_s(G,\Aut(N))\times C^2_{ss}(G,N):\delta_S=C_N\circ\omega,\, d_S\omega=1_N\}
\]are called \emph{smooth factor systems for the pair} $(G,N)$ or \emph{locally smooth $2$-cocycles}.\sindex[n]{$Z^2_{ss}(G,N)$}
\end{definition}

\begin{definition}\index{Crossed Products!of Groups}
For a smooth factor system $(S,\omega)$\sindex[n]{$(S,\omega)$} we write $N\times_{(S,\omega)} G$\sindex[n]{$N\times_{(S,\omega)} G$} for the set $N\times G$ endowed with the group multiplication (\ref{group multiplication}) and call $N\times_{(S,\omega)} G$ the \emph{crossed product} of $N$ and $G$ with respect to $(S,\omega)$.
\end{definition}

For the proofs of the following assertions we refer to [Ne07a], Chapter 2:

\begin{theorem}\label{new lie groups}
Let $G$ and $N$ be Lie groups. Then the following assertions hold:
\begin{itemize}
\item[\emph{(a)}]
If $(S,\omega)$ is a smooth factor system, then $N\times_{(S,\omega)} G$ carries a unique structure of a Lie group for which the map
\[N\times G\rightarrow N\times_{(S,\omega)} G,\,\,\,(n,g)\mapsto (n,g)
\]is smooth on a set of the form $N\times U$, where $U$ is an open identity neighbourhood in $G$. Moreover, the map $q:N\times_{(S,\omega)} G\rightarrow G,\,\,\,(n,g)\mapsto g$ is a Lie group extension of $G$ by $N$.
\item[\emph{(b)}]
Each Lie group extension $q:\widehat{G}\rightarrow G$ of $G$ by $N$ gives rise to a smooth factor system by choosing a locally smooth normalized section $\sigma:G\rightarrow\widehat{G}$ and defining $(S,\omega):=(C_N\circ\sigma,\delta_{\sigma})$. In particular, the extension $\widehat{G}$ is equivalent to $N\times_{(S,\omega)} G$.
\end{itemize}
\begin{flushright}
$\blacksquare$
\end{flushright}
\end{theorem}


\begin{theorem}\label{Orbit of factor systems}
For two smooth factor systems $(S,\omega)$, $(S',\omega')\in Z^2_{ss}(G,N)$ the following statements are equivalent:
\begin{itemize}
\item[\emph{(a)}]
$N\times_{(S,\omega)}G$ and $N\times_{(S',\omega')}G$ are equivalent extensions of $G$ by $N$.
\item[\emph{(b)}]
There exists an element $h\in C^1_s(G,N)$ with $h.(S,\omega):=(h.S,h\ast_S\omega)=(S',\omega')$, where
\[h.S:=(C_N\circ h)\cdot S\,\,\,\text{and}\,\,\,(h\ast_S\omega)(g,g'):=h(g)S(g)(h(g'))\omega(g,g')h(gg')^{-1}.
\]
\end{itemize}
If these conditions are satisfied, then the map
\[\varphi:N\times_{(S',\omega')}G\rightarrow N\times_{(S,\omega)}G,\,\,\,(n,g)\mapsto (nh(g),g)
\]is an equivalence of extensions and all equivalences of extensions $N\times_{(S',\omega')}G\rightarrow N\times_{(S,\omega)}G$ are of this form.
\begin{flushright}
$\blacksquare$
\end{flushright}
\end{theorem}

\begin{corollary}\label{Ext(G,N)} 
Let $\Ext(G,N)$ denote the set of all equivalence classes of Lie group extensions of $G$ by $N$. Then the map
\[Z^2_{ss}(G,N)\rightarrow \Ext(G,N),\,\,\,(S,\omega)\mapsto[N\times_{(S,\omega)}G]
\]induces a bijection
\[H^2_{ss}(G,N):=Z^2_{ss}/C^1_s(G,N)\rightarrow \Ext(G,N).
\]\sindex[n]{$\Ext(G,N)$}\sindex[n]{$H^2_{ss}(G,N)$}
\begin{flushright}
$\blacksquare$
\end{flushright}
\end{corollary}

\begin{remark}\label{fixed outer action} 
The preceding proposition shows that $N\times_{(S,\omega)}G\sim N\times_{(S',\omega')}G$ implies $[S]=[S']$, i.e, equivalent extensions correspond to the same smooth $G$-kernel. In the following we write $\Ext(G,N)_{[S]}$\sindex[n]{$\Ext(G,N)_{[S]}$} for the set of equivalence classes of $N$-extensions of $G$ corresponding to the $G$-kernel $[S]$.
\end{remark}

\begin{theorem}\label{H^2(G,Z(N))=Ext(G,N)_S}
Let $S$ be a smooth outer action of $G$ on $N$ with $\Ext(G,N)_{[S]}\neq\emptyset$. Then the following assertions hold:
\begin{itemize}
\item[\emph{(a)}]
Each extension class in $\Ext(G,N)_{[S]}$ can be represented by a Lie group of the form $$N\times_{(S,\omega)}G.$$
\item[\emph{(b)}]
Any other Lie group extension $N\times_{(S,\omega')}G$ representing an element of $\Ext(G,N)_{[S]}$ satisfies
\[\omega'\cdot\omega^{-1}\in Z^2_{ss}(G,Z(N))_{[S]},
\]and the Lie groups $N\times_{(S,\omega)}G$ and $N\times_{(S,\omega')}G$ define equivalent extensions if and only if
\[\omega'\cdot\omega^{-1}\in B^2_{s}(G,Z(N))_{[S]}.
\]
\end{itemize}
\begin{flushright}
$\blacksquare$
\end{flushright}
\end{theorem}

\begin{corollary}\label{H^2(G,Z(N))}
For a smooth $G$-kernel $[S]$ with $\Ext(G,N)_{[S]}\neq\emptyset$ the map
\[H^2_{ss}(G,Z(N))_{[S]}\times \Ext(G,N)_{[S]}\rightarrow \Ext(G,N)_{[S]},\,\,\,(\beta,[N\times_{(S,\omega)}G])\mapsto [N\times_{(S,\omega\cdot\beta})G]
\]is a well-defined simply transitive action.
\begin{flushright}
$\blacksquare$
\end{flushright}
\end{corollary}

\begin{remark}\label{abelian extensions}{\bf (Abelian extensions).}\index{Extensions!Abelian} Suppose that $N=A$ is an abelian Lie group. Then the adjoint representation of $A$ is trivial and a smooth factor system $(S,\omega)$ for $(G,A)$ consists of a smooth module structure $S:G\rightarrow \Aut(A)$ and an element $\omega\in C^2_{ss}(G,A)$. In this case $d_S$ is the Lie group differential corresponding to the module structure on $A$ (c.f. Definition \ref{lie group cohomology}).

Therefore, $(S,\omega)$ defines a Lie group $N\times_{(S,\omega)}G$ if and only if $d_S\omega=1_N$, i.e., $\omega\in Z^2_{ss}(G,A)$. In this case we write $A\times_{\omega}G$ for this Lie group, which is $A\times G$ endowed with the multiplication
\[(a,g)(a',g')=(a+g.a'+\omega(g,g'),gg').
\]Further $S\sim S'$ if and only if $S=S'$. Hence, a smooth $G$-kernel $[S]$ is the same as a smooth $G$-module structure $S$ on $A$ and $\Ext(G,A)_{S}:=\Ext(G,A)_{[S]}$ is the class of all $A$-extensions of $G$ for which the associated $G$-module structure on $A$ is $S$.

According to Corollary \ref{H^2(G,Z(N))}, the equivalence classes of extensions correspond to cohomology classes of cocycles, so that the map
\[H^2_{ss}(G,A)_S\rightarrow \Ext(G,A)_S,\,\,\,\,\,\,\,\,[\omega]\mapsto[A\times_{\omega}G]
\]is a well-defined bijection.
\end{remark}

\section{Some Lie Algebra Cohomology}\index{Lie!Algebra Cohomology}
In this section we provide the basic tools of Lie algebra cohomology which are necessary in order to understand Lecomte's generalization of the Chern--Weil map and to define secondary characteristic classes of Lie algebra extensions.

\subsection*{The Chevalley-Eilenberg complex}\label{The Chevalley-Eilenberg Complex}

\begin{definition}\label{alternating maps}
Let $V$ and $W$ be vector spaces and $p\in\mathbb{N}$. A multilinear map $f:W^p\rightarrow V$ is called \emph{alternating} if
\[f(w_{\sigma(1)},\ldots,w_{\sigma(p)})=\sgn(\sigma)\cdot f(w_1,\ldots,w_p)
\]for $w_i\in W$ and $\sgn(\sigma)$ is the sign of the permutation $\sigma\in S_p$. We write $\Alt^p(V,W)$\sindex[n]{$\Alt^p(V,W)$} for the set of $p$-linear alternating maps and $\Mult^p(V,W)$\sindex[n]{$\Mult^p(V,W)$} for the space of all $p$-linear maps $V^p\rightarrow W$. For $p=0$ we put 
\[\Mult^0(V,W):=\Alt^0(V,W):=W.
\]
\end{definition}

\begin{definition}\label{Chevalley-Eilenberg complex}{\bf(The Chevalley-Eilenberg Complex).}\index{Chevalley-Eilenberg Complex}
Let $\mathfrak{g}$ be a Lie algebra and $V$ a $\mathfrak{g}$-module. For $p\in\mathbb{N}_0$ we write $C^p(\mathfrak{g},V):=\Alt^p(\mathfrak{g},V)$ for the space of alternating $p$-linear mappings $\mathfrak{g}^p\rightarrow V$ and call the elements of $C^p(\mathfrak{g},V)$ $p$-\emph{cochains}. We also define
\[C(\mathfrak{g},V):=\bigoplus^{\infty}_{p=0}C^p(\mathfrak{g},V).
\]On $C^p(\mathfrak{g},V)$\sindex[n]{$C^p(\mathfrak{g},V)$} we define the \emph{Chevalley-Eilenberg differential} $d$ by
\begin{align}
d\omega(x_0,\ldots,x_p)&:=\sum^p_{j=0}(-1)^jx_j.\omega(x_0,\ldots,\widehat{x}_j,\ldots,x_p)\notag\\
&+\sum_{i<j}(-1)^{i+j}\omega([x_i,x_j],x_0,\ldots,\widehat{x}_i,\ldots,\widehat{x}_j,\ldots,x_p),\notag
\end{align}
where $\widehat{x}_j$ means that $x_j$ is omitted. Observe that the right hand side defines for each $\omega\in C^p(\mathfrak{g},V)$ an element of $C^{p+1}(\mathfrak{g},V)$ because it is alternating. Putting the differentials on all the spaces $C^p(\mathfrak{g},V)$ together, we obtain a linear map $d=d_{\mathfrak{g}}:C(\mathfrak{g},V)\rightarrow C(\mathfrak{g},V)$. The complex 
\[(C^{\bullet}(\mathfrak{g},V),d_{\mathfrak{g}})
\]is called the \emph{Chevalley-Eilenberg complex}.\sindex[n]{$(C^{\bullet}(\mathfrak{g},V),d_{\mathfrak{g}})$}
\end{definition}

\begin{remark}
The elements of the subspace $Z^p(\mathfrak{g},V):=\ker(d_{|C^p(\mathfrak{g},V)})$ are called $p$-\emph{cocycles}, and the elements of the spaces
\[B^p(\mathfrak{g},V):=d(C^{p-1}(\mathfrak{g},V))\,\,\,\text{and}\,\,\,B^0(\mathfrak{g},V):=\{0\}
\]are called $p$-\emph{coboundaries}. It can be shown that $d^2=0$, which implies that $B^p(\mathfrak{g},V)\subseteq Z^p(\mathfrak{g},V)$. Hence it makes sense to define the $p^{\text{th}}$-\emph{cohomology space of} $\mathfrak{g}$ \emph{with values in the module} $V$:
\[H^p(\mathfrak{g},V):=Z^p(\mathfrak{g},V)/B^p(\mathfrak{g},V).
\]\sindex[n]{$Z^p(\mathfrak{g},V)$}\sindex[n]{$B^p(\mathfrak{g},V)$}\sindex[n]{$H^p(\mathfrak{g},V)$}\index{Cocycles}\index{Coboundaries}
\end{remark}

\begin{remark}\label{low degree examples}
In particular, for elements of low degree we have:
\begin{align}
&p=0:\,\,\,d\omega(x)=x\cdot\omega\notag\\
&p=1:\,\,\,d\omega(x,y)=x\cdot\omega(y)-y\cdot\omega(x)-\omega([x,y])\notag\\
&p=2:\,\,\,d\omega(x,y,z)=x\cdot\omega(y,z)-y\cdot\omega(x,z)+z\cdot\omega(x,y)\notag\\
&-\omega([x,y],z)+\omega([x,z],y)-\omega([y,z],x).\notag
\end{align}
\end{remark}

\subsection*{Differential Forms as Lie Algebra Cochains}

Let $M$ be a smooth manifold, $\mathfrak{g}:=\mathcal{V}(M)$ the Lie algebra of smooth vector fields on $M$ and $W$ a vector space. We consider the $\mathfrak{g}$-module $V:=C^{\infty}(M,W)$ of smooth $W$-valued functions on $M$. We want to identify the space $\Omega^p(M,W)$ of $W$-valued $p$-forms with a subspace of the cochain space $C^p(\mathfrak{g},V)$. This is done as follows: To each $\omega\in\Omega^p(M,W)$ we associate the element $\widetilde{\omega}\in C^p(\mathfrak{g},V)$ defined by
\[\widetilde{\omega}(X_1,\ldots,X_p)(m):=\omega_m(X_1(m),\ldots,X_p(m))
\]and observe that $\widetilde{\omega}$ is $C^{\infty}(M)$-multilinear. We then have the following theorem:

\begin{theorem}\label{differential forms as C-infinity(M) lie-algebra cochains}
\,\,\,The map
\[\Phi:\Omega^p(M,W)\rightarrow\Alt^p_{C^{\infty}(M)}(\mathcal{V}(M),C^{\infty}(M,W)),\,\,\,\omega\mapsto\widetilde{\omega}
\]is a bijection.
\end{theorem}

\begin{proof}
\,\,\,For the proof of this theorem we refer to [Ne08b], Theorem 3.2.1.
\end{proof}

\subsection*{Multiplication of Lie Algebra Cochains}

\begin{definition}\label{wedgeproduct}
Let $\mathfrak{g}$ and $V_i$, $i=1,2,3$ be vector spaces and
\[m:V_1\times V_2\rightarrow V_3,\,\,\,(v_1,v_2)\mapsto v_1\cdot_{m}v_2
\]a bilinear map. For $\alpha\in\Alt^p(\mathfrak{g},V_1)$ and $\beta\in\Alt^q(\mathfrak{g},V_2)$ we define
\[\alpha\wedge_m\beta\in\Alt^{p+q}(\mathfrak{g},V_3)
\]by\sindex[n]{$\wedge_m$}
\begin{align}
&(\alpha\wedge_m\beta)(x_1,\ldots,x_{p+q})\notag\\
&:=\frac{1}{p!q!}\sum_{\sigma\in S_{p+q}}\sgn(\sigma)\alpha(x_{\sigma(1)},\ldots,x_{\sigma(p)})\cdot_m\beta(x_{\sigma(p+1)},\ldots,x_{\sigma(p+q)}).\notag
\end{align}
\end{definition}

\begin{remark}\label{definition of alt}
For a $p$-linear map $\alpha:\mathfrak{g}^p\rightarrow V$ and $\sigma\in S_p$ define
\[\alpha^{\sigma}(x_1,\ldots,x_p):=\alpha(x_{\sigma(1)},\ldots,x_{\sigma(p)})
\]and 
\[\Alt(\alpha):=\sum_{\sigma\in S_p}\sgn(\sigma)\cdot\alpha^{\sigma}.
\]In this sense we have 
\[\alpha\wedge_m\beta=\frac{1}{p!q!}\Alt(\alpha\cdot_m\beta),
\]where 
\[(\alpha\cdot_m\beta)(x_1,\ldots,x_{p+q}):=\alpha(x_1,\ldots,x_p)\cdot_m\beta(x_{p+1},\ldots,x_{p+q}).
\]
\end{remark}

\begin{example}\label{lie algebra bracket}
If $V$ is a Lie algebra, then the corresponding Lie bracket
\[L:=[\cdot,\cdot]:V\times V\rightarrow V,\,\,\,L(v,w):=[v,w]
\]is a skew-symmetric bilinear map. In particular, given a vector space $\mathfrak{g}$ and $\alpha\in\Alt^p(\mathfrak{g},V)$ and $\beta\in\Alt^q(\mathfrak{g},V)$, we can build the wedge product 
\[[\alpha,\beta]:=\alpha\wedge_L\beta\in\Alt^{p+q}(\mathfrak{g},V).
\]For $p=q=1$ we obtain
\[[\alpha,\beta](x,y)=[\alpha(x),\beta(y)]-[\alpha(y),\beta(x)].
\]Thus, if $\alpha=\beta$, then
\[[\alpha,\alpha](x,y)=2[\alpha(x),\alpha(y)].
\]
\end{example}

\begin{remark}\label{associativity properties}{\bf(Associativity properties).} Suppose we have four bilinear maps
\[m_{12}:V_1\times V_2\rightarrow W,\,\,\,m_3:W\times V_3\rightarrow U
\]
\[m_{23}:V_2\times V_3\rightarrow X,\,\,\,m_1:V_1\times X\rightarrow U,
\]satisfying the associativity relation
\[v_1\cdot_{m_1}(v_2\cdot_{m_{23}}v_3)=(v_1\cdot_{m_{12}}v_2)\cdot_{m_3}v_3
\]for $v_i\in V_i, i=1,2,3$. Then we obtain for $\alpha_i\in\Alt^{p_i}(\mathfrak{g},V_i)$ the relation
\[(\alpha_1\cdot_{m_{12}}\alpha_2)\cdot_{m_3}\alpha_3=\alpha_1\cdot_{m_1}(\alpha_2\cdot_{m_{23}}\alpha_3),
\]which in turn leads to
\[(\alpha_1\wedge_{m_{12}}\alpha_2)\wedge_{m_3}\alpha_3=\alpha_1\wedge_{m_1}(\alpha_2\wedge_{m_{23}}\alpha_3).
\]
\end{remark}

The following example is an important special case of the previous discussion:

\begin{example}\label{wedgeproduct for composition}
Let $V$ be a vector space and $\End(V)$ the algebra of its linear endomorphisms. Then we have two bilinear maps given by evaluation
\[\ev:\End(V)\times V\rightarrow V,\,\,\,(\varphi,v)\mapsto\varphi(v)
\]and composition
\[C:\End(V)\times\End(V)\rightarrow\End(V),\,\,\,(\varphi,\psi)\mapsto\varphi\circ\psi.
\]They satisfy the associativity relation
\[\ev(C(\varphi,\psi),v)=(\varphi\circ\psi)(v)=\varphi(\psi(v))=\ev(\varphi,\ev(\psi,v)).
\]In particular, for $\alpha\in\Alt^p(\mathfrak{g},\End(V))$, $\Alt^q(\mathfrak{g},\End(V))$ and $\gamma\in\Alt^r(\mathfrak{g},V)$, this leads to the relation
\[\alpha\wedge_{\ev}(\beta\wedge_{\ev}\gamma)=(\alpha\wedge_C\beta)\wedge_{\ev}\gamma.
\]
\end{example}

\begin{remark}\label{commutativity properties}{\bf(Commutativity properties).} Now let $m:V\times V\rightarrow V$ be a bilinear map, $\alpha\in\Alt^p(\mathfrak{g},V)$ and $\beta\Alt^q(\mathfrak{g},V)$. If $m$ is symmetric, then we find for their wedge product:
\[\beta\wedge_m\alpha=(-1)^{pq}\alpha\wedge_m\beta.
\]If $m$ is skew-symmetric, we have
\[\beta\wedge_m\alpha=(-1)^{pq+1}\alpha\wedge_m\beta
\]
\end{remark}

\begin{proposition}\label{wedge-product and differential}
Let $\mathfrak{g}$ be a Lie algebra, $V_i$, $i=1,2,3$, $\mathfrak{g}$-modules and $m:V_1\times V_2\rightarrow V_3$ a $\mathfrak{g}$-invariant bilinear map, i.e.,
\[x.m(v_1,v_2)=m(x.v_1,v_2)+m(v_1,x.v_2),\,\,\,x\in\mathfrak{g}, v_i\in V_i.
\]Then we have for $\alpha\in C^p(\mathfrak{g},V_1)$ and $\beta\in C^q(\mathfrak{g},V_2)$ the relation
\[d_{\mathfrak{g}}(\alpha\wedge_m\beta)=d_{\mathfrak{g}}\alpha\wedge_m\beta+(-1)^p\alpha\wedge_m d_{\mathfrak{g}}\beta.
\]
\end{proposition}

\begin{proof}
\,\,\,For the proof of this useful proposition we again refer to [Ne08b], Proposition 3.3.7.
\end{proof}

\section{Extensions of Lie Algebras}

Extensions of Lie algebras occur quite naturally both in mathematics and physics. For example, given a finite-dimensional simple Lie algebra $\mathfrak{g}$, the corresponding \emph{affine} Lie algebra $\widehat{\mathfrak{g}}$ is constructed as a central Lie algebra extension of the infinite-dimensional Lie algebra $\mathfrak{g}\otimes\mathbb{C}[t,t^{-1}]$ by the one-dimensional Lie algebra $\mathbb{C}$. Note that each affine Kac-Moody algebra can be realized as a vector space direct sum of an affine Lie algebra and $\mathbb{C}$. Moreover, as we have seen in Corollary \ref{1:1 correspondence between section and connections}, there is a one to one correspondence between the connections of a principal bundle $(P,M,G,q,\sigma)$ and the sections of a short exact sequence of Lie algebras commonly known as the \emph{Atiyah sequence} associated to the principal bundle $(P,M,G,q,\sigma)$.

\begin{definition}\label{s.e.s of lie algebras}{\bf (Extensions of Lie algebras).}\index{Extensions!of Lie Algebras}
(a) Let $\mathfrak{g}$ and $\mathfrak{n}$ be Lie algebras. A short exact sequence of Lie algebra homomorphisms
\begin{align}
0\longrightarrow\mathfrak{n}\stackrel{\iota}{\longrightarrow}\widehat{\mathfrak{g}}\stackrel{q}{\longrightarrow}\mathfrak{g}\longrightarrow0\notag
\end{align}
is called an \emph{extension of $\mathfrak{g}$ by $\mathfrak{n}$}. If we identify $\mathfrak{n}$ with its image in $\widehat{\mathfrak{g}}$, this means that $\widehat{\mathfrak{g}}$ is a Lie algebra containing $\mathfrak{n}$ as an ideal satisfying $\widehat{\mathfrak{g}}/\mathfrak{n}\cong\mathfrak{g}$. If $\mathfrak{n}$ is abelian, resp., central in
$\widehat{\mathfrak{g}}$, then the extension is called \emph{abelian}, resp., \emph{central}.  

We call two extensions $\mathfrak{n}\hookrightarrow\widehat{\mathfrak{g}}_1\rightarrow \mathfrak{g}$ and $\mathfrak{n}\hookrightarrow\widehat{\mathfrak{g}}_2\rightarrow \mathfrak{g}$ \emph{equivalent} if there exists a Lie algebra homomorphism $\varphi :\widehat{\mathfrak{g}}_1\rightarrow\widehat{\mathfrak{g}}_2$ such that the following diagram commutes:
$$\xymatrix{ 0 \ar[r]& \mathfrak{n}\ar[r] \ar@{=}[d]& \widehat{\mathfrak{g}}_1\ar[r] \ar[d]^{\varphi}& \mathfrak{n}\ar[r] \ar@{=}[d]& 0 \\
0 \ar[r]& \mathfrak{n}\ar[r] & \widehat{\mathfrak{g}}_2\ar[r] & \mathfrak{g} \ar[r] & 0, }$$
It is easy to see that this implies that $\varphi$ is an isomorphism of Lie algebras.

(b) We call an extension $q:\widehat{\mathfrak{g}}\rightarrow\mathfrak{g}$ with $\ker q=\mathfrak{n}$ \emph{trivial}, or say that the extension \emph{splits}, if there exists a Lie algebra homomorphism $\sigma:\mathfrak{g}\rightarrow\widehat{\mathfrak{g}}$ with $q\circ\sigma=\id_{\mathfrak{g}}$. In this case the map
\[\mathfrak{n}\rtimes\mathfrak{g}\rightarrow\widehat{\mathfrak{g}},\,\,\,(a,x)\mapsto a+\sigma(x)
\]is an isomorphism, where the semidirect sum is defined by the homomorphism
\[\delta:\mathfrak{g}\rightarrow\der(\mathfrak{n}),\,\,\,\delta(x)(a):=[\sigma(x),a].
\]

\end{definition}

\begin{example}\label{Attia sequence}{\bf(The Atiyah sequence).}\index{Atiyah sequence}
In Example \ref{Aut(P) as LGE} we have seen that if $(P,M,G,q,\sigma)$ is a principal bundle with compact base space $M$, then the group $\Aut(P)$ of bundle automorphisms is an extension of Lie groups:
\[1\longrightarrow\Gau(P)\longrightarrow\Aut(P)\longrightarrow\Diff(M)_{[P]}\longrightarrow 1.
\]The derived short exact sequence of Lie algebras is then given by
\[0\longrightarrow\mathfrak{gau}(P)\longrightarrow\mathfrak{aut}(P)\longrightarrow\mathcal{V}(M)\longrightarrow 0.
\]More generally, if $(P,M,G,q,\sigma)$ is a principal bundle with paracompact base space $M$, then Corollary \ref{1:1 correspondence between section and connections} shows that we obtain a short exact sequence of Lie algebras
\[0\longrightarrow\mathfrak{gau}(P)\longrightarrow\mathfrak{aut}(P)\stackrel{q_*}{\longrightarrow}\mathcal{V}(M)\longrightarrow 0.
\]This short exact sequence of Lie algebras is commonly known as the \emph{Atiyah sequence} associated to the principal bundle $(P,M,G,q,\sigma)$.
\end{example}

\chapter{Continuous Inverse Algebras}

In this part of the appendix we introduce an important class of algebras whose groups of units are Lie groups. They may be seen as the infinite-dimensional generalization of matrix algebras and are encountered in K-theory and noncommutative geometry, usually as dense unital subalgebras of $C^*$-algebras. In fact, these algebras play a central role in this thesis, since, for a compact manifold $M$, the space $C^{\infty}(M)$ is the prototype of such a continuous inverse algebra. 

\section{Some Results on Continuous Inverse Algebras}

In this section we will present some results on continuous inverse algebras which will be needed throughout this thesis. For all proofs which are not presented in the following and for a general overview we refer to [Gl02].

\begin{definition}\label{quasi-inv}{\bf(Quasi-invertible elements).}\index{Quasi-Invertible Elements}
In a unital algebra $A$, let $A^{\times}$ denote the corresponding group of units. An element $a$ in the set $Q(A):=1_A-A^{\times}$ is called \emph{quasi-invertible}. Moreover, the map
\[q:Q(A)\rightarrow A,\,\,\,a\mapsto 1_A-(1_A-a)^{-1}
\]is called \emph{quasi-inversion}.\sindex[n]{$Q(A)$}
\end{definition}

\begin{definition}\label{CIA}{\bf(Continuous inverse algebras).}\index{Algebra!Continuous Inverse}
A locally convex unital algebra $A$ is called a \emph{continuous inverse algebra}, or CIA for short, if its group of units $A^{\times}$ is open in $A$ and the inversion 
\[\iota:A^{\times}\rightarrow A^{\times},\,\,\,a\mapsto a^{-1}
\]is continuous at $1_A$.
\end{definition}

\begin{remark}\label{quasi-inv CIA}
A short argument shows that a locally convex unital algebra $A$ is a CIA if and only if the set of quasi-invertible elements is an open neighbourhood of $0_A$ and quasi-inversion is continuous at $0_A$.
\end{remark}

\begin{definition}{\bf (The Gelfand homomorphism for complex algebras).}\index{Gelfand!Homomorphism}
The \emph{Gelfand spectrum} of a complex algebra $A$ is defined as 
\[\Gamma_A:=\Hom_{\text{alg}}(A,\mathbb{C})\backslash\{0\}
\]with the topology of pointwise convergence on $A$. Moreover, each element $a\in A$ gives rise to a continuous function 
\[\widehat{a}:\Gamma_A\rightarrow\mathbb{C},\,\,\,\widehat{a}(\chi):=\chi(a).
\]The function $\widehat{a}$ is called the \emph{Gelfand transform}\index{Gelfand!Transform} of $a$. The map
\[\mathcal{G}:A\rightarrow C(\Gamma_A),\,\,\,a\mapsto\widehat{a},
\]is a homomorphism of unital complex algebras and is called the \emph{Gelfand homomorphism} of the algebra $A$.
\end{definition}

\begin{remark}{\bf(The spectrum of an element).}
At this point we recall that the spectrum of an element $a$ in an arbitrary unital complex associative algebra $A$ is defined as
\[\sigma_A(a):=\sigma(a):=\{\lambda\in\mathbb{C}:\,\lambda\cdot 1_A-a\notin A^{\times}\}.
\]Furthermore, the corresponding spectral radius of $a$ is defined as
\[r_A(a):=\sup\{\vert\lambda\vert:\,\lambda\in\sigma_A(a)\}.
\]
\end{remark}

The following theorem is a generalization of the classical Gelfand-Naimark Theorem for commutative $C^*$-algebras:

\begin{theorem}\label{bi 04}
In a complex CIA $A$, the following assertions hold:
\begin{itemize}
\item[\emph{(a)}]
The Gelfand spectrum $\Gamma_A$ is a compact Hausdorff space.
\item[\emph{(b)}]
If $A$ is commutative, then the Gelfand homomorphism is continuous with respect to the topology of uniform convergence on $C(\Gamma_A)$.
\end{itemize}
\end{theorem}

\begin{proof}
\,\,\,(a) By [Bi04], Lemma 2.1.6 (b), every element $a$ of $A$ has non-empty compact spectrum. We choose $\lambda\in\mathbb{C}\backslash\sigma(a)$ and $\chi\in\Gamma_A$. Since $\lambda-a$ is invertible, the same holds for $\lambda-\chi(a)$, whence $\lambda\neq\chi(a)$. This proves that $\widehat{a}(\Gamma_A)\subseteq\sigma(a)$. Therefore, the Gelfand spectrum is compact, because it is a closed subspace of the product
\[\prod_{a\in A}\sigma(a)\subseteq\mathbb{C}^A.
\]

(b) A proof of this assertion can be found in [Bi04], Theorem 2.2.3 (c).
\end{proof}

\begin{lemma}\label{sub algebras of CIA}
If $A$ is a complex CIA, then so is each closed unital subalgebra $B$ of $A$. 
\end{lemma}

\begin{proof}
\,\,\,We first choose an element $b\in B$ with spectral radius $r_A(b)<1$. Then $1_A-b$ is invertible and the Neumann series $\sum_{n=0}^{\infty}b^n$ converges to $(1_A-b)^{-1}$. Since $B$ is closed, $(1_A-b)^{-1}\in B$, so that $B^{\times}$ is a neighbourhood of $1_A\in B$, hence open. The continuity of the inversion in $B$ follows from the corresponding property of $A$.
\end{proof}

\begin{lemma}\label{quotient algebras}
Let $A$ be a CIA and $I\subseteq A$ a proper closed two-sided ideal. Then the quotient algebra $A/I$ is a CIA with respect to the quotient topology.
\end{lemma}

\begin{proof}
\,\,\,The quotient algebra $A/I$ is a locally convex algebra. The image of the set of quasi-invertible elements of $A$ under the canonical projection of $A$ onto $A/I$ is an open zero-neighbourhood in $A/I$ which consists of quasi-invertible elements. The restriction and corestriction of the quasi-inversion map of $A/I$ to this zero-neighbourhood is continuous. Hence, by Remark \ref{quasi-inv CIA}, $A/I$ is a CIA.
\end{proof}

\begin{lemma}\label{max ideals}
Let $A$ be a commutative CIA and $I$ a maximal proper ideal in $A$. Then $I$ is the kernel of some character $\chi:A\rightarrow\mathbb{K}$.
\end{lemma}

\begin{proof}
\,\,\,Since $I$ is disjoint from the open set $A^{\times}$, the same holds for its closure $\overline{I}$, which is therefore a proper ideal. Maximality of $I$ implies that $I$ is closed. By Lemma \ref{quotient algebras}, the quotient $A/I$ is a CIA. Since $A$ is commutative, the quotient is a field. If $A$ is a complex CIA, [Bi04], Lemma 2.1.6 (d) implies that $A/I$ is isomorphic to $\mathbb{C}$. If $A$ is a real CIA, [Bi04], Lemma 2.1.6 (d) implies that its complexification\footnote{By [Gl02], Proposition 3.4, the universal complexification of a real CIA is again a (complex) CIA.} is isomorphic to $\mathbb{C}$, hence $A/I$ is isomorphic to $\mathbb{R}$. In each case, $I$ is the kernel of a character of $A$.
\end{proof}

\begin{lemma}\label{cont of char of CIA}
If $A$ is a CIA, then each character $\chi:A\rightarrow\mathbb{K}$ is continuous.
\end{lemma}

\begin{proof}
\,\,\,Let $\epsilon>0$. We choose a balanced $0$-neighbourhood $U\subseteq A$ which consists of quasi-invertible elements. Then $\chi(U)\subseteq\mathbb{K}$ is a disc around 0 which consists of quasi-invertible elements and hence does not contain 1. The image of the 0-neighbourhood $\epsilon U\subseteq A$ under $\chi$ is a disc around 0 of radius at most $\epsilon$. We conclude that $\chi$ is continuous at 0 and hence continuous.
\end{proof}

\begin{proposition}\label{M_n(A)}
If $A$ is a continuous inverse algebra, then $\M_n(A)$\sindex[n]{$\M_n(A)$} is a continuous inverse algebra for every $n\in\mathbb{N}$.
\begin{flushright}
$\blacksquare$
\end{flushright}
\end{proposition}

\begin{proposition}\label{C infty (M,A) is CIA}
Let $M$ be a compact smooth manifold and $A$ be a continuous inverse algebra. Then $C^{\infty}(M,A)$\sindex[n]{$C^{\infty}(M,A)$}, equipped with the smooth compact open topology, is a continuous inverse algebra.
\begin{flushright}
$\blacksquare$
\end{flushright}
\end{proposition}

\begin{remark}
The example $C^{\infty}(\mathbb{R},\mathbb{R})$ shows that, if $A$ is a continuous inverse algebra, $C^{\infty}(M,A)$ need not have an open unit group.
\end{remark}



\begin{remark}\label{Mackey complete} \index{Mackey!Cauchy Sequence}\index{Mackey!Complete}
A sequence $(x_n)_{n\in\mathbb{N}}$ in a locally convex space $E$ is called a \emph{Mackey-Cauchy sequence} if $x_n-x_m\in t_{n,m}B$ for all $n,m\in\mathbb{N}$, for some bounded set $B\subseteq E$ and certain $t_{n,m}\in\mathbb{R}$ such that $t_{n,m}\to 0$ as both $n,m\to\infty$. The space  $E$ is called \emph{Mackey complete} if every Mackey-Cauchy sequence in $E$ converges in $E$. 

If $E$ is sequentially complete, then $E$ is Mackey complete. The space $E$ is Mackey complete if and only if the Riemann integral $\int^b_a\gamma(t)\,dt$ exists in $E$ for any smooth curve $\gamma:]\alpha,\beta[\rightarrow E$ and $\alpha<a<b<\beta$.
\end{remark}

\begin{theorem}\label{CIA-Theorem}
If $A$ is a Mackey complete continuous inverse algebra, then $A^{\times}$ is a \emph{(}Baker-Campbell-Hausdorff\emph{)} Lie group, with exponential map
\[\exp_{A^{\times}}=\left.\exp_A\right|^{A^{\times}}
\]
\begin{flushright}
$\blacksquare$
\end{flushright}
\end{theorem}


\section{From Continuous to Smooth Actions on Algebras}

Suppose we are given a unital locally convex algebra $A$ together with a Lie group $G$ acting by automorphisms on $A$, but not necessarily smoothly. For fixed $a\in A$ we define the map 
\[f_a:G\rightarrow A,\,\,\, g\mapsto g.a
\]and consider the subalgebra
\[A^{\infty}:=\{a\in A:\,f_a\in C^{\infty}(G,A)\}
\]of smooth vectors of the action of $G$ on $A$. The natural topology on $A^{\infty}$ is obtained by embedding
\[j:A^{\infty}\rightarrow C^{\infty}(G,A),\,\,\,j(a):=f_a,
\]where $C^{\infty}(G,A)$ is endowed with the smooth compact open topology (cf. Definition \ref{smooth compact open topology}).

\begin{lemma}\label{A is a unital locally convex topological algebra}
Suppose that the Lie group $G$ acts continuously on the unital locally convex algebra $A$ by algebra automorphisms. Then the following assertions hold:
\begin{itemize}
\item[\emph{(a)}]
The subalgebra $A^{\infty}$, endowed with the smooth compact open topology, is a unital locally convex algebra.
\item[\emph{(b)}]
If, in addition, the inversion map $\eta_A:A^{\times}\rightarrow A$, $a\mapsto a^{-1}$ is continuous, then the embedding $i:A^{\infty}\rightarrow A$ is isospectral in the sense that $i^{-1}(A^{\times})=(A^{\infty})^{\times}$ and the inversion map of $A^{\infty}$ is continuous.
\end{itemize}
\end{lemma}

\begin{proof}
\,\,\,We give a sketch of the proof: 

(a) For the first assertion, we recall that the smooth compact open topology turns $C^{\infty}(G,A)$ into a locally convex algebra. 

(b) If, in addition, the inversion in $A$ is continuous, then it is possible to show that
\[C^{\infty}(G,A)^{\times}=\{f\in C^{\infty}(G,A):\,f(G)\subseteq A^{\times}\},
\]and that the inversion in this algebra is continuous. As a subalgebra of $C^{\infty}(G,A)$, $A^{\infty}$ may be described as
\[\{f\in C^{\infty}(G,A):\,(\forall g\in G):\, f(g)=g.f(1_G)\}.
\]This clearly is a subalgebra satisfying $f(G)\subseteq A^{\times}$ if and only if $f(1_G)\in A^{\times}$. For any such $f$ the function $f^{-1}$, obtained by pointwise inversion, also satisfies $f^{-1}(g)=g.f^{-1}(1_G)$, so that $i^{-1}(A^{\times})=(A^{\infty})^{\times}$.
\end{proof}

%

\begin{proposition}\label{smooth vectors are CIA}{\bf($A^{\infty}$ as CIA).}
Suppose $G$ is a Lie group, $A$ a CIA and $\alpha:G\times A\rightarrow A$ a continuous action of $G$ on $A$ by algebra automorphisms. Then the following assertions hold:
\begin{itemize}
\item[\emph{(a)}]
The subspace $A^{\infty}$ of smooth vectors for this action is a CIA, when endowed with the natural topology inherited from the embedding $A^{\infty}\hookrightarrow C^{\infty}(G,A)$.
\item[\emph{(b)}]
The action of $G$ on $A$ leaves $A^{\infty}$ invariant. In particular, we obtain a continuous action of $G$ on $A^{\infty}$ by algebra automorphisms and with smooth orbit maps. 
\item[\emph{(c)}]
The derived action 
\[{\bf L}(\alpha):\mathfrak{g}\times A^{\infty}\rightarrow A^{\infty},\,\,\,(x,a)\mapsto x.a:=\left.\frac{d}{dt}\right|_{t=0}(\exp(tx).a)=T_{1_G}(f_a).x
\]is separately continuous.
\end{itemize}
\end{proposition}

\begin{proof}
\,\,\,(a) In view of Lemma \ref{A is a unital locally convex topological algebra}, it remains to show that the unit group of $A^{\infty}$ is open. The inclusion map $i:A^{\infty}\hookrightarrow A$ is continuous, because it is obtained by restriction of the continuous evaluation map (cf. [NeWa07], Proposition I.2)
\[\ev_{1_G}:C^{\infty}(G,A)\rightarrow A,\,\,\,f\mapsto f(1_G).
\]Now, $(A^{\infty})^{\times}=i^{-1}(A^{\times})$ implies that the unit group of $A^{\infty}$ is open.

(b) For $f_a(g)=g.a$ and $h\in G$ we have $f_{h.a}(g)=(gh).a=f_a(gh)$, so that for each $a\in A^{\infty}$ the orbit map $G\rightarrow A^{\infty}$, $h\mapsto h.a$ corresponds to the map
\[G\rightarrow C^{\infty}(G,A),\,\,\,h\mapsto f_a\circ\rho_h.
\]The smoothness of this map follows from the smoothness of the map
\[G\times G\rightarrow A,\,\,\,(g,h)\mapsto f_a(\rho_h(g))=f_a(gh)
\]and Lemma \ref{smooth exp law}. Therefore, the action of $G$ on $A^{\infty}$ has smooth orbit maps.

(c) From part (b), we see that the action of $G$ on $A^{\infty}$ corresponds to the right translation action of $G$ on $C^{\infty}(G,A)$. In particular, ${\bf L}(\alpha)$ defines a representation of the Lie algebra $\mathfrak{g}={\bf L}(G)$ on $A^{\infty}$, which is a subrepresentation of the representation of $\mathfrak{g}$ on $C^{\infty}(G,A)$, defined by left invariant vector fields. Clearly, $\mathfrak{g}$ acts by continuous operators on $A^{\infty}$, and for each $a\in A^{\infty}$ the map $\mathfrak{g}\rightarrow A^{\infty}$, $x\mapsto x.a$ is continuous linear. In this sense the action map ${\bf L}(\alpha)$ is separately continuous.
\end{proof}

\begin{remark}{\bf(Class of examples).}
The preceding proposition applies in particular to the important class of actions of finite-dimensional Lie groups on Banach algebras.
\end{remark}

\begin{remark}{\bf($A^{\infty}$ is dense).}\label{A infty is dense}
If the finite-dimensional Lie group $G$ acts continuously on a Mackey complete locally convex algebra $A$ by morphisms of algebras, then for each compactly supported smooth function $f\in C_{\text{c}}^{\infty}(G,\mathbb{R})$, we obtain a continuous linear operator on $A$ by
\[\alpha(f).a:=\int_G f(g)\cdot g.a\,dg,
\]where $dg$ is a left Haar measure on $G$. For any smooth Dirac sequence $(f_n)_{n\in\mathbb{N}}$ on $G$, i.e., a sequence $(f_n)_{n\in\mathbb{N}}$ in $C^{\infty}(G,\mathbb{R})$ satisfying
\begin{itemize}
\item
$f_n\geq 0$ for all $n\in\mathbb{N}$,

\item
$\int_G f_n\,dg=1$ for all $n\in\mathbb{N}$,

\item
and for each neighbourhood $U$ of $1_G$ in $G$ the support of $f_n$ is contained in $U$ for $n\in\mathbb{N}$ sufficiently large,
\end{itemize}
we get $\alpha(f_n.)a\rightarrow a$. Since $\im(\alpha)$ consists of smooth vectors, the subspace $A^{\infty}$ is dense in $A$.
\end{remark}

\begin{proposition}\label{automatic smoothness I}
Suppose that $\alpha:G\times E\rightarrow E$ defines a continuous linear action of a Lie group $G$ on a unital locally convex space $E$. If $\alpha$ has smooth orbit maps and the derived action ${\bf L}(\alpha):\mathfrak{g}\times E\rightarrow E$ is continuous, then $\alpha$ is smooth.
\end{proposition}

\begin{proof}
\,\,\,We first note that we have a continuous tangent action given by
\[T\alpha:TG\times TE\rightarrow TE,\,\,\,(g,x).(e,e'):=(g.e,g.e'+g.{\bf L}(\alpha)(x,e')).
\]This implies that $\alpha$ is a $C^1$-action and thus in turn that $T\alpha$ is a $C^1$-action. By iteration we see that the action $\alpha$ is smooth.
\end{proof}

\begin{remark}{\bf(Automatic continuity).}\label{automatic continuity}\index{Automatic!Continuity}
If $G$ is a Fr\'echet-Lie group, $E$ a Fr\'echet space and $\alpha:G\times E\rightarrow E$ a continuous linear action with smooth orbit maps, then the separately continuous bilinear map ${\bf L}(\alpha):\mathfrak{g}\times E\rightarrow E$ is automatically continuous (cf. [Ru73]).
\end{remark}

\begin{corollary}\label{automatic smoothness II}
Suppose that the finite-dimensional Lie group $G$ acts continuously on a Fr\'echet-CIA $A$ by algebra automorphisms. Then the space of smooth vectors $A^{\infty}$ is a CIA and the induced action of $G$ on $A^{\infty}$ is smooth.
\end{corollary}

\begin{proof}
\,\,\,The assertion follows from Proposition \ref{smooth vectors are CIA}, Remark \ref{automatic continuity} and Proposition \ref{automatic smoothness I}.
\end{proof}

\chapter{Topology and Smooth Vector-Valued Function Spaces}

This part of the appendix is devoted to topological aspects and results which are needed throughout this thesis. We start with some results on the extendability of densely defined maps. Then we discuss the projective tensor product of locally convex spaces and further the projective tensor product of locally convex modules. Furthermore, we present some useful results on smooth vector-valued function spaces. In particular, the smooth exponential law will be used several times in this thesis. We finally discuss a special class of locally convex spaces in which the uniform boundedness principle (``Theorem of Banach--Steinhaus") still remains true.

\section{Extendability of Certain Densely Defined Maps}


\begin{definition}\label{uniform cont I}{\bf(Uniformly continuity).}\index{Uniformly Continuity}
A map $f:D\rightarrow H$ from a subset $D$ of an abelian topological group $G$ into the abelian topological group $H$ is called \emph{uniformly continuous} if to each $0$-neighbourhood $V$ of $H$, there exists a $0$-neighbourhood $U$ of $G$ such that $f(x)-f(y)\in V$ for every $x,y\in D$ with $x-y\in U$.
\end{definition}

\begin{proposition}\label{uniform cont II}
Each uniformly continuous map $f:D\rightarrow H$ of a dense subset $D$ of an abelian topological group $G$ into a complete abelian topological group $H$ can be uniquely extended to a \emph{(}uniformly\emph{)} continuous map $\widehat{f}:G\rightarrow H$.
\end{proposition}

\begin{proof}
\,\,\,For a complete proof of this assertion we refer to [BeNa85], Proposition (4.6.1).
\end{proof}

\begin{proposition}\label{uniform cont III}
Let $D$ be a subgroup of an abelian topological group $G$, $H$ an abelian topological group and $f:D\rightarrow H$ a continuous homomorphism. Then the following assertions hold:
\begin{itemize}
\item[\emph{(a)}]
The map $f$ is uniformly continuous. 
\item[\emph{(b)}]
If $D$ is a dense subgroup of $G$ and $H$ is complete, then $f$ extends uniquely to a \emph{(}uniformly\emph{)} continuous map $\widehat{f}:G\rightarrow H$.
\end{itemize}
\end{proposition}

\begin{proof}
\,\,\,(a) Let $V$ be an arbitrary $0$-neighbourhood of $H$. Since $f$ is continuous, there exists a $0$-neighbourhood $U$ of $G$ such that $f(U\cap D)\subseteq V$. Hence, if $x,y\in D$ with $x-y\in U$, then $f(x)-f(y)=f(x-y)\in V$.

(b) The second assertion is a consequence of Proposition \ref{uniform cont II}
\end{proof}

\begin{proposition}\label{ext of dense cont algebra iso}
Let $A$ and $A'$ be complete unital locally convex algebras. Further, let $B$ be a dense subalgebra of $A$ and $B'$ a dense subalgebra of $A'$. If $\varphi:B\rightarrow B'$ is an isomorphism of locally convex algebras, then $\varphi$ can be extended to an isomorphism $\overline{\varphi}:A\rightarrow A'$ of locally convex algebras.
\end{proposition}

\begin{proof}
\,\,\,
(i) We first consider the map $\varphi$ as an embedding $\varphi:B\rightarrow A'$ of abelian groups. By Proposition \ref{uniform cont III} (b), we can extend $\varphi$ to a (uniformly) continuous homomorphism $\overline{\varphi}:A\rightarrow A'$ of abelian groups. Now, the principle of extension of identities immediately implies that $\overline{\varphi}$ is a continuous algebra homomorphism.

(ii) If $\psi:B'\rightarrow B$ denotes the inverse of $\varphi$, then by the same argument as before shows we obtain a continuous algebra homomorphism $\overline{\varphi}:A'\rightarrow A$.

(iii) Finally, the uniqueness statement of Proposition \ref{uniform cont III} (b) leads to the following identities:
\[\overline{\psi}\circ\overline{\varphi}=\id_{A}\,\,\,\text{and}\,\,\,\overline{\varphi}\circ\overline{\psi}=\id_{A'}.
\]Thus, $\overline{\varphi}$ is an algebra isomorphism.
\end{proof}

\begin{corollary}\label{ext of dense cont G-mod iso}
Let $G$ be a topological group. Further, let $A$ and $A'$ be complete unital locally convex algebras with a $G$-module structure. If $B$ is a dense $G$-invariant subalgebra of $A$ and $B'$ is a dense $G$-invariant subalgebra of $A'$, then each $G$-equivariant isomorphism $\varphi:B\rightarrow B'$ of locally convex algebras can be extended to a $G$-equivariant isomorphism $\overline{\varphi}:A\rightarrow A'$ of locally convex algebras.
\end{corollary}

\begin{proof}
\,\,\,This assertion is a consequence of Proposition \ref{ext of dense cont algebra iso} and the principle of extension of identities.
\end{proof}

We now introduce an extremely useful theorem on the extendability of $\mathbb{Z}$-bilinear mappings:

\begin{theorem}\label{completion of top rings I}
Let $G$, $G'$ and $H$ be three complete Hausdorff abelian groups. Further, let $N$ be a dense subgroup of $G$ and $N'$ be a dense subgroup of $G'$. If 
\[f:N\times N'\rightarrow H
\]is a continuous biadditive map, then $f$ can be extended 
to a continuous biadditive map 
\[\overline{f}:G\times G'\rightarrow H.
\]
\end{theorem}

\begin{proof}
\,\,\,Let $(g,g')\in G\times G'$ and $(n_i,n'_i)_{i\in I}$ be a Cauchy net in $N\times N'$ which converges to $(g,g')$. The main idea of the proof is to use the identity
\[f(x,x')-f(y,y')=f(x-y,z')+f(z,x'-y')+f(x-y,x'-z')+f(x-z,x'-y')
\]to show that $(f(n_i,n'_i))_{i\in I}$ is a Cauchy net in $H$. For a complete proof of this theorem we refer to [Bou66], Chapter III, Paragraph 5, Theorem I.
\end{proof}

\begin{corollary}\label{completion of top rings II}
Let $A$ be a complete locally convex space and $B$ be a dense subspace. Further, let $f:B\times B\rightarrow B$ be a continuous bilinear map. Then $f$ can be extended by continuity to a continuous bilinear map 
\[\overline{f}:A\times A\rightarrow A.
\]
\end{corollary}

\begin{proof}
\,\,\,According to Theorem \ref{completion of top rings I}, we can extend the map $f$ to a continuous biadditive map $\overline{f}:A\times A\rightarrow A$. The $\mathbb{C}$-bilinearity is again a consequence of the principle of extension of identities. 
\end{proof}

\section{The Projective Tensor Product}

\begin{definition}{\bf(The projective tensor product).}\label{proj. tensor product III}\index{Projective!Tensor Product}
Let $E$ and $F$ be two locally convex spaces. Further, let $E\odot F$ be the algebraic tensor product of $E$ and $F$. The \emph{projective tensor product} $E\otimes F$ of $E$ and $F$ is the algebraic tensor product $E\odot F$ equipped with the strongest locally convex topology for which the canonical ``projection"
\[\pi:E\times F\rightarrow E\odot F,\,\,\,(e,f)\mapsto e\odot f
\]is continuous. The corresponding completion of $E\otimes F$ with respect to this topology is denoted by $E\widehat{\otimes}F$.\sindex[n]{$\otimes$}\sindex[n]{$\widehat{\otimes}$} 
\end{definition}

\begin{remark}\label{universal property of projective tensor product topology}{\bf(Universal property of the projective tensor product).}\index{Universal Property of!the Projective Tensor Product} 
We recall that the projective tensor product $E\otimes F$ of two locally convex spaces $E$ and $F$ has the following universal property: 

For every locally convex space $Z$ the canonical (algebraic) isomorphism of the space of bilinear maps of $E\times F$ into $Z$ onto the space of linear maps of $E\otimes F$ into $Z$ induces an (algebraic) isomorphism of the space of continuous bilinear maps of $E\times F$ into $Z$ onto the space of continuous linear maps of $E\otimes F$ into $Z$. 

In particular, a bilinear map $\varphi:E\times F\rightarrow Z$ is continuous if and only if there exists a continuous linear map $\psi:E\otimes F\rightarrow Z$ satisfying $\psi\circ\pi=\varphi$. A nice reference for these statements can be found in [Tre67], Part III.
\end{remark}

\begin{remark}\label{associativity of projective tensor product topology}{\bf(Associativity of the projective tensor product).}\index{Associativity!of the Projective Tensor Product}
In the following we shall make use of the associativity of projective tensor product: 

Let $E$, $F$ and $G$ be three locally convex spaces. In order to prove the associativity of $\otimes$, it is enough to show that $E\otimes(F\otimes G)$ is naturally isomorphic to $E\otimes F\otimes G$. This can easily be done by showing that every (continuous) trilinear map of $E\times F\times G$ into a locally convex space $Z$ factors uniquely through a (continuous) linear map of $E\otimes(F\otimes G)$ into $Z$.
\end{remark}

\begin{lemma}\label{continuity of maps EoF to E'oF'}
Let $E,E',F$ and $F'$ be four locally convex spaces. Further, let $\alpha:E\rightarrow E'$ and $\beta:F\rightarrow F'$ be two continuous linear maps. Then the linear map
\[\alpha\otimes \beta:E\otimes F\rightarrow E'\otimes F',\,\,\,\alpha\otimes\beta(e\otimes f):=\alpha(e)\otimes\beta(f)
\]is continuous.
\end{lemma}

\begin{proof}
\,\,\,If $\pi':E'\times F'\rightarrow E'\otimes F',\,\,\,(e',f')\mapsto e'\otimes f'$ denotes the canonical ``projection", then the claim follows from the fact that $\pi'\circ(\alpha\times\beta)$ is continuous as a composition of continuous maps.
\end{proof}

\begin{proposition}\label{EoF as B-module}
Let $A$ be a unital locally convex algebra, $E$ a locally convex space and $F$ a right locally convex $A$-module. Then the map
\[\rho:E\otimes F\times A\rightarrow E\otimes F,\,\,\,\rho(e\otimes f,a):=e\otimes f.a
\]defines the structure of a right locally convex $A$-module on $E\otimes F$.
\end{proposition}

\begin{proof}
\,\,\,(i) Suppose that the right $A$-module structure on $F$ is given by the continuous bilinear map 
\[S:F\times A\rightarrow F,\,\,\,f.a:=S(f,a).
\]Then Remark \ref{universal property of projective tensor product topology} implies that the map $S$ induces a continuous linear map $T$ given by 
\[T:F\otimes A\rightarrow F,\,\,\,f\otimes a\mapsto f.a.
\]

(ii) To show that the map $\rho$ defines on $E\otimes F$ the structure of a right locally convex $A$-module, we use Remark \ref{universal property of projective tensor product topology} again: Indeed, the bilinear map $\rho$ is continuous if and only if the linear map 
\[\psi:E\otimes F\otimes A\rightarrow E\otimes F,\,\,\,e\otimes f\otimes a\mapsto e\otimes f.a
\]is continuous. Now, a short observations shows that $\psi=\id_E\otimes T$, and therefore Lemma \ref{continuity of maps EoF to E'oF'} implies that $\psi$ is continuous.
 \end{proof}

\begin{proposition}\label{AoB as l.c. algebra}
Let $A$ and $B$ be two locally convex algebras. Then $A\otimes B$ becomes a locally convex algebra, when equipped with the multiplication
\[m:A\otimes B\times A\otimes B\rightarrow A\otimes B,\,\,\,m(a\otimes b, a'\otimes b'):=aa'\otimes bb'.
\]
\end{proposition}

\begin{proof}
\,\,\,(i) Suppose that the algebra structure on $A$ is given by the continuous bilinear map 
\[m_A:A\times A\rightarrow A,\,\,\,aa':=m_A(a,a')
\]and that the algebra structure on $B$ is given by the continuous bilinear map 
\[m_B:B\times B\rightarrow B,\,\,\,bb':=m_B(b,b').
\]Then Remark \ref{universal property of projective tensor product topology} implies that the map $m_A$, resp. $m_B$, induces a continuous linear 
\[m^A:A\otimes A\rightarrow A,\,\,\,a\otimes a'\mapsto aa',
\]resp.
\[m^B:B\otimes B\rightarrow B,\,\,\,b\otimes b'\mapsto bb'.
\]

(ii) To show that the map $m$ defines the structure of a locally convex algebra on $A\otimes B$, we use Remark \ref{universal property of projective tensor product topology} again: Indeed, the bilinear map $m$ is continuous if and only if the linear map 
\[n:A\otimes A\otimes B\otimes B\rightarrow A\otimes B,\,\,\,n(a\otimes a'\otimes b\otimes b'):=aa'\otimes bb'
\]is continuous. Here, we have used the canonical isomorphism 
\[A\otimes B\otimes A\otimes B\cong A\otimes A\otimes B\otimes B.
\]Now, a short observations shows that $n=m_A\otimes m_B$, and therefore Lemma \ref{continuity of maps EoF to E'oF'} implies that $m$ is continuous.
\end{proof}

\section{The Projective Tensor Product of Modules}

For the following we recall the projective tensor product topology of Definition \ref{proj. tensor product III}:

\begin{definition}\label{tensor product of $A$ modules}{\bf(The projective tensor product of modules).}\index{Projective!Tensor Product of Modules}
Let $A$ be a unital locally convex algebra, $E$ a right locally convex $A$-module and $F$ a left locally convex $A$-module.  If $I$ is the closure of the subspace of $E\otimes F$ generated by all elements of the form $ea\otimes f-e\otimes af$, $e\in E, f\in F$ and $a\in A$, then the quotient 
\[E\otimes_A F:=E\otimes F/ I
\]is called the (\emph{Hausdorff}) \emph{$A$-tensor product of $E$ and $F$}. The quotient topology of the canonical projection
\[\rho:E\otimes F\rightarrow E\otimes_A F,\,\,\,e\otimes f\mapsto e\otimes_A f:=e\otimes f+I
\]is called the \emph{projective $A$-tensor product topology} on $E\otimes_A F$ and turns $E\otimes_A F$ into a locally convex space. The corresponding completion of $E\otimes_A F$ with respect to this topology is denoted by $E\widehat{\otimes}_AF$.


\end{definition}

\begin{remark}\label{universal prop of tensor prod of mod}{\bf(Universal property of the projective tensor product of modules).} \index{Universal Property of! the Tensor Product of Modules}
We recall that the $A$-tensor product of two locally convex $A$-modules $E$ and $F$ has the following universal property: If 
\[\pr:E\times F\rightarrow E\otimes_A F,\,\,\,(e,f)\mapsto e\otimes_A f
\]denotes the canonical continuous $A$-balanced bilinear ``projection", then each continuous $A$-balanced bilinear map $\alpha:E\times F\rightarrow Z$ into another locally convex space $Z$ factors uniquely through a continuous linear map $\beta: E\otimes_A F\rightarrow Z$, i.e., we have $\alpha=\beta\circ \pr$.
\end{remark}

\begin{proposition}\label{Eo_AF as B-module}
Let $A$ and $B$ be two unital locally convex algebras, $E$ a right locally convex $A$-module and $F$ a locally convex \emph{(}$A,B$\emph{)}-bimodule. Then the map
\[\rho:E\otimes_A F\times B\rightarrow E\otimes_A F,\,\,\,(e\otimes_A f,b)\mapsto e\otimes_A f.b
\]defines the structure of a right locally convex $B$-module on $E\otimes_A F$.
\end{proposition}

\begin{proof}
\,\,\,This assertion immediately follows from Proposition \ref{EoF as B-module} and the definition of the quotient topology on $E\otimes_A F$.
\end{proof}


\begin{corollary}\label{mod and ringhom}
Let $A$ and $B$ be two unital locally convex algebras, $E$ a right locally convex $A$-module and $\varphi:A\rightarrow B$ a morphism of locally convex algebras. Then $E\otimes_A B$ is a right locally convex $B$-module.
\end{corollary}

\begin{proof}
\,\,\,Obviously $B$ is a right locally convex $B$-module. The left locally convex $A$-module structure is given by $a.b:=\varphi(a)b$ for all $a\in A$ and $b\in B$. Now, the claim follows from Proposition \ref{Eo_AF as B-module}.
\end{proof}

\begin{remark}\label{mod and ringhom I}
If $A$ and $B$ are two unital algebras, $E$ a finitely generated projective $A$-module and $\varphi:A\rightarrow B$ a morphism of algebras, then $E\otimes_A B$ is a finitely generated projective $B$-module. In particular, if $E\cong p\cdot A^n$ for some $n\in\mathbb{N}$ and $p\in\Idem_n(A))$, then $E\otimes_A B\cong\varphi(p)\cdot B^n$, where $\varphi(p)$ denotes the matrix in $\Idem_n(B)$ obtained from $p$ by taking in each entry the image of the corresponding entry in $p$ with respect to the map $\varphi$.

Indeed, in view of Corollary \ref{mod and ringhom}, it remains to show that $E\otimes_A B$ is finitely generated and projective: By Corollary \ref{f.g.pr.mod}, $E\cong p\cdot A^n$ (algebraically) for some $n\in\mathbb{N}$ and $p\in\Idem_n(A)$. Then $E':=(\id_A-p)\cdot A^n$ is a (finitely generated and projective) $A$-module such that
\[(E\otimes_A B)\oplus(E'\otimes_A B)\cong A^n\otimes_A B\cong B^n,
\]i.e., $E\otimes_A B$ is a direct summand of a free $B$-module of rank $n$ and hence a finitely generated projective $B$-module (cf. Corollary \ref{f.g.pr.mod} again). If $E\cong p\cdot A^n$ for some $n\in\mathbb{N}$ and $p\in\Idem_n(A)$, then
\[E\otimes_A B\cong p\cdot A^n\otimes_A B\cong(p\otimes_A\id_B)\cdot(A^n\otimes_A B)\cong\varphi(p)\cdot B^n.
\]
\end{remark}

\section{Rudiments on the Smooth Exponential Law}\label{Rudiments on the Smooth Exponential Law}

For arbitrary sets $X$ and $Y$, let $Y^X$ be the set of all mappings from $X$ to $Y$. Then the following ``exponential law" holds: For any sets $X,Y$ and $Z$, we have
\[Z^{X\times Y}\cong(Z^Y)^X
\]as sets. To be more specific, the map
\[Z^{X\times Y}\cong(Z^Y)^X,\,\,\,f\rightarrow f^{\vee}
\]is a bijection, where
\[f^{\vee}:X\rightarrow Z^Y,\,\,\,f^{\vee}(x):=f(x,\cdot).
\]The inverse map is given by
\[(Z^Y)^X\rightarrow Z^{X\times Y},\,\,\,g\mapsto g^{\wedge},
\]where
\[g^{\wedge}:X\times Y\rightarrow Z,\,\,\,g^{\wedge}(x,y):=g(x)(y).
\]Next we consider an important topology for spaces of smooth functions:

\begin{definition}\label{smooth compact open topology}{\bf(Smooth compact open topology).}\index{Smooth!Compact Open Topology}
If $M$ is a locally convex manifold and $E$ a locally convex space, then the \emph{smooth compact open topology} on $C^{\infty}(M,E)$ is defined by the embedding
\[C^{\infty}(M,E)\hookrightarrow\prod_{n\in\mathbb{N}_0}C(T^nM,T^nE),\,\,\,f\mapsto(T^nf)_{n\in\mathbb{N}_0},
\]where the spaces $C(T^nM,T^nE)$ carry the compact open topology. Since $T^nE$ is a locally convex space isomorphic to $E^{2^n}$, the spaces $C(T^nM,T^nE)$ are locally convex and we thus obtain a locally convex topology on $C^{\infty}(M,E)$.
\end{definition}

It is natural to ask if there exists an ``exponential law"  for the space of smooth functions endowed with the smooth compact open topology, i.e., if we always have
\begin{align}
C^{\infty}(M\times N,E)\cong C^{\infty}(M,C^{\infty}(N,E))\label{exp law}
\end{align}
as a locally convex space, for all locally convex smooth manifolds $M$, $N$ and locally convex spaces $E$. In general, the answer is \emph{no}. Nevertheless, it can be shown that \ref{exp law} holds if $N$ is a finite-dimensional manifold over $\mathbb{R}$. To be more precise:

\begin{lemma}\label{smooth exp law}
Let $M$ be a real locally convex smooth manifold, $N$ a finite-dimensional smooth manifold and $E$ a topological 
vector space. Then a map $f:M\rightarrow C^{\infty}(N,E)$ is smooth if and only if the map
\[f^{\wedge}:M\times N\rightarrow E,\,\,\,f^{\wedge}(m,n):=f(m)(n)
\]is smooth. Moreover, the following map is an isomorphism of locally convex spaces:
\[C^{\infty}(M,C^{\infty}(N,E))\rightarrow C^{\infty}(M\times N,E),\,\,\,f\mapsto f^{\wedge}.
\]
\end{lemma}

\begin{proof}
\,\,\,A proof can be found in [NeWa07], Appendix, Lemma A3.
\end{proof}

\section{Function Spaces with Values in Locally Convex Spaces}

\begin{proposition}\label{proj. tensor product I}
Let $M$ be a manifold and $E$ a locally convex space. The subspace generated by the image of the bilinear map $p:C^{\infty}(M)\times E\rightarrow C^{\infty}(M,E),\,\,\,(f,e)\mapsto f\cdot e$ consists of functions whose image is contained in a finite dimensional subspace of $E$, and will denoted by $C_{\text{\emph{fin}}}^{\infty}(M,E)$ $\mathcal{T}$.\sindex[n]{$C^{\infty}(M,E)$}\sindex[n]{$C_{\text{fin}}^{\infty}(M,E)$}
\end{proposition}

\begin{proof}
\,\,\,Clearly the function $f\cdot e$ has its image in a finite-dimensional (in fact, a one-dimensional) subspace.

If, on the other hand, $f\in C^{\infty}(M,E)$ is a function whose image is contained in a finite dimensional subspace $E_0$ of $E$, then we may write $f=\sum_{i=1}^k f_i\cdot e_i$ for $k=\dim E_0$, a basis $e_1,\ldots,e_k$ of $E_0$ and $f_i\in C^{\infty}(M)$. Thus, the function $f$ belongs to the subspace of $C^{\infty}(M,E)$ generated by the functions of the form $f\cdot e$. 
\end{proof}

\begin{proposition}\label{proj. tensor product II}
Let $M$ be a manifold and $E$ a locally convex space. Then the subspace $C_{\text{\emph{fin}}}^{\infty}(M,E)$ of Proposition \ref{proj. tensor product I} is a dense subspace of $C^{\infty}(M,E)$.
\end{proposition}

\begin{proof}
\,\,\,We divide the proof of this proposition into five parts:

(i) If $n\in\mathbb{N}$ and $M=U\subseteq\mathbb{R}^n$ is an open subset, then the assertion follows from [Tre67], Proposition 44.2. 

(ii) To get the result for a general manifold $M$, we first have to show that the space $C^{\infty}_c(M,E)$ of smooth functions with compact support is dense in $C^{\infty}(M,E)$: Indeed, let $(U_n)_{n\in\mathbb{N}}$ be an increasing sequence of relatively compact open subsets of $M$ with $\bigcup_{n\in\mathbb{N}}U_n=M$. For each $n\in\mathbb{N}$ we choose a compactly supported smooth function $g_n$ which is equal to $1$ on $U_n$. Then each $f\in C^{\infty}(M,E)$ is the limit of the sequence $(g_n f)_{n\in\mathbb{N}}$.

(iii) Now, if $f\in C^{\infty}_c(M,E)$ and $K:=\supp(f)$, then there exists a finite open cover $U_1,\ldots,U_m$ of $K$ by charts of $M$ and thus a smooth partition of unity subordinated to this covering, i.e., smooth $\mathbb{R}$-valued functions $\psi_1,\ldots,\psi_m$ on $M$ such that
\[\supp(\psi_i)\subseteq U_i\,\,\,\text{and}\,\,\,\sum_{i=1}^m\psi_i=1
\]for all $i=1,\ldots m$. 

(iv) Next, we write $f_i$ for the restriction of $f$ to $U_i$, i.e., $f_i:=f_{\mid U_i}$ for $i=1,\ldots,m$. Then $f_i\in C^{\infty}(U_i,E)$ and by part (i) of the proof there exists a sequence $(h^i_n)_{n\in\mathbb{N}}$ in $C_{\text{fin}}^{\infty}(U_i,E)$ such that $h^n_i$ converges to $f_i$ in $C^{\infty}(U_i,E)$. Further, a short observation shows that $\psi_i h^i_n\in C_{\text{fin}}^{\infty}(M,E)$ and that $\psi_ih^i_n$ converges to $\psi_if_i$ in $C^{\infty}(M,E)$.

(v) Since $\psi_if_i=\psi_i f$, we conclude from part (iv) that
\[\sum_{i=1}^m\psi_ih^i_n\to\sum_{i=1}^m\psi_if_i=\sum_{i=1}^m\psi_if=f
\]$C^{\infty}(M,E)$. In particular, $C_{\text{\emph{fin}}}^{\infty}(M,E)$ is dense in $C^{\infty}(M,E)$.

\end{proof}

\begin{theorem}\label{proj. tensor product IV}

Let $M$ be a manifold and $E$ a complete locally convex space. Then the following assertions hold:
\begin{itemize}
\item[\emph{(a)}]
The locally convex space $C^{\infty}(M,E)$ is complete.
\item[\emph{(b)}]
The map
\[\phi:C^{\infty}(M)\otimes E\rightarrow C_{\text{\emph{fin}}}^{\infty}(M,E),\,\,\,f\otimes e\mapsto f\cdot e
\]is an isomorphism of locally convex spaces and can therefore be extended to an isomorphism 
\[\Phi:C^{\infty}(M)\widehat{\otimes}E\rightarrow C^{\infty}(M,E)
\]of complete locally convex spaces.
\end{itemize}
\end{theorem}

\begin{proof}
\,\,\,(a) The first assertion follows, for example, from [Gl05], Proposition 4.2.15 (a).

(b) The second assertion can be found as an example in [Gro55], Chapter 2, \S 3.3, Theorem 13. The crucial point here is to use the fact that $C^{\infty}(M)$ is a nuclear space.
\end{proof}

\begin{remark}\label{proj. tensor product V}
Note that the isomorphism $\phi$ of Theorem \ref{proj. tensor product IV} (b) means that the smooth compact open topology and the projective tensor product topology coincides on $C^{\infty}(M)\otimes E$\sindex[n]{$C^{\infty}(M)\otimes E$}.
\end{remark}

\begin{remark}\label{fibres products of manifolds}{\bf(The fibred product of manifolds).}\index{Fibred Product of Manifolds}
Let $N$, $N'$ and $M$ be compact manifolds. Further, let $\pi:N\rightarrow M$ and $\pi':N'\rightarrow M$ be smooth maps. If the fibred product
\[N\times_M N':=\{(n,n')\in N\times N':\,\pi(n)=\pi'(n')\}
\]is a (compact) submanifold of $N\times N'$ (this is, for example, the case if one the smooth maps $\pi$, $\pi'$ is a submersion), then
\[C^{\infty}(N)\widehat{\otimes}_{C^{\infty}(M)}C^{\infty}(N')\cong C^{\infty}(N\times_M N')
\]holds as an isomorphism of unital Fr\'echet algebras. The corresponding isomorphism is on simple tensors defined by
\[f\otimes g\mapsto (f\circ\pr_N)\cdot(g\circ\pr_{N'}).
\]Indeed, a nice reference for this statement is the scholium of [Ky87]. 
\end{remark}


In the remaining part we will show that the space of continuous functions on a topological space with values in a locally convex algebra is again a locally convex algebra. For this we need the following lemma:

\begin{lemma}\label{multiplication and seminorns}
Let $A$ be an algebra endowed with a locally convex topology. Then the multiplication on $A$ is continuous if and only if for each continuous seminorm $p$ on $A$ there exists a continuous seminorm $q$ on $A$ with
\[p(ab)\leq q(a)q(b)\,\,\,\text{for all}\,\,\,a,b\in A.
\]
\end{lemma}

\begin{proof}
\,\,\,($``\Rightarrow"$) Let $A$ be a locally convex algebra and $p$ be a continuous seminorm on $A$. Then the continuity of the multiplication on $A$ in $(0,0)$ implies the existence of a continuous seminorm $q$ on $A$ with
\[p(ab)\leq q(a)q(b)\,\,\,\text{for all}\,\,\,a,b\in A.
\]

($``\Leftarrow"$) If, conversely, $A$ is an algebra endowed with a locally convex topology defined by a system $\mathcal{P}$ of seminorms with the property that for each $p\in\mathcal{P}$ there exists a $q\in\mathcal{P}$ with $p(ab)\leq Cq(a)q(b)$ for all $a,b\in A$, then the multiplication in $A$ is continuous because
\begin{align}
p(ab-a_0b_0)&=p(a(b-b_0)+(a-a_0)b_0)\leq C\left(q(a)q(b-b_0)+q(a-a_0)q(b_0)\right)\notag\\
&\leq C\left(q(a-a_0)q(b-b_0)+q(a_0)q(b-b_0)+q(a-a_0)q(b_0)\right).\notag
\end{align}
\end{proof}

\begin{proposition}\label{C(X,A) loc. con. algebra}
Let $A$ be a locally convex algebra and $X$ a topological space. If we endow the algebra $C(X,A)$ with the topology of uniform convergence on compact subsets, then $C(X,A)$ is a locally convex algebra.\sindex[n]{$C(X,A)$}
\end{proposition}

\begin{proof}
\,\,\,The topology on $C(X,A)$ is defined by the seminorms
\[p_K(f):=\sup_{x\in K}p(f(x)),
\]where $p$ is a continuous seminorm on $A$ and $K\subseteq X$ is a compact subset. In this way $C(X,A)$ clearly becomes a locally convex space.

Now, let $p$ be a seminorm on $A$. According to Lemma \ref{multiplication and seminorns}, there exists a continuous seminorm $q$ on $A$ with $p(ab)\leq q(a)q(b)$ for all $a,b\in A$. Therefore,
\[p_K(fg)\leq q_K(f)q_K(g)
\]holds for all $f,g\in C(X,A)$ and each compact subset $K\subseteq X$. Thus, the claim follows from Lemma \ref{multiplication and seminorns}.
\end{proof}

\section{Some Results on Barreledness}

In this section we present a special class of locally convex spaces in which the ``Theorem of Banach--Steinhaus" still remains true. These are the so-called \emph{barrelled} spaces. In particular, we show that each continuous action of a compact group $G$ on a barrelled unital locally convex algebra $A$ by algebra automorphisms can be extended to a continuous action of $G$ on the completion $\widehat{A}$ of $A$ by algebra automorphisms. We have not found such a theorem discussed in the literature.

\begin{definition}\label{barreldness 1}{\bf(Barrelled).}\index{Barrelled}
A \emph{barrel} $B$ of a locally convex space $E$ is a closed absorbent disc. In particular, a locally convex space $E$ is called \emph{barrelled} if each barrel in $E$ is a zero-neighbourhood.
\end{definition}

\begin{example}\label{barreldness 2}{\bf(Fr\'echet spaces).}
Fr\'echet spaces are barrelled (cf. [BeNa85], Section 11.4).
\end{example}

\begin{proposition}\label{barreldness 3}{\bf(Quotients).}
Every Hausdorff quotient of a barrelled space is barrelled.
\end{proposition}

\begin{proof}
\,\,\,For the proof we refer to [Sch99], Chapter II, \S 7, Corollary 1.
\end{proof}

\begin{remark}\label{barreldness 4}{\bf(F-spaces).}\index{$F$-Spaces}
Recall that a pair $(E,d)$, consisting of a vector space $E$ and a metric $d:E\times E\rightarrow\mathbb{R}$ is a called an F-\emph{space}, if
\begin{itemize}
\item
scalar multiplication and addition are continuous with respect to the metric $d$,

\item
the metric $d$ is translation-invariant, i.e., $d(e+f,e'+f)=d(e,e')$ for all $e,e'$ and $f$ in $E$,

\item
 and the metric space $(E,d)$ is complete.
\end{itemize}
Clearly, each Fr\'echet space is a F-space
\end{remark}

\begin{proposition}\label{barreldness 5}{\bf(Tensor products).}
If $E$ and $F$ are two \emph{F}-spaces, then $E\otimes F$ is metrizable and barrelled. Moreover, $E\widehat{\otimes}F$ is a \emph{F}-space.
\end{proposition}

\begin{proof}
\,\,\,A proof of this statement can be found in [Gro55], Chapter I, \S 1.3, Proposition 5.1.
\end{proof}

\begin{corollary}\label{barreldness 6}
The tensor product $E\otimes F$ of two Fr\'echet spaces $E$ and $F$ is barrelled. Moreover, in this case, $E\widehat{\otimes}F$ is a Fr\'echet space.
\end{corollary}

\begin{proof}
\,\,\,The first assertion directly follows from Remark \ref{barreldness 4} and Proposition \ref{barreldness 5}. The second assertion follows from the fact that $E\widehat{\otimes}F$ is a complete locally convex space whose topology is generated by a countable family of seminorms.
\end{proof}

\begin{remark}{\bf(Pointwise bounded subsets).}
For the following theorem we recall that a subset $H$ of the continuous linear operators $L(E,F)$ between to locally convex spaces $E$ and $F$ is ``pointwise" bounded, if 
\[H(e):=\{h(e):\,h\in H\}
\]is a bounded subset of $F$ for each $e\in E$.
\end{remark}

\begin{theorem}\label{barreldness 7}{\bf(Banach--Steinhaus for barrelled spaces).}\index{Theorem!of Banach--Steinhaus}
Let $E$ be a barrelled space and $F$ an arbitrary locally convex space. If $H\subseteq L(E,F)$ is pointwise bounded, then $H$ is equicontinuous, i.e., for each zero-neighbourhood $V$ of $F$, there is a zero-neighbourhood $U$ of $E$ such that $H(U)\subseteq V$.
\end{theorem}

\begin{proof}
\,\,\,A nice proof can be found in [BeNa85], Section 11.6, (11.6.1).
\end{proof}

\begin{theorem}\label{barreldness 8}{\bf(Seminorm criterion for equicontinuity).}
Let $E$ and $F$ be two locally convex spaces. A subset $H$ of $L(E,F)$ is equicontinuous if and only if for each continuous seminorm $q$ on $F$, there is a continuous seminorm $p$ on $E$ such that
\[(q\circ h)(e)\leq p(e)
\]holds for each $h\in H$ and $e\in E$.
\end{theorem}

\begin{proof}
\,\,\,For the proof of this theorem we refer to [BeNa85], Section 9.5, (9.5.3).
\end{proof}

\begin{proposition}\label{barreldness 9}{\bf(Jointly vs. Separately).}
If $G$ is a compact group and $E$ a barrelled space, then each separately continuous action $\alpha:G\times E\rightarrow E$ of $G$ on $E$ by vector space automorphisms is automatically \emph{(}jointly\emph{)} continuous.
\end{proposition}

\begin{proof}
\,\,\,Since $G$ is compact and $\alpha$ separately continuous, the set $H:=\{\alpha(g):\,g\in G\}\subseteq L(E)$ is pointwise bounded. Thus, Theorem \ref{barreldness 7} implies that $H$ is equicontinuous and we can apply Theorem \ref{barreldness 8}: We fix $g\in G$, $e\in E$ and choose a continuous seminorm $q$ on $E$. Further, if $(g_n)_{n\in\mathbb{N}}$ is a sequence in $G$ with $g_n\to g$ and $p$ is a continuous seminorm on $E$ satisfying $q\circ\alpha(g)\leq p$ for each $g\in G$, then we obtain
\begin{align}
q(g.e-g_n.e')&\leq q((g-g_n).e)+q(g_n(e-e'))\notag\\
&\leq q((g-g_n).e)+p(e-e')\leq\frac{\epsilon}{2}+\frac{\epsilon}{2}\leq\epsilon\notag
\end{align}
for large $n\in\mathbb{N}$ and $e'$ in a sufficiently small neighbourhood $U$ of $e$.
\end{proof}

\begin{theorem}\label{barreldness 10}{\bf(Extending group actions).}\index{Extending!Group Actions}
If $G$ is a compact group, $A$ a barrelled unital locally convex algebra and $\widehat{A}$ its completion, then each continuous action $\alpha:G\times A\rightarrow A$ of $G$ on $A$ by algebra automorphisms can be extended to a continuous action 
\[\widehat{\alpha}:G\times\widehat{A}\rightarrow\widehat{A}
\]of $G$ on $\widehat{A}$ by algebra automorphisms.
\end{theorem}

\begin{proof}
\,\,\,The proof of this theorem is divided into four parts:

(i) According to Proposition \ref{ext of dense cont algebra iso}, we can extend each automorphism $\alpha(g):A\rightarrow A$ to an automorphism $\widehat{\alpha}(g):\widehat{A}\rightarrow\widehat{A}$. In particular, we obtain an action $\widehat{\alpha}:G\times\widehat{A}\rightarrow\widehat{A}$ of $G$ on $\widehat{A}$ by algebra automorphisms. 

(ii) To prove its continuity, we proceed as follows: We first note that completions of barrelled spaces are barrelled again (cf. [NaBe85], Section 11.8, (11.8.1)). Thus, Proposition \ref{barreldness 9} and the construction of $\widehat{\alpha}$ imply that it just remains to verify the continuity of the orbit maps 
\[\widehat{\alpha}_a:G\rightarrow\widehat{A},\,\,\,g\mapsto g.a\,\,\,\text{for}\,\,\,a\in\widehat{A}.
\]

(iii) We fix $a\in\widehat{A}$ and let $(a_i)_{i\in I}$ be a net in $A$ with $a_i\to a$. If $\widehat{\alpha}_i:=\widehat{\alpha}_{a_i}$, then we have $\widehat{\alpha}_i(g)\to\widehat{\alpha}_a(g)$ for each $g\in G$, i.e., $\widehat{\alpha}_i\to\widehat{\alpha}_a$ pointwise. To prove the continuity of $\widehat{\alpha}_a$, it is therefore enough to show that the set $(\widehat{\alpha}_i)_{i\in I}$ is a Cauchy net in the space $C(G,\widehat{A})$: Indeed, if $(\widehat{\alpha}_i)_{i\in I}$ is a Cauchy net in the space $C(G,\widehat{A})$, then the completeness of $C(G,\widehat{A})$ implies that $\widehat{\alpha}_i\to\widehat{\alpha}_a$ uniformly in $C(G,\widehat{A})$. Thus, $\widehat{\alpha}_a$ is continuous by the uniform convergence theorem.

(iv) If $p$ is a continuous seminorm on $\widehat{A}$, then $p_G(f):=\sup_{g\in G}p(f(g))$ defines a continuous seminorm on $C(G,\widehat{A})$. Hence, the compactness of $G$ and the continuity of $\alpha$ imply that
\begin{align}
p_G(\widehat{\alpha}_i-\widehat{\alpha}_j)&=\sup_{g\in G}p(g.a_i-g.a_j)=\sup_{g\in G}p(\widehat{\alpha}(g,a_i-a_j)) \notag\\
&=\sup_{g\in G}p(\alpha(g,a_i-a_j))\leq q(a_i-a_j)\leq\epsilon\notag
\end{align}
for some continuous seminorm $q$ on $A$ and $i,j$ large enough. This shows that $(\widehat{\alpha}_i)_{i\in I}$ is a Cauchy net in the space $C(G,\widehat{A})$ and therefore proves the theorem.
\end{proof}

\chapter{An Important Source of NC Spaces: Noncommutative Tori}

A prominent class of noncommutative spaces are the so called ``noncommutative tori" or ``quantum tori".
These spaces will also provide an important class of examples throughout this thesis. For this reason we now provide an overview on the structure of NC tori.

\section{The Noncommutative 2-Torus}

\begin{definition}\label{NC torus-algebra}{\bf(The noncommutative 2-torus).}\index{Noncommutative!2-Torus}
Let $\theta\in\mathbb{R}$ and $\lambda=\exp(2\pi i\theta)$. The \emph{noncommutative 2-torus algebra} $A^2_{\theta}$\sindex[n]{$A^2_{\theta}$} is the \emph{universal} unital $C^*$-algebra generated by unitaries $U$ and $V$ subject to the relation
\[UV=\lambda VU.
\]The universality property here means that given any unital $C^*$-algebra $B$ with two unitaries $u$ and $v$ satisfying $uv=\lambda vu$, then there exists a unique unital $C^*$-morphism $A_{\theta}\rightarrow B$ sending $U\rightarrow u$ and $V\rightarrow v$.
\end{definition}

\begin{remark}\label{The Schrodinger picture}{\bf(The Schr\"odinger picture of $A^2_{\theta}$):}
In the purely algebraic case any set of generators and relations automatically defines a universal algebra. This is not the case for universal $C^*$-algebras. One has to take care in defining a norm satisfying the $C^*$-identity, and in general the universal problem does not have a solution.
For the noncommutative torus algebra it is possible to proceed as follows:

(a) We first consider the unitary operators $U,V:\ell^2(\mathbb{Z})\rightarrow\ell^2(\mathbb{Z})$ defined by
\[(Uf)(n)=\exp(2\pi i\theta n)f(n)\,\,\,\text{and}\,\,\,(Vf)(n)=f(n-1)
\]and note that they satisfy $UV=\lambda VU$. Let $A_{\theta}$ be the unital $C^*$-subalgebra of $B(\ell^2(\mathbb{Z}))$ generated by $U$ and $V$. 

(b) Using the Fourier isomorphism $\ell^2(\mathbb{Z})\cong L^2(\mathbb{T})$, we see that $A^2_{\theta}$ can also be described as the $C^*$-subalgebra of $B(L^2(\mathbb{T}))$ generated by the operators 
\[(Uf)(z)=f(\lambda z)\,\,\,\text{and}\,\,\,(Vf)(z)=zf(z)\,\,\,\text{for}\,\,\,f\in L^2(\mathbb{T}).
\]
\end{remark}

\begin{remark}
When $\theta=0$, any element $a\in A^2_{\theta}=C(\mathbb{T}^2)$ can be developed as a \emph{Fourier series}:
\[a(\phi_1,\phi_2)\backsim\sum_{k,l\in\mathbb{Z}}a_{kl}\exp(2\pi ik\phi_1)\exp(2\pi il\phi_2),
\]where the series converges uniformly to the function $a$ only under certain conditions (say, if $a$ is piecewise $C^1$).
\end{remark}

\section{The Smooth Noncommutative 2-Torus}
To avoid the severe representation problems that arises even in the classical case, it is convenient to restrict the series developments to the pre-$C^*$- algebra of smooth functions on $\mathbb{T}^2$ (whereby the Fourier series converges to $a$ both uniformly and in $L^2$-norm). Now, $a\in C^{\infty}(\mathbb{T}^2)$\sindex[n]{$C^{\infty}(\mathbb{T}^2)$} if and only if the coefficients $a_{kl}$ belong to the space $S(\mathbb{Z}^2)$\sindex[n]{$S(\mathbb{Z}^2)$} of \emph{rapidly decreasing}\index{Rapidly Decreasing!Double Sequences} double sequences, i.e., those that satisfy
\[\sup_{k,l\in\mathbb{Z}}(1+k^2+l^2)^r|a_{kl}|^2<\infty \,\,\,\text{for all}\,\,r\in\mathbb{N}.
\]

\begin{lemma}\label{t^2 on a^2}
The map $\alpha:\mathbb{T}^2\times A^2_{\theta}\rightarrow A^2_{\theta}$ given on generators by 
\[(z_1,z_2).U^kV^l:=z_1^kz_2^l\cdot U^kV^l
\]is continuous and defines an action by algebra automorphisms.
\end{lemma}

\begin{proof}
\,\,\,This is just a simple calculation.
\end{proof}

\begin{definition}\label{NCQT}{\bf(The smooth noncommutative 2-torus).}\index{Smooth!Noncommutative 2-Torus}
The \emph{smooth noncommutative 2-torus} $\mathbb{T}^2_{\theta}$\sindex[n]{$\mathbb{T}^2_{\theta}$} is the dense (cf. Remark \ref{A infty is dense}) unital $^ *$-subalgebra of smooth vectors for the action of Lemma \ref{t^2 on a^2}. A short observation shows that its elements are given by (norm-convergent) sums
\[a=\sum_{k,l\in\mathbb{Z}}a_{kl}U^kV^l\,\,\,\text{with}\,\,\,(a_{kl})_{k,l\in\mathbb{Z}}\in S(\mathbb{Z}^2)
\](cf. [Tre67], Part III, Chapter 51, Remark 51.1). We conclude from Corollary \ref{automatic smoothness II} that $\mathbb{T}^2_{\theta}$ is a CIA and that the induced action of $\mathbb{T}^2$ on $\mathbb{T}^2_{\theta}$ is smooth.
\end{definition}

\begin{remark}
The algebra $\mathbb{T}^2_{\theta}$ is abelian if and only if $\theta$ is an integer. In this case we may identify $\mathbb{T}^2_{\theta}$ with the algebra $C^{\infty}(\mathbb{T}^2)$ of smooth functions on the 2-torus $\mathbb{T}^2$.
\end{remark}

For rational $\theta$ there is a geometrical description of $\mathbb{T}^2_{\theta}$ in terms of smooth sections on a certain algebra bundle. We first need the following lemma:

\begin{lemma}\label{equivalenz of sections}
Let $(\mathbb{V},M,V,q)$ be a vector or algebra bundle and let $G$ be a finite group which acts freely by diffeomorphisms on $M$. Further, assume that the action of $G$ on $M$ lifts to an action of $G$ on $\mathbb{V}$ by bundle isomorphisms of $\mathbb{V}$. Then the following assertions hold:
\begin{itemize}
\item[\emph{(a)}]
The orbit set $\mathbb{V}/G$ has a unique structure of a vector or algebra bundle over $M/G$ such that 
$$
\begin{xy}
  \xymatrix{
      \mathbb{V} \ar[r] \ar[d]_q    &   \mathbb{V}/G \ar[d]^{\bar{q}}  \\
      M \ar[r]        &   M/G
  }
\end{xy}
$$
commutes and is fibrewise an isomorphism.
\item[\emph{(b)}]
If
\[(\Gamma\mathbb{V})^G:=\{s\in\Gamma\mathbb{V}:\,(\forall g\in G)(\forall m\in M)\,s(m.g)=s(m).g\}
\]denotes the set of $G$-equivariant sections of the bundle, then the map
\[(\Gamma\mathbb{V})^G\rightarrow\Gamma(\mathbb{V}/G),\,\,\,s\mapsto ([m]\mapsto [s(m)])
\]is an isomorphism of $C^{\infty}(M)$-modules.
\end{itemize}
\end{lemma}

\begin{proof}
\,\,\,The proof of this lemma is left as an easy exercise to the reader.
\end{proof}

\begin{proposition}\label{rational quantumtori}{\bf(Rational quantum tori).}\index{Rational Quantum Tori}
If $\theta$ is rational, $\theta=\frac{n}{m}$, $n\in\mathbb{Z}$, $m\in\mathbb{N}$ relatively prime, then $\mathbb{T}^2_{\theta}$ is isomorphic to the algebra of smooth section of an algebra bundle over $\mathbb{T}^2$ with fibre $M_m(\mathbb{C})$.
\end{proposition}

\begin{proof}
\,\,\,We want to apply Lemma \ref{equivalenz of sections} to the trivial algebra bundle
\[\pr_1:\mathbb{T}^2\times M_m(\mathbb{C})\rightarrow\mathbb{T}^2
\]and the group $G=C_m\times C_m$. We proceed as follows:

(i) The group action on $\mathbb{T}^2$ is given by
\[(z_1,z_2).(p,q):=(\exp(2\pi ip\theta)z_1,\exp(2\pi iq\theta)z_2)
\]for $p,q=0,1,\ldots,m-1$. To describe the action of $G$ on $\mathbb{V}:=\mathbb{T}^2\times M_m(\mathbb{C})$ we first define
$$
R:=\begin{pmatrix}
 1      &  &  &  &    \\
         & \exp(2\pi i\theta) &  &  &  \\
         &  & \exp(2\pi i2\theta)  &  &  \\
         &  &  & \ddots &  \\
         &  &  &  &  \exp(2\pi i(m-1)\theta)  
\end{pmatrix}
$$

and

$$
S:=\begin{pmatrix}
 0      & \ldots & \ldots & 0 & 1   \\
 1      & 0 & \ldots & 0 &  0\\
         & \ddots & \ddots & \vdots & \vdots \\
         &  & \ddots& 0 & \vdots \\
 0      &  &  & 1 & 0 
\end{pmatrix}
$$
and note that the matrices $R$ and $S$ are unitary and satisfy the relations $S^m={\bf 1}$, $R^m={\bf 1}$ and $RS=\exp(2\pi i\theta)SR$. Further, they generate a $C^*$-subalgebra which clearly commutes only with the multiples of the identity, so it has to be the full matrix algebra. Then the group action on $\mathbb{T}^2\times M_m(\mathbb{C})$ is given by
\[((z_1,z_2),A).(p,q):=((\exp(2\pi ip\theta)z_1,\exp(2\pi iq\theta)z_2),R^qS^pAS^{-p}R^{-q}).
\]One easily checks the conditions from Lemma \ref{equivalenz of sections}, so that we obtain an algebra bundle \[\mathbb{V}/G\rightarrow\mathbb{T}^2/G\cong\mathbb{T}^2.
\]

(ii) Next, we note that the space of sections of
\[\pr_1:\mathbb{T}^2\times M_m(\mathbb{C})\rightarrow\mathbb{T}^2
\]corresponds to the space 
\[C^{\infty}(\mathbb{T}^2,M_m(\mathbb{C}))\cong C^{\infty}(\mathbb{T}^2)\otimes M_m(\mathbb{C}),
\]which is generated by the elements $u=\exp(2\pi is)$, $v=\exp(2\pi it)$ and $R$, $S$, where the coefficients are again rapidly decreasing with respect to the powers of $u$ and $v$. Again by Lemma \ref{equivalenz of sections}, the sections of the algebra bundle $\mathbb{V}/G\rightarrow \mathbb{T}^2$ correspond to the space $(\Gamma\mathbb{V})^G$ of $G$-equivariant sections of 
\[\pr_1:\mathbb{T}^2\times M_m(\mathbb{C})\rightarrow\mathbb{T}^2
\]

(iii) A section $s:\mathbb{T}^2\rightarrow M_m(\mathbb{C})$,
\[s=\sum_{k,l\in\mathbb{Z}, \atop 0\leq i,j\leq m-1}a_{ijkl}u^kv^lR^iS^j,
\]is $G$-equivariant if and only if the following condition is satisfied:
\[a_{ijkl}\neq 0\,\,\,\text{only if}\,\,\,k\equiv i\,\,\,\text{mod}\,\,m\,\,\,\text{and}\,\,\,l\equiv j\,\,\,\text{mod}\,\,m
\]In this case we may define $a_{kl}:=a_{ijkl}$, where $k\equiv i$ mod $m$ and $l\equiv j$ mod $m$, and the section $s$ can be written as
\[s=\sum_{k,l\in\mathbb{Z}}a_{kl}(uR)^k(vS)^l.
\]

(iv) It remains to note that $U=uR$ and $V=vS$ satisfy only the relations of the noncommutative torus. Hence,
\[\mathbb{T}^2_{\theta}\cong(\Gamma\mathbb{V})^G.
\]
\end{proof}

\begin{corollary}
If $0\neq\theta$ is rational, then $C_{\mathbb{T}^2_{\theta}}$ is isomorphic to $C^{\infty}(\mathbb{T}^2)\cong\mathbb{T}^2_0$.
\end{corollary}

\begin{proof}\label{center of rational quantumtori}
\,\,\,By Proposition \ref{rational quantumtori}, $\mathbb{T}^2_{\theta}\cong\Gamma\mathbb{V}$ for some algebra bundle $\mathbb{V}$ over $\mathbb{T}^2$ with fibre $M_m(\mathbb{C})$. If $s\in C_{\mathbb{T}^2_{\theta}}$, then
\[s(z_1,z_2)\in C_{\mathbb{V}_{(z_1,z_2)}}\cong C_{M_m(\mathbb{C})}\cong\mathbb{C}\cdot{\bf 1}
\]Therefore, 
\[s(z_1,z_2)=f(z_1.z_2)\cdot{\bf 1}_{\mathbb{V}_{(z_1,z_2)}}
\]for a smooth function $f\in C^{\infty}(\mathbb{T}^2)$. 
\end{proof}

\begin{remark}
If $\theta$ is irrational, then it is possible to show that $C_{\mathbb{T}^2_{\theta}}\cong\mathbb{C}$. Moreover, in this case, $\mathbb{T}^2_{\theta}$ is a simple algebra. For more details we refer to [GVF01], Chapter $12$.
\end{remark}

Let $D\in\der(\mathbb{T}^2_{\theta})$ be a derivation 
of the smooth noncommutative 2-torus. Then $D$ is uniquely determined by the values
\[D(U)=\sum_{k,l\in\mathbb{Z}}u_{kl}U^kV^l\,\,\,\text{and}\,\,\,D(V)=\sum_{k,l\in\mathbb{Z}}v_{kl}U^kV^l.
\]The relation
\[D(U)V+UD(V)=\lambda D(V)U+\lambda VD(U),
\]by comparison of coefficients, quickly leads to
\[u_{k(l-1)}(1-\lambda^{1-k})+v_{(k-1)1}(1-\lambda^{1-l})=0.
\]
A short analysis leads to the following proposition:

\begin{proposition}\label{derivations of the quantumtorus}{\bf(Derivations of the quantum torus).}\index{Derivations of the Quantum Torus}
For each $\theta\in\mathbb{R}$ the derivations of $\mathbb{T}^2_{\theta}$ split as a semidirect product of inner and outer derivations, i.e.,
\[\der(\mathbb{T}^2_{\theta})\cong\Inn(\der)(\mathbb{T}^2_{\theta})\rtimes\Out(\der)(\mathbb{T}^2_{\theta}).
\]

Moreover, if $\theta$ is
\begin{itemize}
\item
rational, then the outer derivations correspond exactly to the derivations of the center;
\item
irrational, then the outer derivations are linearly generated by the two derivations $D_U$ and $D_V$ defined by
\[D_U(U)=U\,\,\,\text{and}\,\,\,D_U(V)=0
\]and
\[D_V(U)=0\,\,\,\text{and}\,\,\,D_V(V)=V.
\]
\end{itemize}
\end{proposition}
\begin{proof}
\,\,\,For the proof we refer to the nice paper [DKMM02], Section 3.4.
\end{proof}

\section{The (Smooth) Noncommutative $n$-Torus}

Again, we write ${\bf k}=(k_1,\ldots,k_n)$ for elements of $\mathbb{Z}^n$.

\begin{definition}\label{appendix E noncommutative n-tori}{\bf(The noncommutative $n$-torus).}\index{Noncommutative $n$-Torus}
If $\theta$ is a real skewsymmetric $n\times n$ matrix with entries $\theta_{ij}$, then the \emph{noncommutative $n$-torus algebra} $A^n_{\theta}$\sindex[n]{$A^n_{\theta}$} is the \emph{universal} unital $C^*$-algebra generated by unitaries $U_1,\ldots,U_n$ subject to the relations
\[U_rU_s=\exp(2\pi i\theta_{rs})U_sU_r\,\,\,\text{for all}\,\,\,1\leq r,s\leq n.
\]
\end{definition}

\begin{lemma}\label{t^n on a^n}
The map $\alpha:\mathbb{T}^n\times A^n_{\theta}\rightarrow A^n_{\theta}$ given on generators by 
\[z.U^{\bf k}=z^{\bf k}\cdot U^{\bf k},
\]where $U^{\bf k}:=U^{k_1}_1\cdots U^{k_n}_n$, is continuous and defines an action by algebra automorphisms.
\end{lemma}

\begin{proof}
\,\,\,This is again just a simple calculation.
\end{proof}

\begin{remark}
The \emph{Schwartz space}\index{Schwartz Space} $S(\mathbb{Z}^n)$\sindex[n]{$S(\mathbb{Z}^n)$} is the vector space of rapidly decreasing sequences\index{Rapidly Decreasing!Sequences} sequences indexed by $\mathbb{Z}^n$, i.e., sequences $(a_{\bf k})_{{\bf k}\in\mathbb{Z}^n}$ for which
\[\sup_{{\bf k}\in\mathbb{Z}^n}(1+||{\bf k}||^2)^r|a_{\bf k}|^2<\infty\,\,\,\text{for all}\,\,r\in\mathbb{N}.
\]
\end{remark}

\begin{definition}\label{n-dim QT}{\bf(The smooth noncommutative $n$-torus).}\index{Smooth!Noncommutative $n$-Torus}
The \emph{smooth noncommutative n-torus} $\mathbb{T}^n_{\theta}$\sindex[n]{$\mathbb{T}^n_{\theta}$} is the dense (cf. Remark \ref{A infty is dense}) unital subalgebra of smooth vectors for the action of Lemma \ref{t^n on a^n}. Its elements are given by (norm-convergent) sums
\[a=\sum_{{\bf k}\in\mathbb{Z}^n}a_{\bf k}U^{\bf k},\,\,\,\text{with}\,\,\,(a_{\bf k})_{{\bf k}\in\mathbb{Z}^n}\in S(\mathbb{Z}^n)
\](cf. [Tre67], Part III, Chapter 51, Remark 51.1). We conclude from Corollary \ref{automatic smoothness II} that $\mathbb{T}^n_{\theta}$ is a CIA and that the induced action of $\mathbb{T}^n$ on $\mathbb{T}^n_{\theta}$ is smooth. 
\end{definition}

\subsection*{$K$-Theory of the Smooth Noncommutative $n$-Torus}

At the end of this appendix we ask for the $K$-theory of the noncommutative $n$-torus. It was G. A. Elliot in [Ell84], who was able to compute the $K$-groups by induction on $n$. The surprising result is that the $K$-groups do not depend on the parameter $\theta$ and are isomorphic to the corresponding groups for commutative tori:

\begin{proposition}\label{K-groups of noncommutative n-torus}{\bf($K$-groups of the smooth noncommutative $n$-torus).}\index{$K$-Groups}
\begin{itemize}
\item
$K_0(\mathbb{T}^n_{\theta})\cong K^0(\mathbb{T}^n)=\mathbb{Z}^{2^{n-1}}$,
\item
$K_1(\mathbb{T}^n_{\theta})\cong K^1(\mathbb{T}^n)=\mathbb{Z}^{2^{n-1}}$
\sindex[n]{$K_0(\mathbb{T}^n_{\theta})$}\sindex[n]{$K_1(\mathbb{T}^n_{\theta})$}
\end{itemize}

\begin{flushright}
$\blacksquare$
\end{flushright}
\end{proposition}

\end{appendix}

\printindex[n] 
\printindex[idx] 

\cleardoublepage

\thispagestyle{empty}
\begin{center}
\Huge{\textbf{Lebenslauf}}
\end{center}
$ $\\
$ $\\
$ $\\
$ $\\
\begin{table}[h]
\begin{tabular}{ll}
Juli 2011 & Promotion an der {\bf Universit\"at Erlangen-N\"urnberg}\\
 & Gesamturteil: {\bf Summa cum laude}\\
 & \\ 
Seit April 2010 & {\bf Universit\"at Erlangen-N\"urnberg}\\
 & Doktorand im Department Mathematik\\
 & Titel der Dissertation: \\
 & \emph{A geometric approach to noncommutative principal bundles}\\
 & \\ 
April 2008- & Promotionsstipendiat der {\bf Studienstiftung des deutschen Volkes}\\
M\"arz 2011 & \\
 & \\
April 2007- & {\bf Technische Universit\"at Darmstadt}\\
M\"arz 2010 & Doktorand am Fachbereich Mathematik\\
 & Titel der Dissertation: \\
 & \emph{A geometric approach to noncommutative principal bundles}\\
 & \\
April 2004- & Stipendiat der {\bf Studienstiftung des deutschen Volkes}\\
M\"arz 2007 & \\
 & \\
Oktober 2002- & {\bf Technische Universit\"at Darmstadt}\\
M\"arz 2007  & Studium der Mathematik mit Nebenfach  ``Theoretische Physik"\\
 & Gesamturteil: {\bf sehr gut (1,20)}\\
 & Titel der Diplomarbeit: \emph{Coverings of non-connected Lie groups}\\
 & \\
 September 2001- & {\bf Zivildienst} in einer Tagesf\"orderst\"atte f\"ur schwer \\
 Juni 2002 & behinderte Menschen der AWO im Hainbachtal (Offenbach)\\
 & \\
 September 1992- & {\bf Leibnizgymnasium Offenbach}\\
Mai 2001 & Allgemeine Hochschulreife\\
 & Abiturnote: {\bf 1,3}\\
 & \\
 
 30.05.1982 & Geburt in {\bf Offenbach am Main}
\end{tabular}
\end{table}


\begin{thebibliography}{xxxxxxxxxx}

\bibitem[ABP04]{ABP04}
 Allison, B., S. Berman and A. Pianzola
\emph{Covering Algebras II: Isomorphism of Loop algebras},
Journal f\"ur die reine und angewandte Mathematik (Crelles Journal), {\bf 571}, (2004), 39-71.



\bibitem[AG60]{AG60}
Auslander, M. and  O. Goldman,
\emph{The Brauer group of a commutative ring},
Trans. Amer. Math. Soc., {\bf 97}, (1961), 367-409.

\bibitem[Ba00]{Ba00}
Balachandran, V.K.,
\emph{Topological Algebras}, 
Mathematics Studies, {\bf 185}, North-Holland, 2000.

\bibitem[BCER84]{BCER84}
Batty, C.J.K, A.L. Carrey, D.E. Evans and Derek W. Robinson
\emph{Extending Characters}, 
Publ. RIMS, Kyoto Univ., {\bf 20}, (1984), 119-130.

\bibitem[BeNa85]{BeNA85}
Beckenstein, E. and L. Narici,
\emph{Topological Vector Spaces}, 
Monographs and Textbooks in Pure and Applied Mathematics, {\bf 95}, Dekker, 1985.

\bibitem[Bi04]{Bi04}
Biller, H.,
\emph{Continuous Inverse Algebras and Infinite-Dimensional Linear Lie Groups}, 
Habilitationsschrift, Technische Universit\"at Darmstadt, Juli 2004.


\bibitem[Bos90]{Bos90}
   Bost, J.-B.,   
\emph{Principe d'Oka, $K$-theorie et systmes dynamiques noncommutativs}, 
Invent. Math. {\bf 101} (1990), 261-333.

\bibitem[Bou55]{Bou55}
   Bourbaki, N.,   
\emph{Espaces Vectoriels Topologiques}, 
Livre V, Fascicule de R\'{e}sultats, Hermann, Paris 1955.

\bibitem[Bou66]{Bou66}
    \textemdash,  
\emph{General Topology I}, 
Hermann, Paris 1966.

\bibitem[Br93]{Br93}
 Bredon, G. E.,
\emph{Topology and Geometry}
Graduate Texts in Mathematics, vol. {\bf 139}, Second Edition, Springer-Verlag, New York, 1993.

\bibitem[Bry93]{Bry93}
   Brylinski, J.-L.,   
\emph{Loop Spaces, Characteristic Classes and Geometric Quantization},
Progr. Math. {\bf107}, Birkhuser Verlag, 1993.

\bibitem[Brz97]{Brz97}
   Brzezinski, T.,   
\emph{Quantum Fibre Bundles. An Introduction},
Banach Center Publications, Volume {\bf 39}, Institute of Mathematics, Polish Academy of Sciences,
Warszawa 1997.

\bibitem[BrzMaj93]{BrzMaj93}
   Brzezinski, T. and S. Majid, 
\emph{Quantum Group Gauge Theory on Quantum Spaces},
Comm. Math. Phys. {\bf 157} (1993), 591; ibid. 167 (1995), 235 (erratum).

\bibitem[BrzWi08]{BrzWi08}
   Brzezinski, T. and Pawel Witkowski
\emph{Galois Structures},
Notes by Pawel Witkowski, Januar 2008.


\bibitem[CaMa00]{CaMa00}
 Calow, D., and R. Matthes  
\emph{Covering and gluing of algebras and differential algebras}, 
Journal of Geometry and Physics, {\bf 32}, (2000), 364-396.

\bibitem[Co94]{Co94}
   Connes, A.,   
\emph{Noncommutative Geometry}, 
Academic Press, 1994.

\bibitem[Co06]{Co06}
   \textemdash,  
\emph{Noncommutative Geometry and Physics}, 
Survey Paper, \url{http://www.alainconnes.org/docs/einsymp.pdf}

\bibitem[Co08]{Co08}
   \textemdash,  
\emph{On the Spectral Characterization of Manifolds}, 
arXiv:0810.2088v1 [math.OA], 12 Oct 2008.

\bibitem[Con93]{Con93}
   Conlon, L.,   
\emph{Differentiable Manifolds: A First Course}, 
Birkh\"auser, Boston, 1993.

\bibitem[Conw90]{Conw90}
  Conway, J.B., 
\emph{A Course in Functional Analysis}, 
Graduate Texts in Mathematics, vol. {\bf 96}, Springer-Verlag, 1990.

\bibitem[DK00]{DK00}
  Duistermaat, J.J., Kolk, J.A.C.
\emph{Lie Groups}
Springer, Berlin, 2000.    

\bibitem[DV88]{DV88}
   Dubois--Violette, M.,
\emph{D\'erivations et calcul diff\'erentiel non commutatif},
C.R. Acad. Sci. Paris, 307, S\'erie I, 403-408, 1988.   
      
\bibitem[DKMM02]{DKMM02}
   Dubois--Violette, Michel, Andreas Kriegl, Yoshiaki Maeda, Peter W. Michor,   
\emph{Smooth *-Algebras}, 
Progress of Theoretical Physics Supplement Number {\bf 144} (2001), 54-78.

\bibitem[Du96]{Du96}
Durdevich, M.,
\emph{Geometry of Quantum Principal Bundles I},
Communications in Mathematical Physics, {\bf 175} (1996), 457-521 .

\bibitem[Ell84]{Ell84}
   Elliot, G. A.,   
\emph{On the K-theory of the $C^*$-algebra generated by a projective representation of a torision-free discrete abelian group}, in \emph{Operator Algebras and Group Representations 1}, G. Arsene, ed., (Pitman, London, 1984), pp. 157-184.
Birkh\"auser, Boston, 1993.

\bibitem[ENOO09]{ENOO09}
  Echterhoff, S., Ryszard Nest and Herve Oyono-Oyono 
\emph{Principal non-commutative torus bundles}, in 
Proceedings of the London Mathematical Society, (3) {\bf 99} (2009), 1-31.

\bibitem[EpTh79]{EpTh79}
  Epstein, D. and W. Thurston
\emph{Transformation groups and natural bundles}, in 
Proceedings of the London Mathematical Society, {\bf 38} (1979), 219-236.

\bibitem[Gl01]{Gl01}
   Gl\"ockner, H.,  
\emph{Infinite-dimensional Lie groups without completeness condition}, 
in ``Geometry and Analysis on Finite- and Infinite-Dimensional Lie Groups``,
A. Strasburger et al Eds., Banach Center Publications {\bf 55}, 53-59, Warszawa 2002.

\bibitem[Gl02]{Gl02}
  \textemdash, 
\emph{Algebras whose groups of units are Lie groups},
Studia Math., {\bf 153} ({\bf 2}), 2002.

\bibitem[Gl05]{Gl05}
     \textemdash, 
\emph{Infinite-Dimensional Lie groups}, 
Lecture Notes, Darmstadt, 2005.

\bibitem[GoRo79]{GoRo79}
  Gootman, E. C. and J. Rosenberg
\emph{The Structure of Crossed Product $C^*$-Algebras: A Proof of the Generalized Effros--Hahn Conjecture}, 
Inventiones mathematicae, {\bf 52} (1979), 283-298.




\bibitem[Gro55]{Gro55}
Grothendieck, A.,
\emph{Produits Tensoriels Topologiques et Espaces Nucl\'{e}aires}, 
Memoirs of the American Mathematical Society, Number {\bf 16}, 1955.

\bibitem[GuPo74]{GuPo74}
Guillemin, V. and A. Pollack,
\emph{Differential topology}, 
Prentice-Hall Inc. 1974.

\bibitem[Gue08]{Gue08}
G\"undogan, H.,
\emph{Lie algebras of smooth sections}, 
arXiv:0803.2800v2 [math.RA], 25 Mar 2008.

\bibitem[GVF01]{GVF01}
Gracia-Bondia, J.M., J.C. Vasilly, and H. Figueroa, 
\emph{Elements of Non-commutative Geometry}, 
Birkh\"auser Advanced Texts, Birkh\"auser Verlag, Basel, 2001.

\bibitem[Har87]{Har87}
Hartshore, Robin,
\emph{Algebraic Geometry}, 
Graduate Texts in Mathematics, vol. {\bf 52}, Springer-Verlag, New York, 1987.

\bibitem[HiNe10]{HiNe10}
Hilgert, J. and K.-H. Neeb
\emph{An Introduction to the Structure and Geomerty of Lie Groups and Lie Algebras}, 
To appear.

\bibitem[HoMo06]{HoMo06}
  Hofmann, K.H. and S.A. Morris, 
\emph{The Structure of  Compact Groups} 2nd Revised and Augmented Edition,
de Gruyter Studies in Mathematics {\bf 25}, Cambridge University Press, 2006.

\bibitem[Hu75]{Hu75}
Husemoller, D.,
\emph{Fibre bundles}, 
Graduate Texts in Mathematics, vol. {\bf 20}, Springer-Verlag, New York, 1975.

\bibitem[Hu08]{Hu08}
 \textemdash, 
\emph{Basic bundle theory and K-cohomology invariants}, 
Lecture Notes in Physics, vol. {\bf 726}, Springer-Verlag, Berlin Heidelberg, 2008.

\bibitem[JoSt91]{JoSt91}
Joyal, A. and Ross Street
\emph{An Introduction to Tannaka Duality and Quantum Groups}
Part II of "Category Theory, Proceedings, Como 1990", Lecture Notes in Math {\bf 1488}, 411-492, (Springer-Verlag Berlin, Heidelberg 1991).

\bibitem[Kna02]{Kna02}
Knapp, A.W.,
\emph{Lie Groups Beyond an Introduction},
Birkh\"auser, Boston, Mass., 2002.


\bibitem[Ko70]{Ko70}
   Kostant, B.,
\emph{Quantization and Unitary Representations}, 
Modern Analysis and Appl.3, Lecture Notes in Math., {\bf 170} Springer-Verlag, (1970), 87-208.

\bibitem[KoNo63]{KoNo63}
Kobayashi, S., and K. Nomizu
\emph{Foundations of Differential Geometry vol. 1}, 
Interscience Tracts in Pure and Applied Mathematics, Wiley, 1963.

\bibitem[KoNo69]{KoNo69}
Kobayashi, S., and K. Nomizu
\emph{Foundations of Differential Geometry vol. 2}, 
Interscience Tracts in Pure and Applied Mathematics, Wiley, 1969.

\bibitem[Ky87]{Ky87}
Kyriazis, A.,
\emph{Tensor Products of Function Algebras}, 
Bull. Austral. Math. Soc., {\bf 36}, 417-423, (1987).

\bibitem[La02]{La02}
    Lang, S.,  
\emph{Algebra}, 
Graduate Texts in Mathematics, vol. {\bf 211}, Springer-Verlag, 2002.

\bibitem[LaSu04]{LaSu04}
   Landi, G. and W. van Suijlekom
\emph{Principal Fibrations from Noncommutative Spheres}, 
Communications in Mathematical Physics {\bf 260}, 203-225, (2004).

\bibitem[Lec85]{Lec85}
   Lecomte, P.,  
\emph{Sur la suite exacte canonique associ\'{e}e \`{a} un fibr\'{e} principal}, 
Bull. Soc. Math. Fr. {\bf 13} (1985), 259-271.


\bibitem[MaMa92]{MaMa92}
Marathe, K.B. and G. Martucci
\emph{The Mathematical Foundations of Gauge Theories}, 
North-Holland, 1992.

\bibitem[Ma95]{Ma95}
  Mac Lane, S.,  
\emph{Homology}, Classics in Mathematics, Springer-Verlag, 1995.

\bibitem[Mas08]{Mas08}
  Masson, T.,  
\emph{Noncommutative generalization of $\SU(n)$-principal fibre bundles: a review}, 
Journal of Physics: Conference Series {\bf 103} (2008) 012003.


\bibitem[Mey95]{Mey95}
Meyer, U.,
\emph{Projective Quantum Spaces}, 
Lett. Math. Phys. {\bf 35} (1995), 91.

\bibitem[Mic04]{Mic04}
 Michor, P.,
\emph{Topics in Differential Geometry}, 
Lecture Notes, Wien, 2004.

\bibitem[Mil83]{Mil83}
 Milnor, J.,
\emph{Remarks on infinite-dimensional Lie groups}, 
Proc. Summer School on Quantum Gravity, B. DeWitt ed., Les Houches, 1983.

\bibitem[NaSa03]{NaSa03}
 Navarro Gonz\'alez, J.A. and J.B. Sancho de Salas
\emph{$C^{\infty}$-Differentiable Spaces}, 
Lecture Notes in Mathematics, {\bf 1824}, Springer-Verlag, 2004.


\bibitem[Ne04]{Ne04}
Neeb, K.-H., 
\emph{Abelian extensions of infinite-dimensional Lie groups}, 
Travaux mathmatiques Fascicule {\bf XV} (2004), 
69-194.


\bibitem[Ne07a]{Ne07a}
 \textemdash, 
\emph{Non-abelian extensions of infinite-dimensional Lie-groups}, 
Annales de l'Inst. Fourier {\bf 57} (2007), 209-271.

\bibitem[Ne07b]{Ne07b}
 \textemdash, 
\emph{On the classification of rational quantum tori and the structure of their automorphism groups}, 
Canadian Mathematical Bulletin {\bf 51} (2008), 261-282.

\bibitem[Ne08a]{Ne08a}
 \textemdash, 
\emph{Lie group extensions associated to projective modules of continuous inverse algebras}, 
Archivum Mathematicum (BRNO) Tomus {\bf 44} (2008), 465–489.

\bibitem[Ne08b]{Ne08b}
\textemdash, 
\emph{Differential Topology of Fiber Bundles}, 
Lecture Notes, Darmstadt, 2008.

\bibitem[Ne09]{Ne09}
\textemdash, 
\emph{An Introduction to Unitary Representations of Lie Groups}, 
Lecture Notes, Darmstadt, 2009.

\bibitem[NeSe09]{NeSe09}
Neeb, K.-H. and H. Sepp\"anen
\emph{Borel-Weil Theory for Groups over Commutative Banach Algebras}, 
Journal f\"ur die reine und angewandte Mathematik (Crelles Journal), 2009.

\bibitem[NeWa07]{NeWa07}
Neeb, K.-H. and F. Wagemann
\emph{Lie group structures on groups of smooth and holomorphic maps on non-compact manifolds}, 
Geometriae Dedicata Volume {\bf 134} (2007), Number 1, 17-6.

\bibitem[OP95]{OP95}
Osborn, J. M., and D. S. Passman, 
\emph{Derivations of skew polynomial rings}, 
J. Algebra {\bf 176} (1995), 417-448.

\bibitem[Ou07]{Ou07}
Oubbi, L.,
\emph{Characters on Algebras of Vector-Valued Continuous Functions}, 
Rocky Mountain Journal of Mathematics {\bf 37}, Number 3, 2007.


\bibitem[RaWi98]{RaWi98}
Raeburn, I. and Dana Williams
\emph{Morita Equivalence and Continuous-Trace C*-Algebras}, 
Mathematical Surveys and Monographs {\bf 60}, Amer. Math. Soc., 1998.

\bibitem[ReVa08]{RaVa08}
Rennie, A. and J.C. V\'{a}rilly
\emph{Reconstruction of Manifolds in Noncommutative Geometry}, 
arXiv:math/0610418v4 [math.OA], 4 Feb 2008.

\bibitem[Ru73]{Ru73}
Rudin, W.,
\emph{Functional Analysis}, 
McGraw Hill, 1973.

\bibitem[Sch99]{Sch99}
Schaefer, H.H.,
\emph{Topological Vector Spaces}, 
Graduate Texts in Math., vol. {\bf 3}, Second Edition, Springer-Verlag, New York, 1999.

\bibitem[St51]{St51}
Steenrod, N.,
\emph{The Topology of Fibre Bundles}, 
Princeton University Press, 1951, Princeton, New Jersey.

\bibitem[Swa62]{Swa62}
Swan, R. G.,
\emph{Vector bundles and projective modules}, 
Trans. Amer. Math. Soc. {\bf 105} (1962), 264-277.

\bibitem[Ta74]{Ta74}
Takai, H.,
\emph{On a Duality for Crossed Products of $C^*$-Algebras}, 
Journal of Functional Analysis {\bf 19}, 25-39 (1975).

\bibitem[Tak02]{Tak02}
Takesaki, M.,
\emph{Theory of Operator Algebra I}, 
Encyclopaedia of Mathematical Sciences, Operator Algebras and Non-Commutative Geometry, Springer-Verlag, Berlin, 2002.

\bibitem[Tre67]{Tre67}
Treves, F.,
\emph{Topological Vector Spaces, Distributions and Kernels}, 
Pure and Applied Mathematics, Academic Press, 1967.


\bibitem[To00]{To00}
Tom Dieck, T.,
\emph{Topologie}, de Gruyter,
2. Auflage, 2000.

\bibitem[Va06]{Va06}
  V\'{a}rilly, J. C.,
\emph{An Introduction to Noncommutative Geometry},
EMS Series of Lectures in Mathematics, 2006

\bibitem[Wo07]{Wo07}
 Wockel, C.,
\emph{Lie group structures on symmetry groups of principal bundles}, 
Journal of Functional Analysis {\bf 251}, (2007), 254-288.




\bibitem[Yo49]{Yo49}
  Yoshizawa, H.,  
\emph{Unitary representations of locally compact groups}, 
Osaka Mathematical Journal, Vol. 1, No. 1, 1949, 81-89.

\end{thebibliography}
\end{document}